\documentclass[12pt]{article}
\usepackage{amsmath,amsthm,amsfonts,amssymb,verbatim}
\usepackage[usenames]{color} %added for color

\newtheorem{thm}{Theorem}[section]
\newtheorem{lemma}[thm]{Lemma}
\newtheorem{cor}[thm]{Corollary}

\newtheorem{prop}[thm]{Proposition}
\newtheorem{question}[thm]{Question}

\usepackage{pdfsync}
\usepackage{graphicx}

\theoremstyle{definition}  % comment this out to revert

\numberwithin{equation}{section}
\newtheorem{defn}[thm]{Definition}

\newtheorem{defns}[thm]{Definitions}
\newtheorem{notn}[thm]{Notation}
\newtheorem{snotn}[thm]{Standing Notation}

\newtheorem{exs}[thm]{Examples}
\newtheorem{remark}[thm]{Remark}

\theoremstyle{definition}
\theoremstyle{remark}%%%%%%%%%%%%%%%%%%%%%%

\newcounter{enumitemp}
\newenvironment{enumeratecontinue}{
 \setcounter{enumitemp}{\value{enumi}}
 \begin{enumerate}
 \setcounter{enumi}{\value{enumitemp}}
}
{
 \end{enumerate}
}

\newcommand\pref[1]{(\ref{#1})}

\DeclareMathOperator{\Fix}{Fix}
\DeclareMathOperator{\orb}{orb}

\DeclareMathOperator{\Diff}{Diff}

\DeclareMathOperator{\mcg}{MCG}

\DeclareMathOperator{\Int}{int}
\DeclareMathOperator{\Per}{Per}

\DeclareMathOperator{\diam}{diam}

\DeclareMathOperator{\fr}{fr}

\DeclareMathOperator{\cl}{cl}

\DeclareMathOperator{\dist}{dist}
\DeclareMathOperator{\Stab}{Stab}
\DeclareMathOperator{\core}{ core}

\newcommand{\cA}{{\cal A}}
\newcommand{\B}{{\cal B}}
\newcommand{\C}{{\cal C}}

\newcommand{\U}{{\cal U}}
\newcommand{\R}{\mathbb R}
\newcommand{\E}{{\cal E}}

\newcommand{\cP}{{\cal P}}
\newcommand{\cR}{{\cal R}}
\newcommand{\cS}{{\cal S}}
\newcommand{\cT}{{\cal T}}

\newcommand{\T}{\mathbb T}
\newcommand{\W}{{\cal W}}
\newcommand{\Z}{\mathbb Z}

\def\RH{RH(W,\partial_+ W)}

\def\V{{\cal V}}

\def\tiM{H}
\def\ti{\tilde}
\def\sinfty{S_{\infty}}
\def\sl3z{SL(3, \mathbb Z)}

\def\G{{\cal G}}

\def\L{{\cal L}}
\def\M{{\cal M}}

\def\cS{{\cal S}}

\def\RH{RH(W,\partial_+ W)}

\newcommand{\mrA}{\mathring A} 

\newcommand{\newDiff}{\Diff_{\mu}(S^2,P)}
\newcommand{\newDiffA}{\Diff_{\mu}(A^2,P')}

\title{Entropy zero area preserving diffeomorphisms of $S^2$}
\author{John Franks,\thanks{Supported in part by NSF grant DMS0099640.}\ \ 
Michael Handel\thanks{Supported in part by NSF grant DMS0103435.}}

\begin{document}
\maketitle
\begin{abstract}
In this paper we formulate and prove a structure theorem for area
preserving diffeomorphisms of genus zero surfaces with zero entropy and at least
three periodic points.  As one application we relate the existence of
faithful actions of a finite index subgroup of the mapping class group
of a closed surface $\Sigma_g$ on $S^2$ by area preserving
diffeomorphisms to the existence of finite index subgroups of bounded
mapping class groups $\mcg(S, \partial S)$ with non-trivial first
cohomology. In another application we show that rotation number is defined
and continuous at every point of a zero entropy area preserving diffeomorphism  of the annulus.
\end{abstract}

\section{Introduction and Statement of Results} 

%For any smooth area form $\mu$ on $S^2$, let $\Diff_{\mu}(S^2)$ be the group of $C^\infty$ diffeomorphisms of $S^2$ that preserve $\mu$.  Elements of $\Diff_{\mu}(S^2)$ are said to {\em preserve area} and all such elements preserve the orientation on $S^2$.  
Surface diffeomorphisms with
positive entropy have been studied from both the hyperbolic dynamical
systems point of view and the Nielsen-Thurston point of view.  In this
paper we formulate and prove a structure theorem for %$F\in \Diff_{\mu}(S^2)$ 
area preserving diffeomorphisms of genus zero surfaces with zero entropy. The area preserving assumption is a natural one arising in many dynamical systems and it is an essential ingredient for most of the dynamical structure we investigate here. The genus zero assumption is made to simplify the problem.  There should be a similar theory for higher genus and much of what we show here may well be true for Hamiltonian diffeomorphisms in higher genus. 

If $N$ is a genus zero surface with finitely many boundary components and $F': N \to N$ is a diffeomorphism, then collapsing each component of $\partial N$ to a point produces a homeomorphism $F : S^2 \to S^2$  which restricts to
diffeomorphism on the complement of a finite set.
 For almost all of our analysis we can work directly with $F$ instead of $F'$ and can even forget that $F'$ is smooth but there are two (very important) steps  (see Section~\ref{sec: normal form}  and Lemma~\ref{no doubling}) when we must remember $F'$ and make use of its smoothness.  With this in mind we  make the following definitions.
  
Let $\mu$ be  a measure  on $S^2$ that is topologically conjugate to Lebesgue measure.  A homeomorphism that preserves $\mu$ is said to {\em preserve area}.
Let  $P \subset S^2$ be  a (possibly empty) finite  set and  let $N$ be the  genus zero surface obtained from $S^2$ by blowing up each element of $P$ to a boundary circle.  Inverting this process produces a  quotient map  $\pi_P : N \to S^2$  that restricts to a diffeomorphism from $\Int N$ to $S^2 \setminus P$ and that maps each component of $\partial N$ to an element of $P$.

 Define  $\newDiff$ to be the group of  {\em orientation preserving homeomorphisms of $S^2$ that preserve $\mu$}, that  fix each element of $P$ and for which there is a   $C^\infty$ diffeomorphism   $F':N \to N$ such that $F\pi_P = \pi_P F'$.  Note that if $P = \emptyset$ and $\mu$ is a smooth volume form, then $\newDiff$ is just the group $\Diff_\mu(S^2)$ of $C^\infty$ diffeomorphisms of $S^2$ which preserve $\mu.$

There are certain elements of $\newDiff$ which are trivial from
the point of view of their periodic points.  These include $F \in
\newDiff$ of finite order and $F$ for which $\Per(F)$ contains
only two points.  It is known that an area preserving $F$ must have at
least two fixed points (see \cite{simon:index}).  In the case that
$\Per(F)$ contains exactly two points those points must be fixed.  Blowing up the fixed points as above produces a homeomorphism $F'$ of the closed annulus with every point  having the same irrational
rotation number (see Theorem~\ref{thm: translation interval}). 
This is an interesting topic to
investigate but is not addressed in this article. For the
remainder of this paper we make the following:
\vspace{.1in}

\noindent {\bf Standing Hypothesis:} {\em Assume that $F\in \newDiff$ has infinite order and entropy zero and
that $\Per(F)$ contains at least three points.}

\vspace{.1in}

Suppose that $F \in \newDiff)$ has zero topological entropy and
that $\Fix(F)$ is the set of fixed points for $F$.  To enhance the
topology of the ambient surface, we consider $\M = S^2 \setminus
\Fix(F)$ and $f=F|_{\M} :\M \to \M$.

Disks in this paper are topological objects; they are not assumed to
be round. Every $x \in \M$ has a neighborhood $B$ that is a {\em free
  disk}, meaning that $B$ is an open disk and that $f(B) \cap B
=\emptyset$.  A very weak notion of recurrence for a point $x \in \M$
is to require that there be $n \ne 0$ and a free disk $B$ that
contains both $x$ and $f^n(x).$ We will call such points {\em free
  disk recurrent} and denote the set of these points by $\W_0.$ Each
periodic point is free disk recurrent; a non-periodic $x$ is free disk
recurrent if and only if there is a free disk $B$ which intersects
 the orbit of $x$ in at least two points.  Clearly, if
either the $\alpha$-limit set $\alpha(F,x)$ or the $\omega$-limit set $\omega(F,x)$ contains a point which is not in
$\Fix(F)$ then $x \in \W_0.$ In particular the set $\W_0$ contains the
full measure subset of $\M$ consisting of birecurrent points. The set
$\W_0$ is open and dense in $\M$. It is technically useful to work
with sets that equal the interior of their closure so we define the
larger set $\W$ of {\em weakly free disk recurrent} points as follows.
(We expect that $\W_0 \ne \W$ in general but have not worked out a
specific example.)

For sets $A \subset X$ we denote the {\em interior} of $A$ with
respect to $X$ by $\Int_X(A)$ and the {\em closure} of $A$ with
respect to $X$ by $\cl_X(A)$.  If $X$ is understood then we drop it
from the notation and simply write $\Int(A)$ and $\cl(A)$.

\begin{defn}\label{defn: omega-free}
A point $x \in \M$ is {\em free disk recurrent} for $f$ provided there
exists $n \ne 0$ and a free disk $B$ that contains both $x$ and
$f^n(x).$ The set of free disk recurrent points in $\M$ is denoted
$\W_0.$ If $W_0$ is a component of $\W_0$ and $x \in \M$ is in
$\Int_{\M}(\cl_{\M}(W_0))$, then we say that $x$ is {\em weakly free disk
  recurrent}. The set of weakly free disk recurrent points in $\M$ is
denoted
$\W$. 
\end{defn}

%There are several immediate consequences of the definition which we record for later use.

%\begin{rem}\label{rem: free disk} The set $\W \subset \M$ is open, dense, contains all forward or backward recurrent points for $f$.  Each of its components  is equal to the interior of its closure.  The complement of $\W$ is a closed set with measure zero. If any point of $\omega(x,F)$ or $\alpha(x,F)$ is not in $\Fix(F)$ then the orbit of $x$ returns to a free disk neighborhood of that point infinitely often so $x$ is free disk recurrent.  

The main building block in our structure theorem is a partition of
$\W$ into countably many disjoint $f$-invariant annuli.

\begin{thm} \label{thm: annuli}Suppose $F \in \newDiff$ has
entropy zero, infinite order and at least three periodic points.
Let $f = F|_{\M}$ where $\M = S^2
\setminus \Fix(F)$.   Then there is a countable collection $\cA$ of pairwise disjoint
open $f$-invariant annuli such that 
\begin{enumerate}
\item  $\U = \bigcup_{U \in \cA} U$ is the set $\W$  of weakly free disk 
recurrent points for $f$. 
\item $\cA$ is the set of maximal $f$-invariant open annuli in $\M$.
\item   If $ z \not \in \U  $, there
are components $F_+(z)$ and $F_-(z)$ of $\Fix(F)$ so that $\omega(F,z)
\subset F_+(z) $ and $\alpha(F,z) \subset F_-(z)$.
\item For each $U \in \cA$ and each component $C_{\M}$ of the frontier of
$U$ in $\M$, $F_+(z)$ and $F_-(z)$ are independent of the choice of $z
\in C_{\M}$. 
\end{enumerate}
\end{thm} 

We emphasize the fact that replacing $F$  by an iterate $F^q$  changes   $\M$ and hence changes the annuli of Theorem~\ref{thm: annuli}.

\begin{remark}\label{first centralizer invariance} If $h:S^2 \to S^2$
commutes with $F$ then it preserves $\W$ and hence permutes the open
annuli in the family $\cA.$
\end{remark}

%\begin{remark} We prove in Proposition~(\ref{prop: V in U}) that an open annulus in $M$ is an element of $\cA$ if and only if it is $f$-invariant and is not properly contained in any $f$-invariant open annulus.
%\end{remark}

To see how the elements of $\cA$ arise, consider the special case that
$F$ is the time one map of an area preserving flow $\phi_t$.  Given $x
\in \M,$ choose a free disk neighborhood $B$ of $x$ which is also a
flow box for $\phi_t$.  It is an easy consequence of the
Poincare-Bendixson theorem that if the flow line for $\phi_t$ that
contains $x$ returns to $B$ it closes up into a simple closed curve
$\rho_x$.  In particular, in this case the subsets $\W_0$ and $\W$ are
equal and coincide with the union of the periodic orbits of the flow
which lie in $\M$.  Denote {\em the isotopy class of $\rho_x$} in $\M$
by $[\rho_x]$.  It is clear that $\rho_x$ depends only on the orbit of
$x$ and not $x$ itself and that if $z \in B$ is sufficiently close to
$x$ then $\rho_x$ and $\rho_z$ cobound an annulus in $\M$; in
particular $[\rho_x] =[ \rho_z]$.  In this case $U=\{ y \in \W:
[\rho_y] =[ \rho_x]\}$ is the element of $\cA$ that contains $x$.

For a second special case suppose that $f$ is isotopic to the
identity.  Given $x \in \W_0$, choose $B$ and $n$ as in the definition of
  free disk  recurrent.  If $f_t:\M \to \M$ is an isotopy between $f_0 = $
identity and $f_1 = f$ then the path $\mu_x \subset \M$ defined by
$\mu_x(t) = f_t(x)$ connects $x$ to $f(x)$.  The path $\mu_x \cdot
\mu_{f(x)} \cdot \ldots \cdot \mu_{f^{n-1}(x)}$ can be closed by
adding a path in $B$ connecting $f^n(x)$ to $x$.  Up to homotopy in
$\M$, this closed path is a multiple of some non-repeating closed path
$\rho_x$.  Using the hypothesis that $F$ has entropy zero, one can
show (see \cite{fh:periodic}) that the homotopy class of $\rho_x$ is
represented by a simple closed curve (also written $\rho_x$) that is
independent of $B, n$ and the choice of isotopy $f_t$.  It is easy to see that if
$z \in B$ is sufficiently close to $x$ then $[\rho_x] =[ \rho_z]$.  As
in the previous case, $U=\{ y \in \W: [\rho_y] =[ \rho_x]\}$ is the element of $\cA$ that contains $x$.

In the general case, we make use of the fact (see section~\ref{sec: normal form})  that $f$ is isotopic
to a composition of Dehn twists along a finite set of simple closed
curves $\cR$. 
Cutting along the elements of $\cR$ produces a
decomposition of $\M$ into subsurfaces $\M_i$ such that $f|_{\M_i} : \M_i
\to \M$ is isotopic to the inclusion $\M_i\hookrightarrow \M$.  The main
technical work in this proof is showing that each $\M_i$ is realized,
in a suitable sense, by an $f$-invariant subsurface; see 
 section~\ref{sec:domain covers}.   One then defines $\cA$
in a fashion similar to the second special case.

\vspace{.1in}

Theorem~\ref{thm: annuli} can be applied to $F^q$ for each   $q\ge 2$.  
This gives a countable collection $\cA(q)$ of pairwise disjoint open  $F^q$-invariant annuli that (see Proposition~\ref{prop: V in U}) refines  $\cA$ in the sense that   each $V_j \in \cA(q)$ is contained in some $U_i \in \cA$.  This {\em renormalization process}
can be iterated with  $\cA(q)$  playing the role of $\cA$ and
so on. The $V_j$'s may be essential or inessential in $U_i$.  In the
limit, the former lead to twist-map-like behavior and the latter to
solenoid-like behavior when they are nested infinitely often. 
It is important to note that replacing
$F$ with $F^q$ changes the set of fixed points and hence changes $\M$
and changes the free disk recurrent points of $\M$.
 
We are interested in partitioning $cl(U)$ into sets
analogous to the periodic orbits in the case of the time one map of a
flow.  In particular we would like the rotation number to be constant
on these sets.  The two components of the frontier of $U$ can be
somewhat problematic since such a component could be a single point or
could be a complicated fractal. To deal with this issue we introduce the
{\em annular compactification } $f_c: U_c \to U_c$ of
$f: U \to U$; see Notation~\ref{notn: annular comp} and the paragraph preceding it. The compactification of an end described there is either the
prime end compactification or the compactification obtained by blowing
up a fixed point, whichever is appropriate.

We are now prepared to state the second of our main results. It
describes the finer structure of the dynamics of $f$ on one of the
annuli in $\cA$. The proof is based on renormalization and the details
are in Section~\ref{sec: renormalization}.

\begin{thm}\label{thm: C(x)}
Suppose $F \in \newDiff$ has entropy zero, has
infinite order and at least three periodic points. Let $f =
F|_{\M}$ where $\M = S^2 \setminus \Fix(F)$ and let $\cA$ be
as in Theorem~\ref{thm: annuli}.  For $U \in \cA$, let $f_c
:U_c \to U_c$ be the annular compactification of $f|_U:U \to
U$.  Then
\begin{enumerate}
\item   The rotation number $\rho_{f_c}(x)$  is defined and continuous at every  $x \in U_c$.
\item If $\Fix(F)$ contains at least three points then
$\rho_{f_c}$ is non-constant.
\item  {\em If $C$ is a component of a level set of $\rho_{f_c}$ 
then $C$ is $F$-invariant.  If $C$ does not contain a component
of $\partial U_c$ then it is essential in $U$, meaning that $U_c \setminus C$   has two components each containing a component of $\partial U_c$. }
 \end{enumerate}
\end{thm}

The components $C$ of the level sets of $\rho_{f_c}$ in
Theorem~\ref{thm: C(x)} are the generalizations of the
closed orbits foliating $U$ in the special case that $F$ is
the time one map of a flow.  Of course in the general case
$C $ can be considerably more complicated.  The main example
constructed in \cite{han:pseudo-circle} shows $C$ can be a
pseudo-circle.  It is also possible for $C$ to have
interior.

A heuristic picture of one possibility 
in the  case that $\rho_f|_C$ is rational is an essential  ``necklace'' 
in $U$ consisting of a periodic orbit of  saddle periodic points
each joined to the next by a stable manifold (which is the unstable
manifold of the next) and by an unstable manifold (which is the stable
manifold of the next). This pair, stable and unstable, bound a ``bead'',
an open disk.  The diffeomorphism $f$ permutes the beads and has 
a periodic orbit with one point in each bead.  The set $C$ containing any
$x$ in one of the beads will be the entire necklace.  For such a $C$
there is an $n$ such that $f^n$ will fix each bead and each saddle
point joining them.

Our first application concerns area preserving diffeomorphisms of the closed annulus $A$.   For expected future applications, we state our theorem in a more general context  and  then state the annulus result as a corollary.    

Suppose that $P$ has two preferred elements $p_1,p_2$ and that  $P' = P \setminus \{p_1,p_2\}$.  
  If $H : A \to A$ is the homeomorphism of the closed annulus obtained from some $F\in \newDiff$ by blowing up $p_1$ and $p_2$ then we write $H \in \newDiffA$.   Note that if $P = \{p_1,p_2\}$ then $\newDiffA$ is the group of area preserving $C^\infty$ diffeomorphism of the closed annulus $A$.  
  
  \begin{thm}\label{annulus with singularities}  For each $H \in \newDiffA$ with entropy zero, the rotation number $\rho_H(x)$ is defined and continuous at each $x \in A$. 
\end{thm}

\begin{cor} \label{annulus}  Suppose that $H:A \to A$ is an area preserving $C^\infty$ diffeomorphism of the closed annulus $A$.  If $H $ has entropy zero   then the rotation number $\rho_H(x)$ is defined and continuous at each $x \in A$. 
\end{cor}

%Applying our result In the special case that $P$ has cardinality two, we have the following theorem.

%\jmfc{I think we will in the future want to quote this theorem with a finite number of non-smooth points.  We should probably state it as a theorem about $\Diff_\mu(A, P).$ Or perhaps as a theorem about a genus zero surface with two distinguished boundary components used to define $A$.  I would be ok with a  Remark explaining this.}

%\begin{thm} \label{annulus theorem}  Suppose that $F:A \to A$ is an area preserving $C^\infty$ diffeomorphism of the closed annulus $A$.  If $F $ has entropy zero   then the rotation number $\rho_F(x)$ is defined and continuous at each $x \in A$. 
%\end{thm}

For our next application, recall
 that a group $G$ is   {\em indicable} if there exists a
non-trivial homomorphism $G \to \Z$.  For finitely generated groups
this is equivalent to $H^1(G, \Z) \ne 0$ and equivalent to the
abelianization of $G$ being infinite.  If a finite index subgroup of
$G$ is indicable then we say that $G$ is {\em virtually indicable.}

For $F\in \Diff_{\mu}(S^2)$, denote the centralizer of $F$ in
$\Diff_{\mu}(S^2)$ by $Z(F)$.  As an application of Theorem~\ref{thm:
annuli} and Theorem~\ref{thm: C(x)} we prove

\begin{thm}\label{thm: ent0}
  If $F \in \Diff_{\mu}(S^2)$ has infinite order then each finitely
  generated infinite subgroup $H$ of $Z(F)$ is virtually indicable.
\end{thm}
 
One might expect that Theorem~\ref{thm: ent0} is proved by first
proving the existence of a finite index subgroup $H_0$ of $H$ with
global fixed points and then applying the Thurston stability theorem(
\cite{thurston:stability}; see also Theorem 3.4 of
\cite{franks:distortion}) to produce a non-trivial homomorphism from
$H_0$ to $\Z$.  This is easy to do (see Proposition~\ref{prop: +ent})
in the case that $F$ has
positive entropy but fails when $F$ has zero entropy.  Indeed, there
are examples (see Examples~\ref{example}) for which no finite index subgroup of $Z(F)$ has a
global fixed point.  We prove Theorem~\ref{thm: ent0} by analyzing the
possible ways in which the existence of global fixed points can fail
and by showing that each allows one to define a non-trivial
homomorphism to $\Z$.
 
As an application of  Theorem~\ref{thm: ent0} we have the following 
result about mapping class groups.  

\begin{cor}  \label{no actions} If $\Sigma_g$ is the closed orientable surface of genus $g \ge 2$ then at least one of the following holds.
\begin{enumerate}
\item   \label{no faithful actions} No finite index subgroup of $\mcg(\Sigma_g)$ acts faithfully on $S^2$ by area preserving diffeomorphisms.
\item  \label{non-trivial first cohomology}For all \  $1 \le k \le g-1$,  there is an indicable finite index subgroup $\Gamma$ of  the bounded mapping class group $\mcg(S_k,\partial S_k)$ where $S_k$ is the  surface  with genus $k$ and connected non-empty boundary.
\end{enumerate}
\end{cor}

Corollary~\ref{no actions} relates to the following well known questions about mapping class groups.

 \begin{question}  \label{mcg question} Does $\mcg(\Sigma_g)$, or any of its finite index subgroups,  act faithfully on  a closed surface $S$ by diffeomorphisms?  by area preserving diffeomorphisms?  
 \end{question}

\begin{question}\label{cohomology}     Does every finite index subgroup  $\Gamma$ of $\mcg(\Sigma_g)$ satisfy $H^1(\Gamma, \R) = 0$?
\end{question}

 Question~\ref{mcg question}  is motivated in part by the sections problem (see Problem~6.5 and Question~6.7) of  Farb's survey/problem list \cite{farb:survey} on the mapping class group: which subgroups of  $\mcg(\Sigma_g)$ lift to $\Diff(\Sigma_g)$?  It is also motivated by the analogy between mapping class groups and higher rank lattices and the fact   (\cite{Po},\cite{fh:periodic}, \cite{fh:distortion})  that every action of a non-uniform irreducible  higher rank lattice on   $\Sigma_g$ by area preserving diffeomorphisms factors through a finite group; see Question~12.4 of  Fisher's survey article \cite{fisher:survey} on the Zimmer program. 
 
Question~\ref{cohomology}  is  Problem 2.11 of \cite{kirby};  
see also \cite{Iv2} and \cite{kork2}.   
Corollary~\ref{no actions} \pref{no faithful actions} is a negative
answer to the area preserving, $S = S^2$ case of Question~\ref{mcg
question}.  The answer to Question~\ref{cohomology} is no for genus
$2$ (see \cite{mcC}) but is unknown for genus
at least three.  Presumably a positive answer to
Question~\ref{cohomology}  for genus greater than $3$ would imply that Corollary~\ref{no
actions} \pref{non-trivial first cohomology} does not hold and so
imply that Corollary~\ref{no actions} \pref{no faithful actions} does
hold.

We are grateful to the referee for many very helpful suggestions.

\section{Area preserving  annulus maps.} \label{sec: annulus maps}

We will make use of a number of results on area preserving 
homeomorphisms and diffeomorphisms of the annulus which we cite here.

If $A = S^1 \times [0,1]$ is the annulus, its universal
covering space is $\ti A = \R \times [0,1].$  We will
denote by $p_1$ \  the projection, $p_1 : \R \times [0,1] \to \R$,
of $\ti A$ onto its first factor.

\begin{defn} \label{translation number}
If $f: A \to A$ is an orientation preserving homeomorphism
isotopic to the identity and $\ti f$ is a lift to $\ti A$
then the {\em forward translation interval} $\cT^+_{\ti f}(\ti x)$ of $\ti x \in \ti A$
is defined to be $[a,b]$ where
\begin{align*}
a &=  \liminf_{n \to \infty} 
\frac{p_1(\ti f^n(\ti x)) - p_1(\ti x)}{n} \text{\ \ and}\\
b &= \limsup_{n \to \infty} 
\frac{p_1(\ti f^n(\ti x)) - p_1(\ti x)}{n}.
\end{align*}
If $a = b$ then $\tau^+_{\ti f}(\ti x) =a$
is called the {\em forward translation number} 
of $\ti x \in \ti A$ and 
\[
\tau^+_{\ti f}(\ti x) = \lim_{n \to \infty} 
\frac{p_1(\ti f^n(\ti x)) - p_1(\ti x)}{n}.
\]
The backward translation interval and number
$\cT^-_{\ti f}(\ti x)$ and $\tau^-_{\ti f}(\ti x)$ are defined
analogously.   If $\tau^+_{\ti f}(\ti x)$ and $\tau^-_{\ti f}(\ti x)$ are both defined and if  $\tau^-_{\ti f}(\ti x) =  - \tau^+_{\ti f}(\ti x)$ then we say that $\tau^+_{\ti f}(\ti x)$ is the {\em translation number of $\ti x$} and denote this number by $\tau_{\ti f}(\ti x)$.     All of these definitions are independent of the choice of lift $\ti x$ of $x$ and so may be viewed as functions of $x$.  

The {\em forward rotation interval} $\cR^+_{f}(x)$ 
and {\em forward rotation number}  $\rho^+_f(x)$ of
$x \in A$ are defined to be the projection of
$\cT^+_{\ti f}(x)$ and
$\tau^+_{\ti f}(x)$ respectively in $\T^1 = \R / \Z$.   As the notation suggests, they are independent of the choice of lift $\ti f$ of $f$.    Backward rotation interval, backward rotation number and rotation number are defined  and denoted similarly.  
%The {\em rotation interval} $\cR_{f}(x)$  and {\em rotation number}  $\rho_f(x)$ of $x \in A$ are defined to be the projection of $\cT_{\ti f}(\ti x)$ and $\tau_{\ti f}(\ti x)$ respectively in $\T^1 = \R / \Z$ for any lift $\ti f$ of $f$ and any lift $\ti x$ of $x$ with a similar notational convention (e.g.  $\rho_{f}^+(x)$ denotes the forward rotation number).
\end{defn}

\begin{lemma}  \label{translation numbers exist} Suppose that $f: A \to A$ is an area preserving homeomorphism of the closed annulus which is isotopic to the identity and that  $\ti f: \ti A \to \ti A$ is a lift to its universal covering space. Then $\tau_{\ti f}(\ti x)$ exists for almost all $\ti x \in \ti A.$
\end{lemma}

\proof  This is a standard consequence of the Birkhoff ergodic theorem applied to the function $\phi(x) = p_1(\ti f(\ti x)) - p_1(\ti x)).$
\endproof

The closed interval $\cT(\ti f)$ of the following lemma is called the {\em translation interval of $\ti f$}.  It's projected image  $\cR(f)$ in $\T^1 = \R / \Z$ is called the {\em rotation interval of $f$}.     

\begin{thm}\label{thm: translation interval}  Suppose that $f: A \to A$ is an area preserving homeomorphism of the closed annulus which is isotopic to the identity and that  $\ti f: \ti A \to \ti A$ is a lift to its universal covering space.  Then there is a closed interval  $\cT(\ti f)$   with the following properties.
\begin{enumerate}
\item For each $r\in \cT(\ti f)$  there exists $\ti x \in \ti A$ such that  $\tau_{\ti f}(\ti x) = r$;  if $r =p/q$ is a rational number in lowest terms then one may choose $\ti x$ to be a lift of a periodic point with period $q$.
\item For all $\ti x \in \ti A$,\  \  $\cT^+_{\ti f}(\ti x) \subset \cT(\ti f)$ and $-\cT^-_{\ti f}(\ti x) \subset \cT(\ti f)$.
\end{enumerate}
\end{thm}

\proof    Define $\cT(\ti f)$ to be the set of $r \in \R$ for which there exists $\ti x \in \ti A$ with $\tau_{\ti f}(\ti x) = r$.   Theorem~0.1 from \cite{handel:rotation} implies that $\cT(\ti f)$ is closed.   Suppose that  $r_i \in \cT^+_{\ti f}(\ti x_i)$   for $i=1,2$ and some $\ti x_i  \in \ti A$.   Corollary~2.4  of \cite{franks:recurrence}  implies that  for any
rational in lowest terms $p/q \in [r_1,r_2]$ there is a periodic point $x$ for $f$ with period $q$ and a lift   
$\ti x \in \ti A$ such
that $\tau_{\ti f}(\ti z) = p/q$.  Item (1) and the $\cT^+_{\ti f}(\ti x) \subset \cT(\ti f)$ part of (2) follow immediately.  The symmetric argument with $\ti f$ replaced by $\ti f^{-1}$ proves the $\cT^-_{\ti f}(\ti x) \subset \cT(\ti f)$ part of (2).
\endproof

\begin{prop}\label{prop: +meas}
Suppose $f: A \to A$ is an area preserving homeomorphism
of the closed annulus which is isotopic to the identity.
If there is a subset $Y \subset A$ with Lebesgue
measure $\mu(Y) > 0$ and such that $\rho^+_f(x) = 0$ for almost all
$x \in Y$ then $f$ has a fixed point in the interior of $A$.
\end{prop}
\begin{proof}
By the Birkhoff ergodic theorem $\rho^+_f(x) = \rho^-_f(x)$ for
almost all points of $A$, hence we may assume $\rho_f(x) = 0$
for almost all $x \in Y$.
Since $\mu(Y) >0$ there is a small open disk $D$ 
whose closure is in the interior of $A$ 
with $\mu( Y \cap D) >0.$  If $f$ has no fixed point in
$D$ then by making $D$ smaller we may assume it is a free disk.We let $X = Y \cap D$.
Let $r: X \to X$ be the first return map so $r(x) = f^n(x)$ where 
$n$ is the smallest positive integer such that $f^n(x) \in X$.
The function $r$ is well defined for almost all $x \in X$, so
deleting a set of measure $0$ from $X$ we may assume it defined
for all $x \in X.$

Let $\ti D$ be a lift of $D$.  If $\ti X$ is the set of lifts
to $\ti D$ of points in $X$ then there is a positive measure
subset $\ti X_0 \subset \ti X$ and a lift $\ti f$ of $f$ such that
 $\tau_{\ti f}(x) = 0$ for all $x  \in  X_0.$

Suppose the first return time for $x$ is $n$, so $r(x) = f^n(x)$.
Then $\ti f^n(\ti x) \in T^k(\ti D)$ for a unique integer $k.$ 
We define $h(x, \ti f),$  the {\em homological displacement} of $x$, to 
be $k$.  It depends on $\ti f$ but not on the choice
of lift $\ti D$ of $D.$

It suffices to prove that $h(x, \ti f) = 0$ for some $x \in X_0$
because then $\ti x$ is contained in a periodic disk chain
(see Proposition~(1.3) of \cite{franks:poincare}) and 
$\ti f$ has a fixed point.  We note that
if there are $x,y \in X_0$ such that $h(x, \ti f) > 0$
and $h(y, \ti f) < 0$ then $\ti f$ has a fixed point.
This is a consequence of Theorem~(2.1) of \cite{franks:poincare} 
since there are both
positive and negative recurring disk chains for $f$.
Hence we may assume $h(x, \ti f)$ has a constant sign.

Proposition~(3.2) of \cite{franks:open_annulus} shows that if 
\[
B = \bigcup_{n \in \Z} f^n( X_0)
\]
then 
 
\[
\int_{X_0} h(x, \ti f) \ d\mu = \int_B \tau_{ \ti f}(x)\ d\mu.
\]
Since $\tau_{ \ti f}(x) = 0$ for all $x \in X_0$ we conclude
that $\int_{X_0} h(x, \ti f) \ d\mu = 0.$  Since $h$ has constant
sign it follows that $h(x, \ti f) = 0$ for almost all $x \in X_0.$
\end{proof}

\begin{defn}\label{defn: mean rot}
Suppose $f: A \to A$ is an area preserving homeomorphism
of the closed annulus which is isotopic to the identity
and let $\ti f: \ti A \to \ti A$ be a lift to its universal
covering space. Then the {\em mean translation number}
$\tau_{\mu}(\ti f)$ is 
\[
\int_{X} \tau_{\ti f}(x) \ d\mu
\]
where $X \subset \ti A$ is a fundamental domain for the
universal cover. The {\em mean rotation number} 
$\rho_\mu(\ti f)$ is the coset of 
$\tau_{\mu}(\ti f)$ in $\R/\Z.$
\end{defn}

\begin{prop}\label{prop: mean rot}
Suppose $f: A \to A$ is an area preserving homeomorphism
of the closed annulus which is isotopic to the identity.
If $\rho_\mu(f) = 0$ 
then $f$ has a fixed point in the interior of $A$.
\end{prop}

\begin{proof} Let $\ti f: \ti A \to \ti A$ be the lift of $f$
such that $\tau_\mu(\ti f) = 0$.  If $\tau_{\ti f}$ vanishes on a
set of positive measure then Proposition~(\ref{prop: +meas})
gives the result.  Otherwise there is a set $Y^+$ (resp.  $Y^-$)
with positive measure on which $\tau_{\ti f}$ is positive
(resp. negative). It follows that there is a birecurrent point
$x^+ \in \Int(A)$ (resp. $x^- \in \Int(A)$) 
with a positive (resp. negative) translation
number.  A small free disk $D^+$ containing $x^+$ will be a 
positively recurring disk and similarly there is  a negatively recurring free disk
$D^-$ containing $x^-$.  Theorem (2.1) of \cite{F-Poincare} 
then implies the existence of a fixed point for $f$ in the interior of $A.$
\end{proof}

\begin{notn}\label{notn: annular comp}
Suppose $U \subset S^2$ is an open $f$-invariant annulus.  We would
like to compactify $U$ to a closed annulus for which $f$ has a natural
extension.  The annulus $U$ has two ends which we compactify
separately in a way depending on the nature of the end. We say that an
end of $U$ is {\em singular} if the component of the complement of $U$
in $S^2$ that it determines is a single, necessarily fixed, point $x
\in S^2$.  In this case we compactify that end by blowing up $x$ to
obtain a circle on which $f$ acts by the projectivization of $Df_x$.
If the end is not singular we will take the prime end compactification
(see Mather \cite{mather:prime_end} for properties).  In either case
we obtain a closed annulus $U_c$ whose interior is naturally
identified with $U$ in such a way that $f|_{U}$ extends to a
homeomorphism $f_c: U_c \to U_c.$

We will call $U_c$ the {\em annular compactification}
of $U$ and $f_c: U_c \to U_c$ the annular compactification 
of $f|_U.$  If there is no ambiguity about the choice of
$f$ we will denote the rotation interval
$\cR(f_c)$ by $\rho(U)$ and the two
rotation numbers of the restriction of $f_c$ to its boundary
circles by $\rho(\partial U_c).$
\end{notn}

\begin{lemma}\label{lem: frontier fp}
Let $f$ be an area  preserving  diffeomorphism
of a compact surface.
Suppose $U$ is an open $f$-invariant annulus
and $f_c : U_c \to U_c$ is the extension of $f$ to its
annular compactification.
\begin{enumerate}
\item   If there is a point $x \in U_c$
with $\rho_{f_c}(x) = 0$ then  $\Fix(f_c) \ne \emptyset$. 
\end{enumerate}
If $\bar X$ is the component of the frontier of $U$ corresponding to a component $X$ of $\partial U_c$ then
\begin{enumeratecontinue}
\item If  $\Fix(f_c|_X) \ne \emptyset$ then $\Fix(f|_{\bar X}) \ne \emptyset$.
\item If $\bar X \subset \Fix(f)$ and $\bar X$ contains more than one point
then $X \subset \Fix(f_c)$.
\end{enumeratecontinue}
\end{lemma}

\begin{proof}
(1) follows from Theorem~(\ref{thm: translation interval}).

For (2),  suppose that  $\Fix(f_c|_X) \ne \emptyset$ and note that $\bar X$ is $f$-invariant.
    If $\bar X$ is a single point then (2) is obvious so we may assume that $\bar X$ has more than one point.   Thus   $X$ is the
prime end compactification  and each prime
end $x \in X$ is defined by a sequence of ``cross-cuts'' $\{\gamma_n\}$ where
each $\gamma_n$ is a Jordan arc whose interior is in $U$ and whose
endpoints are in the frontier of $U$.  They satisfy
\begin{description}
\item [(a)]$\displaystyle \lim_{n \to \infty} \diam(\gamma_n) = 0.$
\item [(b)] Each $\gamma_n$ has two complementary components in $U$,
one of which is an annulus and the other of which is an open
disk which we will denote $D_n$.
\item [(c)] The disk $D_{n+1}$ is a subset of $D_n$ and 
$\displaystyle \bigcap_n D_n = \emptyset.$ 
\end{description}
Two such sequences of cross-cuts $\{\gamma_n\}$ and
$\{\gamma_m'\}$ determine the same prime end 
if for each $n$ there is an $m$ with $D_m' \subset D_n$ and
for each $m$ there is an $n$ with $D_n \subset D_m'$.

Let $\{\gamma_n\}$ determine a prime end in $X$ which is fixed by $f_c$.
Then from the fact that $f$ preserves area it follows that
$f(\gamma_n) \cap \gamma_n \ne \emptyset.$  For $n\ge 1$ choose $x_n
\in \Int(\gamma_n).$ From property (a) above it follows that any
point in the limit set of the sequence $\{x_n\}$ is a fixed point of
$f$.  It is clearly in $\bar X.$  This completes the proof of (2).

For (3) suppose that $\bar X \subset \Fix(f)$.  By Lemma 4.1 of 
\cite{handel:commuting}  there is an
isotopy rel $\Fix(f)$ from $f$ to a diffeomorphism
$f'$ that is the identity on a neighborhood of
$\Fix(f)$.  By Theorem~18 of \cite{mather:prime_end}, 
$f_c|_X = f'_c|_X$, which is obviously the identity.
%\jmfc{Lemma 4.1 of \cite{handel:commuting}  is what we quoted in our ``abel'' paper.}
\end{proof}

\begin{cor}\label{cor: frontier gfp}
Let $\G$ be a group of area preserving diffeomorphisms
of $S^2$ or the closed disk $D^2.$ Suppose $U$ is an open
$\G$-invariant annulus and $\G_c$ is the group of homeomorphisms
$g_c: U_c \to U_c$ that are annular compactifications of the
elements $g \in \G.$ If there is a point $x \in \Fix(\G_c)$ then
$cl(U)$ contains a point $\bar x$ of $\Fix(\G).$ If $x$ lies in the
component $X$ of $\partial U_c$ corresponding to a component $\bar X$
of the frontier of $U$ then $\bar x \in \bar X.$
\end{cor}

\begin{proof}  If $x$ is a point of $\Fix(\G_c)$ and $x \in \Int(U_c) =
U$ we are done.  So we may assume it is in a boundary component $X$ of
$U_c$.  If $X$ corresponds to a singular end of $U$ then the point
corresponding to that end is in $cl(U) \cap \Fix(\G).$ Otherwise $X$
is the prime end compactification of an end of $U.$ Let $\{\gamma_n\}$
be a sequence of cross-cuts that determine a prime end in $X$ which is
in $\Fix(\G_c).$ Then, as in the previous lemma, the fact that each $g
\in \G$ preserves area implies that $g(\gamma_n) \cap \gamma_n \ne
\emptyset.$  Also as in the previous lemma we may choose
$\gamma_n$ so that $\displaystyle \lim_{n \to \infty} \diam(\gamma_n) = 0.$
For $n\ge 1$ let $x_n \in \Int(\gamma_n).$
It follows that any point in the limit set of the
sequence $\{x_n\}$ is in $\Fix(g)$.  Since this is independent of the
choice of $g \in \G$ it follows that any point in the limit set is in
$\Fix(\G).$
\end{proof}

\begin{prop}\label{prop: free-disk A}
Suppose $f: A \to A$ is an area preserving homeomorphism
of the closed annulus which is isotopic to the identity
and suppose every point of $A$ has the same forward rotation number. Let
$U = int(A).$  Then either $f$ has a fixed point in $U$ or 
every point of $U$ is free disk recurrent for $f|_U.$
\end{prop}

\begin{proof}
If the forward rotation number of all points of $A$ is $0,$ then
Proposition~\ref{prop: +meas} implies that $f$ has a fixed point in $U$.
Hence we may assume the common rotation number of
the points of $A$ is non-zero and consequently $\Fix(f) = \emptyset.$
Suppose $x \in U$ and $z \in \omega(x) \subset A.$  If
$z \in U,$ then any free disk containing $z$ intersects $\orb(x)$
in infinitely many points.  If $z \in \partial A$ let $V$
be a free half disk neighborhood of $z$ in $A$ and let $V_0 = V \cap U.$
Then $\orb(x)$ intersects $V_0$ infinitely often.
\end{proof}

 The rotation number or rotation interval
of a point $x$ in an open annulus may not be
well defined as in principle it can depend on the compactification 
of the annulus as well as the point.   
The following lemma addresses issue  in
the case that the orbit of $x$ lies in a compact (but not
necessarily invariant) subannulus.

\begin{lemma}\label{lem: 2ann}
Suppose $f_i: A_i \to A_i,\ i=1,2$ are homeomorphisms of closed annuli
which are isotopic to the identity.  Suppose further that $J_i: A_0
\to A_i$ is an essential embedding of a closed annulus in $A_i$
 which is not necessarily $f_i$-invariant and
for some $x \in A_0$ and all $n \in \Z$ we have $J_1^{-1}(f_1^n
(J_1(x))) = J_2^{-1}(f_2^n(J_2(x))).$ Then the rotation interval of
$J_1(x)$ with respect to $f_1$ equals the rotation interval of
$J_2(x)$ with respect to $f_2$.
\end{lemma}

\begin{proof}  Identify $A_0$ with $S^1 \times [0,1]$ and let  $p :A_0 \to S^1$ be projection onto the first coordinate.   For $i=1,2$, extend  $pJ_i^{-1} : J_i(A_0) \to S^1$ continuously to $p_i :A_i \to S^1$.  Rotation intervals for $J_i(x)$ with respect to $f_i$ can be computed using $p_i$.  The lemma therefore follows from the fact that $p_1 f_1^n(J_1(x)) = p_2 f_2^n(J_2(x))$.
\end{proof}

\section{Planar topology} \label{sec:planar}

In this section we record and  prove two useful elementary results.  

Recall that by the Riemann mapping theorem, every open, unbounded,
connected, simply connected subset of $\R^2$ is homeomorphic to
$\R^2$. A closed set $X \subset \R^2$ is said to {\em separate} two
subsets $A$ and $B$ of $\R^2$ provided $A$ and $B$ are contained
in different components of $\R^2 \setminus X.$

\begin{lemma} \label{regular point} If $A$ and $B$ are disjoint
closed connected subsets of $\R^2$ then they are separated
by a simple closed curve or a properly embedded line.
\end{lemma}

\proof Choose a smooth function
$\phi:\R^2 \to [0,1]$ such that $\phi(A) = 0$ and $\phi(B) = 1$ and a
regular value $c \in (0,1)$.  Then $\phi^{-1}(c)$ is a countable
union of properly embedded lines and  simple closed curves.
Each component of $\phi^{-1}(c)$ has a collar neighborhood which
is disjoint from the other components.

Let $U$ denote the component of the complement of $\phi^{-1}(c)$
which contains $B$ and let $X$ denote the frontier of $U$.
Then $X$ separates $A$ and $B$ and $X$
consists of a countable subcollection of the components
of $\phi^{-1}(c),$ each of which is also a component of $X$.  The set
$U$ is the component of $\R^2 \setminus X$ which contains $B.$
Each component $L$ of $X$ separates $\R^2$ into two 
open sets, one of which contains $B$ and $X \setminus L$
and the other of which is disjoint from $X$ and $B$.

Consider a curve $\gamma$ running from a point of $A$ to a point of
$B$ and let $L_0$ be the first component of $X$ which $\gamma$ intersects.
The component $L_0$ is independent of the choice of $\gamma,$
since $L_0$ separates $A$ from all other components of $X.$ It
follows that $A$ and $B$ are in different components of the
complement of $L_0$ since otherwise they could be joined by a 
$\gamma$ which does not intersect $L_0.$
\endproof

\begin{lemma} \label{planar topology}  If $U \subset \R^2$ is open and connected then each component $Z$ of the complement of $U$ has connected frontier and connected complement.
\end{lemma}

\begin{proof}
The complement of $Z$ is the union of $U$ with some of its
complementary components and is therefore connected.  If the frontier
$W$ of $Z$ is not connected then by Lemma~\ref{regular point} there is
a separation of $W$ by a set $Y \subset \R^2 $ that is either a simple
closed curve or a properly embedded line.  Since each component
of $\R^2 \setminus Y$ intersects the frontier of $Z$,  each component must intersect both the interior of $Z$ and
$\R^2 \setminus Z$.  Since $Y$ is disjoint from the frontier $W$ of
$Z$, it is contained in either the interior of $Z$ or in $\R^2
\setminus Z$.  In the latter case $Y$ separates  $Z$ and
in the former case $Y$ separates $\R^2 \setminus Z$.
This contradicts the fact that
$Z$ and $\R^2 \setminus Z$ are connected and so proves that the
frontier of $Z$ is connected.
\end{proof}

\section{Normal Form} \label{sec: normal form}
Let $N$ be the genus zero surface obtained from $S^2$ by blowing up each element of $P$ to a boundary circle   and let  $\pi_P : N^2 \to S^2$ be the \lq inverse\rq\  map that collapses each boundary component to a point in $P$.   Given $F \in \newDiff$ there exists a  diffeomorphism  $F' : N \to N$  such that $F \pi_P = \pi_P F'$.   Identify $N$ with a smooth subsurface of $S^2$  and extend $F' : N \to N$ to a  diffeomorphism $G :S^2 \to S^2$ (This is possible by the isotopy extension theorem; see, e.g., Hirsch's book \cite{Hirsch}.)

Theorem~(1.2) of \cite{fh:periodic} states that a
diffeomorphism $G$ of a closed surface is isotopic relative to its
fixed point set to a homeomorphism with certain nice properties.  The
special case that $G$ is isotopic to the identity is considered in
Lemma~6.3 of that paper.  Together, this theorem and lemma  imply that for any
diffeomorphism $G :S^2 \to S^2$ there is a (possibly empty) finite set
$\cR_G$ of disjoint simple closed curves in $  S^2 \setminus
\Fix(G)$ and a homeomorphism $G_1: S^2 \to S^2$ that is isotopic to
$G$ rel $\Fix(G)$ such that:
\begin{description}
\item [($1_G$)] There are    disjoint open $G_1$-invariant annulus neighborhoods $A_j \subset S^2 \setminus
\Fix(G)$ of the elements $\gamma_j \in \cR_G$.
\item  [($2_G$)] Each component $C_i$ of  $S^2 \setminus \cup A_j$ is $G_1$-invariant.  Moreover, 
if  $G_1|_{C_i} \ne$ identity then $C_i \cap \Fix(G) $ is finite  and $G_1|_{C_i} $ is pseudo-Anosov relative to $C_i \cap \Fix(G) $.
\end{description}
After removing extraneous elements of $\cR_G$ if necessary we may also assume 
\begin{description}
\item [($3_G$)]  The elements of $\cR_G$ are essential, non-peripheral and non-parallel in $S^2 \setminus
\Fix(G)$.  For each $A_l$, if the restriction of $G_1$ to each component of $S^2 \setminus \cup A_j$ that is adjacent to  $A_l$ is the identity, then $G|_{A_l}$ is a non-trivial Dehn twist.
\end{description}

Any simple closed curve in $S^2$ that is  fixed up to isotopy rel $\Fix(G)$  is isotopic rel $\Fix(G)$ into one of the $A_j$'s or into one of the $C_i$'s on which $G$ restricts to the identity.  Applying this to the components of $\partial N$, there is  a diffeomorphism $H :S^2\to S^2$ that is isotopic to the identity rel $\Fix(G)$ and satisfies $H(\cR_G) \cap \partial N = \emptyset$.   After replacing $G_1$ with $HG_1H^{-1}$, we   may assume that each component of $\partial N$ is contained in an $A_j$ or in a  $C_i$ on which $G_1$ restricts to the identity. After an isotopy of $G_1$ we may assume that $G_1$ restricts to the identity on $\partial N$ and hence that items (1) and (2) above hold when $\cR_G$ is replaced by $\cR_G \cup \partial N$.    Let $F'_1  = G_1|_N$ and let $\cR_{F'}$ be the set of simple closed curves in $N$ obtained from  $\cR_G \cap  N$ by removing all peripheral elements.   Then
\begin{description}
\item  [($1_{F'}$)] There are    disjoint open $F'_1$-invariant annulus neighborhoods $A_j \subset  N \setminus (\partial N \cup \Fix(F'))$ of the elements $\gamma_j \in \cR_{F'}$.
\item   [($2_{F'}$)] Each component $C_i$ of  $N \setminus \cup A_j$ is $F'_1$-invariant.  Moreover, 
if  $F'_1|_{C_i} \ne$ identity then $C_i \cap \Fix(F') $ is finite  and $F'_1|_{C_i} $ is pseudo-Anosov relative to $C_i \cap \Fix(F') $.
\item  [($3_{F'}$)] The elements of $\cR_{F'}$ are essential, non-peripheral and non-parallel in $N \setminus (\partial N \cup \Fix(F'))$.  For each $A_l$, if the restriction of $F'_1$ to each component of $N \setminus \cup A_j$ that is adjacent to  $A_l$ is the identity, then $F_1'|_{A_l}$ is a non-trivial Dehn twist.
\end{description}

Let $X = \cup  (C_i \cap \Fix(F') )$ where the union is taken over those $C_i$ for which $F'_1|_{C_i} $ is not the identity.    Blow up each element of $X $ to a boundary circle forming a new compact surface $N^*$ and let $F^*$ and $F_1^*$ be  the diffeomorphisms of $N^*$ induced by $F'$ and $F'_1$ respectively.   Then $F'$ and $F'_1$ are isotopic and    $F_1^* $ is in Thurston canonical form because the non-identity components are now pseudo-Anosov instead of pseudo-Anosov relative to a finite set of fixed points.  If there are any pseudo-Anosov components, then the action of $F_1^*$, and hence $F^*$   on the fundamental group of $N^*$ has exponential growth (see 5.1 of \S V of {\em Expos\'e 11} in \cite{FLP}).     In this case,   Theorem 1 of \cite{bowen} implies that $F^*$,  and hence $F'$, and hence $F$, has positive entropy.  This contradiction implies that   $F'_1 | _{C_i}$ is the identity for each $C_i$.   Thus $N \setminus \cup A_j
\subset \Fix(F'_1)$ and we may assume that $F'_1|_{A_j}$ is a non-trivial Dehn twist
about $\gamma_j$ for each $\gamma_j \in \cR_{F'}$.  

Projecting via $\pi_P$ to $S^2$ we have shown that there is a finite collection $\cR_F$ of essential, non-peripheral, non-parallel simple closed curves in $S^2 \setminus \Fix(F)$ such that $F$ is isotopic rel $\Fix(F)$ to a composition of non-trivial Dehn twists in the elements of $\cR$.

%By Lemma~6.2 of of \cite{fh:periodic}, up to isotopy  rel $\Fix(F)$, there is a unique   $\cR_F$ as above with minimal cardinality.  We will always assume that $\cR_F$ has minimal cardinality.  In particular, the elements of $\cR_F$ are non-parallel,  are {\em essential}, meaning that they do not bound disks in $\M$  and are {\em non-peripheral}, meaning that they do not bound disks in $S^2$ intersecting $\Fix(F)$ in exactly one component.

A result of Brown and Kister \cite{brnkist} implies that $F$ preserves
every component of $\M = S^2 \setminus \Fix(F).$  Given a  component $M$ of
$\M $, let $f =F|_M : M\to M$
and let $\cR$  be the subset of $\cR_F \cap M$ consisting of elements that
are non-peripheral in $M$.    If $\cR = \cR_F \cap M$ then   $F_1|_M$ is a composition of non-trivial Dehn twists along the elements of $\cR$.  If $\cR \ne \cR_F \cap M$ then $F_1|_M$ is isotopic to a composition of non-trivial Dehn twists along the elements of $\cR$.    In either case, $f$ is isotopic to  a composition
of non-trivial Dehn twists along the elements of $\cR$.  The elements of $\cR$ are the {\em reducing curves} for $f : M \to M$; they are non-parallel and non-peripheral.
 % The fact that $F$ preserves $\mu$ implies that $M$ has $f$-recurrent points.  The Brouwer plane translation theorem and the fact that $\Fix(f) = \emptyset$ imply that $M$ is not an open disk.

\section{An intermediate proposition} \label{sec: intermediate}

To clarify the logic of the proof of Theorem~\ref{thm: annuli} we
introduce Proposition~\ref{intermediate} which asserts the existence
of a collection $\cA$ of annuli satisfying the second, third and fourth items
of Theorem~\ref{thm: annuli} plus two additional properties.  What is
missing from this proposition is the fact that the elements of $\cA$
are exactly the components of the set $\W$ of weakly free disk
recurrent points for $f$.  The proof of this missing fact requires
renormalization and so comes at a later stage of the paper.

We have stated Proposition~\ref{intermediate}   in terms of a single component $M$ of $\M$ instead of  all of $\M$ as in Theorem~\ref{thm: annuli}.    This has obvious advantages and can be done without loss.

\begin{prop} \label{intermediate}    Suppose that $F \in \newDiff$ has
entropy zero, has infinite order and at least three periodic points.   Suppose that $M$ is a component of $\M = S^2 \setminus \Fix(F)$ and that $f=F|_M:M \to M$.  Then there is a  countable collection $\cA$ of pairwise disjoint
open $f$-invariant annuli such that  
\begin{enumerate}
\item For each compact set $X \subset M$ there is a constant $K_X$
such that any $f$-orbit that is not contained in some $U \in \cA$
intersects $X$ in at most $K_X$ points.  In particular each  birecurrent point
is contained in some $U \in \cA$. 
\item If $z \in M$ is not contained in any element of $\cA$ then there
are distinct components $F_+(z)$ and $F_-(z)$ of $\Fix(F)$ so that $\omega(F,z)
\subset F_+(z) $ and $\alpha(F,z) \subset F_-(z)$. 
\item For each $U \in \cA$ and each component $C_M$ of the frontier of
$U$ in $M$, $F_+(z)$ and $F_-(z)$ are independent of the choice of $z
\in C_M$. 
\item If $U \in \cA$, and $f_c:U_c \to U_c$ is the extension
to the annular compactification (Notation~\ref{notn: annular comp}) of $U$, then each component of $\partial U_c$
corresponding to a non-singular end of $U$ contains a fixed point of $f_c.$
\item  $\cA$ is the set of maximal $f$-invariant open annuli in $M$.
\end{enumerate}
\end{prop}

Note that it is not possible for $M$ to be simply connected, since the
Brouwer plane translation theorem would then assert that $F|_M$ has a
fixed point in $M$. In the special case that $M$ is an annulus, $\cA$
is the single annulus $M$.  Items (1) - (3) and (5) are obvious and item (4)
follows from Lemma~5.1 of \cite{fh:periodic}.  The constructions and
analysis needed for the case that $M$ is not an annulus are carried
out in sections \ref{sec: endpoints} through \ref{sec: annuli}.  The
final formal proof of Proposition~\ref{intermediate} occurs at the end
of section~\ref{sec: annuli}.

% As an immediate consequence of Proposition~(\ref{intermediate}) we have
% the following.

% \begin{cor}\label{cor: asymp}
% If $x \in S^2$ is in not in any element $U$ of $\cA$ then
% $\alpha(F,x)$ and $\omega(F,x)$ are connected subsets of $\Fix(F).$
% \end{cor}

% \begin{proof} If there exists $y \in \omega(F,x)$ with
% $y \notin \Fix(F)$ then there exists a compact neighborhood
% $Y$ of $y$ in $\M = S^2 \setminus \Fix(F).$ The forward orbit of
% $x$ must intersect this neighborhood infinitely often. This
% contradicts part (1) of Proposition~(\ref{intermediate}).
% This proves that $\omega(F,x) \subset \Fix(F).$ It is then
% straightforward to show it is connected.  The same argument
% applies to $\alpha(x,F).$
% \end{proof}

\section{Hyperbolic Structures}  \label{hyperbolic}

In this section we establish notation and recall
standard results about hyperbolic structures on surfaces.   More details can be found, for example, in  \cite{casble:nielsen}.

Suppose that $M$ is a connected open subset
of $S^2$ that has at least three ends or equivalently  is not homeomorphic to either the open disk or open
annulus.  We say that a simple closed curve $\tau \subset M$ is {\em essential} if it is not freely homotopic to a point and is {\em inessential} otherwise.  Similarly $\tau$  is {\em peripheral} if it is isotopic into arbitrarily small neighborhoods of an end of $M$ and is {\em non-peripheral} otherwise.  Thus $\tau$ is essential if and only if each complementary component contains at least one puncture and is peripheral if and only if one of its complementary components contains exactly one puncture. We say that a  properly embedded line in $M$ is {\em essential} if it is not properly isotopic into arbitrarily small neighborhoods of an end of $M$ or equivalently if each component of its complement contains at least one puncture.   

    If $M$ has infinitely many ends then it can be written as an increasing union of finitely
punctured compact connected subsurfaces $M_i$ whose boundary components
determine essential non-peripheral isotopy classes in $M$.  We may
assume that boundary curves in $M_{i+1}$ are not parallel to boundary
curves in $M_i$.  It is straightforward (see \cite{casble:nielsen}) to put compatible
hyperbolic structures on the $M_i$'s whose union defines a complete
hyperbolic structure on $M$ in which all isolated punctures are cusps.  Of course $M$ also has such a hyperbolic structure when it only has finitely many ends.   In this paper, all hyperbolic structures are assumed to be complete and all isolated punctures are assumed to be cusps.

We use the Poincar\'e disk model for the hyperbolic plane $H$.  In
this model, $H$ is identified with the interior of the unit disk and
geodesics are segments of Euclidean circles and straight lines that
meet the boundary in right angles. A choice of hyperbolic structure on
$M$ provides an identification of the universal cover $\ti M$ of $M$
with $H$.  Under this identification, which we assume throughout this paper, covering translations of $\ti M$  are
isometries of $H$ and geodesics in $M$ lift to geodesics in $H$.  The
compactification of the interior of the unit disk by the unit circle
induces a compactification of $H$ by the \lq circle at infinity\rq\
$\sinfty$.  Geodesics in $H$ have unique endpoints on $\sinfty$.
Conversely, any pair of distinct points on $\sinfty$ are the endpoints
of a unique geodesic. 

Each covering translation $T: \tiM \to \tiM$ extends to a
homeomorphism (also called) $T : H \cup \sinfty \to H \cup
\sinfty$. The fixed point set of a non-trivial $T$ is either one or
two points in $\sinfty$. We denote these point(s) by $T^+$ and $T^-$,
allowing the possibility that $T^+ = T^-$.  If $T^+ = T^-$, then $T$
is said to be {\it parabolic}; a root-free parabolic covering
translation with fixed point $P$ is sometimes written $T_P$.  If $T^+$
and $T^-$ are distinct, then $T$ is said to be {\it hyperbolic} and we
 assume that $T^+$ is a sink and $T^-$ is a source; the unoriented
geodesic connecting $T^-$ and $T^+$ is called the {\em axis} of $T$.
A root-free covering translation with axis $\ti \gamma$ is sometimes
denoted $T_{\ti \gamma}$.
 
   Each essential non-peripheral simple closed curve $\tau' \subset M$ is homotopic to a unique closed geodesic $\tau$.  For each lift $\ti \tau' \subset \tiM$, the homotopy between $\tau'$ and $\tau$ lifts to a bounded homotopy between $\ti \tau'$ and a lift $\ti \tau$ of $\tau$ which is the axis of a hyperbolic covering translation $T$.  The ends of  both  lines $\ti \tau'$ and $\ti \tau$  converge to $T^-$ and $T^+$.    
    
 %Similarly, each essential peripheral $\alpha$ is homotopic to a (non-unique)  horocycle $\gamma$.  For each lift $\ti \alpha \subset \tiM$, the homotopy between $\alpha$ and $\gamma$ lifts to a bounded homotopy between $\ti \alpha$ and a lift $\ti \gamma$ of $\gamma$.    Both ends of  both  lines $\ti \alpha$ and $\ti \gamma$  converge to a single point $T^{\pm}$ that is the unique fixed point of a parabolic covering translation $T$.
 
 Similarly, both ends of a  lift $\ti \tau'$ of a peripheral  simple closed curve $\tau'$ converge to a point that is the unique fixed point of a parabolic covering translation; roughly speaking, this fixed point is  a lift of the isolated puncture of $M$ that is encircled by $\tau$. Conversely, if $T$ is peripheral and $\ti \tau$ is a sufficiently small  horocycle based at $P$ then the image $\tau \subset M$, which we call a {\em horocycle in $M$}, is  a peripheral simple closed curve.    Each simple closed peripheral curve in $M$ is isotopic to a (non-unique) horocycle in $M$.    Each essential properly embedded line in $M$ is properly isotopic to a unique properly embedded geodesic line.

Suppose now that $f : M \to M$ is a homeomorphism.   %Identify $H$ with $\ti M$ and write $\ti f :  H \to H$ for lifts of $f :  M \to M$ to the universal cover. 
    If $f:M \to M$ and $g:M \to M$ are isotopic and $\ti f:H \to H$ is a lift of $f:M \to M$, then the isotopy between $f$ and $g$ lifts to an isotopy between  $\ti f:H \to H$ and  a lift $\ti g : H \to H$ of $g : M\to M$;  we say that $\ti f$ and $\ti g$ are equivariantly isotopic.  A    proof of   the following fundamental
result of Nielsen theory appears in Proposition 3.1 of
\cite{han:fpt}. 

\begin{prop} \label{nielsen}   Every lift $\ti f :  H \to H$ extends
uniquely to a homeomorphism (also called) $\ti f :  H \cup \sinfty \to
H \cup \sinfty$.
   If   $\ti f $ and  $\ti g  $  are equivariantly isotopic lifts of $f:M \to M$ and $g:M \to M$ then $\ti f|_{\sinfty} = \ti g|_{\sinfty}$.

\end{prop}

 %For any extended lift $\ti f :  H \cup \sinfty \to H \cup \sinfty$ there is an {\it associated action $\ti f_\#$ on geodesics [horocycles] in $H$ }defined by sending the geodesic with endpoints $P$ and $Q$ to the geodesic  [horocycle] with endpoints $\ti f(P)$ and $\ti f(Q)$. The action $\ti f_\#$ projects to an {\it action $f_\#$ on geodesics  [horocycles] in $M$}.    Proposition~\ref{nielsen} implies that $f_\#$ depends only on the isotopy class of $f$.  
 
 For any extended lift $\ti f :  H \cup \sinfty \to H \cup \sinfty$ there is an {\it associated action $\ti f_\#$ on geodesics  in $H$ } defined by sending the geodesic with endpoints $P$ and $Q$ to the  geodesic with endpoints $\ti f(P)$ and $\ti f(Q)$. The action $\ti f_\#$ projects to an {\it action $f_\#$ on   geodesics  in $M$}.    Proposition~\ref{nielsen} implies that $f_\#$ depends only on the isotopy class of $f$.    Similarly, if $P$ is the unique fixed point of  the parabolic covering translation $T $  then $\ti f(P)$ is  the unique fixed point of the parabolic covering translation $  \ti f T \ti f^{-1}$.   There is an induced an action $f_\#$ on isotopy classes of   simple closed peripheral curves in $M$ that agrees with the induced action   of $f$ on  isolated punctures in $M$.  Note that if a geodesic or isotopy class of a simple closed peripheral curve is equipped with an orientation then its image under $f_\#$ has a well-defined induced orientation.

The following results are well known and follow easily from the definitions.  

%\begin{lemma} \label{basic lemma 1}Suppose  for $i=1,2$, that $\alpha_i$ is an essential closed curve homotopic to the geodesic [horocycle] $\gamma_i$.  Then $f(\alpha_1)$ is freely homotopic to $\alpha_2$ if and only if $f_\#(\gamma_1) = \gamma_2$.
%\end{lemma}

\begin{lemma} \label{basic lemma 1}
\begin {enumerate}
\item  If $\tau'_1$ and $\tau'_2$ are  essential simple closed curves isotopic to geodesics  $\tau_1$ and $\tau_2$ respectively, then $f(\tau'_1)$ is isotopic to $\tau'_2$ if and only if $f_\#(\tau_1) = \tau_2$.  
\item   If $\gamma'_1$ and $\gamma'_2$ are  properly embedded lines properly isotopic   to geodesics  $\gamma_1$ and $\gamma_2$ respectively, then $f(\gamma'_1)$ is properly isotopic to  $\gamma'_2$ if and only if $f_\#(\gamma_1) = \gamma_2$.  
\item   If $\tau'_1$ and $\tau'_2$ are  simple closed peripheral  curves encircling the punctures $p_1$ and $p_2$ respectively, then $f(\tau'_1)$ is isotopic to $\tau'_2$ if and only if $f(p_1) = p_2$.  
 \end{enumerate}
\end{lemma}

\begin{lemma} \label{basic lemma 2}   For any  extended lift $\ti f :  H \cup \sinfty \to H \cup \sinfty$ and extended covering translation $T: H \cup \sinfty \to H \cup \sinfty$, the  following are equivalent:  
 \begin{enumerate}
  \item $\ti f$ commutes with $T$. 
\item  $\ti f$ fixes $T^{+}$ or  $T^-.$
\item  $\ti f$ fixes $T^{+}$ and  $T^-.$
 \end{enumerate}
\end{lemma}  

\proof  $(3) \implies (2)$ is obvious.  If $\ti f$ commutes with $T$ then it preserves $\Fix(T)$ mapping sources to sources and sinks to sinks.  Thus $(1) \implies (3)$.  If $\ti f$ fixes an element of $\Fix(T)$ then $T$ and $\ti f T \ti f^{-1}$ are covering translations whose axes are asymptotic.  Since these axes are periodic, they are equal and so $\ti f$ fixes both elements of $\Fix(T)$.    Thus $(2) \implies (3)$.
\endproof

 We conclude with a definition and lemma about   isotopy of families of lines.

 %Suppose that $\rho$ and $\rho'$ are non-peripheral essential simple closed curves  or   essential properly embedded lines in $M$.  We say that $\rho$ and $\rho'$ have {\em geodesic-like or minimal intersections} if each component of $M \setminus (\rho \cup \rho')$ whose frontier is the union of an interval $I'\subset \rho'$ and an interval $I \subset \rho$ contains at least one puncture.  Note that:
 Suppose that $\rho$ and $\sigma$ are   essential properly embedded lines in $M$. % and that $\tau$ is an emedded arc in $M$. 
    We say that $\rho$ and $\sigma$ have {\em geodesic-like or minimal intersections} if they intersect transversely and if each component of $M \setminus (\rho \cup \sigma)$ whose frontier is the union of an interval $I\subset \sigma$ and an interval $J \subset \rho$ contains at least one puncture.    %Similarly, if $\tau$ is a compact  embedded arc in $M$  then  $\rho$ and $\tau$ have {\em geodesic-like or minimal intersections} if their intersections are finite and transverse, except perhaps at the endpoints of $\tau$, and if each component of $M \setminus (\rho \cup \tau)$ whose frontier is the union of an interval $I\subset \tau$ and an interval $J\subset \rho$ contains at least one puncture. 
 
 Note that :
  \begin{itemize} 
 \item If $\rho$ and $\sigma$   are geodesics with respect to some hyperbolic structure on $M$ then $\rho$ and $\sigma$  have geodesic-like intersections.
 \item If $\rho$ and $\sigma$   have geodesic-like intersections and $h :M \to M$ is any homeomorphisms then $h(\rho)$ and $h(\sigma)$  have geodesic-like intersections.
 \end{itemize}

% \begin{lemma}  \label{simultaneous isotopy} Suppose that $E_1$ and $E_2$ are locally finite collections of disjoint non-peripheral essential simple closed curves  and    essential properly embedded lines.  Suppose further that for $i=1,2$, the elements of $E_i$  are isotopic to distinct geodesics.  \begin{enumerate} 
% \item  For $i =1$ or $2$, the elements of $E_i$ are simultaneously isotopic to their associated geodesics; i.e.  there is a homeomorphism $g :M \to M$, isotopic to the identity, such that $g(\rho)$ is geodesic for each $\rho \in E_i$.
% \item If each pair $\rho \in E_1$ and $\rho' \in E_2$ have geodesic like intersections, then the elements of $E_1$ and $E_2$ are simultaneously isotopic to their associated geodesics.
% \end{enumerate}
% \end{lemma} 

\begin{lemma}  \label{simultaneous isotopy}  
\noindent
\begin{enumerate} 
\item If  $\E$ is a locally finite collection of disjoint  essential  properly embedded lines in $M$ that determine distinct proper isotopy classes, then the elements of $\E$ are simultaneously isotopic to their associated geodesics; i.e.  there is a homeomorphism $g :M \to M$, isotopic to the identity, such that $g(\rho)$ is geodesic for each $\rho \in \E$.  If the elements of $\E$ are smoothly embedded then we may take $g$ to be a diffeomorphism.
 \item   Suppose that $\E $ and $\L$  are locally finite collections of disjoint   essential  properly embedded   lines that determine distinct proper isotopy classes.  Suppose further that  each element of $\L$ is geodesic and that each element of   $\E$    has minimal intersections with each element of $\L$.  Then  there exists a diffeomorphism $g :M \to M$, isotopic to the identity,  that preserves $\L$ and such that $g(\rho)$ is geodesic for each $\rho \in \E$.

\end{enumerate}
\end{lemma}

\proof   The proofs of Lemma 2.5 and 2.6 of \cite{casble:nielsen} can be modified in a straightforward manner to prove this lemma.  The details are left to the reader.\endproof 
 
 \begin{remark} \label{metric modification}   We will use the first two parts of Lemma~\ref{simultaneous isotopy} to modify metrics so that certain given lines are geodesics in their isotopy classes. The key observation is that  if $\mu$ is a hyperbolic metric on $M$ and $g :M \to M$ is a diffeomorphism   then $\nu = g^*\mu$ is a hyperbolic metric on $M$ and a  line $\ell$ is geodesic in $\nu$ if and only if $g(\ell)$ is geodesic in $\mu$.
 \end{remark}

\section {The endpoint maps $\ti \alpha$ and $\ti \omega$ and annular covers}  \label{sec: endpoints}

In this section we begin the proof of Proposition~\ref{intermediate} in the case that $M$ has at least three ends.  (The annulus case was considered following the statement of the proposition.)  

  Equip $M$ with
a complete hyperbolic structure and identify the universal cover $\ti M$   with the hyperbolic disk $H$  as described in section~\ref{hyperbolic}.  Recall from section~\ref{sec: normal form} that $f$ is isotopic to a homeomorphism $\phi: M \to M$ that is supported on a finite union of disjoint annuli  and that restricts to a non-trivial Dehn twist on each annulus.    We may assume without loss that the core curves of these annuli, which  make up the set $\cR$ of reducing curves for  $f : M \to M$, are geodesics. 
% Let $\cR$ be the set of reducing curves for $f : M \to M$  defined in  section~\ref{sec: normal form} and recall that $f$ is isotopic to a homeomorphism $\phi: M \to M$ that is supported on disjoint annular  neighborhoods of the elements of $\cR$ and that restricts to a non-trivial Dehn twist on each annulus.   
% Equip $M$ with a complete hyperbolic structure as described in section~\ref{hyperbolic}.  Since the elements of $\cR$ are disjoint, essential, non-peripheral and  non-parallel,   they are represented by distinct and disjoint simple closed geodesic curves.
 The full pre-image in $\tiM$ of  $\cR$ is denoted $\ti \cR$.  The closure  of   a component  of $\tiM \setminus \ti \cR$ in $\tiM$ is called a {\em domain}.    If $\cR=\emptyset$ then $\tiM$ is the unique domain but otherwise there are infinitely many domains.  The frontier of a domain  is a union of elements of $\ti \cR$.  If $\cR \ne \emptyset$ then the closure of a domain  in $\tiM \cup \sinfty$ intersects $\sinfty$ in a Cantor set.  The image   of the interior of a domain under projection to $M$ is a component of $M \setminus \cR$.

For each domain $\ti C$ let $\ti \phi_{\ti C}$ be the lift of $\phi$ whose restriction to $\ti C$ is the identity outside of a product neighborhood of the frontier.      If $\ti C_1$ and $\ti C_2$ are adjacent domains that intersect in a common frontier component $\ti \sigma \in \ti \cR$ then $\ti \phi_{\ti C_1} = T_{\ti \sigma}^d \ti \phi_{\ti C_2}$  where $T_{\ti \sigma}$ is a root-free covering translation with axis $\ti \sigma$ and $  |d| > 0$ is the degree of the Dehn twist of $\phi$ around $\sigma$.   It is well known, and straightforward to check, that a point $P \in \sinfty$ is fixed by $\ti f_{\ti C}$ if and only if it is contained in the closure of $\ti C$.   Thus $\Fix(\ti f_{\ti C|_{\sinfty}})$ is a Cantor set if $\cR \ne \emptyset$ and is all of $\sinfty$ if $\cR = \emptyset$. 

 Lifting an isotopy between $f$ and $\phi$ induces   a bijection between the set of lifts $\ti f$ of $f$ and the set of lifts $\ti \phi$ of $\phi$.  Thus $\ti f \leftrightarrow \ti \phi$ if and only if $\ti f$ is equivariantly isotopic to $\ti \phi$.  For each domain $\ti C$ let $\ti f_{\ti C}$ be the lift of $f$ corresponding to $\ti \phi_{\ti C}$.   By Proposition~\ref{nielsen},  $\ti f_{\ti C}|_{S_\infty} =\ti \phi_{\ti C}|_{\sinfty}$ and so $\Fix(\ti f_{\ti C}|_{S_\infty})$  is equal to   the intersection of the closure of  $\ti C$ with $S_\infty$.

The subgroup of covering translations that preserves a domain $\ti C$ is denoted $\Stab(\ti C)$ and called the {\em stabilizer} of $\ti C$.    A  covering translation $T$ is contained in $\Stab(\ti C)$ if and only $\{T^\pm\}$ is contained in the closure of $\ti C$ (which is also equivalent to   the axis of $T$ being contained in $\ti C$).        Lemma~\ref{basic lemma 2} implies that $T \in \Stab(\ti C)$ if and only if    $T$ commutes with $\ti f_C$.

  \begin{lemma}  \label{alpha and omega are points} For each lift $\ti f$ of $f$ 
and each $\ti x \in \tiM$,    $ \alpha(\ti f, \ti x)$ and 
$\omega(\ti f, \ti x)$ are single points in $\sinfty \cap \Fix(\ti f)$.
\end{lemma}

\proof
The Brouwer translation theorem implies that $\omega(\ti f, \ti x)
\subset S_\infty$.  We assume that $\omega(\ti f_, \ti x)$ is not a
single point and argue to a contradiction.  It must be the
case that $\omega(\ti f, \ti x) \subset S_\infty \cap \Fix(\ti f).$  
If not, a non-fixed point $z\in \omega(\ti f_, \ti x)$
would have a free neighborhood whose intersection with $H$ would be a free disk visited by the 
orbit of $\ti x$ more than once (indeed infinitely often).
According to Proposition~(1.3) of \cite{franks:poincare} this
implies $\ti f$ has a fixed point in $H$ -- a contradiction.
Since $\omega(\ti f, \ti x)$ consists of fixed points it is 
straightforward to see that it is also connected.

If $\Fix(\ti f)$ does not contain an interval we are done.  Otherwise,  Lemma~\ref{basic lemma 2} implies that every covering translation with one endpoint in this interval commutes with $\ti f$ and so preserves $\Fix(\ti f)$.  It follows that $\Fix(\ti f) = \sinfty$ and so $f$ is isotopic to the identity.  A proof of the lemma in this special case is given in   Proposition~9.1 of \cite{fh:periodic}.
 \endproof 

In addition to   lifts   of $f$  to the universal cover $H$ we will also use  lifts of $f$ to infinite cyclic covers.

\begin{defns} \label{annulus cover}  Suppose that $\sigma$ is a  closed geodesic that is either
equal to an element of $\cR$ or disjoint from every element of $\cR$.  For
each lift $\ti \sigma$, let $T_{\ti \sigma}$ be a root free
covering translation with axis $\ti \sigma$.  Choose a domain $\ti C$ that contains
$\ti \sigma$.  (If $\sigma \in \cR$ then there are two choices but
otherwise there is just one.) Since $\ti f_{\ti C}$ fixes the
ends of $\ti \sigma,$ it commutes with $T_{\ti \sigma}$ 
by Lemma~\ref{basic lemma 2}.
The {\em annular cover} $A_{\sigma}$ is the closed
annulus that is the quotient space of $(H \cup \sinfty) \setminus
T_{\ti \sigma} ^{\pm}$ by the action of $T_{\ti \sigma}$ and $
f_\sigma : A_{\sigma} \to A_{\sigma}$ is the homeomorphism induced
by $\ti f_{\ti C}$.  For $\ti x \in \tiM$ a lift of $x \in M$, we denote the image of $\ti x$ in $A_\sigma$ by $\hat x$. If
$\alpha(\ti f_{\ti C}, \ti x)$ is not an endpoint of $\ti \sigma$ then
$\alpha(f_\sigma, \hat x)$ is a single point in $\partial A_{\sigma}$
and similarly for $\omega( f_\sigma, \hat x)$.

Similarly, if $\ti \sigma$ is a lift of an embedded  horocycle $\sigma \subset M$ then both
ends of $\ti \sigma$ converge to a point $P \in \sinfty$ and there is
a root free covering translation $T_P$ that preserves $\ti
\sigma$. Let $\ti C$ be the unique domain that contains $\ti
\sigma$.   In this case, the {\em annular cover} $A_{\sigma}$ 
is the half-open annulus that is the
quotient space of $(H \cup \sinfty) \setminus P$ by the action of $T_P
$ and the boundary is a single circle  denoted 
$\partial A_\sigma.$  
As in the previous case, $\ti f_{\ti C}$ induces a homeomorphism $ f_\sigma :
A_{\sigma} \to A_{\sigma}$.  The end of $A_{\sigma}$ corresponding to $P$
projects homeomorphically to the end of $M$ circumscribed by $\sigma$.
We can compactify this end exactly as in Definition~\ref{notn: annular comp} to form a closed annulus
$A_\sigma^c$.  There is an extension of $f_\sigma$ (also called
$f_\sigma$) to a homeomorphism of $A_\sigma^c$.
\end{defns}

 As the notation suggests, $f_\sigma$ is
independent of the choice of $\ti C$ and, up to conjugacy, the choice
of lift $\ti \sigma$.  The former follows from the fact that if $\ti C_1$
and $\ti C_2$ contain $\ti \sigma$ then $\ti f_{\ti C_1}$ and $\ti
f_{\ti C_1}$ differ by an iterate of $T_{\ti \sigma}$ and the latter
from the fact that if $\ti \sigma$ is replaced with $S(\ti \sigma)$
for some covering translation $S$ then $\ti C$ is replaced by $S(\ti
C)$ and $ T_{\ti \sigma}$ is replaced by $ST_{\ti \sigma}S^{-1}$.    

\begin{lemma}\label{twist or not} 
Suppose that $\sigma$ is a  horocycle or a closed geodesic   that  is either
equal to an element of $\cR$ or disjoint from every element of $\cR$.
\begin{enumerate}\item For each closed geodesic $\sigma$,
  $\Fix(f_\sigma|_{\partial A_{\sigma}})$ intersects both components
  of $\partial A_{\sigma}$.  If $\sigma \in \cR$ then $f_\sigma$ is
  isotopic rel $\Fix(f_\sigma|_{\partial A_{\sigma}})$ to a Dehn twist
  of the same index that $f$ twists around $\sigma$.  If $\sigma \not
  \in \cR$ then $f_\sigma$ is isotopic rel $\Fix(f_\sigma|_{\partial
    A_{\sigma}})$ to the identity.
\item For each horocycle $\sigma$, $\Fix(f_\sigma|_{\partial
    A_{\sigma}}) \ne \emptyset$.
\end{enumerate}
\end{lemma}

\begin{proof}
Suppose at first that $\ti \sigma$ is a lift of the closed geodesic $\sigma$.  

If  $\sigma \notin \cR$ then the closure of the domain $\ti C$ that contains $\ti \sigma$  intersects both components of  $\sinfty  \setminus \ti \sigma^{\pm}$.  The points in this intersection are fixed by $\ti f_{\ti C}$ and  project to fixed points $\hat x, \hat y$  for $f_\sigma$ in different components of  $\partial A_{\sigma}$.  A geodesic $\ti \alpha$ connecting $\ti x$ to $\ti y$ in  $\tiM$ projects to  the interior of an embedded arc  $\hat \alpha$ connecting $\hat x$ to $\hat y$ in $A_{\sigma}$ such that $f_\sigma(\hat \alpha)$ is homotopic to $\hat \alpha$ rel endpoints.  It follows that  $f_\sigma$ is isotopic rel $\Fix(f_\sigma|_{\partial    A_{\sigma}})$ to the identity.

 If $\sigma \in \cR$ and $\ti C_1$ and $\ti C_2$ are the domains that contain $\ti \sigma$ then points in the intersection of the closure of $\ti C_1$ with $\sinfty$ are fixed by $\ti f_{\ti C_1}$ and project to fixed points for $f_\sigma$ in one component of $\partial A_{\sigma}$ and points in the intersection of the closure of $\ti C_2$ with $\sinfty$  are fixed by $\ti f_{\ti C_2}$ and project to fixed points for $f_\sigma$ in the other component of $\partial A_{\sigma}$.   If $f$ twists with degree $k$ around $\sigma$ then   $\ti f_{\ti C_1}$ and $ \ti f_{\ti C_2}$ differ by  $T_{\ti \sigma}^{k}$  so  $f_\sigma$ is
  isotopic rel $\Fix(f_\sigma|_{\partial A_{\sigma}})$ to a Dehn twist of index $k$.    This completes the proof of (1).  
 
 The proof for (2)  is similar.
 \end{proof}

\section{Reducing Arcs in Annular Covers}
In this section we recall, adapt and improve definitions and results from  section 10 of \cite{fh:periodic}, where the assumption is that     $F$ is periodic point free and isotopic rel $\Fix(F)$ to the identity  as opposed to our current assumption that  $F$ has zero entropy and  is  isotopic rel $\Fix(F)$ to a composition of Dehn twists on the elements of $\cR$.  In particular, the homeomorphisms $f_\sigma : A_\sigma \to A_\sigma$ of
Definition~\ref{annulus cover}  are periodic point free    in \cite{fh:periodic} and are only fixed point free in our current context.  Switching from  periodic point free to entropy zero requires a change in the proof of  Lemma~\ref{no doubling} but nothing more.   Allowing $\cR$ to be non-empty requires a fair amount of work, most of which is done in later sections.

Of primary interest are the homeomorphisms $f_\sigma : A_\sigma \to A_\sigma$ of
Definition~\ref{annulus cover}.  We frame the discussion more generally for clarity and for possible future applications.

\begin{notn} \label{notn: annulus} 
  We assume throughout this section   that $h:A \to A$ is a    homeomorphism  of the closed annulus $A$  that is isotopic to the identity and whose restriction to the interior $\mrA$ of $A$  is fixed point free and   that $  x_1,\ldots,  
x_r$ are points in $\mrA$  whose $\alpha$-limit sets   $
\alpha(h, x_i)$ are distinct single points in $\partial
A$ and whose $\omega$-limit sets   $
\omega(h, x_i)$  are distinct single points
in $\partial A$.   
  Let   $X \subset \mrA$ be the
union of the $h$-orbits of the $x_i$'s and let  $\mrA_{ X} =\mrA \setminus  X$  equipped with a   hyperbolic structure as in section~\ref{hyperbolic}. 
\end{notn}

Recall that a  properly embedded line $\ell \subset \mrA_{X}$ is   essential
if it is not properly isotopic into arbitrarily small neighborhoods of
some end of $\mrA_{ X}$ and that each essential $\ell$ is properly isotopic
to a unique geodesic. The action of $h$  
on isotopy classes of properly
embedded lines in $\mrA_{X}$ is captured by the map $h _\#$ on
geodesics defined in section~\ref{hyperbolic}.

Suppose that  $\ell$ is a geodesic line in $\mrA_X$ that  separates $\mrA$ into two components,   $U$ and $V$.  Choose an isotopy rel $X$ from $h$ to $h'$ where  $h'(\ell) = h_\#(\ell)$.   The sets  $h'(U)$ and $h'(V)$   are independent of the choice of $h'$ and we write $h_\#(U) = h'(U)$ and $h_\#(V) = h'(V)$.  Thus, $h_\#(U) $ and $h_\#(V)$ are the components of  $\mrA\setminus h_\#(\ell)$. 

\begin{remark}  In general, the hyperbolic metric on $\mrA_X$ is unrelated to $\partial A$.   The ends of a geodesic $\ell \subset \mrA_X$  that is properly embedded in $\mrA$ need not converge to single points in $\partial A$.  Even if the ends  of $\ell$ and $h_\#(\ell)$  converge to single points in  $\partial A$, these pairs of points need not be related by $h|_{\partial A}$.       We will require that  our hyperbolic metrics  satisfy certain extra properties (see Lemma~\ref{proper rays}) to guarantee some compatibility between the metric and the boundary.
\end{remark}
 
 If an embedded  path $\beta \subset \mrA$  has endpoints in $ X$ but is otherwise disjoint from $ X$ then the interior of $\beta$ determines a properly embedded line $\ell \subset \mrA_{ X}$.    Proper isotopy of  $\ell$  in $\mrA_{ X}$ corresponds to isotopy  rel $ X$ of $\beta$ in $\mrA$.   If $\ell$ is essential [resp. a geodesic] in $\mrA_{X}$ then we say that {\em $\beta$ is essential [resp. a geodesic  rel $ X$] in $\mrA$}.   There is an induced map $h_\#$ on geodesics rel $ X$ in $\mrA$ such that $h _\#(\beta)$ is  the   unique geodesic path in the isotopy class rel $ X$ of $h(\beta)$.

\begin{defn}  \label{defn: translation arc geodesic} An arc $\beta' \subset \mrA$ connecting $x \in  X$ to $h(x)$ is called a {\em translation arc} for $x$ if $h(\beta') \cap \beta' =  h( x)$.  If $\beta'$ intersects $X$ only in its endpoints, then the   geodesic rel $X$ in $\mrA$ determined by  $\beta'$ is called a {\em translation arc geodesic} for $x$ relative to $X$.

Assume that $\beta$ is a translation arc geodesic for $x$ relative to
$X$ and let $\beta_j = h_\#^j(\beta)$, a translation arc geodesic for
$h^j(x)$ relative to $X$.  If $B^+ = \cup_{j=0}^\infty \beta_j$ is an
embedded ray in $\mrA$ that converges to $\omega(h, x)$ then we say
that $\beta$ is {\em forward proper} with {\em forward homotopy
streamline} $B^+$.  In this case, $h_\#$ induces a self-map of $ B^+$
that is conjugate to a standard translation of $[0,\infty)$ into
itself.

Assume that $\beta$ is forward proper with forward homotopy streamline
$B^+$ and let $L^+$ be the unique geodesic line in $\mrA_X$ that is
properly embedded in $\mrA$ and such that one of the components,
$V^+$, of $\mrA \setminus L^+$ contains $ \cup_{j=1}^\infty \beta_j$
and intersects $X$ exactly in $ \cup_{j=1}^\infty h_\#^j(x)$.
Topologically, $L^+$ is just the boundary of a sufficiently small
regular neighborhood of $ \cup_{j=1}^\infty \beta_j$ in $\mrA$.
Note that $h_\#(L^+)$ is the unique geodesic line
in $\mrA_X$ that is properly embedded in $\mrA$ and such that one of
the components of $\mrA \setminus L^+$ contains $ \cup_{j=2}^\infty
\beta_j$ and intersects $X$ exactly in $ \cup_{j=2}^\infty h_\#^j(x)$.
In particular $h_\#(L^+) \subset V^+$ and $h_\#(V^+) \subset
V^+$.  %Let $L_j^+ = h_\#^j(L_+)$.  %If $L_j^+$ is the interior of an arc $\bar L_j^+ \subset A$ the ends of $ L_j^+ $ converge to single points $a_j,b_j \in \partial A$ for all $j \ge 0$ and if $a_j,b_j \to \omega(h,x_i)$ in $A$ and if
Let $\cl_A (h^j_\#(V^+)) $ be the closure of $h^j_\#(V^+) $ in $A$.
If both  ends of  each $h_\#^j(L^+)$ converge to $\omega(h, x)$ and if  $\cap_{j=0}^\infty \cl_A(h^j_\#(V^+)) = \omega(h,x)$ then we say
that $B^+$ {\em has a forward translation neighborhoood} and that
$V^+$ is the {\em forward translation neighborhood} determined by
$\beta$.

\bigskip

\includegraphics[width=5.5in]{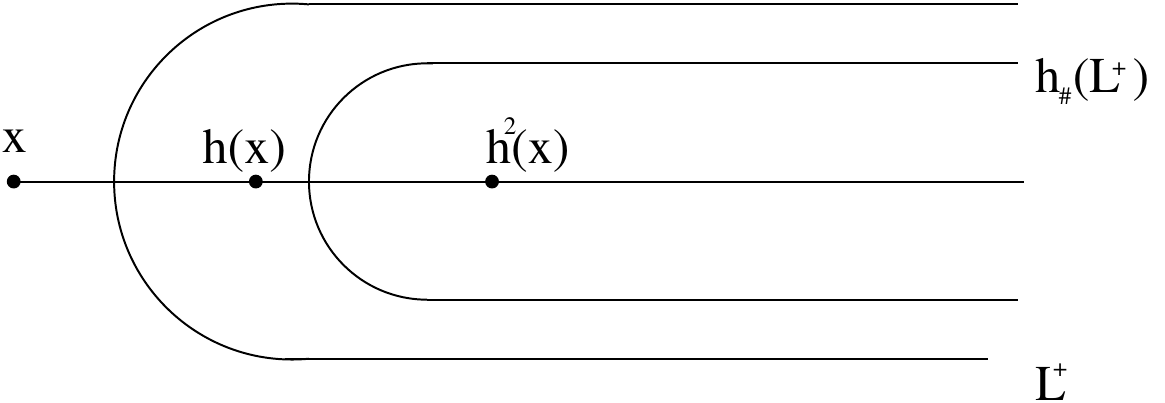}

\bigskip

{\em Backward proper} homotopy translation arcs, {\em backward homotopy streamlines} $ B^- = \cup_{j=0}^\infty h_\#^{-j} (\beta)$ and  {\em backward translation neighborhoods}  $V^-$ with boundary $L^-$ are defined similarly using $h$ instead of $h^{-1}$.
\end{defn}

% \begin{remark} \label{proper streamline}  If $\beta$ is a forward proper translation arc geodesic with forward homotopy streamline $B^+$ then $h_\#$ induces a self-map of   $ B^+$ that is conjugate to a standard tranlslation of  $[0,\infty)$ into itself.   The analogous statement also holds for a backward homotopy streamline  with respect to $h_\#$.  
%  \end{remark} 

\begin{lemma}   \label{proper rays} Assume that $h : A \to A$, $  x_1,\ldots,  
x_r$, $X$  and $\mrA_X$ are as in Notation~\ref{notn: annulus}.  The  hyperbolic metric on $\mrA_X =\mrA\setminus X$  can be chosen so that  for each $1 \le i \le r$
 there  are translation arc geodesics
 $ \beta_i^+$  and  $ \beta_i^-$ for some points, $x_i^+$ and $x_i^-,$ in the
 orbit of $ x_i$ such that  
 \begin{enumerate}
 \item   $\beta^+_i$ is forward proper and the forward homotopy streamline $B_i^+$ has forward translation neighborhood $V^+_i$.   
  \item   $\beta^-_i$ is backward proper and the backward homotopy streamline $B_i^-$ has backward  translation neighborhood $V^-_i$. 
  \item The $B_i^\pm$'s, and hence the $ V_i^{\pm}$'s,
 are all disjoint.
\end{enumerate}
\end{lemma}

\proof      Lemma~10.6 of \cite{fh:periodic}    states that  there  are forward proper translation arc geodesics
 $ \beta_i^+$  and backward proper translation arc geodesics $ \beta_i^-$  such that the $ B_i^{\pm}$'s are all disjoint.   There are two issues that must be discussed before quoting that lemma.  The first is that in the context of \cite{fh:periodic}, $\Per(h) = \emptyset$ and $\cR = \emptyset$.  In proving Lemma~10.6 of \cite{fh:periodic}, the former is used only to conclude that   $\Fix(h) = \emptyset$  and the latter is  not used at all.      Since  our $h$ satisfies $\Fix(h) = \emptyset$ by hypothesis,  we are not quoting out of context. The second issue   is that  the  role of the hyperbolic metric was not explicitly mentioned in either the statement or proof of Lemma~10.6 of \cite{fh:periodic}.  We attend to that now.  

The accumulation set $Y \subset \partial A$ of $X$ is the union of the $\alpha$ and $\omega$ limit sets of the $x_i$'s.  For each $y \in Y$, choose a decreasing sequence of  closed  half disk neighborhoods $U_i(y)$ of $y$ whose intersection is the single point $y$ and whose frontier  $\partial U_i(y)$ in $\mrA$ is disjoint from $X$.    We may assume that  if $y \ne y' \in Y$ then the $\partial U_i(y)$'s are all disjoint from the $\partial U_j(y')$'s and all these lines determine distinct proper isotopy classes in $\mrA_X$.  By Lemma~\ref{simultaneous isotopy} (1), we can simultaneously isotope all of the $\partial U_i(y)$'s to their associated geodesics.  We may therefore modify (see Remark~\ref{metric modification})  the given hyperbolic metric  so that all of the    $\partial U_i(y)$'s  are geodesic.    In particular, if a translation arc for an element of $X$ is contained in the interior of some $U_i(y)$ then the corresponding   translation arc geodesic is also contained in the interior of $U_i(y)$. 

Having chosen the metric with the above properties on translation arc
geodesics, the proof of Lemma~10.6 of \cite{fh:periodic} can be
applied.

We must now arrange that each $B_i^+$ has forward translation
neighborhoods and that each $B_i^-$ has backward translation
neighborhoods.  This will require a further modification of the
metric.  For each $J \ge 0$, choose a smooth properly embedded line
$\sigma^+_{i,J}$ %whose ends converge to single points $a_J,b_J \in \partial A$
such that $\sigma^+_{i,J} \cap B_i^+$ is a single point contained in $
h_\#^{J}(\beta^+_i)$ and such that one of the two complementary
components  $V'^+_J$ of $\sigma^+_{i,J}$ contains $ \cup_{j=J+1}^\infty
h_\#^j(\beta^+_i)$ and intersects $X$ exactly in $ \cup_{j=J+1}^\infty
h^j(x_i)$.   
Assume further that both ends of each $ \sigma_{i,J}^+$ converge to $\omega(h,x)$ and that $\cap_{J=0}^\infty \cl_A( V'^+_J) = \omega(h,x)$.
%the $\sigma^+_{i,J}$'s converge to $\omega(h,x_i)$.  
Define $\sigma^-_{i,J}$ similarly.  We may assume
that all of the $\sigma^\pm_{i,J}$'s are disjoint.  By
Lemma~\ref{simultaneous isotopy} (2), there is an isotopy of $\mrA_X$
that preserves each $B_i^{\pm}$ and that moves each $\sigma^\pm_{i,J}$
to the unique geodesic $L^\pm_{i,J}$ in its proper isotopy class.  We
may therefore change the metric so that each $ \sigma^\pm_{i,J}$ is a
geodesic %(and hence $L^\pm_{i,J} = \sigma^\pm_{i,J}$ for all $i$ and $J$)
 while maintaining the property that $B_i^{\pm}$ is geodesic.
This completes the proof of the lemma.
\endproof

%The ray  $\hat B_i^+$ [resp. $\hat B_i^-$] of Lemma~\ref{proper rays} is called the {\em forward homotopy streamline}  
 
% \begin{remark} \label{proper streamline}Assume the notation of Lemma~\ref{proper rays}.  For each $i$  the map ${f_{\sigma}}_\#$ defined on geodesics rel $\hat X$ induces a self-map of   $ B_i^+$ that is conjugate to a standard tranlslation of  $[0,\infty)$ into itself.   The analogous statement also holds for $B_i^-$ with respect to ${f^{-1}_{\sigma}}_\#$.  
%  \end{remark} 
 
 Further details on the constructions in the next definition can be found in section 10 of \cite{fh:periodic}
 
 \begin{defn}   \label{brouwer subsurface}Assume the metric on $\mrA_X$ has been chosen as in Lemma~\ref{proper rays} and assume the notation of that lemma.  %The translation neighborhoods $V^\pm_i $ for $B^\pm_i$   are disjoint.
 The subsurface $W = \mrA \setminus (X \cup (\bigcup_{i=1}^r V^\pm_i))$ is finitely punctured.   We write $\partial W = \partial_+W \cup \partial_- W$ where $\partial_{\pm}W = \cup_{i=1}^r \partial V_i^\pm$.  Then $h_\#(\partial_+W) \cap W = \emptyset$ and  $\partial_-W \cap h_\#(W) = \emptyset$.  We say that $W$ is the {\em Brouwer subsurface determined by the $ \beta_i^\pm$'s}.

Let $\RH$ be the set of non-trivial relative homotopy classes $[\tau]$ determined by embedded arcs $(\tau, \partial \tau) \subset (W, \partial_+W)$.   Denote $\tau$ with its orientation reversed by $-\tau$ and
 $[-\tau]$ by $-[\tau]$.  By a {\em mulitset} $\cT$ in $ \RH$ we mean a   set, each element of which is a copy of an element of $\RH$.  The {\em multiplicity} of an element of $\RH$ in $\cT$ is the number of copies of that element that appear in $\cT$.  An important tool in our analysis is   a map that assigns to each  finite  multiset ${\cal T}$ in $\RH$    another  finite multiset $h_\#({\cal T} ) \cap W$ in $\RH$.   

Choose a homeomorphism  $g :\mrA \to \mrA$ that is isotopic to $h$ rel $ X$   such that $g(L)  =h_\#(L)$ for each component $L$ of $\partial W$.    For any arc $\tau \subset W$ with endpoints on $\partial_+W$, $  g(\tau)$ is an arc in $  g(W) = h_\#(W)$ with endpoints on $ h_\#(\partial_+W)$; in particular, $  g(\tau) \cap \partial_-W = \emptyset$ and $\partial  g(\tau) \cap W = \emptyset$.  Let $ h_\#(\tau) \subset  h_\#(W)$ be the geodesic arc that is isotopic rel endpoints to $ g(\tau)$.  The components $\tau_1,\ldots,\tau_r$ of $  h_\#(\tau) \cap W$ are arcs in $W$ with endpoints in $\partial_+W$.  Define $ h_\#([\tau]) \cap W =  \{[\tau_1],\ldots,[\tau_r]\}$. It is shown in \cite{han:fpt} (see pages 249 - 250) that $h_\#([\tau]) \cap W$ is well defined.

More generally if ${\cal T} $ is multiset in 
$\RH$ then we define $ h_\#({\cal T}) \cap W = \cup_{[\tau]
\in {\cal T} }h_\#([\tau]) \cap W$.  Note that $
h_\#(\cdot )\cap W$ can be iterated.  Recursively define $
(h_\#)^n([\tau]) \cap W =(h_\#)^{n-1}(
h_\#([\tau]) \cap W) \cap W$.

A finite multiset $\cT$    in $\RH$ is a {\em fitted family} if 
\begin{enumerate}
\item the elements of $\cT$ are represented by   disjoint simple arcs.
\item no element of $\RH$ has multiplicity greater than one in $\cT$.
\item if $[\tau]$ has multiplicity one in $\cT$ then $-[\tau]$ has multiplicity zero in $\cT$.
\item for all $n \ge 0$ and all $t \in \cT$,  each element of $h^n_\#(t) \cap W$ is, up to a change of orientation, a copy of some element of $\cT$.
\end{enumerate}   
\end{defn}

The next lemma states that  one gets the same answer by either  iterating the intersection operator or by first   iterating $h$ and then applying the intersection operator once.  Following this lemma, we will write $h^n_\#(\tau) \cap W$ for  $(h_\#)^n([\tau]) \cap W=(h^n)_\#([\tau]) \cap W$.

\begin{lemma}\label{iteration is well defined}  For all $\tau \in \RH$, $(h_\#)^n([\tau]) \cap W=(h^n)_\#([\tau]) \cap W$.
\end{lemma} 
 
\proof  The statement of this lemma is the same as that of  Lemma 5.4 of \cite{han:fpt}.  Although the setting there is slightly different, the proof given there works here as well.
\endproof

\begin{notn} \label{notn: fitted} Assume the notation of
Lemma~\ref{proper rays} and Definition~\ref{brouwer subsurface}.  Let
${\cal T}_i \subset \RH$ consist of one representative ($[\tau]$ or
$-[\tau]$) of each unoriented homotopy class that is represented by a
component of ${h^n}_\#( \beta_i^-) \cap W$ for some $n> 0$.  The
elements of ${\cal T}_i$ are represented by disjoint arcs.  For any
(not necessarily distinct) components $L_1$ and $L_2$ of $\partial W$,
the number of elements of ${\cal T}_i$ with one endpoint on $L_1$ and
the other on $L_2$ is therefore at most two plus the number of
punctures in $W$.  This is because a region bounded by two such
elements and segments of $L_1$ and $L_2$ must contain at least one
puncture. Thus ${\cal T}_i$ is finite.  Since the fourth item
in the definition of fitted family is satisfied by construction,
$\cT_i$ is a fitted family.  We say that ${\cal T}_i$ is the {\em
  fitted family determined by $ \beta_i^-$}.  The fitted family
determined by $\beta_i^+$ is defined similarly.
\end{notn}

  In Section~\ref{sec: normal form} we used the assumption that $F:S^2 \to S^2$ is smooth and has zero entropy  to conclude that   $f: M \to M$ is isotopic to a composition of Dehn twists along disjoint simple closed curves.   Lemma~\ref{no doubling} below  (c.f. Theorem 5.5(b) of \cite{han:fpt}) is the only other place in which smoothness and the entropy zero hypothesis are applied.

\begin{notn}  We say that an element $[\tau] 
\in \RH$ \emph{eventually doubles} if there exists $n > 0$ so that $ h_\#^n([\tau]) \cap W $ contains $\ [\tau]$ with multiplicity at least two.  %Thus $[\tau]$ \emph{does not eventually double} if for all $n > 0$, $ h_\#^n([\tau]) \cap W $ contains $[-\tau]$ and $[\tau]$ with multiplicity zero or one and they can not both be equal to one.    
\end{notn}

For the rest of the section we will assume that no element of $\RH$
eventually doubles.  Before deducing implications of this assumption
we show that it is satisfied by our primary examples. 
 
   Recall that
$F\in \newDiff$ has entropy zero, that $M$ is a component of
$S^2 \setminus \Fix(F)$ and that $f = F|_M$.  Recall also that $F' :N \to N$ is a $C^\infty$ diffeomorphism of a closed genus zero surface, that $\pi_P :N \to S^2$ collapses components of $\partial N$ to points in $P$ and that $\pi_P F' = F \pi_P$.  In particular $F'$ has zero entropy.    
If $\sigma$ is a
horocycle or   closed geodesic that is either equal to an element of $\cR$ or
disjoint from every element of $\cR$ then $f_\sigma :A_\sigma \to
A_\sigma$ [resp. $f_\sigma :A^c_\sigma \to A^c_\sigma$] is the
homeomorphism of the closed annulus given in Definition~\ref{annulus
  cover}.

\begin{lemma} \label{no doubling} Assume that  $h = f_\sigma: A_\sigma \to A_\sigma$  [resp. $f_\sigma :A^c_\sigma \to A^c_\sigma$] is as in Definition~\ref{annulus cover} and  that $W$  is as in Definition~\ref{brouwer subsurface}.  Then no element of $\RH$ eventually doubles. 
\end{lemma}

Before proving the lemma we state a special case of a theorem of Yomdin \cite{yomdin}.   Suppose that $h : N \to N$ is a  $C^\infty$ diffeomorphism  of a compact surface and   $\nu \subset N$ is a smooth path. 
Let $|\nu|_N$ be the length of $\nu$ in $N$ with respect to some smooth metric on $N$ and define the {\em growth rate for the length of $\nu$ with respect to $h$}   to be
$$gr(\nu,h) = \limsup_{n \to \infty} \frac{\log|{h}^n(\nu)|_N}{n}.$$

\begin{thm}[Yomdin \cite{yomdin}, Theorem 1.4] \label{growth rate} For any $C^\infty$ diffeomorphism $h :N \to N$ of a compact surface and for any smooth path $\nu \subset N$,  $ gr(\nu,h) \le$ entropy($h)$.  
\end{thm}
 
\noindent{\em Proof of Lemma~\ref{no doubling}}\  \  We assume to the contrary that there exist $[\tau] \in \RH$ and $n \ge 1$ such $[\tau]$ has multiplicity at least $2$  in $ ({f_\sigma ^n})_\#([\tau]) \cap W$  and argue to a contradiction. The obvious induction argument on $k$ implies that  $[\tau]$ has multiplicity at least $2^k$ in $ ({f_\sigma ^{kn}})_\#([\tau]) \cap W$.    

 We denote $A_\sigma$ or $A^c_\sigma$ by $A$.  By construction,  the universal covering projection $\tiM \to M$ factors through a covering projection  $\pi_\sigma : \Int A \to M$.   Each smooth path $\nu\subset \Int A$ projects to a smooth path $\pi_\sigma(\mu) \subset M$ whose  length, with respect to   the hyperbolic metric on   $M$, is denoted $|\pi_\sigma(\mu) |_M$.  We will prove that there is a compact set $M_0 \subset M$ and $\epsilon> 0$ so that for all $k \ge 1$ there are at least
$2^k$ disjoint subpaths $\mu_j$ of $ f^{kn}_\sigma (\tau)$ such that $ \pi_\sigma(\mu_j) \subset M_0$ and    $|\pi_\sigma(\mu_j)|_M \ge \epsilon$.  The paths $\pi_{\sigma}(\tau),  \pi_\sigma(f^{kn}_\sigma(\tau))$ and $\pi_\sigma(\mu_j)$ lift  via $\pi_P$ to smooth paths $\tau', {F'}^{kn}(\tau')$ and $\mu'_j$ respectively where the $\mu'_j$'s are disjoint subpaths of ${F'}^{kn}(\tau')$.  Since the $\mu'_j$'s are contained in a compact subset of $\Int N$ there exists $\epsilon' > 0$ so that each $|\mu'_j|_N \ge \epsilon'$.    It follows that the growth rate for the length of $\tau'$ with respect to $F'$ is at least $\log(2) \epsilon'$ contradicting Theorem~\ref{growth rate} and the assumption that $F'$ has entropy zero.

It remains to prove the existence of $M_0$ and $\epsilon$.
 Assume the notation of Definition~\ref{brouwer subsurface}.  Let $L_1$
and $L_2$ be the components of $\partial_+W$ that contain the
endpoints of $\tau$.   Recall that the ends of $L_i$ converges to a single point in $\partial A_\sigma$.

As a first case suppose that $L_1 = L_2$ and that $\tau$ and the
interval in $L_1$ connecting the endpoints of $\tau$ bound a disk $D$
in $\mrA_\sigma$.  Choose an element $x \in X \cap D$ and a compact
essential subannulus $A_1 \subset \mrA_\sigma$   that separates $ x$ from
$L_1$.  There are at least $2^k$ subpaths $\mu_j$ of $
f_\sigma^{kn} (\tau)$ that cross $A_1$.  In this case we let $M_0 =  \pi_\sigma(A_1)$; the existence of a uniform lower bound for  $|\pi_\sigma(\mu_j)|_M$ comes from the compactness of $A_1$, which implies that  there is a uniform lower bound to   $|\mu_j|_A$  and that the restriction of $\pi_\sigma$ to $A_1$ is bi-Lipschitz.

\includegraphics[width = 4in]{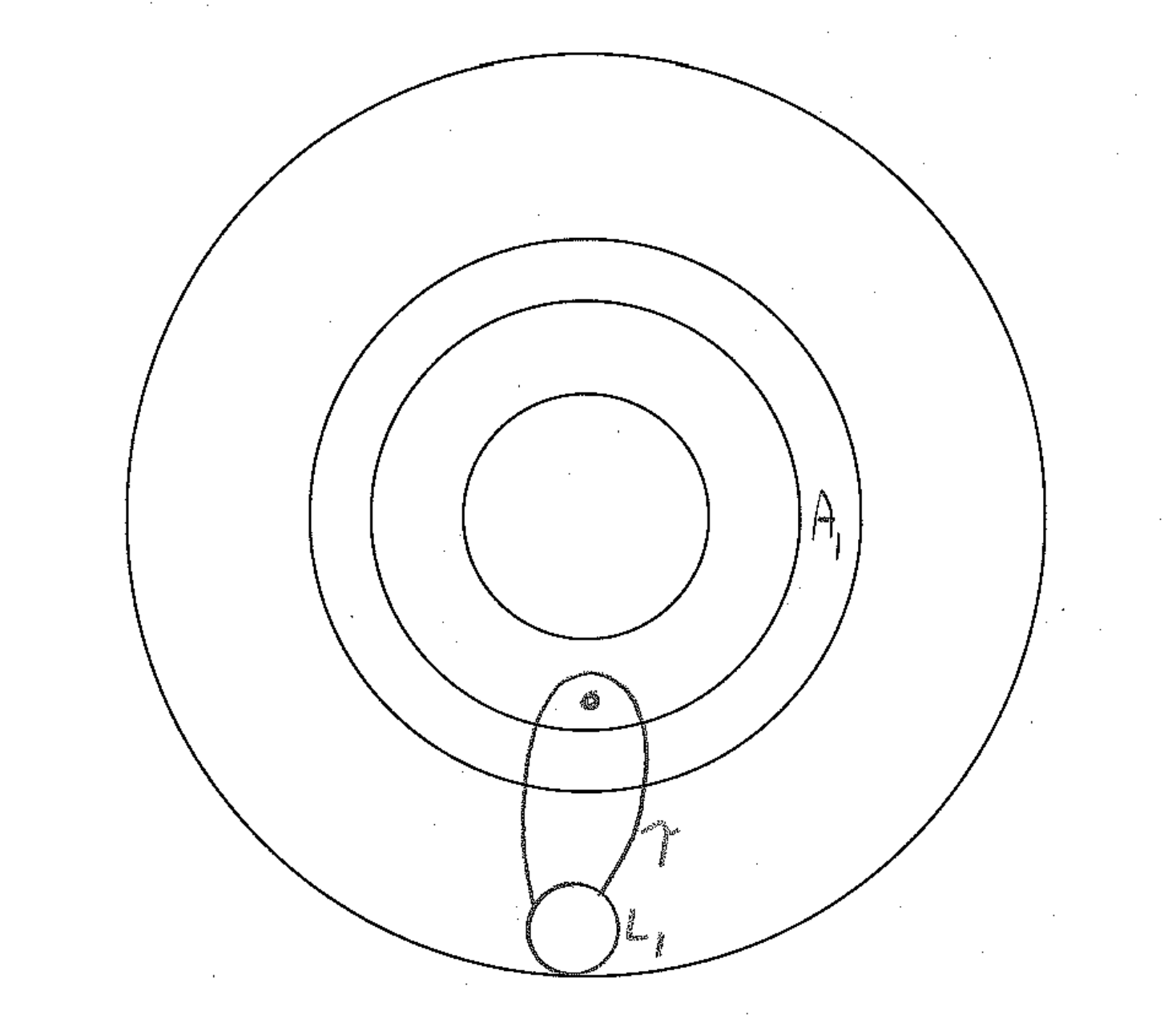}

In the case that  the ends of $L_1$ are in one component
of $\partial A_{\sigma}$ and the ends of $L_2$ are in the other, the same argument works with respect to  a compact essential subannulus $A \subset \mrA_\sigma$ that separates $L_1$ and $L_2$.    

The third case is that  $L_1 = L_2$ and that $\tau$    and the interval in $L_1$ connecting the endpoints of $\tau$ define a simple closed curve that is essential in $A_\sigma$.   Choose a compact essential annulus $A_3\subset \mrA_{\sigma}$ that is disjoint from $L_1 \cup \tau$.  Choose disjoint half-disks  $D_1,D_2$  whose frontiers consist of   intervals $I_1$ and $I_2$ in the component of $\partial A_\sigma$ that contains the endpoint of $L_1$    and   half-circles $\ti \rho_1$ and $\ti \rho_2$ that project to the same simple closed curve $\rho \subset M$.   Assume further that   the closure of $L_1$ is disjoint from $D_1$ and $D_2$. Choose   thickened  arcs  $J_1$ and $J_2$ connecting $\ti \rho_1$  and $\ti \rho_2$ to the far component of  $\partial A_3$.   Thus  $J_1$ and $J_2$ overlap with $A_3$ in rectangles that cross $A_3$. 

\includegraphics[width = 4in]{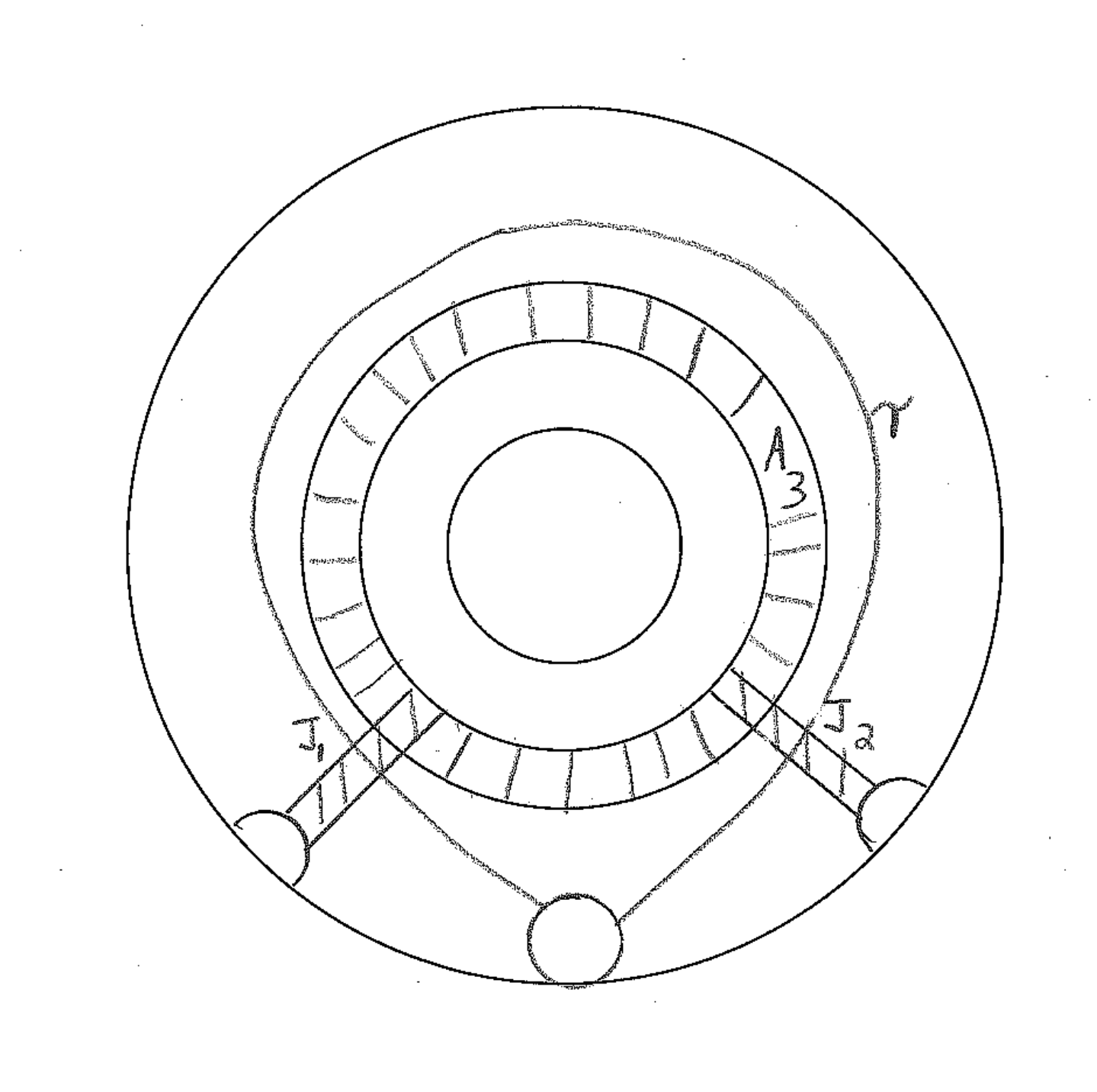}

 There are at least
$2^k$ subpaths $\mu_j$ of $ ({f^{kn}_\sigma}) (\tau)$ such that each
$\mu_j$ either crosses $J_1$, crosses $J_2$, crosses $A_3$ or is an arc
with one endpoint on $\ti \rho_1$ and the other on $\ti \rho_2$.    In this case we let $M_0$ be the union of $\pi_\sigma(J_1 \cup J_2 \cup A_3)$ and   a compact $\delta$-neighborhood of $\rho$ where $\delta> 0$ is so small that this neighborhood is a closed annulus.   The existence of a uniform lower bound for  $|\pi_\sigma(\mu_j)|_M$ comes from the compactness of $J_1 \cup J_2 \cup A_3$ and the fact that  if $\mu_j$ has one endpoint on $\ti \rho_1$ and the other on $\ti \rho_2$ then $|\pi_\sigma(\mu_j)|_M$ has endpoints in $\rho$ but is not contained in the $\delta$-neighborhood of $\rho$.

 The fourth and final case is that  $L_1 \ne L_2$ have endpoints in the same component of $\partial A_\sigma$.  The obvious modification of the argument from the third case applies here.  
  \qed
  
 \vspace{.1in}

   We now return to the more general context of this section.   
       
    Suppose that $W$ is a Brouwer subsurface and that $\tau \in \RH$.    We say that  
  {\em $[\tau] $  disappears under iteration} if $ h_\#^n([\tau]) \cap W  = \emptyset$ for some $n > 0$ and   that {\em  ${\cal T}$ disappears under iteration} if each element in  ${\cal T}$ does.  If  both endpoints of $\tau $ are contained in a single  component $L$ of $\partial_+W$ and the simple closed curve that is the union of $\tau$ with an interval in $L$ does not bound a disk in $A$ then we  say that  $\tau$ is {\em essential}. 
  
  The next lemma is  based on Theorem~5.5(c) of \cite{han:fpt}.    See also Lemma~10.8 of \cite{fh:periodic}.

\begin{lemma} \label{lemma:fitted family}  Suppose  that   $W$ is a Brouwer subsurface   and $\cT$ is a fitted family   as in Definition~\ref{brouwer subsurface}. Suppose further that $\cT$ does not disappear under
iteration and that no element of $\cT$ eventually doubles.  Then there exists $[\tau] \in {\cal T}$ such that
$h_\#([\tau])\cap W$ contains $[\tau]$ with multiplicity one and does not contain $-[\tau]$; moreover,    every other element of $h_\#([\tau])\cap W$ disappears under iteration.    If $\tau$ has both endpoints on the same
component $L$ of $\partial_+W$ then $[\tau]$ is essential.
 \end{lemma}
 
\proof Let $\Gamma$ be the directed graph with one vertex $v_i$ for each
element   ${[\tau_i] \in \cal T}$ and with the number of oriented edges from $v_i$ to $v_j$ equal to the sum of the multiplicities of $[\tau_j]$ and of $-[\tau_j]$ in  $h_\#([\tau_i]) \cap W$.
% and  the multiplicity of $-[\tau_j]$ in  $h_\#([\tau_i]) \cap W$.
%one oriented edge from $v_i$ to $v_j$   for each occurrence of $\pm[\tau_j]$ in
%$h_\#([\tau_i]) \cap W$.  
 By Lemma~\ref{iteration is well defined},  there is a natural bijection between the elements of  $h^n_\#([\tau_i]) \cap W$ and the set of %(terminal endpoints of)
  oriented paths in $\Gamma$ that have length $n$ and begin at $v_i$.

Since no $[\tau_i]$ eventually doubles, each $v_i$  is contained in at most one
non-repeating oriented closed path in $\Gamma$.  The set $\V_0$ of vertices of $\Gamma$ that are contained in at least one 
oriented closed path is non-empty because $\cT$ does not disappear under iteration.   There is a partial
order on the vertices of $\Gamma$ defined by $v_1 > v_2$ if there is
an oriented path in $\Gamma$ from $v_1$ to $v_2$ but no oriented path
from $v_2$ to $v_1$.  Choose $v_p \in \V_0$ so that $v_p> v_q$ implies that $v_q\not \in \V_0$.   Note that if $v_p > v_q$ then  $[\tau_q] $ disappears under iteration.  Indeed, if it does not then there there are arbitrarily long oriented paths in $\Gamma$ beginning with  $v_q$ so there must be an element $v_r \in \V_0$ such that $v_q>v_r$; this contradicts $v_p > v_q$ and the choice of $v_p$.   Note also that if there is an oriented path  from $v_p$ to $v_q$ but $v_p \not > v_q$ then $v_q$ is on the oriented cycle through $v_p$ and so is uniquely determined by the length of the path from $v_p$ to $v_q$.
   %Note also that if there is an oriented path of length $L$ from $v_p$ to $v_q$ but $v_p \not > v_q$ then $v_q$ is on the oriented cycle through $v_p$ and so is uniquely determined by $L$. 

Let $n$ be
the length of the unique oriented non-repeating closed path $\rho$
through $v_p$.  Then $h^n _\#([\tau_p])\cap W$ contains exactly one element that does not disappear under iteration and it is $ \epsilon[\tau_p] $ where $\epsilon=\pm 1$.  To complete the proof of the lemma, it remains to show that $n=1$ and $\epsilon = 1$. 

Let $v_s$ be the endpoint of the unique edge in $\Gamma$ that begins
at $v_p$ and is on the unique oriented closed path through
$v_p$.  It is possible that $v_s= v_p$.  Thus, either $[\tau_s]$ or $-[\tau_s]$
 is the unique element of $h_\#([\tau_p])\cap W$ that
does not disappear under iteration.

  For the remainder of the proof we make use of the end  (pages 253 - 254) of the proof of Theorem~5.5 of \cite{han:fpt}.    The case that $\tau_p$  has endpoints on distinct components of $\partial_+W$ is treated in the last two paragraphs of that proof.   (The labels $\tau_i^k$ in the diagram on the bottom of page 253 are incorrect; they should be $\tau_i^\ast$.) The argument given there applies without change in our present context  so we may assume that   $\tau_p$  has both endpoints on the same component, say $L$,  of   $\partial_+W$.  
  
     The endpoints of $h_\#(\tau_p)$ are contained in the component of the  complement  of $ W$ bounded by $L$.  If $h_\#(\tau_p)$ intersects any other  component of the  complement  of $W$ then  at least two elements of $h_\#([\tau_p]) \cap W$ would be represented by paths with endpoints on distinct components of $\partial_+W$.    Since no such paths disappear under iteration, this can not happen and we conclude that each element of  $h_\#([\tau_p]) \cap W$, and in particular $\tau_s$, has both endpoints on $L$.  
    
    Both ends of $L$ converge to the same component of $\partial A$.   The argument given in   the first and second paragraphs  on page 253 of \cite{han:fpt} (which is a proof by contradiction) carries over without change to this context and proves that  $\tau_p$ is essential.   By symmetry, $\tau_s$  is also essential.  If $[\tau_p] \ne  [\tau_s]$  then   either the interval of $L$ bounded by the endpoints of $\tau_p$ contains the interval of $L$ bounded by the endpoints of $\tau_s$ or vice-versa.   In either case, there is a  rectangle   $D \subset W$  bounded by  $\tau_p, \tau_s$ and intervals in $L$.  It contains finitely many punctures, each of which is mapped to the complement of $W$ by all sufficiently high iterates of $h$.  %There are only finitely many punctures contained in   disks bounded by an element of ${\cal T}$ and an arc in $L$.  
 Thus, for all sufficiently large $k$,   \ $h^{kn}(D)$ does not contain any punctures in $W$.     It follows that either $h^{kn}_\#([\tau_p])\cap W =  h^{kn}_\#([\tau_s])\cap W $ or $h^{kn}_\#([\tau_p])\cap W =  h^{kn}_\#([-\tau_s])\cap W $. But   $\epsilon^k [\tau_p]$ [resp. $[\epsilon^k \tau_s$]\ ] is the unique element of $h^{kn}_\#([\tau_p])\cap W $ [resp. $h^{kn}_\#([\tau_s])\cap W $] that does not disappear under iteration.  This contradicts the assumption that $[\tau_p] \ne  [\tau_s]$.  We conclude that $[\tau_p] =   [\tau_s]$ and hence that $n =1$.   Since $h$ is orientation preserving and $h_\#(L)$ is parallel to $L$, it follows that  $\epsilon = 1$.

%It follows that  $h^{kn}_\#([\tau_p])\cap W = \pm h^{kn}_\#([\tau_s])\cap W $. But   $\pm [\tau_p]$ [resp. $[\pm \tau_s$]] is the unique element of $h^{kn}_\#([\tau_p])\cap W $ [resp. $h^{kn}_\#([\tau_s)\cap W $] that does not disappear under iteration.  This contradicts the assumption that $[\tau_p] \ne \pm[\tau_s]$.  We conclude that $[\tau_p] =  \pm[\tau_s]$ and hence that $n =1$.   Since $h$ is orientation preserving and $h_\#(L)$ is parallel to $L$, it follows that  $\epsilon = 1$. 
\endproof

\begin{defn} \label{disappears}  Suppose that $W$ is a Brouwer subsurface, that ${\cal T}$ is a fitted family and that $[\tau] \in {\cal T}$.     Let  $L_1$ and $L_2$ be the components of $\partial_+W$ that contain the initial  and terminal endpoints $w_1$ and $w_2$ of $\tau$ respectively.  We say that $[\tau]$ is {\em peripheral} if one (and hence all) of the following equivalent conditions are satisfied.
\begin{enumerate}
\item Some component of the complement of $L_1 \cup L_2 \cup \tau$ is contractible in $\mrA_X$.
\item There are rays $R_1 \subset L_1$ and $ R_2 \subset L_2$ whose initial points are $w_1$ and $w_2$     such that   the    line $  R_1^{-1}\tau R_2$  can be  properly isotoped rel $X$ into arbitrarily small neighborhoods of some end of $\mrA$. 
\item If $\ti \tau$ is a lift of $\tau$ to $\tiM$ and $\ti L_1$ and $\ti L_2$ are the lifts of $L_1$ and $L_2$ that contain the endpoints of $\ti \tau$ then $\ti L_1$ and $\ti L_2$ have a common endpoint.
\end{enumerate}
\end{defn}

%\begin{defn} \label{disappears} Suppose that $W$ is a Brouwer subsurface, that ${\cal T}$ is a fitted family and that $[\tau] \in {\cal T}$.     If $L_1$ and $L_2$ are the components of $\partial_+W$ that contain the initial  and terminal endpoints $w_1$ and $w_2$ of $\tau$ respectively, and if some component of the complement of $L_1 \cup L_2 \cup \tau$ is contractible in $\mrA_X$ then we   say that  $[\tau] $ is {\em peripheral}.  Equivalently, $[\tau]$ is peripheral   if there are rays $R_1 \subset L_1$ and $ R_2 \subset L_2$ whose initial points are $w_1$ and $w_2$     such that   the    line $  R_1^{-1}\tau R_2$  can be  properly isotoped rel $X$ into arbitrarily small neighborhoods of some end of $\mrA$. 
%\end{defn}

Our next result  refines Lemma~\ref{lemma:fitted family}.  It is based on Lemma~6.4 of \cite{han:fpt}.

 \begin{lemma}  \label{reducing line}   Suppose  that  $W$, ${\cal T}$ and $[\tau]$ are as in the statement of Lemma~\ref{lemma:fitted family}.    Let $L_1 =\partial V_1$ and $L_2 =\partial V_2$ be the  (possibly equal)  components of $\partial_+W$ that contain the  initial and terminal endpoints  $v_1$ and $v_2$ of $\tau$.  Then
  \begin{enumerate}
  \item  $h_\#([\tau])\cap W = \{[\tau]\}$. 
  \item    If  $[\tau]$ is not peripheral then there are rays $R_1 \subset L_1$ and $R_2 \subset L_2$ such that $R_1^{-1}\tau R_2$  is isotopic to an $h_\#$-invariant geodesic line $\mu$.  %There is an embedded $\fsigmasharp$-invariant geodesic line $\rho$ with one end  asymptotic to an end of $L_1$ and the other end asymptotic to an end of $L_2$.
\end{enumerate}
\end{lemma}
       
\proof  To prove (1) we assume that $\left (h_\#([\tau])\cap W \right ) \setminus  \{[\tau]\} =\{s_1,\dots,s_m\}$ is not empty  and argue to a  contradiction.     Lemma~\ref{lemma:fitted family} implies that each  $s_i$ disappears under iteration and so, in particular,   is represented by a path   with both endpoints on $L_1$ or both endpoints on $L_2$.   %We assume without loss that    both endpoints of $\sigma_1$ are contained in $L_1$.  
  As there is no loss in replacing $h$ by an iterate, we may assume  that  each $h_\#([s_i])\cap W = \emptyset$. %and hence that  $h^2_\#([\tau])\cap W   = h_\#([\tau])\cap W$.  %Since $\sigma$ disappears under iteration its endpoints are either both on $L_1$ or both on $L_2$; we assume for concreteness that they are both on $L_1$.   

%To prove (1) we will assume that $h_\#([\tau])\cap W \setminus  \{[\tau]\} =\{s_1,\dots,s_m\}$ is not empty and argue to a contradiction.

We recall the alternate definition of $h_\#([\tau])\cap W$  given on page 50 of \cite{han:fpt}.      Choose a  lift $\ti h:  \tiM \cup \sinfty \to \tiM \cup \sinfty$ to the compactified universal cover of   $\mrA_X$, choose a lift  $  \ti \tau$ of $\tau$  and for $j =1,2$,  let $\ti L_j$     be  the lift  of $L_j$   that contains  the endpoint $\ti v_j$  of $\ti \tau$.   
The lines $\ti h_\#(\ti L_j)$ are disjoint from the full pre-image $\ti W$ of $W$.  There are finitely many components $\ti W_l$ of $\ti W$ that separate $\ti h_\#(\ti L_1)$ from $\ti h_\#(\ti L_2)$.  Any geodesic path connecting $\ti h_\#(\ti L_1)$ to  $\ti h_\#(\ti L_2)$ crosses through    $\ti W_l$ in a geodesic arc;  the projection of this arc to $W$ determines a well defined element   of $\RH$ and the multi-set of  these elements, obtained by varying $l$, is exactly  $h_\#([\tau])\cap W$.

 %Since $\tau$ is not peripheral,  $\ti L_1$ and $\ti L_2$ do not have a common   endpoint.  The endpoints of $\ti L_1$  and the endpoints of $\ti L_2$ bound disjoint intervals $J_1$ and $J_2$ in   $S_\infty$.    
Since $[\tau] \in h_\#([\tau])\cap W$, we may choose $\ti h$ so that   $\ti h_\#(\ti \tau)$ crosses   $\ti L_1$  and $\ti L_2$  in that order.  For future reference, note that $\ti h^2_\#(\ti \tau)$ crosses   $\ti h  _\#(\ti L_1),  \ti L_1, \ti L_2,$  and $\ti h  _\#(\ti L_2)$  in that order.  

 %Let $\ti W_0$ be the component of $\ti W$ whose frontier contains $\ti L_1$ and $\ti L_2$. 
  Intersection of $\ti h_\#(\ti \tau)$ with $\ti W$ decomposes     $h_\#( \ti \tau)$ into an alternating concatenation of supbaths  $\ti \mu_k$ whose projections   $\mu_k \subset \mrA$   represent elements of $h_\#([\tau])\cap W$ and  subpaths $\ti \nu_k$ whose projections $\nu_k$  are contained in $V_j \setminus h_\#(V_j)$  for $j = 1$ or $2$.  We assume without loss that some $s_i$, say $s_1$, has both endpoints in $L_1$ and hence that some $\mu_k$, say $\mu_1$ is  a path in $W$ that represents $s_1$ and in particular has  both endpoints in $L_1$.  It follows that at least one of the  $\nu_k$'s    is contained in   $V_1 \setminus h_\#(V_1)$  and   has both endpoints in $L_1$.  Note that this is true not only for $\ti h_\#(\ti \tau)$ but also for   any geodesic path that connects $\ti h_\#(\ti L_1)$ to $\ti h_\#(\ti L_2)$.  In particular, $h^2_\#(\tau)$ contains a subpath in $V_1 \setminus h_\#(V_1)$    with both endpoints in $L_1$.

Let $\ti L_1'$ and $\ti L_1''$ be the  lifts of $L_1$ that contain the endpoints of $\ti \mu_1$.  Since $h_\#([s_1])\cap W = \emptyset$,  any geodesic path connecting $\ti h_\#(\ti L_1')$ to  $\ti h_\#(\ti L_1'')$ projects to a path   in $V_1 \setminus h_\#(V_1)$ with both endpoints in  $h_\#(V_1)$.  Since  $\ti h_\#(\ti L_1')$ and  $\ti h_\#(\ti L_1'')$  separate $\ti h^2_\#(\ti L_1)$ from $\ti h^2_\#(\ti L_2)$,    $h^2_\#([\tau])$ contains a subpath  in $V_1 \setminus h_\#(V_1)$ with both endpoints in  $h_\#(V_1)$.  But $S_1 = V_1\setminus h_\#(V_1)$ is a once punctured strip.  A   geodesic arc in  $S_1$ with endpoints in $L_1$ can not be disjoint from a geodesic arc in $S_1$ with endpoints in $h_\#(L_1)$.    Since $h^2_\#(\tau)$ is an embedded geodesic arc, we have reached the desired contradiction and so have proved (1). 

Let $\ti \sigma$ be the subpath of $\ti h_\#(\ti \tau)$ that connects $\ti h_\#(\ti L_1)$ to $\ti L_1$ and let $\sigma$ be its image in $\mrA$.  We now know that $\sigma$ is an arc in $S_1$    with one endpoint in $L_1$ and the other in $h_\#(L_1)$.  Since $S_1$ is a once punctured strip, one of the complementary components of $\sigma$ in $S$ is
unpunctured. There are rays $R' \subset L_1$ and $R'' \subset h_\#(L_1)$  so that ${R'}^{-1} \sigma R''$ is peripheral. Lifting this back to the universal cover, we have that  $\ti L_1$ and $\ti h_\#(\ti L_1)$ are asymptotic.  Their common endpoint $P_1$ is a fixed point for $\ti h|_{\sinfty}$.   Symmetrically, one of the endpoints $P_2$ of $\ti L_2$ is fixed by $\ti h$  and we let $\ti \mu$ be the geodesic connecting $P_1$ to $P_2$ (which are distinct points because $[\tau]$ is not peripheral). This completes the proof of (2).
\endproof

\begin{defn}  An embedded arc $\rho \subset A$ that is disjoint from $X$ and that has   endpoints in $\Fix(h|_{\partial A}) $ is a {\em reducing arc} for $h$ rel $ X$ if  it is $h$-invariant up to isotopy rel  $  X$ and  rel its endpoints  and is non-peripheral in the sense that it is not homotopic rel endpoints and rel $  X$ into $\partial A$.   
\end{defn}

%Note that if the interior of $\rho$ is a geodesic line in then $\rho$ is non-peripheral.

The following lemma is similar to    Proposition 10.10 of \cite{fh:periodic}.    The conclusions of the lemma are more  detailed than those of that proposition and   apply to $h$ and not just some iterate of $h$.     Lemma~\ref{twist or not}  implies that  condition (b)  below is satisfied    in the special case that $h=f_\sigma$ % or $h = f_\sigma^c$. 
for $\sigma \in \cR$.
 Note also that if (b)  is satisfied then all reducing arcs have their endpoints on the same component of $\partial A$.

 \begin{lemma} \label{reducible}    Assume that the hyperbolic metric on $\mrA_X$ has been chosen as in Lemma~\ref{proper rays} and that  $  \beta^\pm_i$ and  $V^\pm_i$ are as in that lemma.  Let $W$ be the associated Brouwer subsurface and assume that no element of $\RH$ eventually doubles.  Let $\alpha =  \cup_{i=1}^r  \alpha(h, x_i) $ and $\omega  =  \cup_{i=1}^r  \omega(h, x_i) $.      Assume    that  for each component $\partial_lA$ of $\partial A$,   $\alpha_l = \alpha \cap \partial_lA$ and $\omega_l = \omega \cap \partial_lA$ have the same cardinality $c_l$ and    that if $c_l > 1$ then the elements of $\alpha_l$ and  $\omega_l$ alternate around $\partial_lA$.  Then
 \begin{enumerate}
  \item There is a reducing arc $\rho$ for $h$ with respect to $X$. 
 \end{enumerate}  
 If  either one of the following conditions is satisfied
 \begin{description}
 \item [(a)]  $r =1$
% \item [(b)]  One component of  $\partial A$ is fixed point free.
 \item [(b)]   If $x$ and $y$ are fixed points in different components of $\partial A$ then $h$ is isotopic rel $\{x,y\}$ to a non-trivial Dehn twist.  
 \end{description}  then
    \begin{enumeratecontinue} 
\item For each $1 \le i \le r$,  $\alpha(h, x_i)$ and $\omega(h,x_i)$ belong to the same component of $\partial A$.
\item     For each $1 \le i \le r$, there is a reducing arc $\rho_i$ whose endpoints are $\alpha(h, x_i)$ and $\omega(h, x_i)$.
\item      For each $1 \le i \le r$,  there is a   translation arc geodesic $  \beta_i$ for $  x_i$  such that $  B_i = \cup_{j=-\infty}^\infty h^j_\#( \beta _i)$   is a properly embedded line  whose initial end  converges  to  $\alpha(h,  x_i)$ and whose terminal end  converges  to    $\omega(h, x_i)$. 
\item The  $B_i$'s are disjoint.   
 \end{enumeratecontinue}  
  If $r =1$
  \begin{enumeratecontinue} 
\item There is a unique   translation arc geodesic for $ x_1$.  
\end{enumeratecontinue}  
  \end{lemma}
  
 \proof     %Choose translation arc geodesics  $  \beta^\pm_i$ and translation neighborhoods $V^\pm_i$ as in Lemma~\ref{proper rays}.   Let  $W$ be the associated Brouwer subsurface and
For each $i$,  let ${\cT}_i$ be the   fitted family  (Notation~\ref{notn: fitted}) determined by $\beta_i^-$.  %We may assume that both ends of $\partial V_i^-$  converge to $\alpha(f_{ \sigma}, \hat x_i)$  and  that both ends of $\partial V_i^+$ converge to $\omega(f_{ \sigma}, \hat x_i)$.   After performing an isotopy rel  $\hat X \cup \partial A_{\sigma}$  we may assume that  $f_{\sigma}(L_i) = {f_{\sigma}}_\#(L_i) $ for each $L_i = \partial V_i^+$.     In particular, the ends of     each $L_i$   are asymptotic to the ends of ${f_\sigma}(L_i)$.
 
 To prove (1), it suffices to show that  there is a properly  embedded non-peripheral  line $\ell \subset \mrA_X$,  whose initial and terminal ends   converge  to  elements of $\Fix(h|_{\partial A}) $ and  such that $h(\ell)$ is properly isotopic in $\mrA_X$  to $\ell$.   The proper isotopy can be chosen so that it extends by the identity on $\partial A$ so we can take $\rho$ to be the closure of $\ell$ in $A$.

If  ${\cal T}_i$ disappears under iteration then  the homotopy streamline $  B_i  = \cup_{n=-\infty}^{\infty} h^n_\#(  \beta_i^-)$ is a properly embedded $h_\#$-invariant line whose ends converge to $\alpha(h,   x_i)$  and $\omega(h, x_i)$ and we let $\ell$  be the line   obtained by pushing  $ B_i$
off of itself; there is always at least one direction to push that results in a non-peripheral line.
If ${\cal T}_i$ does not disappear under iteration, let  $[\tau ]\in {\cal T}_i$  satisfy  the conclusions of   Lemma~\ref{lemma:fitted family} and let $L_p$ and $L_q$ be the components of $\partial_+W$ containing the initial   and terminal  endpoint    of $\tau$ respectively.  We claim that $[\tau] $ is not peripheral  (Definition~\ref{disappears}).      This is obvious if $\omega(h,x_p)$ and $\omega(h,x_q)$ belong to distinct components of $\partial A$.   If $L_p = L_q$ this follows from   Lemma~\ref{lemma:fitted family},  which asserts that $[\tau]$ is essential, and the assumption that the component of $\partial A$ that contains $\omega(h,x_p) = \omega(h,x_q)$ intersects $\alpha$ non-trivially.  In the  final case,  $\omega(h,x_p) \ne \omega(h,x_q)$ belong to the same component of $\partial A$ and so  are separated in that boundary component by elements of $\alpha$;    again $\tau$ is not peripheral.   Lemma~\ref{reducing line}  implies that there  are rays $R_p \subset L_p$ and $R_q \subset L_q$ such that $\ell = R_p^{-1}\tau R_q$, whose ends converge to $\omega(h,x_p)$ and $ \omega(h,x_q)$,  has the desired properties.

We now turn to the proof of (4).    %assuming that (a), (b) or (c) is satisfied and 
   It suffices to show that each $\cT_i$ disappears under iteration.  We   assume that some $\cT_i$  does not disappear under iteration and, continuing  with the above notation,  argue to a contradiction.    Note that $\omega(h,x_p)$ and $\omega(h,x_q)$ lie on the same component, say $\partial_0A$,  of $\partial A$.   This is obvious for (a)  and holds for (b) because  $\rho$ has endpoints   $\omega(h,x_p)$ and $\omega(h,x_q)$ and  is isotopic to $h(\rho$) rel endpoints. Denote the other component of $\partial A$ by $\partial_1A$.

Let $\mu$ be the geodesic determined by $\ell$, let $Y$ be the component of $\mrA_X \setminus \mu$ whose closure
contains $\partial_1A$ and let $Z$ be the other component of $\mrA_X
\setminus \mu$.  Since $Z$ is not contractible, 
it contains at least one orbit of $X$. 
%that $\rho$ separates $X$ in $A$.

We claim that $Y$ also intersects, and hence contains, an orbit of $X$.  If $p \ne q$ this follows from the fact that the endpoints $\omega(h,x_p)$and $\omega(h,x_q)$ of $\ell$  separate $\alpha \cap \partial_0A$.    Suppose then that $p =q$.   Since $\mu$ is $h_\#$-invariant and disjoint from $B^-_i$, each element of $\cT_i$ is represented by an arc   that is disjoint from $\mu$.  Since $\ell$ can be isotoped to be disjoint from $\partial V^+_q$ but cannot be isotoped into $V^+_q$, $\mu$ is disjoint from $ V^+_q$ and hence disjoint from $\tau \cup  V^+_q$.   Since $\tau$ is essential  (Lemma~\ref{lemma:fitted family})   $Y$ contains $V^+_q$ and hence the orbit of $x_q$.    This completes the proof that both $Y$ and $Z$ contain an orbit of $X$.

% By Lemma~\ref{lemma:fitted family}, one of these arcs $\tau$    and an interval in $L_q$ define an essential simple closed curve in $A$.    Since $\mu \cap V^+_q = \emptyset$, $Y$  contains $V^+_q$ and hence the orbit of $x_q$.    This completes the proof that both $Y$ and $Z$ contain an orbit of $X$.  %intersects $X$.  The other component of  $\mrA \setminus \mu$ also intersects $X$ because $\mu$ is not peripheral.   We have now verified that $\mu$ separates $X$. 
   
If $r=1$ then we have reached the desired contradiction and so have proved (4) in this case.   Arguing by induction on $r$, we may assume that $r > 1$ and that 
if one works relative to $X \cap Y$ or relative to $ X \cap Z$ then ${\cal T}_i$ disappears under iteration. In other words, if $ x_i \in Y$ (the argument for
$  x_i\in Z$ is symmetric) then for all sufficiently large $n$,
$h^n_\#( \beta_i^-)$ is isotopic rel $ X \cap Y$ to
an arc $\gamma_{i,n} \subset V_i \subset Y$.  Since $\mu$ is $h_\#$-invariant, $h^n_\#( \beta_i^-) \subset Y$.  It is a
standard fact that the isotopy rel $ X \cap Y$ of
$h^n_\#( \beta_i^-)$ to $\gamma_{i,n}$ can be taken
with support in $Y$.  It follows that this isotopy is rel $X$
which implies that $h^n_\#(\beta_i^-) \subset V_i$ in
contradiction to the assumption that ${\cal T}_i$ does not disappear
under iteration.   This completes the proof of (4).

  Items (3)  and (5) follow from (4).    If $r=1$ then (2) follows from our
assumption that $\alpha_l$ and $\omega_l$ have the same cardinality. If (b)  is satisfied then (2) follows the fact that  $\alpha(h,x_i)$ and $\omega(h,x_i)$ bound a reducing curve.     Thus (2) is satisfied.

 To verify (6), let $B_1$ and $\beta_1$ be as in (4) and denote $h^j_\#(\beta_1)$ by $\beta_{1,j}$.  Thus  $B_1  = \cup_{n=-\infty}^{\infty}  \beta_{1,j}$ and $h_\#(\beta_{1,j}) = \beta_{i,j+1}$.   We assume that there is a translation arc geodesic $  \delta \ne \beta_{1,0}$ for $ x$ and
argue to a contradiction.  Let $ \eta$ be the maximum initial
segment of $ \delta$ whose interior is disjoint from $  B_1$ and
let   $y$ be the terminal endpoint of $\eta$.  Let $\nu$
be the maximum initial segment of $h _\#(  \delta)$
whose interior is disjoint from $B_1$ and let $z$ be the
terminal endpoint of  $\nu$.  If $ y \in X$ then $z
= h(y)$; otherwise   $y$ is in the interior of some $ 
\beta_{1,m}$ and $z$ is in the interior of $\beta_{1,m+1}$.
  
If $y \not \in \beta_{1,-1} \cup \beta_{1,0}$ then the endpoints of
$\eta$ and $ \nu$ are linked in $ B_1$ in contradiction to the fact
that the interiors of $\eta$ and $ \nu$ are disjoint and lie on the
same side of $B_1$.  We may therefore assume that $ y \in \beta_{1,-1}
\cup \beta_{1,0}$.  In this case, the endpoints of $ \eta$ and $ \nu$
bound intervals $I_\eta$ and $I_\nu$ in $ B_1$ that meet in at most one
point.  It follows that either the simple closed curve $ \eta \cup
I_\eta$ or the simple closed curve $ \nu \cup I_\nu$ is inessential in
$A$ and so bounds a disk that is disjoint from $ X$ in contradiction
to the fact that these simple closed curves are composed of two
geodesic segments.  This completes the proof that $ \beta_1$ is the
unique translation arc geodesic and hence the proof of (6).
\endproof

We conclude this section by applying Lemma~\ref{reducible} to the
specific class of annulus homeomorphisms that concern us in this
paper.  Note that the statements are purely topological and so are
independent of hyperbolic metrics used in their proofs.
   
%\begin{notn}   Suppose that  $f_\sigma : A_\sigma \to A_\sigma$ is as in Definition~\ref{annulus cover} and   that $\hat x_1,\ldots, \hat x_r$ are points in the interior of $A_{\sigma}$ such that the   \alpha(f_\sigma,\hat x_i)$'s are distinct points in $\partial A_{\sigma}$ and the $ \omega(f_\sigma,\hat x_i)$'s are distinct points in $\partial A_{\sigma}$.  Let  ${\mrA}_\sigma = \Int(A_{\sigma})$, let   $\hat X \subset {\mrA}_\sigma$ be the union of the $f_\sigma$-orbits of the $\hat x_i$'s and assume that  ${\mrA}_\sigma  \setminus \hat X$  is equipped with a complete hyperbolic structure as in section~\ref{hyperbolic}. 
%\end{notn}

\begin{cor} \label{consistent crossing} Suppose that $\sigma \in
\cR$ and that $f_\sigma : A_\sigma \to A_\sigma$ is as in
Definition~\ref{annulus cover}.  Let ${\mrA}_\sigma =
\Int(A_{\sigma})$.  Then there do not exist $\hat x_1, \hat x_2 \in
{\mrA}_{\sigma}$ such that $\alpha(f_\sigma, \hat x_1)$ and
$\omega( f_\sigma, \hat x_2)$ are contained in one component of
$\partial A_\sigma$ and $\alpha( f_\sigma, \hat x_2)$ and $\omega(
f_\sigma, \hat x_1)$ are contained in the other component of
$\partial A_\sigma$.
\end{cor}

\proof   Let $\hat X \subset {\mrA}_\sigma$ be the union of the $f_\sigma$-orbits of $\hat x_1$ and $\hat x_2$  and assume that  ${\mrA}_\sigma  \setminus \hat X$  is equipped with a complete hyperbolic structure as in section~\ref{hyperbolic}.   Let $W$ be a Brouwer subsurface as in Definition~\ref{brouwer subsurface}. Lemma~\ref{no doubling} implies that no element of $\RH$ eventually doubles.  If there exist  $\hat x_1, \hat x_2 \in A_{\sigma}$  such that $\alpha(f_\sigma, \hat x_1)$ and $\omega( f_\sigma, \hat x_2)$ are contained in one component of $\partial A_\sigma$ and $\alpha( f_\sigma, \hat x_2)$ and $\omega( f_\sigma, \hat x_1)$ are contained in the other component of $\partial A_\sigma$, then the hypotheses of Lemma~\ref{reducible}  are satisfied with $r =2$ and $c_0 =c_1 = 1$.      Lemma~\ref{twist or not} implies that condition (b) of Lemma~\ref{reducible} is satisfied and hence by item (2) of Lemma~\ref{reducible}   that  for $i=1,2$, \  $\alpha(h, x_i)$ and $\omega(h,x_i)$ belong to the same component of $\partial A_\sigma$. This contradiction completes the proof.
\endproof

 The next corollary states that if there is twisting across an annular cover then orbits that start and end on one boundary component can not get to close to the other boundary component.

\begin{cor}  \label{PB}  Suppose that $h: A \to A$   is either 
\begin{enumerate}
\item $f_\sigma:A_\sigma \to A_\sigma$ for some $\sigma \in \cR$.
\end{enumerate} 
or 
\begin{enumeratecontinue} 
\item $f_\sigma: A^c_\sigma \to A^c_\sigma$ for some horocycle $\sigma$ corresponding to an isolated end of $M$. 
\end{enumeratecontinue}
Let $\partial_0 A$ and $\partial_1 A$ be the components of $\partial
A$.  In case (2) assume that $\partial _0 A$ is the unique component
of $\partial A$ and that if $\Fix(f_{\sigma}|_{\partial_1 A})
\ne \emptyset$ then $f_\sigma$ is not isotopic to the identity rel
$\Fix( f_{\sigma}|_{\partial A})$.  Then there is a neighborhood of
$\partial_1 A $ that is disjoint from the $ h$-orbit of any $  \hat x \in A$
for which both $\alpha( h,  \hat x)$ and $\omega(h, \hat  x)$ are
contained in $\partial_0 A$.
\end{cor}

\proof If the corollary fails then there exist $\hat x_t \to \hat P
\in \partial_1 A$ with $\alpha(h, \hat x_t), \omega(h, \hat x_t)
\in \partial_0 A$.  After replacing $h$ by some $h^m$ we may assume
that the rotation number of $h|_{\partial_1A}$ is less than
$\frac{1}{4}$.   In particular there are intervals $J_1 \subset J_2 \subset J_3$ in $\partial_1 A$ connecting $\hat P$ to $h(\hat P), h^2(\hat P)$ and $h^3(\hat P)$  respectively.  These intervals will be trivial if $\hat P$ is fixed by $h$ and non-trivial otherwise. 
%In particular  there is an interval $J_3$ in $\partial_1 A$ that contains $\hat P,h(\hat P),h^2(\hat P)$ and $h^3(\hat P)$ in that order. 
 Additionally, after possibly increasing $m$ further, we may assume
(Lemma~\ref{twist or not}) that if $\eta$ is a path connecting a fixed
point in $\partial_0A$ to $\hat P$ then $h(\eta)$ is not homotopic rel
endpoints to the path obtained by concatenating $\eta$ with $J_1$. %the subpath $J_1 \subset J_3$ connecting $\hat P$ to $h(\hat P)$.  Let  $J_2 \subset J_3$ be the subpath connecting $\hat P$ to $h^2(\hat P)$.

Choose contractible neighborhoods $\hat U_i$ of $J_i$ in $A$ such that
$\hat U_1 \subset \hat U_2 \subset \hat U_3$ and such that $ h(\hat
U_i) \subset \hat U_{i+1}$. Choose lifts $P \in \ti U_1 \subset \ti
U_2 \subset \ti U_3$ in $\tiM \cup \sinfty$ and let $\ti h $ be the
lift of $h$ such that $\ti h(P) \in \ti U_1$.  After passing to a
subsequence, $\hat x_t \to \hat P$ lifts to a sequence $\ti x_t \to P$
such that $\ti x_t, \ti h(\ti x_t) \in \ti U_1$ for all $t$.  Recall
that a translation arc for $\ti x_t$ is a path from $\ti x_t$ to $\ti
h(\ti x_t)$ that intersects its $\ti h$-image only in $\ti h(\ti
x_t)$.  By Lemma 4.1 of \cite{han:fpt}, there is a translation arc
$\ti \delta_t \subset \ti U_2$ for $\ti x_t$.  Let $\hat \delta_t
\subset \hat U_2$ be the projected image of $\ti \delta_t$.  Since
$\ti h(\ti \delta_t) \cup \ti \delta_t \subset \ti U_3$, $h(\hat
\delta_t) \cap \hat \delta_t$ is the projected image of $\ti h(\ti
\delta_t) \cap \ti \delta_t$.  Thus $\hat \delta_t \subset \hat U_2$
is a translation arc for $\hat x_t$.  We now fix such a $\hat x_t$ and
drop the $t$ subscript.

Assume the notation of Lemma~\ref{reducible} applied with $r=1$, 
$  x_1 = \hat x$, $c_0 =1$ and $c_1 = 0$.   Lemma~\ref{no doubling} implies that the hypothesis of  Lemma~\ref{reducible}  are satisfied. The   homotopy streamline $B_1$ produced by item (4) of  Lemma~\ref{reducible} can be
thought of as an arc $\hat \mu$ with initial endpoint $\alpha(\hat
h, \hat x)$ and terminal endpoint $\omega(\hat h, \hat
x)$.  Let $\hat \mu_0$ be the initial subpath of $\hat \mu$ that ends
with $\hat x_1$ and let $\hat \nu \subset \hat U_1$ be a path
connecting $\hat x_1$ to $\hat P$.  The path $\hat \eta = \hat
\mu_0\hat \nu$ connects  $\alpha(\hat
h, \hat x) \in \partial_0 A$ to $\hat P \in \partial_1 A$.
%By Lemma 4.2 of \cite{han:fpt}, $\hat \delta$ is    isotopic rel $\hat X$ to a homotopy translation arc rel $\hat X$ connecting $\hat x$ to $h(\hat x)$. 
 By the uniqueness part of
Lemma~\ref{reducible} (6), $\hat \delta$ is isotopic rel $\hat X$ to
the subpath of $\hat \mu$ connecting $\hat x$ to $h(\hat x)$.
It follows that the  path $\hat \eta^{-1}h(\hat \eta)$ connecting $P$ to $h(P)$ is
homotopic rel endpoints  to $\hat \nu^{-1}\hat \delta h(\hat \nu) \subset
\hat U_2$. Hence  
$h(\hat \eta) $   is homotopic rel endpoints to $\hat \eta J_1$. This contradiction completes the proof .
 \endproof

\section{ $\omega$-lifts} \label{sec: omega lifts}

We assume throughout this section that $\cR \ne \emptyset$.    Recall from Section~\ref{sec: endpoints}  that the closure of a component of  $\tiM \setminus \ti \cR$ in $\tiM$ is called a domain.
We will assign a domain or a pair of domains to each $\ti x \in \tiM$ based on  its forward $\ti f$-orbit. %the forward $\ f$-orbit of its image $x \in M$.  
 By symmetry, we can assign a domain or a pair of domains to each $\ti x \in \tiM$ based on its backward $\ti f$-orbit.%the backward $f$-orbit  of $x$. 
  In the next section  (Corollary~\ref{pretracking})  we show that these two methods give the same  domain or pair of domains when $x$ is birecurrent.

Suppose that $\ti C$ is a domain and that $\ti \sigma \in \ti\cR$ is a frontier component of $\ti C$.  Let $I_{\ti \sigma}$ be the component of $\sinfty \setminus \Fix(\ti f_{\ti C})$ bounded by the endpoints of $\ti \sigma$.  We write $\ti \sigma_{\ti C}$ for $\ti \sigma$ equipped with the orientation which makes every point in $ I_{\ti \sigma}$ move  away  from   the backward endpoint of $\ti \sigma$ toward the forward endpoint of $\ti \sigma$ under the action of $\ti f_{\ti C}$.   Equivalently, the orientation on $\ti \sigma$ is chosen so that a turn from inside $\ti C$ along $\ti \sigma$ in the direction (left or right) of the Dehn twist of $f$ across $\sigma$ has one moving toward the forward end of $\ti \sigma_{\ti C}$.  % The direction (left or right) of twisting on $\sigma \in \cR$ and the choice of a domain $\ti C$ containing a lift $\ti \sigma$ induce an orientation on  $\ti \sigma$.  We sometimes write $\ti \sigma_{\ti C}$ for $\ti \sigma$ equipped with this orientation.  

We say that a pair of disjoint oriented distinct geodesics   in $H$ are {\em anti-parallel} if either of the following conditions are satisfied.
\begin{itemize}
\item The four endpoints  in $S_\infty$ are distinct with the pair of initial endpoints  separating  the pair of terminal endpoints.
\item The initial endpoint of one of the geodesics equals the terminal endpoint of the other.
\end{itemize}
  
 \begin{lemma} \label{orientations on sigma} The orientations on $\ti \sigma$ induced from the two domains that contain it are opposite.  % If $\ti \sigma_1$ and $\ti \sigma_2$ are components of $\partial \ti C$ that project to the same element of $\cR$ then $\ti \sigma_1$ and $\ti \sigma_2$ are anti-parallel with the orientations induced from $\ti C$.
 \end{lemma}

\proof  This follows  from the fact that left [or right]  turns from the two domains containing $\ti \sigma$ result in motion in different directions along $\ti \sigma$.  %The second  also uses  the fact that $\sigma$ separates $S^2$.
\endproof

Recall from Lemma~\ref{alpha and omega are points} that for all lifts $\ti f$ and all $\ti x \in \tiM$, $\alpha(\ti f,\ti x)$ and $\omega(\ti f, \ti x)$ are single points in $\sinfty \cap \Fix(\ti f)$.

\begin{lemma}  \label{moving on} Suppose that $\ti C_1$ and $\ti C_2$ are  domains with intersection $\ti \sigma \subset \ti \cR$,  that $\ti f_i =\ti f_{\ti C_i}$  and that $\ti x \in \tiM$.     If $\omega(\ti f_1,\ti x) \ne \ti \sigma_{\ti C_1}^+$   then   $ \omega(\ti f_2,\ti x) = \ti \sigma_{\ti C_2}^+ = \ti \sigma_{\ti C_1}^-$.     Symmetrically, if $\alpha (\ti f_1,\ti x) \ne \ti \sigma_{\ti C_1}^-$   then   $\alpha(\ti f_2,\ti x) = \ti \sigma_{\ti C_2}^- = \ti \sigma_{\ti C_1}^+$.
\end{lemma}

\proof Let $ T_{\ti \sigma}$ be the root free covering translation with axis $\ti \sigma$ and orientation induced by $\ti C_2$. Then  $\ti f_2^n = T_{\ti \sigma}^{dn} \ti f_1^n$ where   $d>0$ is the degree of Dehn twisting about $\cR$.    By hypothesis and by  Lemma~\ref{orientations on sigma},  $\omega (\ti f_1, \ti x) \ne   \ti \sigma_{\ti C_1}^+=  \ti \sigma_{\ti C_2}^- =T_{\ti \sigma}^{-} $.     Since  $\ti f_1^n(\ti x)$ converges to  $\omega(\ti f_1, \ti x)$ it follows that $T_{\ti \sigma}^{dn} \ti f_1^n(\ti x) \to T^+_{\ti \sigma}$.  This proves that $\omega (\ti f_2, \ti x) =  \ti \sigma_{\ti C_2}^+$.   
\endproof

\begin{lemma} \label{bounded distance}  There is a constant $D_1>0$ so that for all domains $\ti C$ and all $\ti x \in \tiM$ such that   $\dist(\ti x,\ti C) > D_1$,      at least one of $  \alpha(\ti f_{\ti C}, \ti x)$ and $  \omega(\ti f_{\ti C}, \ti x)$ is an endpoint of the component  $\ti \sigma$  of $\partial \ti C$ that is  closest to $\ti x$.
\end{lemma} 

\proof   Up to the action of covering translations there are   only finitely many elements of $\ti R$.  Thus, if the lemma is false  there exists a domain $\ti C$ and a frontier component $\ti \sigma$ of $\ti C$ and a sequence $\ti x_k \in \tiM$  such
  that
\begin{itemize}
\item  $\ti \sigma$ is the component of $\partial \ti C$ closest to $\ti x_k$.
\item Neither $\alpha(\ti f_{\ti C} , \ti x_k)$ nor $\omega(\ti f_{\ti C}, \ti x_k)$ is an endpoint of $\ti \sigma.$
\item $\dist(\ti x_k,\ti C) \to \infty.$
\end{itemize}

Consider the annular cover $A_\sigma$ and the induced map $f_\sigma :
A_\sigma \to A_\sigma$.  Let $\hat x_k$ be the image of $\ti x_k$ in
$A_\sigma$.  By the second item, $\alpha(\ti f_{\ti C} , \ti x_k),
\omega(\ti f_{\ti C}, \ti x_k)\in \Fix(\ti f_{\ti C}) \cap (\sinfty
\setminus T^\pm_{\ti \sigma})$; in particular $\alpha(\ti f_{\ti C} ,
\ti x_k)$ and $\omega(\ti f_{\ti C}, \ti x_k)$ belong to the same
component of $\sinfty \setminus T^\pm_{\ti \sigma}$ because
$\Fix(\hat f_{\ti C}) \cap \sinfty$ consists of  ends of $\ti C$ in $\sinfty$
and $\ti C$ lies on one side of $\sigma.$
It follows that $\alpha( f_\sigma, \hat x_k)$ and
$\omega( f_\sigma, \hat x_k)$ belong to the same component of
$\partial A_\sigma$.   From the first and third items we conclude
that every neighborhood of the other component of $\partial A_\sigma$
contains $\hat x_k$'s for all sufficiently large $k$ in contradiction
to Corollary~\ref{PB}.
\endproof

%\begin{lemma} \label{crossing arc}  Suppose that $\ti C_1$ and $\ti C_2$ are  domains with intersection $\ti \sigma \subset \ti \cR$ and   that $\ti f_i =\ti f_{\ti C_i}$.  Let    $\partial_i A_{ \sigma}$ be the boundary component that contains points that lift into the   closure of $\ti C_i$.  If neither $\alpha(\ti f_1,  \ti x)$ nor $\omega(\ti f_2,  \ti x)$ is  an endpoint of $\ti \sigma$     then $\alpha(f_\sigma, \hat x) \in \partial_1 A_{\sigma}$ and $\omega( f_\sigma, \hat x) \in \partial_2 A_{\sigma}$.
%\end{lemma}

%\proof  This is an immediate consequence of the definitions.
%\endproof

For $\ti C$ a domain and $D > 0$ we let $N_D(\ti C)$ be the set of points in $\tiM$ whose distance from $\ti C$ is less than or equal to $D$.

\begin{cor}  \label{detecting home domain} Suppose that $D_1$ is the constant of Lemma~\ref{bounded distance}, that $\ti C$ is a domain and that $\ti x \in \tiM$.    If neither  $  \alpha(\ti f_{\ti C}, \ti x)$ nor $  \omega(\ti f_{\ti C}, \ti x)$ is an endpoint of  a component     of $\partial \ti C$  then $\ti f_{\ti C}(\ti x) \in N_{D_1}(\ti C)$.
\end{cor}

\proof  This is an immediate consequence of  Lemma~\ref{bounded distance}.
\endproof

 \begin{cor} \label{canonical lift} For all $\ti x \in \tiM$ either 
 \begin{enumerate}
 \item There is a   domain $\ti C$ such that $ \omega(\ti f_{ \ti C}, \ti x)$ is not an endpoint of a  component of $\partial  \ti C$.
 \end{enumerate}
 or
 \begin{enumeratecontinue} \item There is a   component $\ti \sigma$ of $\ti \cR$ such that both $ \omega(\ti f_{\ti  C_1}, \ti x)$ and $\omega(\ti f_{\ti  C_2}, \ti x)$ are endpoints of $\ti \sigma$    where $\ti  C_1$ and $\ti  C_2$ are the two domains that contain $\ti \sigma$ in their boundaries.
 \end{enumeratecontinue}
 Morever,  if (1) is satisfied then $\ti C$ is unique and (2) is not satisfied and if (2) is satisfied then $\ti \sigma$ is unique and (1) is not satisfied.
  \end{cor}
  
  \begin{remark}   Suppose that $A$ is a closed  $F$-invariant annulus in $S^2$ such that $\Fix(F)$ is disjoint from the interior $\mrA$ of $A$  but intersects both components of $\partial A$.  If $F|_A$ is isotopic to the identity rel $\Fix(F|_A)$ then the core curve $\sigma$ of $A$ is not an element of $\cR$ and  item (1) of Corollary~\ref{canonical lift} is satisfied for each $\ti x \in \tiM$ that projects into $\mrA \subset M$.  In the remaining case, $F|_A$ is isotopic    rel $\Fix(F|_A)$ to a non-trivial Dehn twist,       $\sigma \in \cR$ and item (2) of Corollary~\ref{canonical lift} is satisfied for each such $\ti x $.    
\end{remark}

 \begin{remark} In case (2),  we expect (but have not proven) that $ \omega(\ti f_{\ti C_1}, \ti x)$ and
$\omega(\ti f_{\ti C_2}, \ti x)$ are distinct endpoints of $\ti
\sigma$.
\end{remark}
 
 \vspace{.1in}
\noindent {\em Proof of Corollary~\ref{canonical lift}} \ \ The moreover part of Corollary~\ref{canonical lift} follows from Lemma~\ref{moving on} and the obvious induction argument .   It therefore suffices to find $ \ti C$ satisfying (1) or $\ti \sigma$ satisfying (2).  
 
 Choose a domain $\ti C_1'$.   If $  \omega(\ti f _{\ti C_1'}, \ti x)$ is not an endpoint of a  component of $\partial  \ti C_1'$ we are done.  Otherwise,  $  \omega(\ti f _{\ti C_1'}, \ti x)$ is an endpoint of a component  $\ti \sigma_1$ of  $\partial \ti C_1$ and we let $\ti C_2'$ be the domain whose intersection with $\ti C_1'$  is   $\ti \sigma_1$.  If $\ti \omega(\ti   C_2', \ti x)$ is either not the endpoint  of a  component of $\partial  \ti C_2'$ or is an endpoint of $\ti \sigma_1$  we are done.  Otherwise,  let $\ti C_3'$ be the domain whose intersection with $\ti C_2'$  is the component  $\ti \sigma_2$ of $\partial \ti C'_2$ whose endpoint set   contains $  \omega(\ti f _{\ti C_2'}, \ti x)$.   Iterating this procedure we either reach the desired conclusion or produce distinct domains $\ti C_k'$  such that   $\omega(\ti f _{\ti C_k'}, \ti x)$  is an endpoint of  $\ti \sigma_{k} = \ti C_k' \cap \ti C'_{k+1}$.   By Lemma~\ref{bounded distance}, $\alpha(\ti f _{\ti C_k'}, \ti x)$ is an endpoint of  $\ti \sigma_{k-1}$ for all sufficiently large $k$.  
 
 Let $ f_k : A_k \to  A_k$ be the homeomorphism of the annular cover  determined by $\ti \sigma_k$, let $\ti f_k =  \ti f _{\ti C_k'}$ and  let $\partial_- A_k$ and $\partial_+ A_k$ be the components of $\partial A_k$ that contain points that lift into the closure of $\ti C'_k$ and $\ti C'_{k+1}$ respectively.  As usual, $\hat x \in A_k$ is the image of $\ti x \in \tiM$.  Then  $\alpha(f_k, \hat x) \in \partial_- A_k$ and $\omega(f_k, \hat x) \in \partial_+ A_k$.  The former follows from the fact that  $\alpha(\ti f _k, \ti x) )\in \Fix(\ti f _k) \cap (\sinfty \setminus T^\pm_{\ti \sigma_k})$ and the latter from the fact that  $\omega(\ti f _{k+1} , \ti x) )\in \Fix(\ti f_{k+1}) \cap (\sinfty \setminus T^\pm_{\ti \sigma_k})$. 
 
     Choose $j < l$ so that $\ti \sigma_j$ and $\ti \sigma_l$ project to the same element   $\sigma \in \cR$ but  $\ti \sigma_k$  projects to a different element of $\cR$ for all $j < k < l$.   %Let $T$ be the covering translation satisfying $T(\sigma_l) = T(\sigma_j)$. 
%Since $\cR$ consists of a finite set of simple closed geodesics thereis an $l$ which is the first integer greater than $k$ such that  $\sigma_l = \sigma_k$.  
 Choose an   arc $\ti \tau \subset \tiM$ with one endpoint on $\ti \sigma_j$,    the other on $\ti \sigma_l$ and with interior disjoint from $\ti \sigma_j \cup \ti \sigma_l$.  Then $\ti \tau$ projects to a path $\tau \subset M$ with endpoints in $\sigma$ and with interior disjoint from $\sigma$.  Since $\sigma$ disconnects $S^2$, both ends  of $\tau$ belong to the same component  $X$ of $S^2 \setminus \cR$.  Let $Y \ne X$ be the other component of  $S^2 \setminus \cR$ that contains $\sigma$ in its closure.  The interiors of the domains $\ti C_{j+1}$ and $\ti C_l$ both project to $X$ and the interiors of $\ti C_{j }$ and $\ti C_{l+1}$ both project to   $Y$.  A covering translation $T$  satisfying $T(\sigma_j) = \sigma_l$ also satisfies $T(\ti C_{j+1}) = \ti C_{l}$ and  $T(\ti C_{j}) = \ti C_{l+1}$.  It follows that    $T \ti f_{j+1}T^{-1} = \ti f_{l}$ and $T \ti f_{j} T^{-1} = \ti f_{l+1}$.  Letting $\ti y = T(\ti x)$, we have 
 $\omega(\ti f_l , \ti y)  = T\omega(\ti f_{j+1},\ti x) )\in \Fix(\ti f_l) \cap (\sinfty \setminus T^\pm_{\ti \sigma_l})$ and 
 $\alpha(\ti f_{l+1} , \ti y)  = T\alpha(\ti f_{j},\ti x) )\in \Fix(\ti f_{l+1}) \cap (\sinfty \setminus T^\pm_{\ti \sigma_l})$.  Thus  $\omega(f_l, \hat y) \in \partial_- A_l$ and $\alpha(f_k, \hat x) \in \partial_+ A_l$, which contradicts Corollary~\ref{consistent crossing} and the fact that $\alpha(f_l, \hat x) \in \partial_- A_l$ and $\omega(f_l, \hat x) \in \partial_+ A_l$.  
 
  The process therefore terminates after finitely many steps.   \qed

\begin{defn} If Corollary~\ref{canonical lift} (1) is satisfied then we say that  $\ti C$ is the {\em $\omega$-domain for $\ti x$} and  $\ti f_{\ti C}$ is the {\em $\omega$-lift for $\ti x$}.     Otherwise, Corollary~\ref{canonical lift} (2) is satisfied and  we say that $\ti C_1$ and $\ti C_2$ are the {\em $\omega$-domains for $\ti x$} and 
 $\ti f_{\ti C_1}$ and $\ti f_{\ti C_2}$ are the {\em $\omega$-lifts for $\ti x$}. 
 \end{defn}

\begin{cor}   \label{stays close}  Let $D_1$ be the constant of Lemma~\ref{bounded distance}.
\begin{enumerate}
\item   If   $ \ti C$ is the unique $\omega$-domain for $\ti x$ then   $\ti f_{\ti C}^n(\ti x) \in N_{D_1}(\ti C)$ for all sufficiently  large $n$. 
\item   If     $\ti C_1$ and $\ti  C_2$  are $\omega$-domains for   $\ti x$ with intersection $\ti \sigma \in \ti \cR$ then $\ti f_{\ti C_i}^n(\ti x) \in N_{D_1}(\ti C_1 \cup \ti C_2)$ for $i=1,2$ and all sufficiently  large $n$. 
\end{enumerate}
\end{cor} 

\proof  If  $ \ti C$ is the unique $\omega$-domain for $\ti x$ and (1) fails then there exist  arbitrarily large $n$  such that  $\ti f_{\ti C}^n(\ti x) \not \in N_{D_1}(\ti C)$.   The component $\ti \sigma_n$ of $\partial \ti C$ that is closest to  $\ti f_{\ti C}^n(\ti x)$ takes on infinitely many values as $n \to \infty$.   By restricting to large $n$, we may  assume that $\alpha(\ti f_{\ti C}, \ti x)$ is not an endpoint of $\ti \sigma_n$.  By hypothesis,  $\omega(\ti f_{\ti C}, \ti x)$ is not an endpoint of $\ti \sigma_n$. This contradiction to   Lemma~\ref{bounded distance}  completes the proof of (1).

Suppose now that  (2) fails.   Since $\ti f_{\ti C_1}$ and $\ti f_{\ti C_2}$ differ by an iterate of $T_{\ti \sigma}$ and since  $T_{\ti \sigma}$ preserves both $\ti C_1$ and $\ti C_2$,  it follows that $\ti f_{\ti C_1}^n(\ti x) \not \in N_{D_1}(\ti C_1 \cup \ti C_2)$ if and only if  $\ti f_{\ti C_2}^n(\ti x) \not \in N_{D_1}(\ti C_1 \cup \ti C_2)$.   We may therefore assume that   there exist  arbitrarily large $n$  such that    $\ti f_{\ti C_1}^n(\ti x) \not \in N_{D_1}(\ti C_1 \cup \ti C_2)$ and such that the component $\ti \sigma_n$ of $\partial \ti C$ that is closest to  $\ti f_{\ti C_1}^n(\ti x)$ is not $\ti \sigma$.  Since $\ti f_{\ti C_1}^n(\ti x)$ converges to an endpoint of $\ti \sigma$, $\ti \sigma_n$  takes on infinitely many values as $n \to \infty$.   By restricting to large $n$, we may  assume that $\alpha(\ti f_{\ti C_1}, \ti x)$ is not an endpoint of $\ti \sigma_n$.  This contradicts Lemma~\ref{bounded distance} and the assumption that $\omega(\ti f_{\ti C_1}, \ti x)$ is   an endpoint of $\ti \sigma \ne \ti \sigma_n$.   
%Assume the notation of (2) and that  (2) fails for $i =1$; the $i=2$ case is symmetric.   There exist  arbitrarily large $n$  such that    $\ti f_{\ti C_1}^n(\ti x) \not \in N_{D_1}(\ti C_1 \cup \ti C_2)$.  Lemma~\ref{bounded distance} implies that $\ti \sigma$ is the component of $\partial \ti C_1$ that is closest to $\ti f_{\ti C_1}^n(\ti x)$ for all sufficiently large $n$.  Since $\ti f_{\ti C_1}$ and $\ti f_{\ti C_2}$ differ by an iterate of $T_{\ti \sigma}$ and since  $T_{\ti \sigma}$ preserves both $\ti C_1$ and $\ti C_2$, $\ti f_{\ti C_2}^n(\ti x)\not \in N_{D_1}(\ti C_1 \cup \ti C_2)$ and $\ti \sigma$ is not the  component of $\partial \ti C_2$ that is closest to $\ti f_{\ti C_2}^n(\ti x)$.  This contradicts Lemma~\ref{bounded distance} and so completes the proof of (2).
\endproof

 %\proof   Suppose that  $ \ti C$ is the unique $\omega$-domain for $\ti x$.    If $\alpha(\ti f_{\ti C}, \ti x)$ is  not an endpoint of some $\ti \sigma \in \ti R$ then (1) follows from Lemma~\ref{bounded distance}.  If $\alpha(\ti f_{\ti C}, \ti x)$ is an endpoint of  $\ti \sigma \in \ti R$, let $Z$ be the component of the complement of $\ti C$ whose nearest frontier component is $\ti \sigma$.    If $n$ is sufficiently large  then  $\ti f_{\ti C}^n(\ti x) \not \in Z$.  For all such $n$,   Lemma~\ref{bounded distance} implies that $\ti f_{\ti C}^n(\ti x) \in N_{D_1}(\ti C)$.    

% Item (1)(b) is an immediate consequence of Lemma~\ref{bounded distance}.  For (1)(a), we may assume that  $\alpha(\ti f_{\ti C}, \ti x)$ is an endpoint of some $\ti \sigma \in \ti R$.    Let $Z$ be the component of the complement of $\ti C$ whose nearest frontier component is $\ti \sigma$.      If $n$ is sufficiently large  then  $\ti f_{\ti C}^n(\ti x) \not \in Z$.  For all such $n$,   Lemma~\ref{bounded distance} implies that $\ti f_{\ti C}^n(\ti x) \in N_{D_1}(\ti C)$.    

%The proof of (2) is similar.
%\endproof

We record the following observation for easy reference.

\begin{lemma} \label{omega condition}  If $\ti f_{\ti C}^{k_i}(\ti x) \in N_D(\ti C)$  for some $D > 0$ and some $k_i \to \infty$  then $\ti C$ is an $\omega$-domain for $\ti x$.
\end{lemma}
 
\proof It suffices to show that if $\omega(\ti f_{\ti C},\ti x)$ is an
endpoint of $\ti \sigma \in \ti \cR$ and $C'$ is the other domain
whose frontier contains $\ti \sigma$ then $\omega(\ti f_{\ti C'},\ti
x)$ is an endpoint of $\ti \sigma$.  The covering translation $T_{\ti
  \sigma}$ preserves $N_D(\ti C)$.  Since $\ti f_{\ti C'}^{k_i}$ and
$\ti f_{\ti C}^{k_i}$ differ by an iterate of $T_{\ti \sigma}$, it
follows that $\ti f_{\ti C'}^{k_i}(\ti x) \in N_D(\ti C)$ and hence
that $\omega(\ti f_{\ti C'},\ti x)$ lies in the Cantor set of
ends of $\ti C$ and in the ends of $\ti C'$.  Since the ends of $\it \sigma$
are the only points in the intersection of these Cantor sets, 
$\omega(\ti f_{\ti C'},\ti x)$ is an  endpoint of $\ti \sigma$.
\endproof

\section{Domain Covers} \label{sec:domain covers}

% We assume throughout this section that $\cR \ne \emptyset$.

Let $\ti C$ be a domain and let $C$ be its image in $S$.    Recall that  $\Stab(\ti C)$ is the subgroup of covering
translations that preserve $\ti C$ and  that elements of 
$\Stab(\ti C)$ commute with $\ti f_{\ti C}$.  We cannot
restrict $f$ to $C$ because $C$ is not $f$-invariant and we can not
replace $f$ by an isotopic map that preserves $C$ because we might
lose the entropy zero property.  Instead we lift to the $\pi_1(C)$
cover $\bar C$ of $S$.  More precisely we make the following definitions.

\begin{defns} \label{bar construction}
Define $  \bar C$ to be the quotient space
of $ \tiM$ by the action of $\Stab(\ti C)$ and  $ \bar f_C : \bar C
\to \bar C$ to be the homeomorphism induced by $\ti f_{\ti C}$.   
Up to conjugacy,   $ \bar f_C : \bar C \to \bar C$ is independent of
the choice of lift $\ti C$ of $C$.   Define $\bar C_{\core} \subset \bar C$ to be the quotient space
of $\ti C \subset \tiM$ by the action of $\Stab(\ti C)$. 
\end{defns}

\begin{snotn}  Our convention will be that if
$\ti x \in \ti C$ then its image in $M$ is $x$ and its image in $\bar
C$ is $\bar x$.  
\end{snotn}

Note that $\bar C_{\core}$ is homeomorphic to $C$ and that if $\cR \ne \emptyset$ then  (topologically) $\bar C$ is obtained from $ \bar C_{\core}$ by adding collar neighborhoods to  each component of $\partial \bar C_{\core}$.  
Note also  that $\bar f_C$ is isotopic to the identity.  
\vspace{.1in}

If $\ti C$ is both an $\alpha$-domain and an $\omega$-domain for $\ti x$ then we say that $\ti C$ is a {\em home domain} for $\ti x$.   Denote the set of birecurrent points for $f$ and $\bar f_C$ by $\B (f)$ and $\B(\bar f_C)$ respectively.  Denote  the  full pre-image in $\tiM$  of $\B(f)$ by  $\ti \B(f)$.
The following proposition, whose proof is delayed until  the end of the section, is the main result of this section.

 \begin{prop}   \label{home lift}   If  $\ti C$ is an $\omega$-domain for $\ti x \in  \ti \B(f)$  then  $\bar x \in \B(\bar f_C)$ and $\ti C$ is a home  domain for $\ti x$.  Moreover if $\ti \omega(\ti f_{\ti C}, \ti x)$ is an endpoint of $\ti \sigma \in \ti \cR$ then $\ti \alpha(\ti f_{\ti C}, \ti x)$ is also an endpoint of $\ti \sigma$.
\end{prop}

As an immediate corollary we have

\begin{cor} \label{pretracking} For each  $\ti x \in \ti B(f)$ one of the following is satisfied.
\begin{enumerate}
\item There is a unique home domain $\ti C$ for $\ti x$; neither $\ti \alpha(\ti f_{\ti C}, \ti x)$  nor $\ti \omega(\ti f_{\ti C}, \ti x)$ is the endpoint of a component of $\partial \ti C$.
\item There are two home domains $\ti C_1$ and $\ti C_2$ for $\ti x$.  The intersection $\ti C_1 \cap \ti C_2$ is a component $\ti \sigma$ of $\ti \cR$ and for $i=1,2$, both $\ti \alpha(\ti f_{\ti C_i}, \ti x)$  and $\ti \omega(\ti f_{\ti C_i}, \ti x)$ are endpoints of $\ti \sigma$.
\end{enumerate}
\end{cor}

The following definition is   key to  the proof of Proposition~\ref{home lift}.

 \begin{defn}  \label{defn:near cycle} A covering translation $T : \tiM \to \tiM$ is a {\em near-cycle of period $m$ for $\ti x \in \tiM$ with respect to $\ti f_{\ti C}$} if there is    a free disk $U$ for $f$   and a lift $\ti U$ that contains $\ti x$ such that $\ti f_{\ti C}^m(\ti x) \in T(\ti U)$.  If  $m$ is  irrelevant then we simply say that {\em $T$ is a near-cycle  for $\ti x \in \tiM$ with respect to $\ti f_{\ti C}$}. 
\end{defn}

\begin{remark} \label{near cycle is open} It is an immediate
consequence of the definitions that if $T : \tiM \to \tiM$ is a
near-cycle of period $m > 0$ with respect to $\ti f_{\ti C}$ for $\ti x$
then it is also a near-cycle of period $m$ with respect to $\ti f_{\ti C}$ 
for all points in a neighborhood of $\ti x$.  Moreover, it is clear
that by shrinking the free disk $U$ slightly to $U_0,$ 
we may assume that $cl(U_0)$ is contained in a free disk and we still
have $\ti f_{\ti C}^m(\ti x) \in T(\ti U_0).$
\end{remark}

\begin{remark} A point $\ti x \in \tiM$ has at least one near cycle with respect to $\ti f_{\ti C}$ if and only if its image $x\in M$ is free disk recurrent. 
\end{remark}

\begin{remark} The only near-cycles for $\ti x \in \tiM$ with respect
to $\ti f_{\ti C}$ that we make use of are those that are contained in
$\Stab(\ti C)$. 
\end{remark}

The following lemma is essentially the same as Lemma~10.5 of \cite{fh:periodic}.  We reprove it here because our assumptions have changed.

\begin{lemma} \label{near cycle} If $T \in \Stab(\ti C)$ is a near-cycle  for $\ti x \in \tiM$ with respect to $\ti f_{\ti C}$     then $\alpha(\ti f_{\ti C}, \ti x)$ and $ \omega(\ti f_{\ti C}, \ti x)$  can not both lie in the same component of $S_{\infty} \setminus \{T^+,T^-\}$.
\end{lemma}

\proof If $T$ is parabolic let $\ti \sigma$ be a horocycle preserved
by $T$; otherwise let $\ti \sigma$ be the axis of $T$.  From $T \in
\Stab(\ti C)$ it follows that $\ti \sigma$ is either an element of
$\ti \cR$ or disjoint from $\ti \cR$.  Let $f_\sigma :A_\sigma \to
A_\sigma$ be as in Definition~\ref{annulus cover}.  We assume
the result is false and argue to a contradiction.  By Lemma~\ref{no doubling}, we may apply 
Lemma~\ref{reducible} with $h = f_\sigma$, $r=1$ and $\hat x_1$ the
image of $\ti x$ in $A_\sigma$.  Assume   the notation of that
lemma. The lifts $\ti B_1$ and $\ti B_1'$ of $\hat B_1$ that
contain $\ti x$ and $T(\ti x)$ respectively are disjoint and $\ti
f$-invariant up to isotopy rel the orbits of $\ti x$ and $T(\ti x)$.
Lemma~8.7 (2) of \cite{fh:periodic} implies that $\ti B_1$ and $\ti
B_1'$ have parallel orientations.  But it follows from the fact
that the endpoints of $\ti B_1$ are $\alpha(\ti f_{\ti C}, \ti x)$ and
$ \omega(\ti f_{\ti C}, \ti x)$ and the endpoints of $\ti B_1' $ are
$T\alpha(\ti f_{\ti C}, \ti x)$ and $T \omega(\ti f_{\ti C}, \ti
x)$, that these four points must occur in a configuration in $\sinfty$
which implies that $\ti B_1$ and $\ti B_1'$ have anti-parallel orientations.
This contradiction completes the proof.  \endproof

%The next   lemma will be used to show that certain $\bar f_C$-orbits are birecurrent.  

\begin{remark} \label{parabolic near cycle} In the case that the covering translation $T$  is parabolic,   Lemma~\ref{near cycle} asserts that at least one of $\alpha(\ti f_{\ti C}, \ti x)$ and $ \omega(\ti f_{\ti C}, \ti x)$ must equal $T^\pm$.
\end{remark}
\begin{lemma}  \label{peripheral curve}Suppose  $\ti C$ is an $\omega$-domain for $\ti x $,  that $\omega(\ti f_{\ti C},\ti x)$ is an endpoint of $\ti \sigma \in \ti \cR$ and  that   $\bar x$ is $\bar f_C$-recurrent.       Then every near cycle $T \in \Stab(\ti C)$  for a point in the $\ti f_{\ti C}$-orbit of $\ti x$ is hyperbolic with axis $\ti \sigma$. \end{lemma}

\proof   To simplify notation we write $\ti f = \ti f_{\ti C}$.  There is no loss in assuming that $T$ is a near cycle for $\ti x$.  Let $U$ be the free disk with respect to which $T$ is defined,   let $\ti U$ be the lift of $U$ containing $\ti x$ and let $n$ satisfy $\ti f^n(\ti x) \in T(\ti U)$.    There is a neighborhood $x \in V \subset U$ such that $f^n(V) \subset U$.  Let $\ti V$ be the lift of $V$ contained in $\ti U$. By Remark~\ref{near cycle is open} we may assume that the diameter of $\ti U$ in the hyperbolic metric is finite.

If $\alpha(\ti f, \ti x)$ is an endpoint of $\ti \sigma$ then  Lemma~\ref{near cycle} and the fact that $\ti \sigma \subset \partial \ti C$ complete the proof.    Suppose then that $\alpha(\ti f, \ti x)$ is not an endpoint of $\ti \sigma$ and in particular,   $\alpha(\ti f, \ti x) \ne \omega(\ti f, \ti x)$.

Since $\bar x$ is $\bar f$-recurrent, there exist $n_i \to \infty$ and
$S_i \in \Stab(\ti C)$ such that $\ti f^{n_i}(\ti x) \in S_i(\ti V)
\subset S_i(\ti U)$.  From $\ti f^n(S_i \ti V) = S_i \ti f^n(\ti V)
\subset S_iT(\ti U)$ we see that $\ti f^n(\ti f^{n_i}(\ti x)) \in
S_iTS_i^{-1}(S_i(\ti U))$ and hence that $T_i = S_iTS_i^{-1}$ is a
near cycle for $\ti f^{n_i}(\ti x)$.  Note also that both $\ti
f^{n+n_i}(\ti x)$ and $T_i\ti f^{n_i}(\ti x)$ are contained in
$T_i(S_i(\ti U))$ and so $\dist(\ti f^{n+n_i}(\ti x),T_i\ti f^n(\ti x))$ 
is bounded independently of $n_i$.
 
If $T$, and hence each $T_i$,  is parabolic then $\omega(\ti f_{\ti C},\ti x) \ne T_i^\pm$ because $\omega(\ti f_{\ti C},\ti x)$ is an endpoint of the axis $\ti \sigma$ of a hyperbolic covering translation.     Lemma~\ref{near cycle}  (see also Remark~\ref{parabolic near cycle}) therefore implies that each  $T_i^{\pm} = \alpha(  \ti f, \ti x)$.     In this case the $T_i$'s are  iterates of a single parabolic covering translation and there is a neighborhood of $\omega(\ti f, \ti x)$ that is moved off of itself by every $T_i$.  This contradicts  $\lim \ti f^{n_i}(\ti x)  = \omega(\ti f, \ti x)$  and   $\lim T_i(\ti f^{n_i}(\ti x))  = \lim (\ti f^{n+n_i}(\ti x)) = \omega(\ti f, \ti x)$.
    We conclude that $T$ and   each  $T_i$ are hyperbolic.   Let $A_T$ be the axis of $T$ and $A_i = S_i(A_T)$  the axis of $T_i$.    
 
To complete the proof we assume that $A_T \ne \ti \sigma$ and argue to a contradiction.
   
   We claim that $A_i \ne   \ti \sigma$.   This is obvious   if  $A_T$ is not an element of $\ti \cR$ so we assume that   $A_T$ is an element of $\ti \cR$ and that $\ti \sigma = A_i = S_i(A_T)$ for some $S_i \in \Stab(\ti C)$ 
 and argue to a contradiction.   Keeping in mind that $\ti \sigma$ and $A_T$ are distinct components of the frontier of $\ti C$,   Lemma~\ref{near cycle} implies that $\alpha(\ti f, \ti x)$, which by Lemma~\ref{alpha and omega are points} is a single point in the intersection of $\sinfty$ with the closure of $\ti C$, is an endpoint of $A_T$.  The axis of $S_i$ is contained in $\ti C$ and is not $\ti \sigma$ or $A_T$.  It follows that the axis of $S_i$  is disjoint from $A_T$ and $\ti \sigma$ and  has no endpoints in common with either.  Since    $\ti \sigma = S_i(A_T)$,  the axis of $S_i$ does not separate $A_T$ from $\ti \sigma$ and so does not separate  $\alpha(\ti f, \ti x)$ from $\omega(\ti f_{\ti C},\ti x)$.   This contradicts Lemma~\ref{near cycle} applied to the near cycle $S_i$ and so completes the proof that $A_i \ne \ti \sigma$.        
   
   After passing to a subsequence we may assume that either the $A_i$'s are all the same or all different.   In the former case, $T_i$ is independent of $i$ and there is a neighborhood of $\omega(\ti f, \ti x)$ that is moved off of itself by each $T_i$.   As above this contradicts the fact that $T_i(\ti f^{n_i}(\ti x)) \to \omega(\ti f, \ti x)$.    We may therefore assume that the $A_i$'s are distinct lifts of a closed curve in $M$ and hence, after passing to a subsequence, converge to some point $Q \in \sinfty$.    If $Q \ne \omega(\ti f, \ti x)$ then there is a neighborhood of $\omega(\ti f, \ti x)$ that is moved off of itself by each $T_i$ and we have a contradiction.    Thus $Q = \omega(\ti f, \ti x)$.
   
    For sufficiently large $i$   the endpoints of $A_i$ are contained in a neighborhood  of $\omega(\ti f, \ti x)$ that does not contain $\alpha(\ti f, \ti x)$  and does not contain the other endpoint of $\ti \sigma$.  Since $A_i $ is disjoint from $\ti \sigma  $,   it does not separate $\alpha(\ti f, \ti x)$  from  $\omega(\ti f, \ti x)$.  This contradicts Lemma~\ref{near cycle} applied to $T_i$ since    neither $\alpha(\ti f, \ti x)$ nor  $\omega(\ti f, \ti x)$ is an endpoint of $A_i$.
     \endproof
     
\begin{lemma}  \label{recurrence condition} Suppose that $U$ is a free disk, that $x \in U$ is  recurrent [birecurrent] with respect to $f$ and that the set of lifts of $  U$ to $\tiM$ that intersect $\{\ti f_{\ti C}^k(\ti x): k \ge 0\}$  is finite up to the action of $\Stab(\ti C)$.  Then   $\bar x \in \bar C$   is  recurrent [birecurrent] with respect to $\bar f : \bar C \to \bar C$. 
\end{lemma}

\proof  The set of lifts of $  U$ to $\tiM$ that intersect $\{\ti f_{\ti C}^k(\ti x): k \ge 0\}$  is finite up to the action of $\Stab(\ti C)$  if and only if the set of lifts of $U$ to $\bar C$ that intersect $\{\bar f_C^k(\bar x): k \ge 0\}$  is finite.   We may therefore replace the former with the latter in the hypotheses of this lemma.

Suppose that  $x$ is recurrent.  We must prove that   $\bar x$ is recurrent and that if $x$ is recurrent with respect to $f^{-1}$ then $\bar x$ is recurrent with respect to $\bar f^{-1}$.
 
Let $\bar U_1, \ldots, \bar U_m$ be the lifts of $U$ to $\bar C$ that
intersect $\{\bar f_{C}^k(\bar x): k \ge 0\}$ and let $\bar x_j \in
\bar U_j$ be the corresponding lifts of $x$.  We may assume that $\bar
x_1 = \bar x$.  Choose a sequence $n_i \to \infty$ such that $f^{n_i}(
x) \to x $ and such that each $f^{n_i}( x) \in U$.  After passing to a
subsequence we may assume that $\bar f^{n_i}_C(\bar x_1) \in \bar U_s$
where $s$ is independent of $i$.  Then $\bar f^{n_i}_C(\bar x_1) \to
\bar x_s$ and we are done if $s=1$.  Otherwise by renumbering
we may assume that
$s=2$.  Since $\bar x_2$ is in the $\omega$-limit set of $\bar x_1$,
each point in $\{\bar f_{C}^k(\bar x_2): k \ge 0\}$ that projects to
$U$ is contained in some $\bar U_j$.  We may therefore apply the
previous argument with $\bar x_2$ in place of $\bar x_1$.  After
passing to a further subsequence we may assume that $\bar
f^{n_i}_C(\bar x_2) \to \bar x_t$ where $t \ne 2$ because $\bar
f_C^{n_i}(\bar x_1)$ is the unique point in $\bar U_2$ that projects
to $f^{n_i}(x)$.  If $t=1$ then $\bar x_1$ is in the $\omega$-limit
set of $\bar x_1$ and we are done.  Otherwise we may assume $t=3$.
After iterating this argument at most $m$ times, we have shown that
$\bar x$ is recurrent.
  
  From the recurrence of  $\bar x$, it follows that a lift of $U$ to $\bar C$ intersects $\{\bar f_C^k(\bar x): k \ge 0\}$ if and only if it intersects $\{\bar f_C^k(\bar x): k \in \Z\}$.   In particular, the set of lifts of $U$ to $\bar C$ that intersect $\{\bar f_C^{-k}(\bar x): k \ge 0\}$  is finite.   If  $x$ is recurrent with respect to $f^{-1}$ then by the above argument   $\bar x \in \bar C$   is  recurrent   with respect to $\bar f^{-1} : \bar C \to \bar C$ as desired. 
\endproof

 \begin{remark}\label{cosets}  If $\ti U$ is a lift of a disk $U$ and  $T_1,T_2$ are covering translations then   $T_1(\ti U)$ and $\ti T_2(\ti U)$  are in the same $\Stab(\ti C)$-orbit if and only if $T_2 T_1^{-1} \in \Stab(\ti C)$.  Thus a collection of lifts $\{T_m(\ti U)\}$ of $U$ is finite up to the action of $\Stab(\ti C)$  if and only if the $T_m$'s determine   only finitely many right cosets of $\Stab(\ti C)$.
 \end{remark}

\noindent{\bf Proof of Proposition~\ref{home lift}}    Let $U$ be a free disk of bounded diameter that contains $x$ and let $\ti U$ be the lift that contains $\ti x$.  

As a first case suppose that $\ti \omega(\ti f_{\ti C}, \ti x)$ is not
an endpoint of an element of $\ti \cR$. Corollary~\ref{stays close} (1)
implies that for some $D$ and all $k \ge 0$, $\ti f_{\ti C}^k(\ti x)
\in N_D(\ti C)$ or equivalently, $\bar f_{C}^k(\bar x) \in N_D(\bar
C_{\core})$.  Since $N_D(\bar C_{\core})$ is compact, it follows that $\{\bar
f_{C}^k(\bar x) \ |\ k\ge 0\}$ intersects only finitely many lifts of $U$.
Equivalently, $\{\ti f_{\ti C}^k(\ti x): k \ge 0\}$
intersects only finitely many lifts of $ U$ to $\tiM$ up to the action
of $\Stab(\ti C)$. Lemma~\ref{recurrence condition} implies that $\bar
x $, and hence $\bar f_C^k(\bar x)$ for all $k$, is   recurrent under $\bar f_C$. Since the forward $\bar f_C$-orbit of $\bar f_C^k(\bar x)$ is eventually contained in $N_D(\bar C_{\core})$, it follows that $\bar
f_{C}^k(\bar x) \in N_D(\bar C_{\core})$ for all $k$ and hence that
$\ti f_{\ti C}^k(\ti x) \in N_D(\ti C)$ for all $k$.  Lemma~\ref{omega
  condition} applied to $\ti f_{\ti C}^{-1}$ implies that $\ti C$ is an $\alpha$-domain for $\ti x$ and hence a  home domain for $\ti x$.

We assume now that $\omega(\ti f_{\ti C}, \ti x)$ is an endpoint of $\ti
\sigma \in \ti \cR$ and that $\ti C_1$ and $\ti C_2$ are the two
domains that contain $\ti \sigma$ in their frontier.  We will treat
$\ti C_1$ and $\ti C_2$ symmetrically and prove that the proposition
holds for $\ti C = \ti C_1$ and $\ti C = \ti C_2$.    Denote $\ti f_{\ti
  C_1}$ by $\ti f_1$ and $\ti f_{\ti C_2}$ by $\ti f_2$.  When near
cycles are defined with respect to $\ti f_i$ we refer to them as $\ti
f_i$-near cycles.  Let $S$ be a root-free covering translation with
axis $\ti \sigma$.  Corollary~\ref{stays close} (2) implies that $\ti
f_1^k(\ti x), \ti f_2^k(\ti x) \in N_{D}(\ti C_1 \cup \ti C_2)$ for
some $D$ and all $k \ge 0$.  We may assume without loss that $\ti U
\subset N_{D}(\ti C_1) \cap N_{D}(\ti C_2)$.

After interchanging  $\ti C_1$ with $\ti C_2$ if necessary, we may assume by Lemma~\ref{moving on} that  $\alpha(\ti f_1, \ti x)$ is  an endpoint of  $\ti \sigma$.  Lemma~\ref{near cycle}  implies that every $\ti f_1$-near cycle  $T \in \Stab(\ti C_1)$ for a point in the $\ti f_1$-orbit of $\ti x$  is an iterate of $S$.      We will apply this as follows.  If $T_1$ and $T_2$ are $\ti f_1$-near cycles for $\ti x$ and if $T_1 T_2^{-1}$ (which is a near cycle for a point in the $\ti f_1$-orbit of $\ti x$)   is an element of $\Stab(\ti C_1)$ then  $T_1 T_2^{-1}$ is an iterate of $S$.  In particular, if   $T_1$ and $T_2$ determine the same right coset of $\Stab(\ti C_1)$ then they also determine the same right coset of $\Stab(\ti C_2)$,  

%Let $\U_i$ be the set of lifts  of $U$ that intersect  $N_{D}( \ti C_i)$ and contain $\ti f_2^k(\ti x) $ for some $k \ge 0$.  Each element of $\U_i$ has the form $T(\ti U)$ for some covering translation $T$ and we let  $\cT_i$ be the set of all such $T$.   To prove that $\bar x$ is $\bar f_2$-birecurrent it suffices by Lemma~\ref{recurrence condition} to prove that $\U_1 \cup  \U_2$ is finite up to the action of $\Stab(\ti C_2)$ or equivalently (Remark~\ref{cosets})   that  $\cT_1 \cup \cT_2$ determine only finitely many  right cosets of $\Stab(\ti C_2)$.       As above, the compactness of    $N_D(\bar C_{2_{\core}})$ implies that   $\U_2$ is finite up to the action of $\Stab(\ti C_2)$.  It is therefore is sufficient to prove that $\cT_1$ determines only finitely many  right cosets of $\Stab(\ti C_2)$.   

 %  Each $T_m \in \cT_1$ is an $\ti f_2$-near cycle for $\ti x$.   Since $\ti f_2$ and $\ti f_1$ differ by an iterate of $S$,  there exists $j_m$ such that $S^{j_m}T_m$ is an $\ti f_1$-near cycle for $\ti x$.    The sets $\{S^{j_m}T_m\}$ and $\{T_m\}$ determine the same  right cosets of $\Stab(\ti C_2)$ so it suffices to show that  $\{S^{j_m}T_m\}$  determines only finitely many right cosets of $\Stab(\ti C_2)$.    We know that  $\{S^{j_m}T_m\}$  determines only finitely many right cosets of $\Stab(\ti C_1)$.   As observed above this implies that  $\{S^{j_m}T_m\}$  determines only finitely many right cosets of $\Stab(\ti C_2)$.  This completes the proof that $\bar x_2$ is $\bar f_2$-birecurrent.    

Let $\U_i$ be the set of lifts  of $U$ that intersect  $N_{D}( \ti C_i)$ and contain $\ti f_2^k(\ti x) $ for some $k \ge 0$.    To prove that $\bar x$ is $\bar f_2$-birecurrent it suffices by Lemma~\ref{recurrence condition} to prove that $\U_1 \cup  \U_2$ is finite up to the action of $\Stab(\ti C_2)$.    As above, the compactness of    $N_D(\bar C_{i_{\core}})$ implies that   $\U_i$ is finite up to the action of $\Stab(\ti C_i)$.   

Each element of $\U_i$ has the form $T(\ti U)$ for some covering translation $T$; let  $\cT_i$ be the set of all such $T$.      Each $T_m \in \cT_1$ is an $\ti f_2$-near cycle for $\ti x$.   Since $\ti f_2$ and $\ti f_1$ differ by an iterate of $S$,  there exists $j_m$ such that $S^{j_m}T_m$ is an $\ti f_1$-near cycle for $\ti x$.   Since $\U_1 =\{T_m(\ti U)\}$    is finite up to the action of $\Stab(\ti C_1)$, Remark~\ref{cosets} implies that $\{ T_m\}$, and hence $\{S^{j_m}T_m\}$,  determine only finitely many right cosets of $\Stab(\ti C_1)$.   As observed above, this implies that  $\{S^{j_m}T_m\}$, and hence  $\{T_m\}$,  determine  only finitely many right cosets of $\Stab(\ti C_2)$.     Lemma~\ref{recurrence condition} and a second application of Remark~\ref{cosets}  complete the proof that   $\bar x$ is $\bar f_2$-birecurrent.    

%To prove that $\bar x$ is $\bar f_2$-birecurrent it suffices by Lemma~\ref{recurrence condition} to prove that $\U_1 \cup  \U_2$ is finite up to the action of $\Stab(\ti C_2)$ or equivalently (Remark~\ref{cosets})   that  $\cT_1 \cup \cT_2$ determine only finitely many  right cosets of $\Stab(\ti C_2)$.       As above, the compactness of    $N_D(\bar C_{2_{\core}})$ implies that   $\U_2$ is finite up to the action of $\Stab(\ti C_2)$.  It is therefore is sufficient to prove that $\cT_1$ determines only finitely many  right cosets of $\Stab(\ti C_2)$.   

%   Each $T_m \in \cT_1$ is an $\ti f_2$-near cycle for $\ti x$.   Since $\ti f_2$ and $\ti f_1$ differ by an iterate of $S$,  there exists $j_m$ such that $S^{j_m}T_m$ is an $\ti f_1$-near cycle for $\ti x$.  Remark~\ref{cosets} implies that  The sets $\{S^{j_m}T_m\}$ and $\{T_m\}$ determine the same  right cosets of $\Stab(\ti C_2)$.   so it suffices to show that  $\{S^{j_m}T_m\}$  determines only finitely many right cosets of $\Stab(\ti C_2)$.     We know that  $\{S^{j_m}T_m\}$  determines only finitely many right cosets of $\Stab(\ti C_1)$.   As observed above this implies that  $\{S^{j_m}T_m\}$  determines only finitely many right cosets of $\Stab(\ti C_2)$.  This completes the proof that $\bar x_2$ is $\bar f_2$-birecurrent.    

Having established that $\bar x$ is recurrent for $\bar f_2^{-1}$,  there exists $m_j \to \infty$ and $T'_j \in \Stab(\ti C_2)$ such that $\ti f_2^{-m_j}(\ti x) \in T_j'(\ti U)$.   Since $\ti U$ has bounded diameter, the distance between   $\ti f_2^{-m_j}(\ti x)$ and $T'_j(\ti x)$ is  bounded independently of $j$.  It follows that $T'_j(\ti x) \to \alpha(\ti f_2, \ti x)$.  Lemma~\ref{peripheral curve} implies that each $T'_j$     is  an iterate of $S$.  We conclude that $\alpha(\ti f_2, \ti x) $ is  an endpoint of the axis  $\ti \sigma$ of $S$. This completes the proof for $\ti C_2$. 

%Since $x$ is recurrent with respect to $f^{-1}$ there are near cycles $T'_j$ with respect to $\ti f_2^{-1}$ for $\ti x$ such that $T'_j(\ti x) \to \alpha(\ti f_2, \ti x)$.  Having established that $\bar x$ is recurrent for $\bar f_2^{-1}$, Lemma~\ref{recurrence condition} and Remark~\ref{cosets} imply that, after passing to a subsequence, we may assume that all $T'_j$'s belong to the same right coset of $\Stab(\ti C_2)$.  If $\ti f_2^{-n_1}(\ti x) \in T'_1(\ti U)$ then each $T'_j{T'_1}^{-1} \in \Stab(\ti C_2)$ is an $\ti f_2$-near cycle for $\ti f_2^{-n_1}(\ti x)$.  By Lemma~\ref{peripheral curve},    $T'_j{T'_1}^{-1}$ is  an iterate of $S$.  We conclude that $\alpha(\ti f_2, \ti x) = \alpha(\ti f_2,\ti f_2^{-n_1}(\ti x))$ is also an endpoint of $\ti \sigma$.  This completes the proof for $\ti C_2$. 

 Now that we have established that  $\alpha(\ti f_2, \ti x)$ is  an endpoint of  $\ti \sigma$, this same argument can be applied to $\ti C_1$. 
 \qed

 \begin{lemma}\label{not home} \begin{enumerate} \item If $\ti C$ is not a home domain for $\ti y \in \ti B(f)$ then $\alpha(\ti f_{\ti C}, \ti y)$ and $\omega(\ti f_{\ti C}, \ti y)$ are both endpoints of the component of $\partial \ti C$ that is closest to the home domain for $\ti y$. 
\item  If $\ti y \in \ti B(f)$, $\ti C$ is any domain  and    either $\alpha(\ti f_{\ti C} \ti y)$ or $\omega(\ti f_{\ti C}, \ti y)$ is an endpoint of a frontier component $\ti \sigma$ of $\ti C$ then both $\alpha(\ti f_{\ti C} \ti y)$ and $\omega(\ti f_{\ti C}, \ti y)$ are  endpoints of   $\ti \sigma$. \end{enumerate}
\end{lemma}

\proof Item  (1) follows from  the existence of a home domain for $\ti y$,    Lemma~\ref{moving on}  and  the obvious induction argument on the number of domains that separate $\ti C$ from a home domain for $\ti y$.  Item (2) follows from (1) if $\ti C$ is not a home domain for $\ti y,$ and from  Proposition~\ref{home lift} otherwise.
\endproof

We conclude this section by strengthening Corollary~\ref{stays close}.   
 
\begin{cor}   \label{symmetric  close}  Suppose that $\ti x \in \B(\ti f)$ and that  $D_1$ is the constant of Lemma~\ref{bounded distance}.
\begin{enumerate}
\item   If $ \ti C$ is the unique home domain for $\ti x$  then    $\ti f_{\ti C}^n(\ti x) \in N_{D_1}(\ti C)$ for all  $n$. 
\item  If   $\ti C_1$ and $\ti  C_2$  are home domains for   $\ti x$ with intersection $\ti \sigma \in \ti \cR$
 then $\ti f_{\ti C}^n(\ti x) \in N_{D_1}(\ti C_1 \cup \ti C_2)$ for all  $n$. 
 \item \label{item:near the punctures}  If  $\dist(\ti x, \ti \cR) >D_1$ then the domain that contains $\ti x$ is a home domain for $\ti x$.  %If $\ti y \in \ti B(f)$ and $\dist(\ti y, \ti \cR) >D_1$ then the domain that contains $\ti y$ is a home domain for $\ti y$.
\end{enumerate}
\end{cor} 

\proof  Suppose  that $\ti C$ is the unique home domain for $\ti x$ and that  $\ti f_{\ti C}^n(\ti x) \not \in N_{D_1}(\ti C)$.      Choose $\epsilon$ less than the distance from  $\ti f_{\ti C}^n(\ti x) $ to $N_{D_1}(\ti C)$.  Proposition~\ref{home lift} implies that $\bar x \in \B(\bar f_{C})$ and hence that there exist arbitrarily large $k$ and $S_k \in \Stab(\ti C)$ such that the distance from $\ti f_{\ti C}^k(\ti x)$   to $S_k\ti f_{\ti C}^n(\ti x)$ is less than $\epsilon$.  Since $S_k$ preserves distance to $\ti C$,    $\ti f_{\ti C}^k(\ti x) \not \in N_{D_1}(\ti C)$.  This contradicts Corollary~\ref{stays close} and so completes the proof of (1).

Assuming the notation of (2), suppose that  $\ti f_{\ti C}^n(\ti x) \not \in N_{D_1}(\ti C_1 \cup \ti C_2)$.   There is no loss in assuming  that $\ti f_{\ti C}^n(\ti x)$ is closer to $\ti C_1$ than $\ti C_2$.    If $S_k \in \Stab(\ti C_1)$ then $S_k\ti f_{\ti C}^n(\ti x)$ is closer to $\ti C_1$ than $\ti C_2$ and has distance greater than $D_1$ from $\ti C_1$.  The argument given for (1) therefore applies in this context as well.

If $\dist(\ti x, \ti \cR) >D_1$ and $\ti C$ is a domain that does not contain $\ti x$ then $\ti x \not \in  N_{D_1}(\ti C)$.  Item (3) therefore follows from items (1) and (2).
\endproof
 
\section{Some Results when  $\cR= \emptyset$}

%Before restricting to the $\cR = \emptyset$ case we make a pair of general definitions.

We say that a  point $P \in \sinfty$ {\em projects to} a  puncture $c$ in $M$ if  some  (and hence every) ray in $\tiM$ that converges to $P$ projects to a ray in $M$ that converges to $c$.  Note that if $P$ is the  fixed point of a parabolic covering translation then $P$ projects to an isolated puncture in $M$.

\begin{defn}  \label{defn:iterates about 2} Suppose that $\ti C$ is a home domain for  a lift $\ti x$ of $x \in M$ and that $\alpha(\ti f_{\ti C}, \ti x) =  \omega(\ti f_{\ti C}, \ti x) =P$.  If there is a parabolic covering translation $T_P$ that fixes $P$  such that  every near cycle $S \in \Stab(\ti C)$ for  every  $\ti f_{\ti C}^k(\ti x)$ is a positive iterate of $T_P$  then we  say that {\em $\ti x$ tracks $P$}.   If $c$ is the isolated puncture in $M$ to which $P$ projects, then we also say that   {\em $x$ rotates about $c$}.  (The latter is well defined because (Corollary~\ref{canonical lift})  $\ti C$ is the unique home domain for $\ti x$.)
%Assume that $\ti C$ is a home domain for  a lift $\ti x$ of $x \in M$.   We say that {\em $x$ rotates about  an isolated puncture $c$} if for some (and hence all) lifts $\ti x \in \tiM$ there exists   $P \in \sinfty$ that projects to $c$ and a parabolic covering translation $T_P$ that fixes $P$ such that   every near cycle $S \in \Stab(\ti C)$ for every  $\ti f_{\ti C}^k(\ti x)$ is a positive iterate of $T_P$.    We also say that {\em $\ti x$ tracks $P$}.
\end{defn}

 \begin{defn}      If $\ti C$ is a home domain for   $\ti x \in \tiM$  and $\alpha(\ti f_{\ti C}, \ti x)\ \ne \omega(\ti f_{\ti C}, \ti x)$ let $\ti \gamma(\ti x)$  be the oriented geodesic  with endpoints  $\alpha(\ti f_{\ti C}, \ti x)$ and $\omega(\ti f_{\ti C}, \ti x)$.  Corollary~\ref{pretracking} implies that $\ti \gamma(\ti x)$ is independent of the choice of $\ti C$  in the case that $\ti x$ has two home domains.   Let   $\gamma(x) \subset M$ be the unoriented geodesic that is the projected image   of $\ti \gamma(\ti x)$.  We say that   {\em $\ti x$ tracks $\ti \gamma(\ti x)$ %under iteration by $\ti  f_{\ti C}$
 } and that    {\em $x$ tracks $\gamma(x)$}.    Note that $\gamma(x)$ is independent of the choice of lift $\ti x$ and the choice of home domain for $\ti x
$; the latter would not be true if we imposed an orientation on $\gamma(x)$.      
 \end{defn}

If   $f$ is isotopic to the identity then $\cR = \emptyset$, $\tiM$  is the only   domain  and there is a lift $\ti f_{\tiM}$  that commutes with all covering translations and fixes every point in $\sinfty$.  In the notation of Defintion~\ref{bar construction},   $\bar M = M$   and    $\bar f: \bar M \to \bar M$  is just  $f : M \to M$.  We sometimes refer to  $\ti f_{\tiM}$ as {\em the preferred lift} of $f$ and sometimes drop the $\tiM$ subscript.

In this section we import some results from \cite{fh:periodic} that apply to the case that $f$   is isotopic to the identity.

 \begin{lemma}  \label{import1}  Assume that $f$ is isotopic to the identity and periodic point free.  Suppose that $x \in \B(f)$,  that $\ti f :\tiM \to \tiM$ is the preferred lift to the universal cover, that   $\ti x$ is a lift of $x$   and that $\alpha(\ti f, \ti x)\ = \omega(\ti f, \ti x) =P$.    Then $P$ projects to an isolated puncture $c$ and $x$ rotates about $c$. 
\end{lemma}

\proof  This is Lemma 11.2 of  \cite{fh:periodic}.
\endproof

   \begin{lemma}  \label{import2}  Assume that $f$ is isotopic to the identity and periodic point free.   If  $x \in \B(f)$ tracks $\gamma(x)$ then  $\gamma(x)$ is   a simple closed curve.     If in addition $y \in \B(f)$ tracks $\gamma(y)$  then $\gamma(x)$ and $\gamma(y)$ are either disjoint or equal.
\end{lemma}

\proof  All references in this proof are to  \cite{fh:periodic}.  By  Lemma 10.2(1)  and Lemma 11.6(2),   $\gamma(x)$ is simple and birecurrent.   If $\gamma(x)$ is not a closed curve then by Lemma 11.6(3) there is a simple closed geodesic $\alpha$ such that $\alpha$ and $\gamma(x)$ intersect transversely and non-trivially  and such that with respect to given orientations on $\alpha$ and $\gamma(x)$ all intersections have the same intersection number.  This can not happen on a genus zero surface since $\alpha$ must separate.  Thus $\gamma(x)$ is a simple closed curve.  The second assertion of the lemma follows from  Lemma 10.2(2).
\endproof

 To make use of these lemmas in our present context we use the following consequence of  Lemmas~\ref{no doubling}, \ref{lemma:fitted family} and  \ref{reducing line}.
 
 \begin{lemma}  \label{replacement} Assume that  $h = f_\sigma: A_\sigma \to A_\sigma$  [resp. $f_\sigma :A^c_\sigma \to A^c_\sigma$] is as in Defnition~\ref{annulus cover} and  that $W$  is as in Definition~\ref{brouwer subsurface}.      If $\cT$   is a fitted family that does not disappear under iteration then there exists       an % non-peripheral
  element  $[\tau] \in \cT$ such that   $h_\#([\tau])\cap W = \{[\tau]\}$. 
\end{lemma}
 
The proofs of Lemmas~\ref{import1} and \ref{import2} in
\cite{fh:periodic} quote Lemmas 10.2, 11.2 and 11.6 of
\cite{fh:periodic}. The hypothesis that $f$ is periodic point
free is only directly applied in the proofs of those three lemmas to
prove Lemma~10.8 of \cite{fh:periodic}, whose conclusion is a
weaker version of the conclusion of Lemma~(\ref{replacement})
above.  Thus in each place that Lemma~10.8 of
\cite{fh:periodic} is applied in proving Lemmas~(\ref{import1}) and
(\ref{import2}) above we can replace it with
Lemma~(\ref{replacement}).  This justifies the following lemma.
 
 \begin{lemma}\label{still true}  Lemmas~\ref{import1} and \ref{import2} remain true if the hypothesis that $f$ is periodic point free is replaced by the hypothesis that the topological entropy of $F$ is zero.
\end{lemma}

\begin{remark}  The proof of Lemma~10.8 of \cite{fh:periodic} is a pointer to the proof of Theorem 5.5 of \cite{han:fpt}.    That theorem has three parts.  The first two state that no element of $\RH$ doubles.  The third uses the first two to prove the existence of $[\tau]$ as in Lemma~\ref{replacement}.  Thus our dividing the argument into  Lemmas~\ref{no doubling} and  \ref{lemma:fitted family} follows the original proof.
\end{remark}

\section{Two compactifications}\label{sec: fh-periodic}

We now return to the general case, allowing the possibility that $\cR = \emptyset$.  Our goal in this section is to extend Lemma~\ref{still true} to the case that $\cR \ne \emptyset$.  Our strategy is to apply  Lemma~\ref{still true} to $\bar f_{\bar C}$ which is isotopic to the identity.  Before doing so, we must address the fact that  if $\cR \ne \emptyset$ then two different compactifications of the universal cover of $\bar C$ are being used.

In the {\em extrinsic compactification}, the universal cover of $\bar C$ is  metrically  identified with the universal cover $\ti M$  of $M$, which is metrically identified with $\tiM$   and is compactified by    $\sinfty$.  The covering translations of the universal cover of $\bar C$ are identified with the subgroup $\Stab(\ti C)$ of covering translations of the universal cover of $M$;  the closure in $\sinfty$ of the fixed points of the elements of $\Stab(\ti C)$  is a Cantor set $K$ whose 
convex hull projects to $C_{core} \subset \bar C$. 
 
   In the {\em intrinsic  compactification}, $\bar C$ is viewed without regard to $M$ and is equipped with a hyperbolic structure in which the ends   corresponding to the components of $\partial C$ are  cusps.  The universal cover of $\bar C$ is then metrically identified with $\tiM$ and compactified with $\sinfty$.  In this case,   the set of fixed points of  covering translations is dense in $\sinfty$.    Topologically the intrinsic  compactification of the universal cover is obtained from the extrinsic compactification by collapsing the closure of  each   component  of $\sinfty \setminus K$ to  a point.  
   
   We have defined $\bar C$ using the extrinsic metric so that  geodesics in $\bar C_{\core}$ correspond exactly to geodesics in $C \subset M$. If  one considers $\bar f: \bar C \to \bar C$ as a homeomorphism of a punctured surface without reference to $M$, as one should do when applying   results from \cite{fh:periodic}, then the intrinsic metric   is used.  To help separate the two,   write  $g: N \to N$ for  $\bar f: \bar C \to \bar C$ when $\bar C$ has the intrinsic metric.   Since $g$ is isotopic to the identity there is a preferred lift 
$\ti g : \ti N \to \ti N$ to the universal cover that commutes with all covering translations.  The \lq identity map\rq\   $p: \ti M \to \ti N$ conjugates   $\ti f_{\ti C} : \ti M \to \ti M$ to   $\ti g : \ti N \to \ti N$.      The homeomorphism $p$, which is not an isometry, extends over the compactifying circles but not by a homeomorphism; it collapses the closure of  each   component  of $\sinfty \setminus K$ to  a point.   In particular, $p|_K$   identifies a pair of points if and only if they bound a  component of $\partial \ti C$.    

Let   $T(\ti N)$ be the group of covering translations of  $\ti N$  and let $B : \Stab(\ti C) \to T(\ti N)$ be the bijection induced by  $p$.   The  following properties are satisfied by $S,S' \in  \Stab(\ti C)$.  
\begin{description}
\item [(a)]  If $S $ is parabolic then $B(S)$ is parabolic.
\item  [(b)]   If $S $ is hyperbolic then $B(S)$ is hyperbolic unless the axis of $S$ is a component of $\partial \ti C$, in which case it is parabolic.
\item  [(c)]   If  $S^\pm= \{\alpha(\ti f_{\ti C}, \ti x), \omega(\ti f_{\ti C}, \ti x)\}$ %does not bound a component of $\partial \ti C$
 then $B(S)^\pm = \{\alpha(\ti g, \ti x), \omega(\ti g, \ti x)\}$.
 \item  [(d)]   If $S$ and $B(S)$ are hyperbolic then the axis of $S$ projects to a simple closed curve  if and only if  the axis of $B(S)$ projects to a simple closed curve.
 \item  [(e)]    If $S,S',B(S)$ and $B(S')$ are hyperbolic then the axes of $S$ and $S'$ are equal or disjoint if and only if the axes of $B(S)$ and $B(S')$ are equal or disjoint. 
\end{description}

 % $\omega(\ti f_{\ti C}, \ti x)$ and $\alpha(\ti f_{\ti C}, \ti x)$ bound a component of $\partial \ti C$ then  $\omega(\ti g, p(\ti x)) =\alpha(\ti g, p(\ti x))$.

% \begin{remark}  When applying these definitions to $g :N \to N$ in place of $f:M \to M$, $\ti C$ is all of $\ti N$ and the lift $\ti f_{\ti C}$ of $f$ is replaced by the   preferred lift $\ti g$ of $g$.
% \end{remark}

 \begin{lemma}  \label{isolated puncture}Suppose that $x \in \B(f)$,  that $\ti C$ is a home domain for  a lift $\ti x$  and that $\alpha(\ti f_{\ti C}, \ti x)\ = \omega(\ti f_{\ti C}, \ti x) =P$.    Then $P$ projects to an isolated puncture $c$ and $x$ rotates about $c$. %Then $\ti x$ iterates around $P$.
\end{lemma}
      
\proof %  If $\cR =\emptyset$ then this is    Lemma~\ref{replacement}.  

%Suppose then that $\cR \ne \emptyset$ and assume  notation as above.

Since $\alpha(\ti f_{\ti C}, \ti x)\ = \omega(\ti f_{\ti C}, \ti x)
=P$, it follows that $\alpha(\ti g, p(\ti x))= \omega(\ti g, p(\ti x))
=p(P)$.  By Lemma~\ref{still true}, \ $p(P)$ projects to an isolated
puncture $c'$ in $\bar C$ and there is a parabolic covering
translation $T'$ that fixes $p(P)$ such every near cycle for every
point in the orbit of $p(\ti x)$ is a positive iterate of $T'$.
 
If $T \in \Stab(\ti C)$ is the covering translation  corresponding to $T'$ then every near cycle in $\Stab(\ti C)$  for every point in the orbit of $\ti x$ is a positive iterate of $T$.  It suffices to show that $T$ is parabolic.  Let  $U$ be a free disk for $x$ with compact closure and let $\ti U$ be the lift that contains $\ti x$.   Since (Proposition~\ref{home lift}) $\bar x \in \B(\bar f)$, there exist $n_i, a_i,m_j,b_j \to \infty$ such that $\ti f_{\ti C}^{n_i}(\ti x) \in T^{a_i}(\ti U)$ and $\ti f_{\ti C}^{-m_j}(\ti x) \in T^{-b_j}(\ti U)$.  It follows that  $P = \alpha(\ti f_{\ti C}, \ti x) =T^-$ and $P = \omega(\ti f_{\ti C}, \ti x) =T^+$.  Thus   $T$ is parabolic and we are done.  
\endproof

   \begin{lemma}  \label{disjoint scc}  If  $x \in \B(f)$ tracks $\gamma(x)$ then  $\gamma(x)$ is   a simple closed curve.     If in addition $y \in \B(f)$ tracks $\gamma(y)$  then $\gamma(x)$ and $\gamma(y)$ are either disjoint or equal.
\end{lemma}

\proof  %If $\cR =\emptyset$ then this is    Lemma~\ref{replacement}. 
We may assume without loss that the axes of $\gamma(x)$ and $\gamma(y)$ are not components of $\partial   C$ because such curves are simple and do not transversely intersect any other geodesics in $C$.
 Lemma~\ref{still true} implies that the lemma holds with $\ti f_{\ti C}$ and $\ti x$ replaced by $\ti g$ and $p(\ti x)$.   Items (b), (d) and (e) above therefore complete the proof.
 \endproof

%Suppose then that $\cR \ne \emptyset$. Choose a lift $\ti x$ and home domain $\ti C$ for $\ti x$.  We may assume without loss that $\ti C$ is also a home domain for a lift $\ti y$ of $y$.    By Lemma~\ref{home lift}, $\bar x, \bar y \in \B(\bar f_C)$.  

%Assume  notation as in the beginning of this section.  If $\alpha(\ti g, p(\ti x))= \omega(\ti g, p(\ti x))$ then  $\alpha(\ti f_{\ti C}, \ti x)$ and $\omega(\ti f_{\ti C}, \ti x)$ bound a component of $\partial \ti C$.  In this case $ \gamma(x) \in \cR$ and the lemma is obvious.  We may therefore assume that $\alpha(\ti g, p(\ti x)) \ne \omega(\ti g, p(\ti x))$ and similarly for $\ti y$.  As in the  $\cR =\emptyset$ case, Lemma~\ref{replacement} completes the proof. \endproof

The following corollary generalizes Lemma~\ref{peripheral curve} which only applies when $\ti \gamma$ is a component of $\partial \ti C$.

\begin{cor} \label{limited near cycles} Suppose that $x\in \B(f)$,  that $\ti C$ is a home domain for  a lift $\ti x$  and that $\ti x$   tracks $\ti \gamma$.     Then   every $\ti f_{\ti C}$-near cycle $S \in \Stab(\ti C)$ for a point in the orbit of $\ti x$   is an iterate of $T_{\ti \gamma}$.  
\end{cor}

\proof     We make use of the following consequences of Lemma~\ref{disjoint scc} above and Lemmas  8.7(2), 8.9 and 8.10 of \cite{fh:periodic}.
\begin{enumerate}
\item Suppose that $\ti x, \ti z$ have $\ti C$ as  a home domain and that $\ti \gamma(\ti x)$ and $\ti \gamma(\ti y)$  are disjoint and anti-parallel.  Then $\ti x$ and $\ti z$ are not contained in any free disk for $\ti f_{\ti C}$.
\item    Suppose that $\ti x,\ti y,  \ti z$ have $\ti C$ as  a home domain, that  $\ti \gamma(\ti y)$ separates $\ti \gamma(\ti x)$  and $\ti \gamma(\ti z)$ and is anti-parallel to both lines.   Then $\ti x$ and $\ti z$ are not contained in any free disk for $\ti f_{\ti C}$.
\end{enumerate}

 We may assume without loss that $S$ is a near cycle for $\ti x$.   There exist $m > 0$ and  a lift $\ti U$ of a free disk $U \subset M$
such that $\ti x \in \ti U$ and  $\ti f_{\ti C}^m(\ti x) \in S(\ti U)$.  Let $\ti z = S(\ti x)$.  Since $S \in \Stab(\ti C)$, $S$ commutes with $\ti f_{\ti C}$.   Thus  $\ti C$ is a home domain for $\ti z$ and $S (\ti \gamma) = \ti \gamma(\ti z) \subset \ti C$.    Lemma~\ref{disjoint scc} implies that $\ti \gamma$ and $S(\ti \gamma)$ are disjoint or equal (up to perhaps a change of orientation).   In the latter case we are done so we  assume the former and argue to a contradiction.  By (1),  $\ti \gamma$ and $S(\ti \gamma)$ are parallel.  Since $M$ has genus zero there is an anti-parallel translate $S'(\ti \gamma)$ that separates $\ti \gamma$ and $S(\ti \gamma)$.    Let $\ti y = S'(\ti x)$.   We have $S' \in \Stab(\ti C)$ because  $S'(\ti \gamma) \subset \ti C$.  Thus  $S' (\ti \gamma) = \ti \gamma(\ti y)$ in contradiction to  (2).
\endproof

 \section{The Set of Annuli $\cA$}  \label{sec: annuli}

\begin{defns}\label{defn: A1} 
Let $\Gamma$ be the set of simple closed curves that are
tracked by at least one element of $\B(f)$.  For each lift
$\ti \gamma$ of $\gamma \in \Gamma$, choose a domain $\ti C$
that contains $\ti \gamma$ and let $\ti U(\ti \gamma)$ be
the set of points in $\tiM$ which have a neighborhood $\ti
V$ such that every point in $\ti V \cap \ti \B(f)$ tracks
$\ti \gamma$.  We say that $\ti C$ is a {\em home domain}
for $\ti U(\ti \gamma)$, that $\ti \gamma$ is the {\em
  defining parameter} of $ \ti U(\ti \gamma)$ and that
$T_{\ti \gamma}$ is {\em the covering translation associated
  to $\ti U(\ti \gamma)$}.

For each $\gamma \in \Gamma$ define $U(\gamma)$ to be the projected image of $\ti U(\ti \gamma)$ for any lift $\ti \gamma$.  We say that $C$  is a {\em home domain} for $U(\gamma)$ and that  $\gamma$  is the {\em defining parameter} of  $U(\gamma)$.
\end{defns} 

We show in Lemma~\ref{U covers} that $U(\gamma) \ne \emptyset$.

\begin{remark}   As the notation suggests, $\ti U(\ti \gamma)$ depends only on $\ti \gamma$ and not on the choice of $\ti C$.  Indeed, if $\ti C$ is not unique then $\ti \gamma \in \ti \cR$ and (Corollary~\ref{pretracking}) every element of  $\ti V \cap \ti \B(f)$ has exactly two home domains $\ti C$ and $\ti C'$ (where  $\ti C'$  is the other domain that contains $\ti \gamma$)  and  both $\{\alpha(\ti f_{\ti C}, \ti z), \omega(\ti f_{\ti C}, \ti z)\}$ and $ \{\alpha(\ti f_{\ti C'}, \ti z), \omega(\ti f_{\ti C'}, \ti z)\}$ are contained in $\{\ti \gamma^\pm\}$.   $U(\gamma)$ is well defined because $\ti U(S(\ti \gamma)) = S\ti U(\ti \gamma)$ for any covering translation $S$.  
\end{remark}
   
 \begin{defns}\label{defn: A2} 
Let $\C$ be the set of   isolated punctures $c$ in 
$M$ for which there is at least
one element of $\B(f)$ that rotates about $c$.  For each $P \in
\sinfty$ that projects to $c \in \C$, let $\ti C$ be the unique domain whose closure contains
$P$ and let $\ti U(P)$ be the set of points in $\tiM$ for which
there is a neighborhood $\ti V$ such that every point in $\ti V \cap
\ti \B(f)$ tracks $P$.  We say that $\ti C$ is the {\em home
  domain} for $\ti U(P)$, that $P$ is the {\em defining parameter} of
$\ti U = \ti U(P)$ and that $T_{P}$ is {\em the covering translation
associated to $\ti U(P)$}.
  
  For each $c \in \C$ define $U(c)$ to be the projected image of $\ti U(P)$ for any puncture $P$ that projects to $c$.   We say that $C$  is the {\em home domain} for $U(c)$ and that  $c$  is the {\em defining parameter} of  $U(\gamma)$.    As in the previous remark,  $U(c)$ is well defined.   We show in Lemma~\ref{U covers} that $U(c) \ne \emptyset$.
  
 Let $\ti \cA$ be the set of all $\ti U(\ti \gamma)$'s and $\ti U(P)$'s
  and let 
\[
\ti \U = \bigcup_{\ti \gamma} \ti U(\ti \gamma) \cup \bigcup_P \ti U(P).  
\]
Let $\cA$ be the set of all $U(\gamma)$'s and $U(c)$'s and let
$\U$ be the projection of $\ti \U$ into $M$.

\end{defns}  
\bigskip

\begin{lemma} \label{U tilde properties} 
\begin{enumerate}
\item  Each $\ti U \in \ti \cA$ is open   and invariant by both $T$ and $\ti f_{\ti C}$ where   $\ti C$ is a home domain for $\ti U$ and $T$ is the covering translation associated to $\ti U$.   
\item If  $\ti U, \ti U' \in \ti \cA$  have different  defining parameters  then   $\ti U \cap \ti U' = \emptyset$.
\item  If $\ti U  \in\ti  \cA$  and $S$ is a covering translation then $S(\ti U) \cap \ti U \ne \emptyset$ if and only if  $S$ is an iterate of the covering translation associated to $\ti U$.    
\item Each $U \in \cA$ is open and $f$-invariant;  if $U_1$ and $U_2$ have different defining parameters then $U_1 \cap U_2 = \emptyset$.

\end{enumerate}
\end{lemma}

\proof  (1) and (2) are immediate from the definitions.    (3)  follows from (2)  and the fact that $S$ maps the   defining parameter for $\ti U$ to the defining parameter for $S(\ti U)$.    (4) follows from (1) -  (3).   \endproof

\begin{cor}  \label{second centralizer invariance}  If $h:M \to M$
commutes with $f$ then $h$ permutes the elements of $\cA$. 
\end{cor}

\proof Since  $h_\#(\cR)$ is   a reducing set for $hfh^{-1} = f$ and since reducing sets are unique, $\cR$ is  $h_\#$-invariant.  It follows that  both $\ti \cR$ and the set of domains for $f$ are $\ti h_\#$-invariant for any lift  $\ti h :\tiM \to \tiM$ of $h$.

If $\ti C$ is a home domain for $\ti x \in \tiM$  and $\ti x$ tracks $\ti \gamma$ [resp. $P$] under iteration by $\ti f_{\ti C}$ then
 $\ti h \ti f_{\ti
  C} \ti h^{-1} = \ti f_{\ti C'}$ for some domain $\ti C'$ that is a
home domain for $\ti h(\ti x)$ and $\ti h(\ti x)$ tracks $\ti
h_{\#}(\ti \gamma)$ [resp. $\ti h(P)$] under iteration by $\ti f_{\ti C'}$.  This proves
that $h(U(\gamma)) = U(h_\#(\gamma))$. 
\endproof

As a special case, our next lemma   shows that $\B(f) \subset \U$.

\begin{lemma}  \label{U covers} If either the $\alpha$-limit set  $\alpha(f,y)$ or the $\omega$-limit set   $\omega(f,y)$ of the $f$-orbit of $y$ is non-empty,  then $y$ is contained in an element $U$ of $\cA$.    In particular, each $y \in \B(f)$ is contained in some $U \in\cA$.
\end{lemma}

\proof   The two cases are symmetric so we may assume that $\omega(f,y) \ne \emptyset$.  Choose  $z \in \omega(f,y)$ and  a free disk neighborhood $V$ of $z$ with compact closure. 
  After replacing $y$ by some $f^k(y)$, we may assume that   $y \in V$.    Since $z \in \omega(f,y)$ there exist $m_i \to \infty$ such that  $f^{m_i}(y) \to z$ and such that  each $f^{m_i}(y) \in V$.  Choose a lift $\ti V$ of $V$ and let $   \ti y, \ti z \in \ti V$ be  lifts of $y$ and $z$. 

 By Corollary~\ref{symmetric close}, the distance between a point in $ \ti \B( f)$ and a home domain for that point is uniformly bounded.     It follows that there are only finitely many home domains for elements   $\ti x_l \in \ti\B(f) \cap \ti V$  and so we may choose a sequence $\ti x_l \to \ti y$ all of which have the same home domain(s) $\ti C$ and $\ti C'$, where we allow the possibility that $\ti C  = \ti C'$.  By Corollary~\ref{symmetric close}   the distance between $\ti f_{\ti C}^{m_i}(\ti x_l)$ and $\ti C \cup \ti C'$ is uniformly bounded.    It follows that the distance between $\ti f_{\ti C}^{m_i}(\ti y)$ and $\ti C \cup \ti C'$ is uniformly bounded.  After  passing to a subsequence of the $m_i$'s and   interchanging $\ti C$ and $\ti C'$ if necessary,   we may assume that the  distance between $\ti f_{\ti C}^{m_i}(\ti y)$ and $\ti C$ is uniformly bounded.  
 
      Let $S_i$ be the covering translation such that $\ti f_{\ti C}^{m_i}(\ti y) \in S_i(\ti V)$ and note that the  distance between $S_i(\ti z)$ and $\ti C$ is uniformly bounded.       Up to the action of $\Stab(\ti C)$,  the number of translates  of $\ti z$ that have uniformly bounded distance from $\ti C$ is finite.  We may therefore choose $k>j$ such that $S = S_kS_j ^{-1}\in \Stab(\ti C)$.   Let $\ti W = S_j(\ti V)$ and let $\ti W '\subset \ti W$ be a neighborhood of $\ti f_{\ti C}^{m_j}(\ti y)$ such that $\ti f^{m_k-m_j}(\ti W')  \subset S(\ti W)$.  Then $S$ is  a $\ti f_{\ti C}$-near cycle  for every point in $\ti W'$ and in particular for $\ti f_{\ti C}^{m_j}(\ti x_l) $ for all sufficiently large $l$.  Choose such an $\ti f_{\ti C}^{m_j}(\ti x_l) $  and denote it simply by $\ti x$.

      To prove that $\ti f_{\ti C}^{m_j}(\ti y)$, and hence $\ti y$,  is contained in an element of $\ti \U$ with home domain $\ti C$ it suffices to show that if $\ti w \in \ti \B(f) \cap \ti W'$ then $\ti C$ is a home domain for $\ti w$ and  $\{\alpha(\ti f_{\ti C}, \ti x),\omega(\ti f_{\ti C}, \ti x) \}  = \{\alpha(\ti f_{\ti C}, \ti w),\omega(\ti f_{\ti C}, \ti w)\}  $.
      
         We proceed with a case analysis.     As a first case suppose that $\ti x$ tracks a geodesic $\ti \gamma(\ti x)$. Corollary~\ref{limited near cycles} implies that $S$ is an iterate of $T_{\ti \gamma(\ti x)}$.      As a first subcase suppose that $\ti C$ is a home  domain for $\ti w$.  Since $S \in \Stab(\ti C)$ is a near cycle for $\ti w$,  Lemma~\ref{isolated puncture} implies that $\alpha(\ti f_{\ti C}, \ti w) \ne \omega(\ti f_{\ti C}, \ti w)$ and  Corollary~\ref{limited near cycles} implies that $\ti w$ tracks $\ti \gamma(\ti x)$.
         
 %        $T_{\ti \gamma(\ti x)}$ is not parabolic, Lemma~\ref{isolated puncture} implies that $\ti w$ tracks some geodesic $\ti \gamma(\ti w)$.  Corollary~\ref{limited near cycles} implies that $\ti \gamma(\ti w) = \ti \gamma(\ti x)$ as desired.  
      
      The remaining subcase is that $\ti C$ is not a home  domain for $\ti w$.  Lemma~\ref{not home} implies that $\{\alpha(\ti f_{\ti C}, \ti w),\omega(\ti f_{\ti C}, \ti w)\}$ is contained in the set of endpoints for some $\ti \sigma$ in the frontier of $\ti C$.  Lemma~\ref{near cycle} then implies that $\ti \sigma = \ti \gamma(\ti x)$.  Let $\ti C'$ be the other domain that contains  $\ti \gamma(\ti x)$.  Since some iterate of $T_{\ti \gamma(\ti x)}$ is a near cycle for $\ti w$ with respect to $\ti f_{\ti C}$, the same is true with respect to $\ti f_{\ti C'}$.   Lemma~\ref{near cycle}  implies that  $\{\alpha(\ti f_{\ti C'}, \ti w),\omega(\ti f_{\ti C'}, \ti w) \} \cap \{\ti \gamma^\pm(\ti x)\} \ne \emptyset$ and Lemma~\ref{not home}  implies that  both $\alpha(\ti f_{\ti C'}, \ti w)$ and $\omega(\ti f_{\ti C'}, \ti w) $ are endpoints of $\ti \gamma(\ti x)$.  This contradicts the assumption that  $\ti C$ is not home domain for $\ti w$ and so proves that the second subcase never occurs.

By Lemma~\ref{isolated puncture}, the only remaining case is that $\alpha(\ti f_{\ti C}, \ti x) = \omega(\ti f_{\ti C}, \ti x) =P$ and that $S$ is an   iterate of $T_P$. Lemma~\ref{near cycle} implies that $P \in \{\alpha(\ti f_{\ti C'}, \ti w),\omega(\ti f_{\ti C'}, \ti w) \}$,  Lemma~\ref{not home} implies that $\ti C$ is a home domain for $\ti w$ and  Lemma~\ref{disjoint scc} implies that  $\alpha(\ti f_{\ti C'}, \ti w) = \omega(\ti f_{\ti C'}, \ti w) = P$. \endproof
           
\begin{cor} \label{interior of closure} Each $\ti U \in \ti \cA$ is the interior of its closure in $\ti M$.
\end{cor}

\proof   Since $\ti U$ is obviously contained in the interior of its closure, it  suffices to show that  if $\ti y$ is in the interior of the closure of $\ti U$ then $\ti y \in \ti U$.   Choose a neighborhood $\ti V $ of $\ti y$ that is contained in the closure of $\ti U$.    Since the elements of $\cA$ are open and either disjoint or equal and since  each $\ti z \in \ti B(f) \cap \ti V$ is contained in some element of $\cA$, it follows that $\ti B(f) \cap \ti V \subset \ti U$.   If $\ti \gamma$ [resp. $P$] is the defining parameter for $\ti U$ then each element of $\ti B(f) \cap \ti V$ tracks $\ti \gamma$ [resp. $P$].   By definition, $y \in \ti U$.
\endproof

\begin{lemma} \label{more home domains}    Let $Y = M \setminus \U$ and let $\ti Y \subset \tiM$ be the full pre-image of $Y$.\begin{enumerate}
\item For each $\ti y \in \ti Y $ there is a domain $\ti C$ that is the unique $\alpha$-domain, unique $\omega$-domain and   unique home domain for $\ti y$; both  $\alpha(\ti f_{\ti C}, \ti y)$ and $\omega(\ti f_{\ti C}, \ti y)  $ project to punctures in $M$. 
   Moreover,   $\ti y$ has a neighborhood $\ti W$ so that $\ti C$ is a home domain for   all points in $\ti W \cap \ti B(f)$.
   \item If $\ti C$ is the home domain for $\ti y \in \ti Y$ then $\ti y$ has no $\ti f_{\ti C}$-near cycles in $\Stab(\ti C)$.
\item   For any compact subset $X \subset M$ there is a constant $K_X$ such that for each $y \in Y$,  $ f^i(y) \in X$ for at most $K_X$ values of $i$.
 \item There exists $\epsilon > 0$ so that if  $\ti y_1, \ti y_2 \in \ti Y $ and $\dist(\ti y_1, \ti y_2) < \epsilon$ then $\ti y_1$ and $\ti y_2$ have the same home domain.  As a consequence, points in the same   component of $\ti Y$ have the same   home domain.
  \end{enumerate}
\end{lemma}

\proof  Suppose at first that  $\cR = \emptyset$ and hence that there is only one domain.  Items (1) and   (4) are obvious.   %Item (1) follows from Lemma~\ref{U covers}. 
Every neighborhood of $\ti y \in \ti Y$ contains points in $\ti B(f)$ that are contained in different elements of $\cA$.  Lemma ~\ref{isolated puncture} and  Corollary~\ref{limited near cycles}   imply that such points have no common  near cycles.   Item (2) therefore follows from Remark~\ref{near cycle is open}.     Item (3) follows from item (2) and the fact that every compact set has a finite cover by free disks.

We now assume that $\cR \ne \emptyset$.  Write $M$ as an increasing
sequence of compact connected subsurfaces $M_1 \subset M_2 \subset
\ldots$ such that 
$$N_{D_1}\cR \subset M_1 \qquad \text{ and } \qquad M_{i} \subset \Int(M_{i+1})$$
for all $i$  where $D_1$ is the
constant of Lemma~\ref{bounded distance}  and so that every
component of $M\setminus M_i$ contains a puncture.  Moreover we
choose $M_i$ so that for any sequence $\{V_i\}$ of components of
$M\setminus M_i$ satisfying $V_{i+1} \subset V_i$ we have
 $  f(V_{i+1}) \subset V_i.$ We also assume without loss that the frontier
$\partial M_i$ of $M_i$ is a finite union of geodesics and
horocycles.

%   Let $W_2'$ be the component of  $ \tiM \setminus M_2$ containing $\ti y$ and let $\mu \subset W'_2$ be a ray connecting $y$ to a puncture $c'$.           After replacing $y$ by some point in its forward orbit, we may assume that    $f^j(y) \in    M\setminus M_2$ for     all $j \ge 0$.    Let $W_1$ and $W_2'$ be, respectively, the components of     $ M \setminus M_1$ and $ M \setminus M_2$ that contain $y$ and let $\mu \subset W'_2$ be a ray connecting $y$ to a puncture $c'$.  Note that $f(\mu) \subset W_1$.    

    Since $y \not \in \U$,  Lemma~\ref{U covers}  implies that $  \omega(f,y) = \emptyset$ and hence that the forward  orbit of $y$ intersects each $M_i$ in a finite set.       %       After replacing $y$ by some point in its forward orbit, we may assume that $f^j(y) \in    W_1$ for some component  $W_1$ of $ \tiM \setminus M_1$ and    all $j \ge 0$.    Let $W_2'$ be the component of  $ \tiM \setminus M_2$ containing $\ti y$ and let $\mu \subset W'_2$ be a ray connecting $y$ to a puncture $c'$.      
    After replacing $y$ by some point in its forward orbit, we may assume that    $f^j(y) \in    M\setminus M_2$ for     all $j \ge 0$.    Let $W_1$ and $W_2'$ be, respectively, the components of     $ M \setminus M_1$ and $ M \setminus M_2$ that contain $y$ and let $\mu \subset W'_2$ be a ray connecting $y$ to a puncture $c'$.  Note that $f(\mu) \subset W_1$.

Given a lift $\ti y$, let $\ti C $ be the domain that contains $\ti y$
and let $\ti W'_2 \subset \ti W_1$ be the lifts that contain $\ti y$.
Since the distance from a point in $\ti W_1$ to a domain other
than $\ti C$ is greater than $D_1$, Corollary~\ref{symmetric close}
implies that $\ti C$ is a home domain for every point in $\ti B(f)
\cap \ti W_1$.  The lift $\ti \mu$ of $\mu$ that begins at $\ti y$
converges to some $Q \in \sinfty$ that belongs to the closure of $\ti
C$ because $\ti \mu$ does not cross any element of $\ti \cR$.  In
particular, $\ti f_{\ti C}(Q) = Q$.

 If $\partial W_1$ is a horocycle then $\partial \ti W_1$ is a single
lift of $\partial W_1$ with both endpoints at $Q$.  Otherwise
$\partial W_1$ is a single simple closed geodesic, $\partial  \ti W_1$ has
countably many components and the closure of $\partial  \ti W_1$ intersects
$\sinfty$ in a Cantor set that  contains $Q$.  In both cases,
$\ti W_1$ is the only lift of $W_1$ that contains $Q$ in its
closure. It follows that $\ti f_{\ti C}(\ti \mu) \subset \ti W_1$ and
in particular that $\ti f_{\ti C}(\ti y) \in \ti W_1$.

%An immediate consequence of the above claim is that    $\ti f_{\ti C}(\ti \mu) \subset \ti W_1$ and hence that $\ti f_{\ti C}(\ti y) \in \ti W_1$.
Applying this argument to $\ti f^j$ for $j \ge 2$, perhaps with $W_2'$ replaced by some other component of $M \setminus M_2$ that depends on $j$,  shows that $\ti f_{\ti C}^j(\ti y) \in \ti W_1$ for all $j \ge 0$.  There exists $J_2$ so that $f^j(y)   \in M\setminus
M_3$ for all $j \ge J_2$.  Let $\ti W_2$ be the component of $\tiM
\setminus \ti M_2$ that contains $\ti f_{\ti C}^{J_2}(\ti y)$.  By
the same argument, $\ti f_{\ti C}^j(\ti y) \in \ti W_2$ for all $j
\ge J_2$.  Continuing in this manner, we can choose a decreasing
sequence of components $\ti W_i$ of $\tiM \setminus \ti M_i$ such
that for all $i$, $\ti f_{\ti C}^j(\ti y) \in \ti W_i$ for all
sufficiently large $j$. One may therefore choose a ray $\ti \tau$ that converges to $  \omega(\ti f_{\ti C},\ti y) $ so that the terminal end of the projected ray $\tau \subset M$ lies in the complement of each $M_i$.  Thus  $\tau$  converges to a puncture $c$  which lifts to   $  \omega(\ti f_{\ti C},\ti y) $. 
%whose intersection with so that for each $M_i$,   the intersection of $M_i$ with the  projected   image $\tau \subset M$ of $\ti \tau$ is contained in a finite subpath of  $\tau $.  The terminal end of $\tau$ therefore converges to some puncture $c$  which lifts to   $  \omega(\ti f_{\ti C},\ti y) $. 
  It follows (Corollary~\ref{canonical lift}) that
$\ti f_{\ti C}$ is the unique $\omega$-lift for $\ti y$ 
and $\ti C$ is its unique $\omega$-domain.
%Corollary~\ref{symmetric close} implies that $\ti C$ is a home domain for every point in $\ti B(f) \cap \ti W_1$.

By the symmetric argument applied to $ f^{-1}$, there is a unique domain $\ti C^\ast$ that is an $\alpha$-domain for $\ti y$; moreover
%such that $\ti f_{\ti C^\ast}^{-n}(\ti y) \to P^\ast$ where $P^\ast$ projects to some end in $M$ and 
there is a neighborhood of $\ti y$ such that   $\ti C^\ast$ is a home domain for every
birecurrent point in this neighborhood.  To complete the proof of (1) it suffices to prove that 
%We claim that  
$\ti C = \ti C^\ast$.
%and hence that     $\ti C$ is a home domain for $\ti y$.  
If $\ti C \ne \ti C^\ast$,  then both $\ti C$ and $\ti C^*$ are home domains for every birecurrent point in a neighborhood of $\ti y$.  But then (Corollary~\ref{pretracking})  $y \in    U(\sigma)$ where $\sigma = \ti C \cap \ti C^\ast$   contradicting the assumption that $y$ is not contained in any $U \in \U$.    This completes the proof of (1).

%  We have now shown that $\ti C$ is a home domain for  $\ti y$.    Corollary~\ref{stays close} implies that $\ti f_{\ti C}(\ti y) \in N_{D_1}(\ti C)$ for all $i$.  This property determines $\ti C$ and so proves that $\ti y$ has a unique home domain.  Corollary~\ref{stays close} also implies that $\ti C$ is a home domain for all points in a neighborhood of $\ti y$.  This completes the proof of (1). 

Every neighborhood of $\ti y$ contains points in $\ti B(f)$ that are contained in different elements of $\U$.  Lemma ~\ref{isolated puncture} and  Lemma~\ref{limited near cycles}   imply that such points have no common $\ti f_{\ti C}$-near cycle in $\Stab(\ti C)$.  Item (2) now follows from Remark~\ref{near cycle is open}.

Any compact $X \subset M$ has a cover by finitely many, say $D$, free
disks with compact closure.  Since $N_{D_1}(\bar C_{\core})$ is a
compact subset of $\bar C$, there is a constant $L$ so that for each of these $D$  free disks $B$, there are at most $L$ disjoint lifts of $B$ to $\bar C$ that intersect  $N_{D_1}(\bar C_{\core})$.  Equivalently, there are at most $L$\ $\Stab(\ti C)$-orbits of lifts of $B$ to $\tiM$ that intersect $N_{D_1}(\ti C)$.  Item (1) and Corollary~\ref{detecting home domain} imply  that $\ti f_{\ti C}^j(\ti y) \in N_{D_1}(\ti C)$ for all $j$.
Item (2) therefore implies that  there are at most $K_X= DL$ values of $j$
such that $f^j(y) \in X$.  This proves (3).
       
It remains to prove (4).    Corollary~\ref{symmetric  close} \pref{item:near the punctures} implies that any  two elements of $\ti B(f)$ in the same component of $\tiM \setminus \ti M_1$  %two points in the same component of $M\setminus M_1$
 have the same home domain.  We may therefore assume that $\ti y_1,\ti y_2$ project into  $M_1$.   Since the forward orbit of $y_1$ intersects $M\setminus M_1$,   there exists $\epsilon(y_1)$ such that $\dist(\ti y_1, \ti y_2) < \epsilon(y_1)$  implies that   $\ti y_1$ and $\ti y_2$ have the same home domain.  Since $M_1$ is compact,  we may choose $\epsilon(y_1)$ independently of $y_1$.  This completes the proof of (4).
  %     Applying (3) to    $X= M_2$ we have that for each $y \in Y$ there exists \jmfnew{$j = j(y)$ such that  $0 \le j \le K_{M_2}$ and  $f^j(y) \in M \setminus M_2 $.} \jmfc{I changed $y_1$ to $y$ here.}  Choose $\epsilon > 0$ so that 
%\begin{itemize}
%\item     $z \notin M_2 \implies  N_{\epsilon}(z) \cap M_1 = \emptyset$.   
%\item $z \in M_2$, $1 \le j \le K_{M_2}$ and  $f^j(z) \notin M_2  \implies f^j(N_{\epsilon}(z)) \cap  M_1 = \emptyset$. 
%\end{itemize}
%  Suppose now that $ \ti C'$ is the home domain for $\ti y_1$ and that  $\dist(\ti y_1, \ti y_2) < \epsilon$.    There exists $0 \le j \le K_{M_2}$ such that $\ti f_{\ti C'} ^j(\ti y_1)$ and $\ti f_{\ti C'}^i(\ti y_2)$ belong to the same component of $\tiM \setminus \ti M_1$.  Corollary~\ref{detecting home domain} 
%\jmfnew{together with Lemma~\ref{omega condition} implies that} the domain  that contains     $\ti f_{\ti C'}^j(\ti y_1)$ and $\ti f_{\ti C'}^j(\ti y_2)$ is the home domain for these points \jmfc{Do we need to apply this to $f^{-1}$ also?} so    must be $ \ti C'$.  This implies that $\ti C'$ is also the home domain for $\ti y_2$.      
\endproof

% \begin{remark}  If $\ti \sigma \in \ti \cR$ then by Lemma~\ref{more home domains}-(1), the home domain for a point in the frontier of $\ti U(\ti\sigma)$ is one of the two   home domains of  $\ti U(\ti\sigma)$. 
% \end{remark}

\begin{cor} \label{frontier home domains} Suppose that  $\ti V$ is a component of $\ti U \in \ti \cA$ and that the union $\ti V'$ of $\ti V$ with all of its bounded complementary components has  finite area.   Then each point  in the frontier $\fr(\ti V)$ of $\ti V$ has the same home domain.
\end{cor}

\proof Choose $\epsilon > 0$ as in Corollary~\ref{more home
  domains} (4).  It suffices to show that   $\fr(\ti V)$ can not be written as a union of two non-empty sets $X_1$ and $X_2$  whose $\epsilon/2$ neighborhoods are disjoint.   We assume that such $X_1$ and $X_2$ exist and argue to a contradiction.
  
Since $\ti V'$ is simply connected it is the union of
an increasing sequence of compact disks $\{B_i, i=1\ldots \infty\}$.
Since $\ti V'$ has finite area we may assume that each $\partial B_i
\subset N_{\epsilon/2}(\fr (\ti V'))$ and hence that $\partial B_i
\cap N_{\epsilon/2}(X_1)$  and $\partial B_i\cap N_{\epsilon/2}(X_2))$
 is an open cover of $\partial B_i$.   Since $\partial B_i$ is connected one of these sets must be empty.  But this can only happen for all $B_i$ if one of the sets $X_1$ and $X_2$ is empty.

%  Since $\ti V'$ is simply connected it is the union of
%an increasing sequence of compact disks $\{B_i, i=1\ldots \infty\}$.
%Since $\ti V'$ has finite area we may assume that each $\partial B_i
%\subset N_{\epsilon}(\fr (\ti V'))$.  
%For any pair of points $\ti y_1,
%\ti y_2 \in \fr(\ti V')$ there exists $k$ such that $\partial B_k$
%intersects both $N_{\epsilon}(\ti y_1) $ and $N_{\epsilon}(\ti y_2) $.
%It follows that $N_{\epsilon}(\fr (\ti V'))$ is connected.
%Lemma~\ref{more home domains}-(4) implies that each point in $\fr(\ti
%V')$ has the same unique home domain $\ti C$ and hence,  by
%Lemma~\ref{bounded distance}, that $\fr(\ti f_{\ti C}^i(\ti V') )
%= \ti f_{\ti C}^i(\fr(\ti V')) \subset N_{D_1}(\ti C)$ for all
%$i$.  It follows that $ \ti f_{\ti C}^i(\ti V') \subset N_{D_1}(\ti
%C)$ for all $i$ which implies by Lemma~\ref{omega condition} that   $\ti C$ is a home domain for every point in $\ti V'$.

\endproof

Item (4) of  Proposition~\ref{intermediate} asserts that if $f_c : U_c \to U_c$ is the annular compactification (Notation~\ref{notn: annular comp}) of   $U \in \cA$, then a component of $\partial U_c$ corresponding to a non-singular end of $U$ contains fixed points for $f_c$.   We will prove this by viewing $U$ as an essential subannulus of the annular cover    determined by  the defining parameter of $ U  $. 

%Each $U \in \cA$ is associated to two homeomorphisms of the closed annulus.  One is the annular compactification $f_c : U_C \to U_c$ (Notation~\ref{notn: annular comp}) and the other is the annular lift  $f_\sigma : A_\sigma \to A_\sigma$ or $f_\sigma : A^c_\sigma \to A^c_\sigma$ (Definitions~\ref{annulus cover}).  We showed in   Lemma~\ref{twist or not}  that the component of $A_\sigma$ or $A^c_\sigma$ corresponding to a non-singular  end of $U$ has a fixed point.   Item (4) of the following lemma is the analogous statement for $\partial U_c$. }

\begin{defn} \label{annular projection}  If $\ti U  = \ti U(\ti \gamma)$ choose a parameterization of  the annular cover    $A_\gamma$ (see Definition~\ref{annulus cover})  as $S^1 \times [0,1]$ with $S^1$ having circumference one.  Lift this to a parameterization of $(\tiM \cup \sinfty)\setminus \ti \gamma^{\pm}$ as $\R \times [0,1]$ and let $\pi: (\tiM \cup \sinfty) \setminus \ti \gamma^{\pm} \to \R $ be projection onto the $\R$ factor.   (Alternately, one can define this directly as orthogonal  projection onto $\ti \gamma$ parameterized as $\R$ and with fundamental domain having length one.)   If $\ti U  = \ti U(P)$ where $P$ projects to an isolated end $M$ with horocycle $\tau$ define $\pi : (\tiM \cup \sinfty) \setminus P \to \R $  as above using the compactified annular cover $A_P^c = A_\tau^c$.   In both case we say that the $\pi$ is the {\em projection associated to the defining parameter of $\ti U$.}
\end{defn}

\begin{cor}  \label{cor: frontier-bound}   Suppose that $T$ is  the covering translation associated to   $\ti U \in \ti \cA$,  that $\pi$ is  the  projection associated to the defining parameter of  $\ti U$, and that $\ti C$ is a home domain for  $\ti U$.   Given   $p,q>0$ define $\ti g = T^{-p} \ti f_{\ti C}^q$.    Then   there exists $r > 0$ so that  $\pi(\ti g^r(\ti y)) < \pi (\ti y) -1$
 for all $\ti y  \in \fr(\ti U)$ for which $\ti C$ is a home domain.    
\end{cor}

\proof  
To simplify notation slightly, we let $h = f^q$ and $\ti h = \ti f_{\ti C}^q$.   Increasing $p$ makes the desired inequality easier to satisfy so we may assume that $p=1$ and  $\ti g =  T^{-1}\ti h.$   
 The goal is to prove the existence of $r$ such that
\begin{equation}\label{eqn: frontier-bound}
\pi(\ti h^r(\ti y)) < \pi (\ti y) +r-1
\end{equation}
for all  $\ti y$.

Choose compact subsurfaces $M_1 \subset M_2 \subset M$ such that  $$N_{D_1}\cR  \subset M_1 \qquad \text{ and } \qquad M_1 \subset \Int(M_2)$$ and so that the following hold for each component $W_1$ of $M \setminus M_1$ and each component $W_2$ of $M \setminus M_2$.
\begin{enumerate}
\item  $W_i$   contains  at least one puncture.
\item $\partial W_i$ is connected and is either a geodesic or a horocyle.
\item  $W_2 \subset W_1 \implies   h(W_2) \subset W_1$. 
\end{enumerate}

% Choose compact subsurfaces $M_1 \subset M_2 \subset M$  whose boundary components are either geodesics or horocycles so that the following hold for each component $W_1$ of $M \setminus M_1$ and each component $W_2$ of $M \setminus M_2$.
%\begin{enumerate}
%\item  $W_1$ and $W_2$ contain  at least one puncture.
%\item  $W_2 \subset W_1 \implies   h(W_2) \subset W_1$. 
%\end{enumerate}
The existence of $r$ is independent of the exact choice of projection $\pi$ so we may assume
\begin{enumeratecontinue}
\item  If $\ti U = \ti U(\ti \gamma)$  then $\pi$ is orthogonal projection onto $\ti \gamma$; if $\ti U = \ti U(P)$ then there is a  horocycle $\ti \nu$ whose ends converge to $P$ such that the restriction of $\pi$ to the component of $\tiM \setminus  \ti \nu$ whose closure contains  $\sinfty \setminus P$ is orthogonal projection onto  $\ti \nu$.
\end{enumeratecontinue} 
We will eventually add one more property satisfied by $M_1$.  Namely,
\begin{enumeratecontinue}

 \item  For any lift $\ti W_1$ of a component of  $M \setminus  M_1$,  any $\ti y \in \fr(\ti U)$  and for  all $J_1 < J_2$,  $$\ti h^j(\ti y) \in \ti W_1  \text { for all } %\forall 
 J_1 \le j \le J_2   \implies \pi(\ti h^{J_2}(\ti y)) - \pi(\ti h^{J_1}(\ti y)) \le 1 +(J_2-J_1)/10  $$

\end{enumeratecontinue}  
%and hence that for  any component $\ti W_1$ of $\tiM \setminus \ti M_1$, the diameter of   $\pi(\ti W_1)$   is at most the length $L$ of a fundamental domain of $\ti \gamma$. 
Assuming (5) for now, we complete the proof of the corollary. 

Suppose that $ h^j(y) \not \in M_2$ for  some $J_1<J_2$ and all $J_1 \le j \le J_2$.  Let $W_1$ be the component of $M\setminus M_1$ that contains $h^{J_1}(y)$ and let $\ti W_1$ be the lift of $W_1$  that contains $\ti h^{J_1}(\ti y)$.  Arguing exactly as in the proof of Lemma~\ref{more home domains}, we conclude that
$\ti  h^j(\ti y) \in \ti W_1$ for all $J_1 \le j \le J_2$.   By (5)   $$\pi(\ti h^j(\ti y)) \le \pi(\ti h^{J_1}(\ti y) ) +1 +(j-J_1)/10$$ for all $J_1 \le j \le J_2$. 

   By Lemma~\ref{more home domains} (3) and the assumption that $\ti y \in \fr(\ti U) $, there is a constant $K$ such that there are at most $K$ values of $j$ with $ f^j(y)  \in M_2$.  There is a constant $B$ so that $\pi(\ti h(\ti y))  < \pi(\ti y) + B$ for all $\ti y \in \tiM $. Thus
\[
\pi(\ti h^r(\ti y)) < \pi (\ti y) + K B + (K+1)  +  r/10
\]
for all $r.$ A straightforward calculation shows that
inequality~\ref{eqn: frontier-bound} therefore holds for 
\[
r > \frac{10(KB + (K+1) +1)}{9}. 
\]

It remains to verify (5).   If $\ti U = \ti U(\ti \gamma)$ then by enlarging $M_1$ we may assume that  $\gamma \subset \Int M_1$.  
Each component of $\partial \ti W_1$ is disjoint from $\ti \gamma$.  There is a component $\ti \delta$ of $\partial \ti W_1$ that separates $\ti \gamma$ from all other components of $\partial \ti W_1$.    Since $\delta$  is a simple geodesic or horocycle, $\ti \delta \cap  T_{\ti \gamma}(\ti \delta) = \emptyset$.     It follows that $\ti W_1 \cap  T_{\ti \gamma}(\ti W_1) = \emptyset$    and hence that  the diameter of $\pi(\ti W_1)$ is less than one.  (Recall that we have normalized the projection so that a fundamental domain of $\ti \gamma$ has length one.)  This completes the proof of (5) in the $\ti U = \ti U(\ti \gamma)$ case.

%Thus, each component $\ti \delta$ of $\partial \ti M_1$ has both endpoints in   the same component of $\sinfty \setminus T_{\ti \gamma}^\pm$.   Since $\delta$  is a simple geodesic or horocycle, $\ti \delta \cap  T_{\ti \gamma}(\ti \delta) = \emptyset$.     It follows that $\ti W_1 \cap  T_{\ti \gamma}(\ti W_1) = \emptyset$    and hence that  the diameter of $\pi(\ti W_1)$ is less than one.  (Recall that we have normalized the projection so that a fundamental domain of $\ti \gamma$ has length one.)  This completes the proof of (4) in the $\ti U = \ti U(\ti \gamma)$ case.

Suppose then that $\ti U =\ti U(P)$ and that $\ti \nu$ is as in (4).
Assuming without loss that $\ti \nu$ projects to a simple closed curve
$\nu \subset M_1$, the previous argument applies to all lifts of $W_1$
except the one $\ti W_1$ whose closure contains $P$.  It therefore
suffices to verify (5) for this one lift $\ti W_1$ and for this we are
allowed to enlarge $M_1$ if necessary.

 Let $c$ be the puncture that lifts to $P$. % and let $W_1'$ be the component of $M \setminus M_1$ that lifts to $\ti W_1'$.  %Then $W_1'$ is an open disk with one puncture at $c$ and with $\fr(W_1') =  \tau$.  %We may assume by Corollary~\refl{symmetric  close}\pref{near the punctures} that $\ti C$ is a home domain for every element of $\ti B(f) \cap \ti W_1'$.   
 If $U$ contains a neighborhood of $c$ then we may assume that $W_1\subset U$ in which case  $(5)$ is vacuously true.  We may therefore assume that  $U$ does not contain a neighborhood of $c$ and hence that there exist $\ti z_i \in \ti B(f)$ such that $\ti z_i\to P$ and $\ti z_i \not \in \ti U$.   By   Lemma~\ref{U covers}, $\ti z_i$ belongs to some element of $\ti \cA$ and so   $\alpha( \ti f_{\ti C},\ti z_i)$ and $\omega( \ti f_{\ti C},\ti z_i)$ 
are both  unequal to $P$.

Let $f_\nu: A_\nu^c \to A_\nu^c$ be the homeomorphism of the
compactified annular cover $A_\nu^c$ (see Definitions~\ref{annulus  cover}), 
let $\partial_1 A_\nu^c$ be the component of $\partial
A_\nu^c$ that corresponds to $c$
and let $\partial_0A_\nu^c$ be the other component of $\partial
A_\nu^c$.  The projected images $\hat z_i \in  A_\nu^c$ 
  of
$\ti z_i$ satisfy $\alpha( f_\nu, \hat z_i), \omega( f_\nu,\hat z_i)
\in\partial_0A_\nu^c$  and   any given neighborhood of $\partial_1 A_\nu^c$ contains $\hat z_i$ for all sufficiently large $i$.  Corollary~\ref{PB} therefore implies that
$\Fix( f_\nu |_{\partial A_\nu^c}) $ intersects both components of
${\partial A_\nu^c} $ and that $f_\nu$ is isotopic to the identity
relative to $\Fix( f_\nu |_{\partial A_\nu^c}) $.

 Let $\ti f_\nu: \ti A_\nu^c \to \ti A_\nu^c$ be the lift to the universal cover that fixes points in both components of $\partial \ti A_\nu^c$.   Then  $\ti f_\nu|_{\Int(\ti A_\nu^c)}$   is naturally identified with $\ti f_{\ti C}$ by construction and so $\ti h$ is naturally identified with $\ti f^q_\nu|_{\Int(\ti A_\nu^c)}$.   Since $\ti f_\nu|_{\partial_1 \ti A_\nu^c}$ has translation number zero, we can enlarge $M_1$ to arrange that   (5)  is satisfied.
\endproof

\endproof

\begin{lemma} \label{U is an annulus}Suppose that   $U \in \cA$.
\begin{enumerate}
\item  $U $ is an open annulus that is essential in $M$.
\item  If $U =U(\gamma)$ then each simple closed curve in $U$ that is essential in $U$ is isotopic to $\gamma$.  If $U =U(P)$ then each simple closed curve in $U $ that is essential in $U$  is isotopic to a horocycle surrounding the isolated end of $M$ corresponding to $P$.

\item   
If $U =U(P)$ and  $\C$ is the component of $\Fix(F)$ whose corresponding puncture in $M$ lifts to $P$ then $\C$   contains  a component of  the frontier of $U$ in $S^2$.  In other words, $U$ contains a deleted neighborhood of  $\C$. 
 
\item   Each component of $\partial U_c$ corresponding to a non-singular end of $U$  has a fixed point for $f_c$. 
\end{enumerate}
  \end{lemma}

\proof Choose $\ti U \in \ti \cA$ projecting to $U$ and let $T $ be
the covering translation associated to $\ti U$.  We will prove that
$\ti U$ is connected and simply connected.  
 The first and third items of
Lemma~\ref{U tilde properties} then imply that $U$ is an open annulus and that (2) is satisfied.  Since (2) implies that $U$ is essential in $M$, (1) is also proved. 

As part of our proof that  $\ti U$ is simply connected we will show
that each component $\ti V$ of $\ti U$  is : 
\begin{description}
\item [(a)] unbounded
\item [(b)] simply connected
\item [(c)]  $T$-invariant.
\end{description}

We verify (a) by assuming  that $\ti V$ is bounded and arguing to a contradiction. Let
$\ti f = \ti f_{\ti C}$ where (Corollary~\ref{frontier home domains})
$\ti C$ is a home domain for each point in the frontier of $\ti V$.
Since $f$ preserves area there exists $q > 0$ and a covering
translation $S$ so that $\ti f^q(\ti V) \cap S(\ti V) \ne \emptyset$.
Lemma~\ref{U tilde properties} (3) implies that $S = T^p$ for some $p
\in \Z$.  After replacing $T$ with $T^{-1}$ if necessary we may assume
that $p \ge 0$.  From the fact that $\ti f^q(\ti V )$ and $S(\ti V)$ are
both components of $\ti U$, it follows that $\ti f^q(\ti V ) = S(\ti
V) = T^p(\ti V)$.  Thus $\ti V$ is $\ti g$-invariant where $\ti g =
T^{-p}\ti f^q$.  If $p=0$ then $\ti f$ has bounded orbits (since
we are assuming $\ti V$ is bounded) and hence
fixed points by the Brouwer plane translation theorem.  Since $\ti f$
is fixed point free, $p \ne 0$.  This contradicts 
Corollary~\ref{cor: frontier-bound} and so 
completes the proof  of (a).

If (b) fails then some component of the complement of $\ti V$ is
bounded.  Thus there is a closed disk $D$  that is not contained in $\ti U$ but whose boundary is contained in $\ti U$.  By the definition of $\ti U$ there exist $\ti z \in\ti B(f) \cap D$ such that $\ti z \not \in \ti U$.  By Lemma~\ref{U covers} there is $U' \in \cA$ such that $\ti z \in \ti U' $.  But then the component of $\ti U'$ containing $\ti z$ is bounded in contradiction to (a).  This prove (b).

We next  assume that (c) fails and argue to a contradiction.   A closed curve homotopic to an iterate of $\gamma$ contains a closed curve homotopic to $\gamma$. Thus  $T^p(\ti V) \ne \ti V$ for all $p \ne 0$.  Lemma~\ref{U tilde properties}(3) implies that $\ti V$ is moved off itself by every covering translation.  In particular, $\ti V$ has finite area  because the covering projection into $M$ is injective on $\ti V$.     Define $\ti f = \ti f_{\ti C}$  where $\ti C$ is a home domain   for each point in the frontier of $\ti V$.  As in the previous argument, there exists an integer $p$ and a positive integer $q$ so that $\ti f^q(\ti V) = T^p(\ti V)$.   If $p = 0$, then $\ti V $     has recurrent points, and hence fixed points for $\ti f$, which is impossible.  Thus $p \ne 0$ and we assume without loss that $p > 0$.

 Let $\pi $ be the projection associated to the defining parameter of $\ti U$ and let 
 $\ti g = T^{-p}\ti f^q $.  Then    
 $\ti g( \ti V) = \ti V$ and by Corollary~\ref{cor: frontier-bound},  there is an $r>0$ such that $\pi(\ti g^r(\ti y)) < \pi(\ti y)  -1$ for every $\ti y$ in $\partial \ti V.$    The function $\pi \ti g^r  - \pi$ is defined on the universal cover of a compact annulus (either $A_\gamma$ or $A_P^c$ in the notation of Definition~\ref{annular projection}) and is invariant under the cyclic group of covering translations of that covering space.   It follows that   $\pi \ti g^r  - \pi$ is   uniformly continuous.   Consequently, there is $\delta>0$ such that every
$\ti x\in \ti V$ which is within $\delta$ of $\partial \ti V$ satisfies
$\pi(\ti g^r(\ti x)) <  \pi(\ti x) -1.$

Let $\ti V_{n} = \{ \ti x \in \ti V\ |\ \pi(\ti x) < -n\}.$    Then $\{\ti V_n\}_{n \ge 0}$ is a nested family whose intersection
  is empty.  Moreover,  each $\ti V_{n}$ is non-empty because $\ti V$ is $\ti
  g$-invariant and $\lim_{n\to \infty} \pi \ti g^{nr}(\ti y) =
  -\infty$ for all $\ti y \in \partial \ti V$.
Since $\ti V$ has finite area there exists $N>0$ such
  that $\ti V_N$ contains no ball of diameter $\delta$, and hence
  every point of $\ti V_{N}$ must be within $\delta$ of $\partial \ti
  V.$ We conclude the $\ti g^r(\ti V_{N}) \subset \ti
  V_{N+1} \subset V_N$.     But then $\ti g^r(\ti V_{N})$ is a proper open subset of  $V_N$ with the same finite area as $V_N$.  This contradiction  completes the proof of (c).

We have now proved that each component of $U$ contains a simple closed curve that is essential in $M$ and that all such simple closed curves in $U$ are in the same isotopy class. Moreover  if $U' \in \cA$ and $U \ne U'$ then $U$ and $U'$ do not contain isotopic simple closed curves.    If $U$ has more than one component then there is an unpunctured annulus $A$  whose boundary curves are in $U$ and whose interior intersects a component of $\fr(U)$ and hence intersects the interior of some $U' \ne U$.    It follows that $A$ contains a component of $U'$ and hence contains an essential simple closed curve not isotopic to the components of $\partial A$. This contradiction implies that $U$ and hence $\ti U$ is connected.  Item (b) therefore implies that $\ti U$ is simply connected.    This  completes the proof of (1) and (2).    %This same argument proves  (3).

A similar argument proves (3): If $U$ does not contain a  neighborhood of the puncture $c$ corresponding to $\C$ then the    once punctured disk  neighborhood of $c$  determined by a core curve $\tau $ of $U$  contains some $U' \ne U$ and hence contains an essential simple closed curve that is not isotopic to $\tau$. This contradiction proves (3).

We now consider (4). Suppose that $\partial_0 U_c$ is a component of $\partial U_c$ corresponding to a non-singular    end of $U$, meaning that the   corresponding component $Z$  of the frontier of $U$ in $S^2$ is not a single point.   The compactification of this  end of $U$ is by prime ends.    By  Lemma~\ref{lem: frontier fp}(3) we may assume that $Z \not \subset  \Fix(F)$ or equivalently that $M \cap Z \ne \emptyset.$  Let $\ti f_c : \ti U_c \to \ti U_c$ be the lift to the universal cover such that $\ti f_c|_{\ti U} = \ti f_{\ti C}|_{\ti U}$.   We will prove that $f_c |_{\partial_0U_c}$ has a fixed point by showing that  the translation number $\tau$  for $\ti f_c |_{ \partial_0\ti U_c}$ (see Definition~\ref{translation number}) is zero.    By symmetry, it suffices to  assume that $\tau >0$ and argue to a contradiction.  

Choose  a  degree one closed path  $\mu$  with embedded interior in $U$   and with   both endpoints   at $z \in M \cap Z$.   Let  $\ti \mu_0$ be a lift of the interior of  $\mu$ to $\ti U$.  Since   $\mu$ has degree one, the ends of $\ti \mu_0$ converge to  lifts   $\ti z$ and $T (\ti z)$  of $z$ in  the frontier of $\ti U$ in $\tiM$.     Denote the bounded area   component of $\ti U \setminus \ti \mu_0$  by $D_0$.   For each $k$, let $\ti \mu_k = T^k (\ti \mu_0)$  and $D_k = T^k(D_0)$.
  
 From the point of view of $\ti U_c$, $D_0$ is the interior of a
half-disk $D_0^c$ whose frontier is the union of   $\ti \mu_0$ and an interval $I_0 \subset \partial_0\ti
  U_c$ that is a fundamental domain for the action on
$\partial_0 \ti U_c$ of the covering translation 
$T_c : \ti U_c \to \ti U_c$ corresponding to $T$.  Let $D_k^c = T_c^k(D_0^c)$.
Choose $0 < p/q < \tau$,   let $\ti g =
T^{-p} \ti f_{\ti C}^q $ and let $\ti g_c = T_c^{-p} \ti f_c^q$.   Identify $\partial_0 \ti U_c$ with $\R$. Under the action of $\ti g_c$, points in $\partial_0 \ti U_c$  move in the positive direction at an average rate of $\tau - p/q > 0$.  In particular, 
 given any $\bar z$ in the interior of
$I_0$ and any $L> 0$, there exists $j > 0$ so that for any
sufficiently small half disk neighborhood $B$ of $\bar z$ in $\ti
U_c$, we have $\ti g_c^j(B) \subset D^c_l$ for some $l \ge L$.
  
From the point of view of $\pi$,  $D_0$ is not so small.   The image under $\pi$ of $\ti \mu_0$ is   bounded  so the image under $\pi$ of $\ti \mu_l$ goes to infinity with $l$.   The  frontier
of the set $B$ from the previous paragraph is the union of an interval in $\partial_0 \ti U_c$ with an
open embedded path $\ti \nu_c \subset \Int(\ti U_c)$. We may choose
$B$ so that, under the identification of $\Int(\ti U_c)$ with $\ti U$, $\ti
\nu_c$ corresponds to the interior of a path $\ti \nu$ with endpoints
  in $M \cap Z$. Corollary~\ref{cor: frontier-bound} implies
that the $\pi$-image of the endpoints of $\ti g^j(\nu)$ decrease
linearly in $j$.     Since $L$ can be arbitrarily large,  this proves that there is no uniform bound to the diameter of the  image under $\pi$ of $D_l$.  Since $T(D_l) = D_{l+1}$, this diameter is independent of $l$  and  we conclude that each  $\pi(D_l)$ is not bounded below.

\includegraphics[width = 6in]{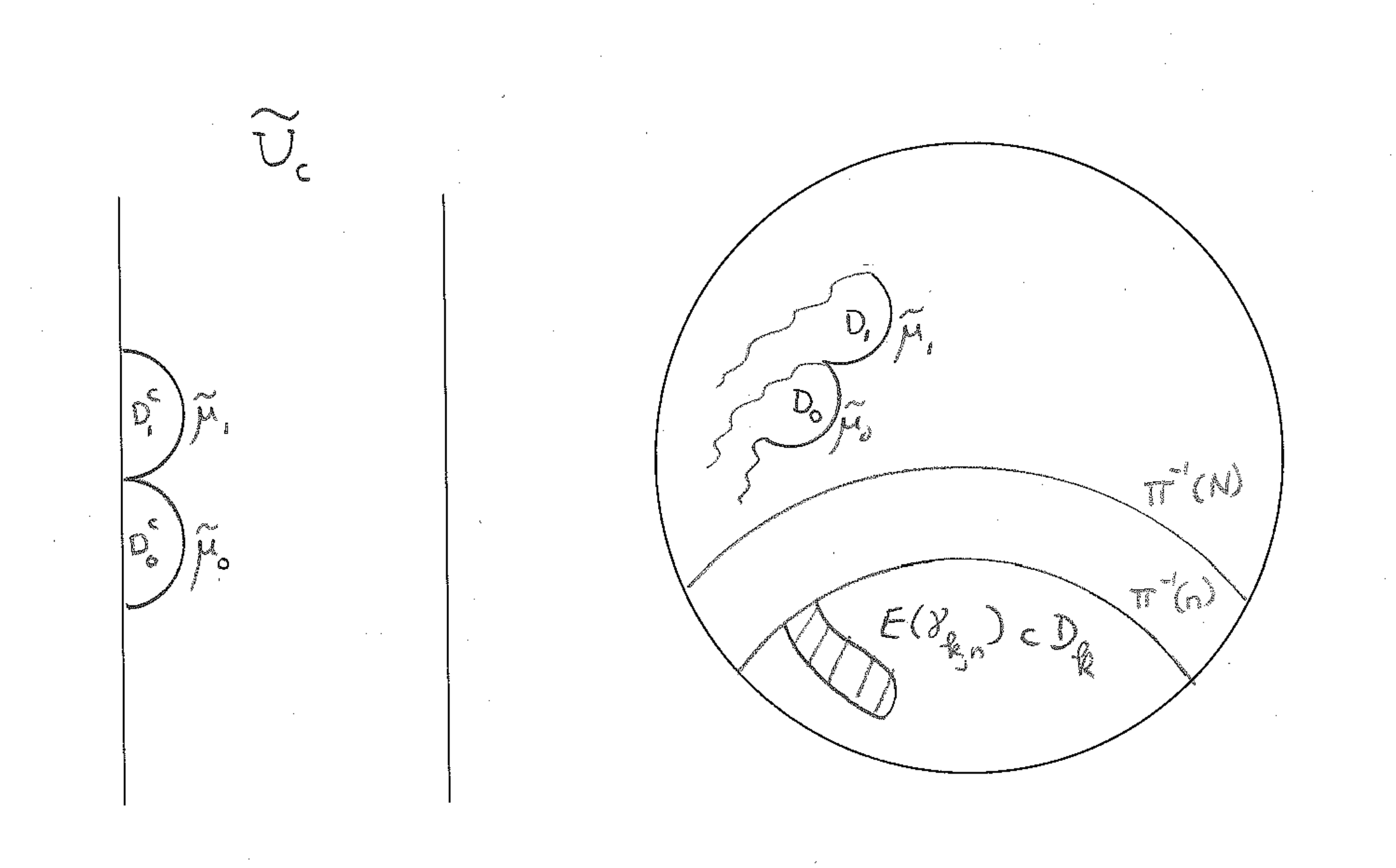}

%   Choose $0 < p/q < \rho(\partial_0 U)$ and  let $\ti g = T^{-p}\ti f^q $. A cross cut in $D_k$ is an embedded open arc $\ti \nu \subset D_k$ whose ends converge to points $\ti z_1, \ti Z_2$ in the frontier of $\ti U$ If $\nu$   is a cross cut in $D_k$ then the $\pi$-image  of the endpoints of  $\ti g^j(\nu)$ decrease linearly in $j$ by Corollary~\ref{cor: frontier-bound}.   On the other hand, for arbitrarlily large $j$ one can choose $\nu$ so that $\ti g^j(\nu)$ is a cross cut in $D_l$ where $l$ increases linearly in $j$.    It follows that $\pi(D_k) $ is not bounded below.

%    If $\nu$ is a   short cross cut in $D_k$ then $\ti g^j(\nu)$ is a cross cut in $D_{l}$ where $l$ increases  linearly in $j$.  On the other hand the image under $\pi$ of the endpoints of  $\ti g^j(\nu)$ decrease linearly in $j$ by Corollary~\ref{cor: frontier-bound}.   It follows that $\pi(D_k) $ is not bounded below.

Choose a positive integer $N$ so that $\pi(\ti \mu_0) > -N$.  For
every fixed $n>N$ and $k>0$ consider all cross cuts
$\gamma_{k,n}\subset D_k$ such that $\pi(\gamma_{k,n}) =-n$.  (In
other words,  $\gamma_{k,n}$ is a non-trivial component of the intersection of $D_k$ with the  properly embedded line  $\pi^{-1}(-n)$.)
Let $E(\gamma_{k,n})$ be the complementary component of $\gamma_{k,n}$
that is contained in $D_k$ and let $ d_{k,n}$ be the maximum area of
all such $E(\gamma_{k,n})$.  To see that this maximum is achieved, it
suffices to show that any ascending chain $E(\gamma^1_{k,n}) \subset
E(\gamma^2_{k,n}) \subset E(\gamma^3_{k,n}) \subset \dots$ is finite.
Suppose not.  Let $E$ be the union of an infinite ascending chain.
Choose $\ti w \in E(\gamma^1_{k,n})$,
choose $\ti w' \in \ti U \setminus E$ and choose
a path $\ti \rho \subset \ti U$ connecting $\ti w$ to $\ti w'$.  Then
$\ti \rho$ intersects $\gamma^i_{k,n}$ for all $i$.  Choose a point
$\ti v_i \in \ti \rho \cap \gamma^i_{k,n}$ for each $i$ and a limit
point $\ti v$ of some subsequence of the $\ti v_i$'s.  Then $\ti v \in
\ti U$ because $\ti \rho \subset \ti U$ is compact.
 However, this is impossible because $\pi^{-1}(-n)$ 
is a properly embedded line so 
the $\ti v_i$'s converge to $\ti v$ in this line and 
$\ti v$ is in one component of the open subset 
$\pi^{-1}(-n) \cap \ti U$ of this line while each $\ti v_i$ is in
a different component. This contradiction shows that $ d_{k,n}$ is well
defined.

We have 
$$
d_{k,n} = d_{k+1,n-1}  \  > \  d_{k+1,n}
$$
The equality follows from the fact that   $\gamma_{k+1,n-1} = T(\gamma_{k,n}) \subset D_{k+1}$ is a cross cut   with $\pi(\gamma_{k+1,n-1}) = -n+1$.  The inequality follows from the fact that each $E(\gamma_{k,n})$ is contained in some $E(\gamma_{k,n-1})$.  

     Fix $k$ and  choose  $\gamma_{k,n}$ so that $d_{k,n} = E(\gamma_{k,n})$.  Since $D_k$ has finite area,  we have $\lim _{n \to \infty} d_{k,n} = 0$.    Arguing as in the proof of (c), there exits $r > 0$ and $N' > N$ such that $\pi(\ti g^r(\gamma_{k,n})) <  -n-1$ for all $n > N'$.    Our choice of $N$ guarantees that $\ti g^r(\gamma_{k,n}) \cap \ti \mu_l = \emptyset$ for   $l \ge k$.  Since  the endpoints of $ \gamma_{k,n}$  move upward under the action of $\ti g^r_c$, 
      it follows that $\ti g^r(E(\gamma_{k,n}))$ is contained in $D_l$ for some $l \ge k$ and hence that   $\ti g^r(E(\gamma_{k,n}))$ is contained in  some $E(\gamma_{l,n+1})$.    This contradicts the fact that   $d_{l,n+1} < d_{k,n}$ for all $l \ge k$.       
\endproof     

\section{Proof of  Proposition~\ref{intermediate} }
 
%Having established that each $U \in \cA$ is an open annulus, we prove   Proposition~\ref{intermediate}.  
%Recall (Notation \ref{notn: annular comp}) that the annular compactification $U_c$ of $U \in \cA$ is a closed annulus and that $f_c :U_c \to U_c$ is the homeomorphism induced by $f_{|U}$.

\begin{lemma}\label{maximal annuli}  $\cA$ is the set of maximal $f$-invariant  open annuli in $M$.
\end{lemma}

\proof By Lemma~\ref{U is an annulus},  the elements
of $\cA$  are disjoint $f$-invariant open annuli.  It therefore  suffices
to show that for every $f$-invariant open annulus $V$ there
exists $U \in \cA$ such that $V \subset
U$. %, and for this it suffices, by the definition of $\cA$, to show that $V \cap \B(f) \subset U$.

If $V $ is inessential  in $M$ then the union of $V$ with one of its
complementary components in $M$ is an $f$-invariant open
disk.  Since $f$ preserves area, the Brouwer plane
translation theorem implies that this open disk contains a
fixed point which is impossible because $M$ is fixed point
free.  We conclude that $V$ is essential in $M$.

Let $\alpha$ be an essential simple closed curve in $V$ and
let $\gamma$ be either a simple closed geodesic or a
horocycle in $M$ that is isotopic to $\alpha$.  Since $V$ is
$f$-invariant, $\gamma$ is isotopic to $f(\gamma)$ and so
does not cross any reducing curves.

Choose a lift  $\ti \gamma \subset \tiM$  of $\gamma$  and let $T $ be a root free covering translation that preserves ${\ti \gamma}$.  The ends of $\ti \gamma$ converge to the (possibly equal) endpoints $T^\pm$ of $T$.
If $\gamma$ is not   a reducing curve then $\ti \gamma$ lies in
a unique domain $\ti C$.  
The lift $\ti f_1 = \ti f_{\ti C}$ of $f$ fixes  $T^\pm$ and  so commutes with $T$ by Lemma~\ref{basic lemma 2}. 
If $\gamma$ is a reducing curve then $\ti \gamma$ is the common frontier of
two domains $\ti C_1$ and $\ti C_2$.  Let $\ti f_j,\ j = 1,2$
be the lift which  fixes the ends of $\ti C_j$.  In this case too $\ti f_j$  fixes   $ \ti \gamma^{\pm}$ and commutes with $T$.

The components of the full pre-image of $V$ are copies of
the universal cover of $V$; we refer to each component as
a lift of $V$.  
There is a compactly supported homotopy from  $\gamma$ to $\alpha$ which lifts to a homotopy between $\ti \gamma$ and a lift $\ti \alpha$ of $\alpha$.  Let $\ti V$ be the lift of $V$ that contains $\ti
\alpha$.
 Since the lifted homotopy moves points a uniformly bounded distance, the ends of $\ti \alpha$ converge
to $T^{\pm} $.    Since this uniquely determines $\ti \alpha$ and since the ends of $T(\ti \alpha)$ converge to $T^{\pm} $, it follows that $T(\ti \alpha) = \ti \alpha$ and hence that $T(\ti V)  = \ti V$.   For the same reason, $\ti f_j(\ti \alpha)$ is the unique lift of $f(\alpha)$ whose ends converge to  $T^{\pm} $.   Since there is  such a lift of $f(\alpha)$ in $\ti V$, it follows that $\ti V$ and  $\ti f_j(\ti V)$ have non-trivial intersection and so, being lifts of $V$,    are equal.

%\jmfnew{There is a compactly supported homotopy from$\gamma$ to $\alpha$ which} lifts to a homotopy between $\ti \gamma$ and a lift $\ti \alpha$ of $\alpha$.  Since the lifted homotopy moves points a uniformly bounded distance and commutes with $T$, the ends of $\ti \alpha$ converge to $\ti \gamma^{\pm} $ and $\ti \alpha$ is $T$-invariant. The components of the full pre-image of $V$ are copies of the universal cover of $V$; we refer to each component as a lift of $V$.  The lift $\ti V$ of $V$ that contains $\ti \alpha$ is $T$-invariant and is the only lift of $V$ whose closure contains $\ti \gamma^{\pm}$.  Since $\ti f_j$ fixes $\ti \gamma^{\pm}$ it follows that $\ti V$ is $\ti f_j$-invariant.

Given $\ti x \in \ti B(f) \cap \ti V$ projecting to  $x \in \B(f) \cap V$, let $W \subset V$ be a
free disk neighborhood of   $x $ with compact closure and
let $\ti W\subset \ti V$ be the lift of $W$ 
that contains $\ti x$.  There exist $k_i \to \infty$ such
that $f^{k_i}(x) \in W$ and covering translations $S_i$
satisfying $\ti {f_j}^{k_i}(\ti x) \in S_i(\ti W)$.
Since $\ti {f_j}^{k_i}(\ti x) \in \ti V$, $S_i$
preserves $\ti V$ and so must be an iterate of $T$.  After
passing to a subsequence and reversing the orientation of
$T$ if necessary, we may assume that $S_i = T^{m_i}$ for
$m_i\to \infty$.  In particular, the distance between $\ti {f_j}^{k_i}(\ti x)$ and $\ti C$  is uniformly bounded.  Lemma~\ref{omega condition} implies that
$\ti C$ is an $\omega$ domain and hence
(Proposition~\ref{home lift}) a home domain for $\ti x$. Lemma~\ref{U covers} imples that $\ti x$ is contained in some $\ti U \in \cA$; 
Lemma~\ref{isolated puncture} and Corollary~\ref{limited
  near cycles}   imply that $T$ is the covering translation associated to $\ti U$.    Since $T$ is independent of the
choice of $\ti x$, $\B(f) \cap \ti V \subset \ti U$.  The interior of the closure of   $\B(f) \cap \ti V $ contains $\ti V$ so  $\ti V \subset \ti U$ by  Corollary~\ref{interior of closure}. This completes the proof.
 \endproof
 
Recall (see Notation~\ref{notn: annular comp}) 
that for any open $f$-invariant annulus $V \subset M$
there is a natural annular compactification of $V$ 
denoted $V_c$ and an extension of $f$ to the closed annulus
$f_c : V_c \to V_c.$   See   Definition~\ref{translation number} for the definition of translation number, translation interval, rotation number and rotation interval.

 \begin{lemma}\label{lem: U rot int}
Suppose that $U \in \cA$  and that $X$ is a component
of $\partial U_c$ corresponding to a
non-singular end.   Then the
translation number $\tau(\ti f_c|_{\ti X})$ of any lift of  $f_c$
restricted to the universal covering space $\ti X$ is an integer
$p.$ Moreover the translation interval $\cT(\ti {f_c})$ is a non-trivial 
interval containing $p$ as an endpoint and having length
at most $1$.
\end{lemma}

\begin{proof}
 No integer can be in the interior of the translation
interval $\cT(\ti f_c).$  To see this
we suppose to the contrary that an integer (which without loss we
assume is $0$) is in the interior of $\cT(\ti f_c)$ and show 
this leads to a contradiction.  In this case by Theorem~\ref{thm: translation interval}
there would be periodic points in $U$ with both positive and negative
rotation numbers.  Theorem~(2.1) of \cite{franks:poincare} then implies
that $f$ has a fixed point in the open annulus $U$, which is a contradiction.

By part (4) of 
Lemma~\ref{U is an annulus},   $f_c$ has a fixed point
in $X$.  It follows that the translation number of the lift 
$\ti f_c|_{\ti X} : \ti X \to \ti X$ is an integer, 
say $p.$ Hence $p \in \cT(\ti {f_c}).$      
There is a point in the interior of $U$ with a well defined
non-integer translation number.  This is because almost all points of
$U$ have a well defined translation number by Theorem~\ref{translation numbers exist}
and if these were all integers then Proposition~\ref{prop: +meas}
would imply $U$ contains a fixed point -- a contradiction.
Since $p \in \cT (\ti f_c)$ and no integer can be in its interior, it
follows that $\cT(\ti {f_c})$ is non-trivial, $p$ is one endpoint and
it must be contained in either $[p, p+1]$ or $[p-1, p].$
\end{proof}

Suppose that $\mu_1$ and $\mu_2$ are disjoint non-homotopic essential oriented simple closed curves in $M$ and that $\ti \mu_1$ and $\ti \mu_2$ are lifts to $\tiM$.  The initial and terminal ends of $\ti \mu_i$ converge to the fixed  points  $T_i^- ,T_i^+\in\sinfty$ respectively of some covering translation $T_i$.  If $\mu_1$ and $\mu_2$ are non-peripheral then   $T_1$ and $T_2$ are hyperbolic and  the four endpoints are distinct.  Moreover, $\ti \mu_1$ and $\ti \mu_2$ are anti-parallel if $\{T^-_1,T^-_2\}$ links $\{T^+_1,T^+_2\}$ and  parallel otherwise.  If either $\mu_1$ or $\mu_2$ is peripheral then it requires more care to decide if $\ti \mu_1$ and $\ti \mu_2$ are anti-parallel. 

%Suppose  that $\ti \gamma_1$ and $\ti \gamma_2$ are disjoint axes of hyperbolic covering translations $T_1$ and $T_2$.      Orient $\ti \gamma_i$ so that  its initial end converges to $T_i^-$  and its terminal end converges to  $T_i^+$.   Then $\ti \gamma_1$ and $\ti \gamma_2$ are parallel if $\{T^-_1,T^-_2\}$ links $\{T^+_1,T^+_2\}$ and anti-parallel otherwise.  This allows us to If either $T_1$ or $T_2$ is parabolic then there are associated horocycles rather than an axis and it requires more care to define the orientation on the axes and to decide if they are parallel or anti-parallel.

\begin{defn} \label{oriented mu} Suppose  that $T$ is the covering translation associated to $\ti U \in \ti \cA$ and that $\ti C$ is a home domain for $\ti U$.   Let $\ti f = \ti f_{\ti C}$. %  If $T  =  T_{\ti \gamma}$, let $\tau \subset U$ be a simple closed curve isotopic to $\gamma$; otherwise $T = T_P$ and we let $\tau \subset U$ be a horocycle.   In both cases, let $\ti \tau \subset \ti U$ be the lift of $\tau$.   
Identify the annular compactification $U_c$ with $S^1 \times [0, 1]$ and so the  
universal cover of  $U_c$  with   $\R \times [0,1]$.   Let  $p_1 : \R \times [0,1] \to \R$ be projection onto the first coordinate. 
   
Since there are no fixed points for $f_c$ in $U$,
Proposition~\ref{prop: +meas} implies that the set of points in
$U_c$ with zero rotation number has measure zero.  Thus there is a
full measure set $\cP \subset U$ consisting of points in $B(f)$
which have a well defined non-zero rotation number for $f_c:U_c \to
U_c$.  Each lift $\ti x \in \ti U$ of each $x \in \cP$ has a well
defined non-zero translation number with respect to $\ti f_c$.  These
translation numbers must either all be positive or all be negative
since the existence of a point with positive translation number and a
point with negative translation number would imply the existence of
positively and negatively recurring free disks in $U$ and then Theorem
2.1 of \cite{franks:poincare} implies the existence of a fixed point.

Let $\mu \subset U$ be an essential simple closed curve and let
$\ti\mu$ be its lift to $\ti U$.  If all $\ti x$ as above have
positive translation number then we orient $\ti \mu$ so that the $
p_1$-image of its initial end converges to $-\infty$ and the $
p_1$-image of its terminal end converges to $+\infty$.  Otherwise, all
$\ti x$ as above have negative translation number and we orient $\ti
\mu$ so that the $ p_1$-image of its initial end converges to
$+\infty$ and the $ p_1$-image of its terminal end converges to
$-\infty$.  We say that $\ti \mu$ has the {\em orientation determined
  by $\ti C$.}  (If $\ti U$ has two home domains then the orientations
that they induce on $\ti \mu$ are opposite from each other.)
\end{defn}

  For any pair of    disjoint properly embedded oriented lines $\ell_1,\ell_2$  in $\R^2$ there is an ambient isotopy  that moves $\ell_1$ and $\ell_2$ to a pair of oriented  horizontal lines.  If the horizontal lines are both oriented to the right or both oriented to the left then we say that    $\ell_1$ and $\ell_2$ are parallel.  Otherwise, we say that   $\ell_1$ and $\ell_2$ are anti-parallel.     It is easy to check that this is well defined.

 \begin{lemma} \label{free disk means parallel}  Suppose that $\ti U_1$ and $\ti U_2$ are distinct elements of $\ti \cA$ that have a common home domain $\ti C$.  Suppose further that both $\ti U_1$ and $\ti U_2$ intersect a lift   $\ti D \subset \tiM$  of some free disk $D \subset M$.  For $i =1,2$, let   $\mu_i $ be an  essential simple closed curve in $U_i$ and let   $\ti \mu_i \subset \ti U_i$ be its lift  endowed with the orientation determined by $\ti C$.   Then $\ti \mu_1$ and $\ti\mu_2$ are parallel.   %Suppose that   $\ti D \subset \tiM$ is a lift of a free disk $D \subset M$ and that $\ti x_1, \ti x_2 \in \ti D \cap \ti  B(f)$ have a common home domain $\ti C$ but belong to distinct elements $\ti U_1$ and $\ti U_2$ of $\ti \cA$.   For $i =1,2$, let   $\mu_i $ be an  essential simple closed curve in $U_i$ and let   $\ti \mu_i \subset \ti U_i$ be its lift  endowed with the orientation determined by $\ti C$.   Then $\ti \mu_1$ and $\ti\mu_2$ are parallel.
 \end{lemma}
 
 \proof   Following Definition~\ref{oriented mu}, we let $\cP_i$ be the full measure subset of  $U_i$ consisting of points with well defined non-zero rotation number for $f_c : {U_i}_c \to {U_i}_c$.  Since $U_i \cap D$ is an open set we may choose $x_i \in  \cP_i$ and lifts $\ti x_i \in \ti D$.
 
 Let $\ti f = \ti f_{\ti C}$.     By Theorem~2.6 of \cite{han:fpt} there exists an oriented properly embedded line $L_i $ with the following properties.
 \begin{enumerate}
 \item $L_i$   contains the $\ti f_{\ti C}$-orbit of $\ti x_i$.
 \item The initial and terminal ends of $L_i$ converge to $\alpha(\ti f_{\ti C},\ti x_i)$ and $\omega(\ti f_{\ti C}, \ti x_i)$ respectively. 
 \item  If $i < j$ then $ \ti f^i(\ti x_1) < \ti f^j(\ti x_1)$  in the ordering induced on $\ti L$ by its orientation.
 \item $L_i$ is $\ti f$-invariant,   up to isotopy rel the orbit of $\ti x_i$.   
 \end{enumerate}
 
As $L_1$ is only defined up to isotopy rel the orbit of $\ti x_1$, we may assume that $L_1 \subset \ti U_1$.  
The lines $L_1$ and $\ti f(L_1)$ are isotopic rel the orbit of $\ti x_1$.  Since $L_1$ and $\ti f(L_1)$  are both contained in $\ti U_1$ and the orbit of $\ti x_2$ is disjoint from $\ti U_1$, $L_1$ and $\ti f(L_1)$ are isotopic rel the orbits of $\ti x_1$ and $\ti x_2$.  
By symmetry we may assume that $L_2 \subset \ti U_2$ is $\ti f$-invariant up to isotopy rel the orbits of $\ti x_1$ and $\ti x_2$. %
%
%Lemma~8.7 of \cite{fh:periodic} implies that $L_1$ is unique up to isotopy rel the orbit of $\ti x_1$ and so   is independent of the exact choice of $\ti x_2$.    We may therefore assume (in the notation of Definition~\ref{oriented mu}) that $x_1 \in \cP_1$  and hence that orientation on $\ti L_1$ is the one determined by $\ti U_1$.  Symmetrically   the  orientation on $\ti L_1$ is the one determined by $\ti U_1$.    
 Since $\ti x_1, \ti x_2$ are contained in a   free disk for $\ti f$,   Lemma~8.7(2) of \cite{fh:periodic} implies that $L_1$ and $L_2$ are parallel.   Items (2) and (3) imply that  the orientation on $L_i$ is consistent with one one on  $\ti \mu_i$ determined by $\ti C$ and we conclude that $\ti \mu_1$ and $\ti \mu_2$ are parallel.
 \endproof

 \begin{lemma} \label{ U is interior of closure}   Each $U \in \cA$ is the interior of its closure in $M$.
\end{lemma}

\proof   It is obvious that $U \subset \Int(\cl(U))$ so it  suffices to show that if $x \in \fr(U)$ then every neighborhood   of $x$ intersects some element $U' \ne U$ of $\cA$.

Choose  $\ti U \in \ti \cA$ projecting to $U$ and a lift  $\ti x \in \fr(\ti U)$.   By Lemma~\ref{more home domains} (1) there is a free disk neighborhood $D$ of $x$ lifting to a neighborhood $\ti D$ of  $\ti x$ and there is a domain $\ti C$ that is a home domain for each point in $\ti D \cap \ti \B(f)$ and hence a home domain for every element of $\ti A$ that intersects $\ti D$.      Let $\ti f = \ti f_{\ti C} : \tiM \to \tiM$.  Since $\ti x$ is in the frontier of $\ti U$, $\ti D$ intersects  at least one element $\ti U' \ne \ti U$ of $\ti \cA$.
%Note that if $\ti U$ is in the only element of $\ti \cA$ intersecting $\ti D$ then by the definition of $\ti U$ the point $\ti x$ would lie in $\ti U$,
%a contradiction. Hence $\ti D$ must intersect at least one $\ti U' \ne \ti U.$ 
We must show that for any $\ti D$ there is such a $\ti U'$ whose projection
$U'$ in $\cA$ is not equal to $U.$
Let  $\cS$ be the set of covering translations $S$ such that $\ti U' =S(\ti U)$ intersects $\ti D$ but is not equal to   $\ti U$.  
It suffices to show  that $\cS= \emptyset$, and we do this by    assuming that $\cS $ contains at least one element $S$ and arguing to a contradiction.

     If $\ti C \ne S(\ti C)$ then they must have a common frontier component $\ti \sigma$ because there are points for which they are both home domains.  But $\ti \sigma$ projects to a simple closed curve $\sigma$   that separates $M$ with the interior of $\ti C$ projecting into one side and the interior of $S(\ti C)$ projecting to the other in contradiction to the fact that $S$ is a covering translation.  We conclude that $\ti C = S(\ti C)$.  Thus $S \in \Stab(\ti C)$ and $S$ commutes with $\ti f$.   It follows that if $\ti \mu \subset \ti U$ and $\ti \mu' \subset S(\ti U)$ are lifts of a simple closed curve $\mu \subset U$ equipped with the orientation determined by $\ti C$   as in  Definition~\ref{oriented mu} then $S$ maps $\ti \mu$ to $\ti \mu'$   and preserves orientations.  
     
      If $\mu$ and $\mu'$ are anti-parallel we have contradicted     Lemma~\ref{free disk means parallel}.
  If $\mu$ and $\mu'$ are parallel then there is a covering translation $S'$ such that  $ S'(\ti \mu)$ separates $\ti \mu$ and $\ti \mu'$ and such that  the orientation on $S'(\ti \mu) $  is anti-parallel to that of $\ti \mu$; the existence of $S'$ follows from the fact that   $M$ has genus zero.    We are now reduced to the previous case and so are done.
  \endproof

We are now able to prove Proposition~\ref{intermediate}.
\vspace{.1 in}

\noindent {\bf Proposition~\ref{intermediate}}{\em \ \ Suppose that $F
  \in \newDiff$ has entropy zero, has infinite order and at
least three periodic points.  Suppose that $M$ is a component of $\M
= S^2 \setminus \Fix(F)$ and that $f=F|_M:M \to M$.  Then $\cA$ (see
Definition~\ref{defn: A2}) is a countable collection of pairwise
disjoint essential open $f$-invariant annuli in $M$ such
that
\begin{enumerate}
\item { For each compact set $X \subset M$ there is a constant $K_X$
such that any $f$-orbit that is not contained in some $U \in \cA$
intersects $X$ in at most $K_X$ points.  In particular each  birecurrent point
  is contained in some $U \in \cA$.}
\item {  If $z \in M$ is not contained in any element of $\cA$ then there
are components $F_+(z)$ and $F_-(z)$ of $\Fix(F)$ so that $\omega(F,z)
\subset F_+(z) $ and $\alpha(F,z) \subset F_-(z)$.} 
\item{  For each $U \in \cA$ and each component $C_M$ of the frontier of
$U$ in $M$, $F_+(z)$ and $F_-(z)$ are independent of the choice of $z
\in C_M$. }
\item{   If $U \in \cA$, and $f_c:U_c \to U_c$ is the extension
to the annular compactification (Notation~\ref{notn: annular comp}) of $U$, then each component of $\partial U_c$
corresponding to a non-singular end of $U$ contains a fixed point of $f_c.$}
\item {  $\cA$ is the set of maximal $f$-invariant open annuli in $M$}
\end{enumerate}}

%\vspace{.1in}

%\noindent {\bf  Proposition~\ref{intermediate}}\ \      {\em \ \ Suppose that $M$ is a component of $\M = S^2 \setminus \Fix(F)$ and that $f=F|_M:M \to M$.  Then there is a  countable collection $\cA$ of pairwise disjoint
%open $f$-invariant annuli such that }
%\begin{enumerate}
%\item {\em For each compact set $X \subset M$ there is a constant $K_X$
%such that any $f$-orbit that is not contained in some $U \in \cA$
%intersects $X$ in at most $K_X$ points.  In particular each  $x \in
%\B(f)$ is contained in some $U \in \cA$.}
%\item {\em For each $U \in \cA$ and   $z$  in the frontier of $U$ in $S^2$, there
%are components $F_+(z)$ and $F_-(z)$ of $\Fix(F)$ so that $\omega(F,z)
%\subset F_+(z) $ and $\alpha(F,z) \subset F_-(z)$.}
%\item {\em For each $U \in \cA$ and each component $C_M$ of the frontier of
%$U$ in $M$, $F_+(z)$ and $F_-(z)$ are independent of the choice of $z
%\in C_M$. }
%\item{\em  If $U \in \cA$, and $f_c:U_c \to U_c$ is the extension
%to the annular compactification of $U$, then each component of $\partial U_c$
%corresponding to a non-singular end of $U$ contains a fixed point of $f_c.$}
%\end{enumerate}

\noindent{\bf Proof of Proposition~\ref{intermediate}}
The case in which $M$ has less than three ends is proved in section~\ref{sec: intermediate} so we may assume that $M$ has at least three ends.   

%The index set $\cA$ is defined in Definitions~(\ref{defn: A1}) and (\ref{defn: A2}).  
The elements of $\cA$ are essential 
open annuli by Lemma~\ref{U is an  annulus} (1) 
and are disjoint and $f$-invariant by Lemma~\ref{U tilde properties} (4).  
Lemma~\ref{more home domains} (3) implies
(1) which implies (2).   Item (4) follows from Lemma~\ref{U is an  annulus} (4). 
Item (5) is Lemma~\ref{maximal annuli}.

We now turn to (3).  Let $Z$ be a component of the frontier of $U$ in $M$, let $\ti U$ be a component of the full pre-image of $U$ and let $\ti Z$ be a component of the frontier of $\ti U$ that projects onto $Z$.

    Given $\ti z \in \ti Z$, let $\ti C$ be the unique (Lemma~\ref{more home domains}) home domain for $\ti z$, let $\ti f = \ti f_{\ti C}$ and let $\ti D$ be a neighborhood of $\ti z$ that projects to a  free disk for $f$ and is disjoint from a lift $\ti \mu$  of a simple closed curve $\mu \subset U$ that is essential in $U$.   
 By Lemma~\ref{more home domains} (1), we may assume that $\ti C$ is a home domain for each 
element of $\ti B(f) \cap \ti D$  and hence a home domain for each element of $\ti A$ that intersects $\ti D$.  We claim that $\ti D$ intersects exactly one component $V_{\ti z}$ of  $\tiM \setminus \cl(\ti U)$.   If the claim is false   then there exist $\ti U' , \ti U'' \in \cA$ that intersect $\ti D$ and that are contained in  distinct components of  $\tiM \setminus \cl(\ti U)$.    Let $\ti \mu'\subset \ti U' $and $\ti \mu''\subset \ti U'' $  be  lifts of   essential simple closed curves $\mu' \subset U'$ and $\mu''\subset U''$.    Equip  $\ti \mu, \ti \mu'$ and $\ti \mu''$ with the orientation determined by $\ti C$.  Since  $\ti \mu'$ and $\ti \mu''$ are contained in distinct components of  $\tiM \setminus \cl(\ti U)$ and are contained in the same component of the complement of $\ti\mu$, no one of these three lines separates the other two. It follows that two of these lines are anti-parallel in contradiction to   Lemma~\ref{free disk means parallel}.    This completes the proof of the claim.   We conclude  that each $\ti z \in \ti Z$ has a
neighborhood that intersects exactly one component $V_{\ti z}$ of  $\tiM \setminus \cl(\ti U)$. 

%The first step in the proof of (3) is to show that each $\ti z \in \ti Z$ has a neighborhood that intersects exactly one component $V_{\ti z}$ of  $\tiM \setminus \cl(\ti U_1)$.  \jmfc{What is $U_1$?}  Let $\ti C$ be the unique (Lemma~\ref{more home domains}) home domain for $\ti z$, let $\ti f = \ti f_{\ti C}$ and let $\ti D$ be a neighborhood of $\ti z$ that projects to a  free disk for $f$ and is disjoint from $\ti \mu_1$.    By Lemma~\ref{more home domains} -(1), we may assume that $\ti C$ is a home domain for each element of $\ti B(f) \cap \ti D$  and hence a home domain for each element of $\ti A$ that intersects $\ti D$.   If the claim is false \jmfc{Is the claim the ``first step'' or all of (3)?} then there exist $\ti U_2, \ti U_3 \in \cA$ that intersect $\ti D$ and that are contained in  distinct components of  $\tiM \setminus \cl(\ti U_1)$.    For $i =2,3$    let $\ti \mu_i \subset \ti U_i$  be the  lift of an essential simple closed curve $\mu_i \subset U_i$.    Equip each $\ti \mu_i$ with the orientation determined by $\ti C$.  Since  $\ti \mu_2$ and $\ti \mu_3$ are contained in distinct components of  $\tiM \setminus \cl(\ti U_1)$ and are contained in the same component of the complement of $\ti\mu_1$, no one of these three lines separates the other two. It follows that two of these lines are anti-parallel in contradiction to   Lemma~\ref{free disk means parallel}.    This completes the proof of the first step.

 The next step in the proof of (3) is to show that  the intersection $B$ of $\sinfty$ with the closure of $\ti Z$ cannot have more than  two components.   The open set $V_{\ti z} =
V_{\ti Z}$ depends only on $\ti Z$ and not on $\ti z$.  In particular,
$\ti Z$ is contained in the frontier of $V_{\ti Z}$.
Let $W$ be the component of the complement of $\ti U$ that contains
$V_{\ti Z}$ and so contains $\ti Z$.  Lemma~\ref{planar topology}
implies that the frontier of $W$ is connected and hence is contained
in a component of the frontier of $\ti U$.  Thus $\ti Z$ is the
frontier of $W$. The  argument in the preceding paragraph shows that $W \setminus \ti Z$ is connected.   Lemma~\ref{planar topology} also implies that  the complement of $W$ is connected and hence that  the 
complement of $\ti Z$ in $\tiM$  has exactly two components.  If $B$ has  more than  two components   there would be two components of $\sinfty \setminus B$ with neighborhoods contained in the same component of $\tiM \setminus \ti Z$ and so there would be a line in  $\tiM \setminus \ti Z$ that separates $\ti Z$.  This contradiction completes the second step.    

The third step  is to prove that  each component of $B$ is a single point.  If $\cR \ne \emptyset$ then this follows from the fact (Corollary~\ref{symmetric  close}) that  $\ti U$, and hence $\ti Z$, is contained in a uniformly bounded neighborhood of either one or two domains.  Similarly, we are done if there is an essential non-peripheral simple closed curve $\tau$ in $M$ that is  contained in an element of  $\cA$, for in this case  each interval in $\sinfty$ contains the endpoints of a lift of $\tau$ that is disjoint from $\ti U$.   

 We are now reduced to the case that  that $f$ is isotopic to the identity and that  $U $ is peripheral.  Choose compact subsurfaces $M_1 \subset M_2 \subset M$ %such that  $$N_{D_1}\cR  \subset M_1 \qquad \text{ and } \qquad M_1 \subset \Int(M_2)$$ and 
 so that $M \setminus M_1$ has at least three components and so that the following hold for each component $W$ of $M \setminus M_1$ and each component $V$ of $M \setminus M_2$.
\begin{itemize}
%\item  $V$  [resp.$W$]  contains  at least one puncture.
\item $\partial V$  [resp. $\partial W$]  is   connected and is either a geodesic or a horocyle.
\item  $V\subset W \implies   h(V) \subset W$. 
\end{itemize}
 By Lemma~\ref{U is an annulus} (3), $U$ contains a deleted neighborhood of some puncture $p$.  We may   assume without loss that the component $W_p$ of  $M \setminus M_1$ that contains $p$ contains no other puncture and is contained in $U$.    Let $\ti W_P$ be the component of the full pre-image of $W_p$ that is contained in $\ti U$ and let $P \in \sinfty$ be the fixed point of the covering translation $T_P$ corresponding to $\ti U$.  (Thus $P$ projects to $p$.)     
 
 For $\ti \rho$ a path in $\tiM$,  define $d(\ti \rho)$ to be the total length of the maximal subpaths of $\ti \rho$ that are contained in the full pre-image $\ti M_1$ of $M_1$.  Equivalently, project $\ti \rho$ to a path $ \rho \subset M$ and take the total length of $\rho \cap   M_1$.   For all $\ti x \in \tiM$, let $\ti \rho_{\ti x}$ be a path connecting $\ti x$ to $\partial \ti W_P$ such that $d(\ti x) := d(\ti \rho_{\ti x})$ is minimal among all such paths.   Since there is a lower bound to the distance between components of $M\setminus M_1$, $\ti \rho_{\ti x}$ decomposes as a finite alternating concatenation of subpaths in $M \setminus \ti M_1$ and subpaths in $M_1$ with all intersections with $\partial M_1$ being orthogonal.  Note also that $|d(\ti y_1) - d(\ti y_2)| \le \dist(\ti y_1,\ti y_2)$.  If $\tau$ is an essential closed curve in $M_1$ that is non-peripheral  in $M_1$ then an endpoint in $\sinfty$ of any lift of $\tau$  is the limit of points $\ti y_j$ with $d(\ti y_j) \to \infty$.  The third step will therefore be completed once we show that there is a uniform bound to $d(\ti x)$  for  birecurrent $\ti x \in \ti U$ and hence for all $\ti x \in \ti U$. 
 
 %For all $\ti x \in \tiM$, let $\ti \rho_{\ti x}$ be the shortest  geodesic arc  connecting $\ti x$ to $\partial \ti W_P$.  Define $D(\ti x)$ to be the total length of the maximal subpaths of $\ti \rho_{\ti x}$ that are contained in the full pre-image $\ti M_1$ of $M_1$.   Equivalently, project $\ti \rho_{\ti x}$ to a path $ \rho_x \subset M$ and take the total length of $\rho_x \cap   M_1$.   We claim that there is a uniform bound to $D(\ti x)$  for  birecurrent $\ti x \in \ti U$ and hence for all $\ti x \in \ti U$.  Since every interval  in $\sinfty$ is the endpoint of ray that is disjoint from $\ti M_1$, this will complete the proof of step three.
 
 Let $Q \in \sinfty$ be a translate of $P$, let $\ti W_Q$ be the  horodisk neighborhood of $P$ that is a lift of $W_p$, let   $\ti \mu$ be the geodesic connecting $P$ to $Q$ and let $\ti F$ be the fundamental domain for the action of $T_P$ on $\tiM$ that  is bounded by $\partial \ti F = \ti \mu \cup T(\ti \mu)$.  We may assume without loss that $\ti x\in  \ti F$.  Since $f$ is isotopic to the identity, there exists a constant $C_0$ so that $\dist(\ti y, \ti f(\ti y)) < C_0$ for all $\ti y$.     It follows that if $\ti y \in \ti F$ and $\dist(\ti y, \partial \ti F) > C_0$ then    $\ti f(\ti y) \in \ti F$.    Applying this to the orbit of $\ti x$  we conclude that there exists $m \ge 0$ so that  $\ti f^j(\ti x) \in \ti F$ for all $0 \le j \le m$ and so that $\dist(\ti f^m(\ti x), \partial \ti F) \le  C_0$.  

 Letting $C_1$ be the length of the finite arc $\ti \mu \cap \ti M_1$,  we have  $d(\ti f^j(\ti x)) \le   C_1 + \dist(\ti f^j(\ti x), \partial \ti F)$ for all $j$.    Since $M_2$ is compact, it is covered by finitely many, say $K$, free disks.    If the orbit of $\ti x $ contains more than $K$ points in $\ti F \cap \ti M_2$ then there would be a near cycle $S$ for some points in the orbit of $\ti x$ that was not an iterate of $T_P$ in contradiction to Lemma~\ref{import1}.    Thus the orbit of $\ti x $ intersects $\ti F \cap \ti M_2$ in at most $K$ points.    Suppose that $d(\ti x) > (2K+1) C_0 + C_1$.   If  both  $\ti f^j(\ti x)$ and $\ti f^{j+1}(\ti x)$ are contained in $\tiM \setminus \ti M_2$ then both   $\ti f  ^j(\ti x)$ and $\ti f^{j+1}(\ti x)$ are contained in the same component   of $\tiM \setminus \ti M_1$ and so $d(\ti f^j(\ti x)) =d(\ti f^{j+1}(\ti x))$.   It follows that $d(\ti f^m(\ti x)) \ge d(\ti f(\ti x)) -  2KC_0 > C_0 + C_1$ and hence that $\dist(\ti f^m(\ti x), \partial \ti F) > C_0$.  This contradiction shows that $d(\ti x)$ is bounded above and so completes the proof of step 3.

If $B$ is a single point $P$,  then $P$ is also the intersection of $\sinfty$ with   the closure of one of the complementary components of $\ti Z$.  It follows that   $\alpha(\ti f, \ti y) = \omega(\ti f, \ti y) = P$ for each $\ti y \in  \B(\ti f)$ contained in this component and hence that this component is $\ti U(P)$.  Projecting to $M$, we have by Lemma~\ref{U is an annulus} (3), that $Z$ is disjoint from a neighborhood of  the puncture   to which $P$ projects.   This contradicts (2) and the fact that $\alpha(\ti f, \ti z) = \omega(\ti f, \ti z) = P$ for all $\ti z \in \ti Z$.  %The only possibility is that this complementary component is $\ti U(P)$.  But this contradicts (2) applied to  the projection $Z$ in $M$ of $\ti Z$ since $Z$ is disjoint from $U(c)$  which is a   neighborhood of the end corresponding to $P$ by Lemma~\ref{U is an annulus} (3).   
We conclude that the limit set of $\ti Z$ in $\sinfty$ is a pair of points, say $a$ and $b$.  % Let $T_i$ be the covering translation associated to $\ti U_i$ and note that $ T^\pm_1$  and $T^\pm_2$ are contained in distinct components of  $\sinfty \setminus \{a,b\}$.  \jmfc{I think $U_i$ is not defined until the next paragraph or is it from step 1?}

The final step in the proof of (3) is to show that either $\alpha(\ti f, \ti z) = a$ and $\omega(\ti f, \ti z) = b$ for all $\ti z \in \ti Z$ or $\alpha(\ti f, \ti z) = b$ and $\omega(\ti f, \ti z) = a$ for all $\ti z \in \ti Z$.   Since $\ti Z$ is connected, it suffices to verify this for all $\ti z \in \ti D$.    

Choose $\ti z \in \ti D$.  For $i=1,2$, choose $\ti y_i \in \ti D \cap B(\ti f)$ and $\ti U_i \in \ti A$ such that $\ti y \in \ti U_i$ and such that $\ti U_1$ and $\ti U_2$  are in different components of $\ti M \setminus \ti Z$.
%  For notational convenience denote $\ti U$ by $\ti U_1$.  Choose $\ti y_1 \in \ti D \cap \ti B(f) \cap \ti U_1$ and $\ti y_2 \in \ti D \cap \ti B(f) \cap \ti U_2$ where $\ti U_2 \ne \ti U_1$.   Let $T_i$ be the covering translation associated to $\ti U_i$ and note that $ T^\pm_1$  and $T^\pm_2$ are contained in distinct components of  $\sinfty \setminus \{a,b\}$.  
As shown in the proof of Lemma~\ref{free disk means parallel}  there  are  oriented lines   $L_1$ and $L_2$ 
 with the following properties.
 \begin{description}
 \item [(a)] $L_i \subset U_i$   contains the $\ti f$-orbit of $\ti y_i$.
 \item [(b)]  The initial and terminal ends of $L_i$ converge to $\alpha(\ti f,\ti y_i)$ and $\omega(\ti f, \ti y_i)$ respectively. 
% \item  If $i < j$ then $ \ti f^i(\ti x_1) < \ti f^j(\ti x_1)$  in the ordering induced on $\ti L$ by its orientation.
 \item  [(c)]$L_i$ is $\ti f$-invariant,   up to isotopy rel the orbits of $\ti y_1$ and $\ti y_2$.   
% \item  [(d)]  $L_1$ and $\ti \mu_1$ are parallel to  $L_2$ and  $ \ti \mu_2$. 
  \end{description}
 
 The isotopy of (c) between $L_i$ and $\ti f(L_i)$ can be taken with compact support in $\ti U_1 \cup \ti U_2$.   We may therefore assume
\begin{description}
 \item [(d)]  the isotopy of (c) is rel the orbits of $\ti y_1,\ti y_2$ and $\ti z$.  
 \end{description}

 By Theorem~2.2 of \cite{han:fpt}  there is an oriented line $L_3$ satisfying:
  \begin{description} 
 \item [(e)]  $L_3$  contains the $\ti f$-orbit of $\ti z$.
 \item [(f)] The initial and terminal ends of $L_3$ converge to $\alpha(\ti f,\ti z)$ and $\omega(\ti f, \ti z)$ respectively. 
% \item  If $i < j$ then $ \ti f^i(\ti x_1) < \ti f^j(\ti x_1)$  in the ordering induced on $\ti L$ by its orientation.
 \item [(g)]$L_3$ is $\ti f$-invariant,   up to isotopy rel the orbit of $\ti z$.   
 \end{description}
 
 As $L_3$ is only defined rel the orbit of $\ti z$ we may assume that 
 $L_3$ is disjoint from $L_1$ and $L_2$.   By (d), $f(L_3)$ is isotopic  rel    the orbits of $\ti y_1,\ti y_2$ and $\ti z$ to a line $L_3'$ that is disjoint  from $L_1$ and $L_2$.   Item (g) implies that  $L_3$ is isotopic to $L_3'$ rel the orbit of $\ti z$.  This isotopy can be chosen to leave $L_1$ and $L_2$ invariant so $L_3$ is isotopic to $L_3'$ rel the orbits of $\ti y_1,\ti y_2$ and $\ti z$. In other words $L_3$  is $\ti f$-invariant rel the orbits of $\ti y_1,\ti y_2$ and $\ti z$.   
 
Lemma~8.7(2) of \cite{fh:periodic} implies that $L_3$ is parallel to $L_1$ and $L_2$.   It follows that the ends of $L_3$ converge to distinct points and that the orientation on $L_3$ is independent of $\ti z \in \ti D$.  
\endproof 
 
%\MH{As above,} choose a small disk $\ti D$  that projects to a free disk $D \subset M$ and contains an arc $\beta_0 $ that has endpoints in distinct components of $\tiM \setminus \ti Z$.  Extend $ \beta_0$ to a properly embedded line $ \beta \subset H$ that intersects $\ti Z$ only in $ \beta_0 \cap \ti Z$ and separates $a$ from $b$.  Let $\ti Z_a$ be the set of points in $\ti Z $ that are in the same component of the complement of $\beta$ as $a$; define $\ti Z_b$ similarly. Perform an isotopy   \MH{ rel $\ti Z \cup \{a,b\}$ }from $\ti f_{\ti C}$ to a homeomorphism $g$ such that $g(\beta) \cap \beta = \emptyset$.  After interchanging $a$ and $b$ we may assume that $g(\beta)$ separates $b$ from $\beta$.  Since $g|_{\ti Z} = \ti f_{\ti C}|_{\ti Z}$ it follows that $\ti f_{\ti C}(\ti Z_b)\subset Z_b$ which implies that $b= \omega(\ti z, \ti f_{\ti C})$ for each $\ti z \in Z_b$.  Varying the location of $\ti D$ we conclude that $b= \omega(\ti z, \ti f_{\ti C})$ for all $\ti z \in \ti Z$.  Applying the same argument to $\ti f^{-1}_{\ti C}$ we see that $a= \alpha(\ti z, \ti f_{\ti C})$ for all $\ti z \in \ti Z$ proving   (3). \qed

\section{Renormalization.}\label{sec: renormalization}

In this section we study the finer structure of $f|_U$, the restriction
of $f$ to one of the annuli $U \in \cA$.

%The following lemma provides a criterion for deciding when two points belong to the same element of $\cA$.  

%\begin{lemma}\label{same U} \MHC{new lemma}  Suppose that  $x,y \in \B(f)$ and that there exist a path $\rho \subset M$ with endpoints $x$ and $y$,   a sequence $k_i \to \infty$ and a positive constant $D$ such that  $f^{k_i}(\rho)$ is homotopic rel endpoints to a path of length at most $D$.  Then $x$ and $y$ belong to the same element of $\cA$.
%\end{lemma}

%\proof   Let $\ti \rho$ be a lift of $\rho$ and let $\ti x$ and $ \ti y $ be the endpoints of $\ti \rho$, let $\ti C$ be a home domain (Proposition~\ref{home lift}) for $\ti x$ and let $\ti U$ be the element of $\ti \cA$  (Lemma~\ref{U covers}) that contains $\ti x$.      By hypothesis,  $\dist(\ti f_{\ti C}^{k_i}(\ti x), \ti f_{\ti C}^{k_i}(\ti y) \le D$ for all $k_i$.     It follows that $ \omega(\ti f_{\ti C}, \ti x) = \omega(\ti f_{\ti C}, \ti y) $.   If $\ti C$ is not the unique home domain for $\ti x$ then the  equality of $\omega$ limit sets  holds for the other home domain as well.  Thus   $\ti C$ is an $\omega$ domain,  and hence a home domain (Proposition~\ref{home lift}) for $\ti y$.     Since the element of $\ti \cA$ that contains $\ti y$ is determined by $\omega(\ti f_{\ti C}, \ti y) = \omega(\ti f_{\ti C}, \ti x)$,  we have $\ti y \in \ti U$ as desired. 
%  \endproof 

%We now turn to the renormalization process.
 For each $q \ge
1$ let $\M_q = S^2 \setminus \Fix(F^q) \subset S^2 \setminus
\Fix(F) = \M$.  Recall that by the main theorem of
\cite{brnkist}, each component $M$ of $\M$ is $F$-invariant
and similarly each component $M_q$ of $\M_q$ is
$F^q$-invariant. Let $\cA(q)$ be the family of open
$F^q$-invariant annuli obtained by applying  Definition~\ref{defn: A2}
%Proposition~\ref{intermediate} \jmfc{Why not Thm 1.3?}  
to the restriction of $F^q$
to a component $M_q$ of $\M_q$ that is contained in the
component $M$ of $\M$.  See Proposition~\ref{intermediate} 
for several useful properties of $\cA(q)$.

\begin{lemma}  \label{permutes}$f$ permutes the elements of $\cA(q)$.
\end{lemma}

\proof Since $f$ commutes with $f^q$ this follows from Corollary~\ref{second centralizer invariance} applied to $\cA(q)$.
\endproof

\begin{lemma}\label{lem: ess => inv}
If $V \in \cA(q)$ is essential in $M$ then $V$ is $f$-invariant.
\end{lemma}
\begin{proof} Lemma~\ref{permutes} implies that $f(V)$ is an element of $\cA(q)$ and hence that 
$f(V)$ is either equal to
or disjoint from $V$.  Since $V$ is essential in $M$, $f(V)$ is
essential in $M$.  If $f(V)$ is disjoint from $V$ then 
$f$ maps one component of the complement of $V$ to a
proper subset of itself because every component of the complement
of $V$ in $S^2$ contains fixed points of $F$.  This
contradicts the fact that $f$ preserves area.  
\end{proof}

The following proposition   shows that elements
of the family $\cA(q)$ refine the elements of $\cA$.
\begin{prop}\label{prop: V in U}
Each $V \in \cA(q)$ is a subset of some $U \in \cA.$
\end{prop}

\proof The case that $V$ is essential follows from Lemma~\ref{lem: ess => inv} and  Lemma~\ref{maximal annuli} so we may assume that   $V  $ is inessential in $M$.
An essential closed curve in $V$ bounds a closed disk in $M$ and we
let $W$ be the open disk that is the union of $V$ and this disk.   Since $f$ preserves area and $W$ is open and
invariant under $f^q$ there is a periodic point $p \in W \cap
\Fix(f^q)$ by the Brouwer plane translation theorem.  

Let $\ti p  \in \ti W \subset \tiM$ be   lifts of $p\in W$,   let $\ti C$ be a home domain for $\ti p$ and let $\ti U$ be the element of $\ti \cA$ that contains $\ti p$.  We will  show that   $\ti W \cap \ti B(f^q) \subset \ti U$.  Corollary~\ref{interior of closure} then implies that $\ti W \subset \ti U$ and hence that $V \subset W \subset U$.

%  By Corollary~\ref{interior of closure}, it suffices to show that $\ti W \cap \ti B(f^q) \subset \ti U$.   
 Given $ z \in  W \cap   B(f^q)$, choose $k_i \to \infty$ such that %$f^{qk_i}(z) \in W$ and 
 each  $f^{qk_i}(z)$ is connected to $z$ by a path in $W$ of length less than $1$.  Also, choose $d$ so that $z$ is connected to $p$ by a path in $W$ of length less than $d$.   If $\ti z$ is the lift of $z$ into $\ti W$ then $\dist(\ti f_{\ti C}^{qk_i}(\ti z), \ti f_{\ti C}^{qk_i}(\ti p)) < d+1$.  It follows that $ \omega(\ti f_{\ti C}, \ti z) = \omega(\ti f_{\ti C}, \ti p) $.   If $\ti C$ is not the unique home domain for $\ti p$ then the  equality of $\omega$ limit sets  holds for the other home domain as well.  Thus   $\ti C$ is an $\omega$ domain,  and hence a home domain (Proposition~\ref{home lift}) for $\ti z$.     Since the element of $\ti \cA$ that contains $\ti z$ is determined by $\omega(\ti f_{\ti C}, \ti z) = \omega(\ti f_{\ti C}, \ti p)$,  we have $\ti z \in \ti U$ as desired. 
\endproof

%Let $U$ be the element of $\cA$ that contains $p$.  It suffices to show that $\B(f) \cap W \subset U$.

%Given $x \in \B(f) \cap W$ there exist $k_i \to \infty$ such that $\dist(f^{k_i}(x), x) < 1$.  Let $\rho \subset W$ be an embedded path connecting $x$ to $p$.   For each $k_i$ the path $f^{k_i}(\rho)$ is homotopic to a path $\rho_{k_i}$ that is the concatenation of a path from $f^{k_i}(x)$ to $x$ with $\rho$.  In particular the length of $\rho_{k_i}$ is uniformly bounded.  Lemma~\ref{same U} implies that $x\in U$ as desired.
%\end{proof}

\begin{remark}   $V \in \cA(q)$ is essential in $\M$ if and only if it is
essential in the unique $U \in \cA$ containing it,  since
Proposition~\ref{intermediate} asserts $U$ is essential in $M$.
In this case we   will  simply say  that $V$ is essential.
\end{remark}

The next  lemmas   provide information about
the translation and rotation intervals of the extension  $f_c: U_c \to U_c$   of $f$ to the
annular compactification of $U$.

\begin{lemma} \label{lem: not in V}   
Suppose that $q>1$  and that $x \in U$ is not contained
in any $V \in \cA(q)$.  Then $\omega(f_c, x)$ is either contained in
 a component of  $\partial U_c \cup\Fix(f_c^q).$
 Moreover,
\begin{enumerate}
\item The forward rotation number, $\rho^+_{f_c}(x),$
with respect to $f_c$ is well defined.
\item If $\omega(f_c, x) $ contains a point of
$U$ then   $\rho^+_{f_c}(x) = p/q$ for some $0 < p < q$. 
\item  If $\omega(f_c, x)$ is contained in a component of $\partial U_c$ corresponding to a non-singular end of $U$ then  $\rho^+_{f_c}(x) = 0$.
\item   If $\omega(f_c, x)$ is contained  in  a component $B$ of $\partial U_c$ corresponding to a  singular end of $U$ then   $\rho^+_{f_c}(x) = \rho(B)$.
\end{enumerate}
An analogous statement holds for backward rotation number.
\end{lemma}

\begin{proof} Lemma~\ref{U covers} implies that $\omega(f^q,x) \cap \M_q = \emptyset$.  Thus  $\omega(f^q,x) \subset \Fix(f^q)$ and 
$\omega(f_c^q,x) \subset \Fix(f_c^q) \cup \partial U_c$.  Since each component of $\Fix(f_c^q) \cup \partial U_c$ is $f_c^q$-invariant,  $\omega(f_c^q,x)$   is contained in   a component  $K$ of $\Fix(f_c^q) \cup \partial U_c$.  

The rotation number $\rho_{f_c}$ is well defined and
constant on each component of $\partial U_c$.  It is also
well defined and locally constant on $\Fix(f_c^q)$.  Since
both sets are closed, $\rho_{f_c}$ is locally constant on
their union and hence constant on $K$, say $\rho(f_c|_K)
=\rho_K$.  It follows that $\rho(f_c^q|_K) = q \rho_K.$

In fact, more is true.  There is a lift $\ti f_c : \ti U_c
\to \ti U_c$ of $f$ such that $\tau_{\ti f_c}(\ti y) =
\rho_K$ for each $\ti y$ that projects into $K$.  Let $p_1 :
\ti U_c \to \R$ be the projection used to define $\tau_{\ti f_c}$.
Then for any $k \in \Z$
$$
|p_1\ti f^{kq}(\ti y) - p_1\ti y - kq\rho_K| < 1
$$
for  any   $\ti y$ that projects into $K$.  For any fixed $k$,  this inequality holds for any point $\ti z$ that projects  into a neighborhood, say $W_k$, of $K$.   Suppose that $z$ and the forward $f^q$-orbit of $z$ is contained in $W_k$.   Then by applying the above inequality with $\ti y$ equal, in order, \ to $\ti z,\ti  f^{kq}_c( \ti z), \ti  f^{2kq}_c( \ti z), \ldots, \ti  f^{(j-1)kq}_c( \ti z)$ 
 and summing, we obtain
$$
|p_1\ti f^{qjk}(\ti z) - p_1\ti z - qjk\rho_K| < j
$$
for all $j$.   Setting $n=jk$ and dividing by $n$ we obtain
\[
\Big |\frac{p_1\ti f^{nq}(\ti z) - p_1\ti z}{n} - q\rho_K \Big | < \frac{1}{k}
\]
for all $n$ which are multiples of $k.$ An easy computation for $n$ which
are not multiples of $k$
proves that   
$$
q\rho_K -1/k \le  \liminf_{n \to \infty} \frac{p_1(\ti f^{nq}(\ti z)) - p_1(\ti z)}{n} \le  
 \limsup_{n \to \infty} \frac{p_1(\ti f^{nq}(\ti z)) - p_1(\ti z)}{n}
\le  q\rho_K + 1/k
$$
for all $z$ with $\omega(f_c^q,z) \subset W_k.$ 
Since $\omega(f_c^q,x) \subset W_k$ for all $k$ it follows that  
$\rho^+_{f_c^q}(x) = q\rho_K$ and hence that   
$\rho^+_{f_c}(x) = \rho_K.$   This completes the proof of (1).

If $K$ contains a point $y$ in the interior of $U_c$ then $y \in
\Fix(f_c^q)$ and $\rho(f_c|_K) =p/q$ for some $0 \le p < q$.  If $p =
0$ then there is a lift $\ti f : \ti U \to \ti U$ and a lift $\ti y
\in \Fix(\ti f^q)$ in contradiction to the Brouwer translation theorem
applied to $\ti f$ and the fact that $\Fix(\ti f) = \emptyset$.  Thus
$0 < p < q.$ 
     
If $K$ is a component of $\partial U_c$ corresponding to a
non-singular end then $\rho_K = 0$ by
Proposition~\ref{intermediate} (4). This proves (3) and (4) is
clear.  The analogous result for backward rotation numbers comes
from considering $f^{-1}.$
\end{proof}

\begin{cor} \label{cor: V(q)}
Suppose  $q>1$ and $V \in \cA(q)$  is essential in $U \in \cA$ and 
$U$ has a non-singular end. If $f_c : V_c \to V_c$
is the extension to its annular compactification
then for any lift $\ti f_c$ to the universal covering space $\ti V_c,$
there is $p \in \Z$ such that the translation interval
$\cT(\ti f_c)$ is a non-trivial 
subinterval of  $[p/q, (p+1)/q]$ and contains at least one of its endpoints.
\end{cor}
\begin{proof}
Lemma~\ref{lem: ess => inv} says $V$ is $f$-invariant.
The fact that $U$ has a non-singular end implies that $V$ 
does also. The result now 
follows from Lemma~\ref{lem: U rot int} since
\[
\cT(\ti f_c) = \frac{\cT(\ti f_c^q)}{q}.
\]
\end{proof}

\begin{lemma}\label{lem: cA_q}
If $V \in \cA(q)$ is a  proper subset of $U \in \cA,$ then there is a full
measure subset $W \subset V,$ containing $V \cap \Per(f),$ with the following properties.
\begin{enumerate}
\item If $V$ is essential in $U$, then every $x \in W$
has the same well defined rotation number in $V_c$
as in $U_c$.
\item The annulus $V$ is inessential in $U$ if and only if 
there is $p \in \Z$ with $0< p  < q,$ such that
every $x \in W$ has rotation number $p/q$ in $U_c$.
\end{enumerate}
\end{lemma}

%In fact if $V \in \cA(q)$ then {\em all} points of $V$ have a well defined rotation number, and the same rotation number in $U_c$ and $V_c.$  If $V$ is inessential in $U$ then all points of $V$ have the same rotation number. \jmfnew{If $U= V$ the first of these statements is trivial and if  $V$ is a proper subset of $U$,$V$ has at least one non-singular end, so this statement is the content of Proposition~\ref{prop: well-defined}.} The second statement will follow when we show as part of Theorem \ref{thm: C(x)} that $\rho_f$ is continuous on $U,$ since the full measure set $W$ is dense in $V$.

\begin{proof}  Let $W$ be the  full measure subset  of $V$ which consists
of birecurrent points which have a well defined rotation numbers
in both $U_c$ and $V_c$.  Suppose that $x \in W$.

If $V$ is essential in $U$ then it is $f$-invariant by
Lemma~\ref{lem: ess => inv}. The inclusion of $V$ in $U$
induces an isomorphism on the fundamental group.
%There is a full measure subset $W$ of $V$ which consists of birecurrent points which have a well defined rotation numbers in both $U_c$ and $V_c$. 
 The rotation numbers in the two annuli
can be computed along a subsequence of iterates which recurs.
More precisely we may join $f^{n_i}(x)$ to $x$ by arcs $\alpha_i$ in $V$
whose lengths are bounded uniformly in $i.$  If $\ti f$ is a lift
to the universal cover $\ti V$ then we can join a lift $\ti x$ of $x$
to $\ti f^{n_i}(\ti x)$ by an arc $\ti \beta_i$ in $\ti V$.  The projection
$\beta_i$ of $\ti \beta_i$ in $V$ concatenated with $\alpha_i$
forms a closed loop in $V$ and the rotation number is the limit of $1/n_i$ times
the homology class of this loop as $n_i$ goes to infinity.  This is easily 
seen to be independent of the choices of $\ti f$ and $\ti x.$  This
homology class is the same in $U$ and $V$ so (1) follows.

If $V$ is inessential in $U$ there is a component $X$ of its
complement in $S^2$ contained in $U$.  The set $Q = V \cup X$ is an
open disk invariant under $f^q$. % for some $q.$
Since $f$ is area preserving and $f^q(Q) = Q,$ by the Brouwer plane
translation theorem there is a point $x_0 \in Q \cap \Fix(f^q).$ 
 Since $x_0 \in \Fix(f^q)$ it is not in any $V \in \cA(q)$, so
by Lemma~\ref{lem: not in V} (2) the
rotation number of $x_0$ in $U$ is $p/q$ for some $0 < p < q$. 
Let $N$ be a compact disk neighborhood
of $x$ in $V$ and let $\{n_i = k_iq\}$ be a sequence such that
$f^{n_i}(x) \in N$.  Choose a path $\sigma \subset Q$ connecting $x_0$
to $x$.  Since $Q$ is contractible, $ f^{n_i}(\sigma)$ is homotopic
rel endpoints to the concatenation of $\sigma$ with a path in $N$.
Choose a lift $\ti f_c : \ti U_c \to\ti U_c$ and a lift $\ti N$ of
$N$.  Let $\ti x, \ti x_0 \in \ti N$ be lifts of $x,x_0$ and let $p_1
: \ti U_c \to \R$ be the projection used to define $\tau_{\ti f_c}$.
 The $p_1$-image of $\ti N$ is a bounded subset of $\R.$ It follows that
$p_1(\ti f_c^{n_i}(\ti x) )- p_1(\ti f_c^{n_i}(\ti x_0))$ is 
bounded uniformly in $i$ and hence that $x$ and $x_0$ have the same rotation
number in $U$, namely $p/q.$  

This proves that if $V$ is
inessential in $U$ then all the points of $W \cap V$ have the same
rotation number $p/q.$ To show the converse observe that since $V$
is a proper subset of $U$ it has a non-singular end.  If $V$ is
essential then Lemma~\ref{lem: U rot int} applied to $f^q|_V$  
asserts that $f|_V$ has a non-trivial rotation interval and, in
particular, by Theorem~\ref{thm: translation interval}, there are periodic
points in $V$ with infinitely many distinct rotation numbers.

%If $x \in V$ is in $W$ there is a constant $K_x$ and a
%$f^{n_i}(x)$ to $f^{n_i}(x_0)$ is less than $K_x$.  \jmfnew{Since the covering
%projection restricts to an isometry on each component of the full
%pre-image of Q,} the distance in the covering space $\ti U_c$ from $\ti
%f^{n_i}(\ti x)$ to $\ti f^{n_i}(\ti x_0)$ is less than $K_x$, where
%$\ti f, \ti x,$ and $\ti x_0$ are lifts of $f, x,$ and $x_0$
%respectively.  It follows that $x$ and $x_0$ have the same rotation
%number in $U$, namely $p/q.$ \jmfnew{This proves that if $V$ is
%inessential in $U$ then all the points of $W \cap V$ have the same
%rotation number $p/q.$ To show the converse observe that if $V$ is
%essential then Lemma~(\ref{lem: U rot int}) applied to $f^q|_V$
%asserts that $f|_V$ has a non-trivial rotation interval and, in
%particular, by Theorem~(\ref{thm: rot num}), there are periodic
%points in $V$ with infinitely many distinct rotation numbers.}
\end{proof}

 \begin{cor} \label{cor: ess V}  If  $x \in U$ and
$\omega(f_c,x)$ is not contained in $\partial U_c$, 
then there is a positive integer $r = r(x)$
so that if $q > r$ and either $r=1$ or $q$ is relatively prime 
to $r$, then there is $V \in \cA(q)$ which is essential in $U$
and contains $x.$
\end{cor}

\begin{proof}  If there is  a positive integer $n$
such that $\omega(f_c,x) \subset \Fix(f_c^n) \cup \partial U_c$,  let $r_0 = r_0(x)$ be the smallest such $n$ and note that $r_0> 1$  since $\Fix(f) \cap U = \emptyset$  and $\omega(f_c,x)$ is not contained in $\partial U_c$.   If there is no such $n$, let $r_0 = 1$.  Note  that if $k \ge 1$ and $\omega(f_c,x) \subset \Fix(f_c^k) \cup \partial U_c$ then $r_0 >1$ and $k$ is a multiple of $r_0$.   

   Suppose now that $r =r_1r_0$ for some positive integer $r_1$ (to be chosen below),   that $q > r$ and that either $r=1$ or $q$ is relatively prime to $r$.    By the above observation,   $\omega(f_c,x) \not \subset \Fix(f_c^q) \cup \partial U_c$ so   %Since $\omega(f_c,x)$  is not contained in $\partial U_c$, 
   Lemma~\ref{lem: not in V} implies that   $x$ is contained in some $V_q \in \cA_q$.  It remains to show that if $r_1$ is properly chosen then  
$V_q$   is essential.

If there is a positive integer $m$ such that $x$ is contained in an
inessential element of $ \cA(m)$,  note that $m \ge 2$ since
$U$ is essential and let $r_1$ be the smallest such $m$.  If there
is no such $m$, let $r_1 = 1$.  In this case we are done so
suppose that $r_1 > 1$ and hence that there is an inessential $V_{r_1}
\in \cA(r_1)$ with $x \in V_{r_1}$.  We complete the proof by assuming
that $V_q$ is inessential and arguing to a contradiction.  By
Lemma~\ref{lem: cA_q} (2), a full measure subset of the non-empty
open set $V_q \cap V_{r_1}$ consists of points with the same rotation
number in $U$ .  Moreover by the same result this number must have the
form $p/q$ and $p'/{r_1}$ with $p,p' \ne 0$.  Since $q$ is relatively
prime to $r_1$ this is impossible and we have reached the desired
contradiction.\end{proof}

 Recall (see Notation~\ref{notn: annular comp}) that to simplify notation we denote the rotation interval $\cR(f_c)$ by $\rho(U)$ when there is no ambiguity about the choice of diffeomorphism $f$ but various annuli $U$ are under consideration.

\begin{lemma}\label{lem: one V}
 Suppose  $U \in \cA$ has a non-trivial rotation interval $\rho(U).$
For any sufficiently large prime $q$ there is 
$V \in \cA(q)$ which is essential in $U$
and satisfies $\cl(V) \subset  U.$
\end{lemma}

\begin{proof}
Since $\rho(U)$ is non-trivial, by  Theorem~\ref{thm: translation interval} we may choose three periodic points $\{x_i\},\ i=1,2,3$ in $U$ whose rotation numbers 
$\{p_i/q_i\}$ have distinct denominators and are contained in the interior of $\rho(U).$
Choose a prime $q,$  larger than each $q_i$ and
sufficiently large that any three
intervals of length $1/q$ containing the three numbers
$\{p_i/q_i\}$, must be pairwise disjoint and must
lie in the interior of $\rho(U).$ The points $\{x_i\}$ 
lie in  elements of $\cA(q)$ by  Lemma~\ref{lem: not in V}
and these elements must be distinct by Corollary~\ref{cor: V(q)}.
They are essential in $U$ by Lemma~\ref{lem: cA_q}.
At least one of these annuli, say $V$,  must be separated by the other two from
the components of the complement of $U$.  Hence $\cl(V) \subset U.$ 
\end{proof}

We will write $|\rho(V)|$ for the 
length of the interval $\rho(V)$. 

\begin{lemma}\label{lem: monotonic}  Suppose that $Y$ is a component of the frontier of $U$ in $S^2$ and that 
 $\{V_i\} $ is an infinite  sequence of distinct essential elements 
of $\cA(q)$ such that $V_{i+1}$ separates $V_i$ from $Y$.  Then
\[
\lim_{i \to \infty} |\rho(V_i)| = 0.
\]
\end{lemma}

\begin{proof}
Let $W_i$ be the open annulus that is the union of $V_1, V_i$  and a closed annulus bounded by an essential curve in $V_1$ and an essential curve in $V_i$.   The complementary components of $W_i$ in $S^2$ are the component of  $S^2 \setminus V_i$ that contains $Y$ and the   component of $S^2 \setminus V_1$ that contains the other component of the frontier of $U$. In particular, these complementary components are compact  and connected.    Let 
$W \subset U$ be the union $\cup_i W_i.$   Since the nested intersection of compact connected sets is compact and connected,  $W$ is open and has two
complementary components in $S^2$ so it  must be an open annulus.  

Let $B$ be the 
 boundary component   of $\partial W_c$ corresponding to the end of $W$ that is disjoint from $V_1$.   Every neighborhood of $B$ contains $V_i$ for all sufficiently large $i$.         It follows that if $x_i \in V_i$ is  periodic, then the rotation number of $x_i$ with respect to $f$ converges to the rotation number $a$ of the restriction of $f_c$ to $B$.   Theorem~\ref{thm: translation interval} therefore implies that  the interval $\rho(V_i)$ converges to the point $a$  and so has length tending to zero.
\end{proof}

\begin{lemma}\label{lem: ends}
Suppose $\rho(U)$ is non-trivial and $\partial_0 U_c$ is a component of
$\partial U_c$. Let $\rho_0$ be the rotation number of $f_c$
on $\partial_0 U_c.$ There exists $Q > 0$ such that
\begin{enumerate}
\item   If $q$ is any
product of primes, each bigger than $Q$, 
then $\partial_0 U_c$ is a frontier component of a 
(necessarily  unique) essential $V_0(q) \in \cA(q)$ with $V_0(q) \subset U$.
\item The rotation number of the homeomorphism induced by
$f$ on $\partial_0 V_0(q)_c$ is $\rho_0$. 
In particular $\rho_0 \in \rho(V_0(q))$. 
\item  If $\rho_0 \ne p/q$ for some $0 < p < q$ then $\cl_U(V_0(q')) \subset V_0(q)$ for all sufficiently large $q'$.
\end{enumerate}
\end{lemma}

\begin{proof} The first step in the proof of (1) is to prove
that for sufficiently large $Q$ and for $q$ as in the hypothesis
there exists a (necessarily unique) essential $V_0(q) \in \cA(q)$ that is
not separated from $\partial_0 U_c$ by any other essential element in
$\cA(q)$.

By Lemma~\ref{lem: one V} we may assume that there exists
an essential $V_1 \in \cA(q)$ whose closure is contained in $U$.  Let
$\rho(\partial U_c) = \{\rho_0,\rho_1\}$.  We may assume that $Q$ is so large
that neither $\rho_0$ nor $\rho_1$ has the form $p/q$ with $0< p < q.$ Choose
$\delta$ such that $\delta < |\rho_0 -p/q|, |\rho_1-p/q|$ for all $0 < p <q.$

The proof is by contradiction: assuming that no such $V_0(q)$ exists
we will inductively define an infinite sequence $\{V_i\} $ of distinct
essential elements of $\cA(q)$ such that $V_{i+1}$ separates $V_i$
from $\partial_0 U_c$ and such that $|\rho(V_i)| > \delta/2$ in
contradiction to Lemma~\ref{lem: monotonic}. It suffices to assume
that $V_1,\ldots V_{i-1}$ have been defined for $i \ge 2$, and define
$V_i$. By the assumption we wish to contradict, any element of
$\cA(q)$) is separated from $\partial_0 U_c$ by another element of
$\cA(q)$.  In particular $V_{i-1}$ is separated from $\partial_0
U_c$, say by $V^*_{i}.$  Since $V_i^*$ is also separated from
$\partial_0 U_c$ by (yet another) element of $\cA(q),$ it is contained
between two open essential annuli in $U$.  Hence each component of its
frontier is contained in the interior of $U$.  Lemma~\ref{lem: not in
  V} implies that the rotation number of the restriction of $f_c$ to a
component of $\partial {V^*_i}_c$ has the form $p/q$ with $0 <p < q.$

Choose an essential closed curve $\alpha$ in $V^*_i$ and let $W_i$ be the union of $V^*_i$ with the component of $U \setminus \alpha$ that does not contain $V_{i-1}$.    Then $W_i$ is an open annulus whose frontier components are $\partial_0 U_c$ and a component of the frontier of $V^*_i$.    Theorem ~\ref{thm: translation interval} implies that $W_i$ contains a periodic point $z$ whose rotation number  has distance less than $\delta/2$  from $\rho_0$  and so is not of the form $p/q$. 

In particular, $z \in \M_q$ and is contained in some $V_i \in \cA(q)$
by Lemma~\ref{lem: not in V}.  Lemma~\ref{lem: cA_q} (2) implies that
$V_i$ is essential and hence separates $\partial_0 U_c$ from
$V_{i-1}$.  The rotation number of the restriction of $f_c$ to a
component of $\partial {V_i}_c$ has the form $p/q$ with $0 <p < q$ for
the same reason that components of $\partial {V^*_i}_c$ satisfy this
property.  Since $z \in V_i$, it follows that $|\rho(V_i)| >
\delta/2$.  This completes the induction step and hence shows
the existence an infinite family $\{V_i\}$ contradicting
Lemma~\ref{lem: monotonic}. We conclude there is a unique essential
$V_0(q) \in \cA(q)$ that is not separated from $\partial_0 U_c$ by any
essential element of $\cA(q)$.
  
Since there exists $ V_2 \in \cA(q)$ whose closure is
contained in $U$, the component $B(q)$ of $\fr(V_0(q))$
which is separated from $\partial_0 U_c$ by $V_0(q)$ is
contained in $U$.  Lemma~\ref{lem: not in V} implies
that if $x \in B(q)$ then $\omega(f_c,x)$  is contained in a component of $\Fix(f^q)$ that is disjoint from $\Fix(f)$.  It follows that if $q'$ is a product
of primes all greater than $Q$ and with $q$ and $q'$
relatively prime, then $B(q) \cap B(q') = \emptyset$. Let
$W(q)$ be the open sub-annulus of $U_c$ bounded by
$\partial_0 U_c$ and $B(q)$. Note that either $B(q)$
separates $B(q')$ from $\partial_0 U_c$ or $B(q')$
separates $B(q)$ from $\partial_0 U_c$.  Hence either
$\cl_U(W(q)) \subset W(q')$ or $\cl_U(W(q')) \subset
W(q).$

Theorem ~\ref{thm: translation interval} implies that $W(q)$ contains a periodic
point $w$ whose rotation number is arbitrarily close to $\rho_0$, but not
equal to $\rho_0$, and in
particular not of the form $p/q$.  Lemma~\ref{lem: not in V} and
Lemma~\ref{lem: cA_q} (2) imply that $w$ is contained in some element
of $ \cA(q)$ that is essential and hence this element must be $V_0(q)$.
Theorem~\ref{thm: translation interval} implies that $\rho_0 \in \rho(V_0(q)).$
Now choose $q'$ sufficiently large that $1/q' < |\rho_0 - \rho_{f_c}(w)|.$
Since $\rho_0 \in \rho_{f_c}(V_0(q'))$, by the same argument used
for $V_0(q),$ we  conclude that $\rho_{f_c}(w) \notin \rho(V_0(q'))$
and hence $w \notin V_0(q').$ It follows that
$\cl_U(W(q')) \subset W(q).$

Items (1) and (3) will
follow once we prove that $W(q) = V_0(q)$ (which is equivalent to
showing that $\partial_0 U_c$ is the boundary component $B'(q)$ of
$V_0(q)$ which is not $B(q)$). In particular this will
show that there are no inessential elements of $\cA(q)$ contained
in $W$. In order to show this we first prove the following.

{\bf Claim:} {\em If $q'$ is sufficiently large
then for any open set $P$ in  $W(q)\setminus V_0(q)$ we have 
$P \cap W(q') = \emptyset$.}

We choose $q'$ so that, in addition to its
properties above, it is large enough that
$\rho(V_0(q'))$ does not contain a point of the 
form $p/q$ with $0<p<q$. Assuming that there is an open $P \subset
W(q)\setminus V_0(q)$ with $P \cap W(q') \ne \emptyset$, we will argue to a contradiction, thus proving the claim. The open set $P \cap W(q')$ 
(like any open subset of $U_c$)  contains a positive measure
subset $P_0$ of points which are birecurrent and have well
defined rotation numbers in $U_c$. By Lemma~\ref{U  covers}, 
the fact that $\rho(V_0(q'))$ does not contain a point of the 
form $p/q$  implies that $V_0(q') \subset \M_q$ and hence 
there is a positive measure subset $P_1 \subset P_0$
contained in some $V \in \cA(q)$.  This $V$ is necessarily
inessential since otherwise it would separate $\partial_0 U_c$ and
$V_0(q).$ Lemma~\ref{lem: cA_q} (2) therefore implies 
that points in a full measure subset of $P_1$ 
have the same rotation number which is of the form $p/q$ with
$0<p<q$.

It follows that there is a positive measure subset $P_2
\subset P_1$ which is not contained in $V_0(q')$ but is
contained in an essential element of $\cA(q')$, (by
Lemmas~\ref{lem: not in V} and \ref{lem: cA_q} (2) again).    
This contradicts the assumption that $P_0 \subset W(q')$ 
and so verifies the claim.

One consequence of the claim is that $\partial _0
U_c \subset B'(q)$.  This is because if 
$x \in \partial_0 U_c \setminus B'(q)$
then $x$ has a neighborhood which is disjoint
from $V_0(q)$ but intersects $W(q')$ in an open set
contradicting the claim.

We want now to prove $B'(q) \subset \partial _0 U_c$ and
hence that $B'(q) = \partial _0 U_c$.
We note that $\partial_0 U_c$ has neighborhoods
which are disjoint from $\Fix(f_c^q) \setminus \Fix(f_c)$,
since otherwise points in $\partial_0 U_c$ would
have a rotation number of the form $p/q,\  p\ne 0.$ We
let $W_0$ denote such a neighborhood which is chosen
sufficiently small that it is a subset of  $W(q') \cup \partial_0 U_c$
and let $x$ be a point of $B'(q) \cap W_0.$ 
Any point on the frontier  (in $U_c$) of an element of $\cA(q)$  
is either in $\Fix(f^q)$ or $\partial U_c$  
or has arbitrarily small neighborhoods meeting more than
one element of $\cA(q).$ This is because each element of 
$\cA(q)$ is the  interior of its closure by 
Corollary~\ref{interior of closure} and the union of
all elements of $\cA(q)$ and $\Fix(f^q)$ is dense in $S^2.$
But $x$ has a neighborhood which intersects no element of 
$\cA(q)$ other than  $V_0(q)$, because 
otherwise there would be an open $P \subset W(q')$
which is disjoint from $V_0(q)$ contradicting the
claim above.

We conclude 
that $x \in \partial_0 U_c \cup \Fix(f^q).$  Since $x \in W_0$
implies $x \notin \Fix(f_c^q) \setminus \Fix(f_c) $ and
$\Fix(f_c) \subset \partial U_c,$ we conclude that 
$x \in \partial_0 U_c.$ We have shown that $B'(q) \cap W_0 
= \partial_0 U_c$, but $B'(q)$ is connected, so in fact
$B'(q)= \partial_0 U_c.$ This completes the proof
of \ (1) and (3).

Finally to prove (2) we observe that one component
of the complement of $U$ in $S^2$ coincides with a component
of the complement of $V_0(q)$, namely the component corresponding
to $\partial_0 U_c.$  Indeed $V_0(q)$ is a neighborhood of the
corresponding end of $U.$  It follows that $\partial_0 U_c$ and
$\partial_0 V_0(q)_c$ (in the annular compactifications
$U_c$ and $V_0(q)_c$ respectively) can be naturally identified.  Hence
the rotation number of the map induced by $f$ on $\partial_0 V_0(q)_c$
is $\rho_0.$
\end{proof}

\begin{notn} For each $U \in \cA$ there are two components of its
frontier in $S^2$.  Associated to each component and each $q$
satisfying the hypothesis of Lemma~\ref{lem: ends}, there is
 an element of $\cA(q)$ as described in this lemma.  We will
refer to these as the {\em end elements} of $\cA(q)$ and denote them
$V_0(q)$ and $V_1(q).$ They are neighborhoods of the ends
of $U$. We label them $V_0$ and $V_1$ consistent with a
transverse orientation; i.e., for any $q,q',$ we have $V_0(q)
\cap V_0(q') \ne \emptyset$. Any element of $\cA$ which is not an
end element will be called an {\em interior element}.
\end{notn}

\begin{lemma}\label{lem: end limit}
Suppose $x \in U \in \cA$.
If $\{q_n\}$ is a sequence of primes tending to 
infinity and $x \in V_0(q_n)$ for all $n$ then
$\rho_{f_c}(x)$ is well defined and equal to the 
rotation number $\rho_0$ of the component of $\partial U_c$
corresponding to the end elements $V_0(q_n).$ 
\end{lemma}

\begin{proof} 
For $n$ sufficiently large, $\rho_0 \ne p/q_n$ for
$0 < p < q_n.$ Hence by Lemma~\ref{lem: ends} (3),
by choosing a subsequence we may assume
$\cl_U(V_0(q_{n+1})) \subset V_0(q_n).$

 We now apply Lemma~\ref{lem: 2ann} letting 
 $A_0$ be a closed
annulus in the annular compactification of $V_0(q_n)$ which
has $\partial_0 U_c$ as one boundary component and the other an essential
closed curve in $V_0(q_n) \setminus \cl_U(V_0(q_{n+1})).$
We are identifying $\partial_0 U_c$ as a component of the boundary of
both $U_c$ and  the annular compactification of $V_0(q_n).$
Lemma~\ref{lem: 2ann} implies the rotation interval of
$x$ is the same in $U_c$ as it is in the  compactification of $V_0(q_n)$.  Since this holds for all $q_n$ and since $\rho(V_0(q_n))$ contains $\rho_0$ and has length $\le 1/q_n$, \ $\rho_{f_c}(x_0) = \rho_0$.%, are the same, but $|\rho(V_0(q_n))| < 1/q_n$. Hence the rotation interval of $x$ has length $0$ and hence $x$ has a well defined rotation number in $U_c$.  Since the points of $\partial_0 U_c$ also lie in the compactification of $V_0(q_n)$, the value of this rotation number is $\rho_0.$
\end{proof}  

\begin{lemma}\label{lem: cl(V) in U}
Suppose $U \in \cA.$
If $V \in \cA(q)$ is essential in\ $U$
and $\rho(V)$ is disjoint from $\rho(\partial U_c)$, then
$\cl(V) \subset  U.$ Moreover, if $x \in U$ and the rotation interval
for $x$ in $U_c$ is disjoint from $\rho(\partial U_c)$, then for 
every sufficiently large prime $q$  
there exists an essential $V \in \cA(q)$ such that
$x \in V,\ cl(V) \subset U.$ 
\end{lemma}  

\begin{proof}
By Lemma~\ref{lem: ends} the fact that $\rho(V)$ is
disjoint from $\rho(\partial U_c)$ means that $V$ is neither
$V_0(q)$ nor $V_1(q),$ the two end elements whose frontiers
contain the components of the boundary of $U_c.$ It follows
that $V$ is separated by the essential annuli $V_0(q)$ and
$V_1(q)$ from the boundary of $U_c$ and hence $\cl(V) \subset
U.$

To see the moreover part we observe that if the rotation interval
for $x$ in $U_c$ is disjoint from $\rho(\partial U_c)$, then
for sufficiently large $q$, the rotation interval of 
$x$ will be disjoint from $\rho(V_0(q))$ and $\rho(V_1(q))$.  
If the rotation interval of $x$ is a single rational point
choose $q$ larger than its denominator; otherwise choose any $q > 1$.  By Lemma~\ref{lem: not in V}
this will guarantee that $x$ lies in some $V \in \cA(q).$ 
Corollary~\ref{cor: ess V} implies that if $q$ is sufficiently large
this $V$ will be essential. This $V$ is disjoint 
$V_0(q)$ and $V_1(q)$ and hence  will satisfy $\cl(V) \subset U.$
\end{proof}

In principle a point $x \in V \in \cA(q)$ might have
a different rotation interval when viewed in $V$ than
when viewed in $U.$  The following proposition shows this
does not happen, and as a consequence every point of $U$
has a well defined rotation number.

\begin{prop}\label{prop: well-defined}
Suppose  $x \in U \in \cA$.
Then the rotation number $\rho_{f_c}(x)$ of $x$
with respect to $f_c : U_c \to U_c$ exists.
Moreover, if $x$ is in an essential $V \in \cA(q)$ and 
$\rho_{f_c}(x) \notin \rho(\partial V_c)$, then
$\rho_{f_c}(x) = \rho_{h_c}(x)$ where $h = f|_V.$ 
\end{prop}

\begin{proof}
We will first show that $\rho_{f_c}(x)$ exists.
Given $\epsilon >0 $ it suffices to  show that the rotation interval of $x$ in
$U_c$ has length $< \epsilon.$  %To do this we first 
Choose an
integer $Q$ 
such that $1/Q < \epsilon$ and such that $Q$ is greater
than the number $r(x)$ from Corollary~\ref{cor: ess V}.
Hence  if $q'$ is any product of primes each
of which is $>Q$ then $q'$ is relatively prime to $r(x)$ so
$x$ is contained in an essential element of $\cA(q).$

%More precisely 
Choose three primes $q_i,\ i = 1,2,3$ all greater
than $Q.$ 
Let $V^i$ be the essential element of $\cA(q_i)$ which contains
$x$.  We may assume that each of the $V^i$ are interior in $U_c,$
as otherwise for every sufficiently large prime $q$ the
$V'$ in $\cA(q)$ containing $x$ is an end element and it
follows from Lemma~\ref{lem: end limit} that $x$ has a well
defined rotation number which equals the rotation number of
one boundary component of $U_c.$

Let $V$ be the  essential annulus in $\cA(q_1q_2q_3)$
which contains $x$. %It exists and is essential because $q_1q_2q_3$ is relatively prime to $r$ and $s$.
Note that $V \subset V^i$ and if we define
$h_i$ to be the restriction of $f^{q_i}$ to $V^i$ then
$V$ can be considered as an element of $\cA(q_2q_3, h_1),\ 
\cA(q_1q_3, h_2),$ and $\cA(q_1q_2, h_3).$ 

Suppose for one choice of $i,$ say $i = 1,$ the element $V$ is an
interior element of $\cA(q_2q_3, h_1),$ i.e., $\cl(V) \subset V^1.$ 
Then there exists  an annulus $A_0$
whose boundary consists of two essential simple closed curves, one in
each component of $V^1 \setminus \cl(V).$ The orbit of $x$ lies in
$A_0$ and $A_0$ is an essential closed annulus embedded in both $U_c$
and $V^1_c$. Lemma~\ref{lem: 2ann}  implies that  the rotation interval
of $x$ is the same when calculated in $V^1_c$ as when calculated in
$U_c$ and since $|\rho(V^1)| < 1/Q <\epsilon$, the fact that $\epsilon$
is arbitrary proves that the rotation number of $x$ in $U_c$ is well
defined.

%We will then apply Lemma~(\ref{lem: 2ann}) by finding a closed subannulus
%$A_0$ which lies in both $V^1_c$ and $U_c$ and which contains $V$ and
%hence the orbit of $x.$ It will then follow that the rotation interval
%of $x$ is the same when calculated in $V^1_c$ as when calculated in
%$U_c$ and since $|\rho(V^1)| < 1/Q <\epsilon$, the fact that $\epsilon$
%is arbitrary proves that the rotation number of $x$ in $U_c$ is well
%defined. We can apply Lemma~(\ref{lem: 2ann}) because $V$ is interior
%in $V^1$ so $\cl(V) \subset V^1$ and we can form an annulus $A_0$
%whose boundary consists of two essential simple closed curves, one in
%each component of $V^1 \setminus \cl(V).$ The orbit of $x$ lies in
%$A_0$ and $A_0$ is an essential closed annulus embedded in both $U_c$
%and $V^1_c.$

We are left with the possibility that 
each of $\cA(q_2q_3, h_1),\ \cA(q_1q_3, h_2),$ and $\cA(q_1q_2, h_3)$
has $V$ as an end element.  We will show this leads to a contradiction.
We chose three primes $q_i$ in order to guarantee that $V$ corresponds
to the same end (i.e., a $V_0$ or a $V_1$) for two of the 
$h_i: V^i \to V^i.$
Hence we may assume without loss of generality that
$V = V_0(q_2q_3) \in \cA(q_2q_3, h_1)$ and $V = V_0(q_1q_3) \in
\cA(q_1q_3, h_2).$   But Lemma~\ref{lem: ends} (2) applied
to $f^{q_1}$ implies that the rotation numbers of the maps induced
by $f$ on $\partial_0 V^1_c$ and $\partial_0 V_c$ coincide.
Likewise, so do the rotation numbers on 
$\partial_0 V^2_c$ and $\partial_0 V_c.$
Since $V^i$ is interior in $U_c$ both its ends are non-singular.
But it follows from Lemma~\ref{lem: U rot int} applied to 
$h_i$ that the rotation number of the map induced by $f$
on $\partial_0 V^i_c$ has the form $p_i/q_i$ with $0< p_i < q_i$
and hence it is not possible for two of these rotation numbers
to coincide. This contradiction completes the proof that 
$\rho_{f_c}(x)$ exists.

To prove the second assertion of the proposition we note that
$\rho_{h_c}(x)$ exists by the first part applied to $f^q|_V.$
The fact that $\rho_{f_c}(x) \notin \rho(\partial V_c)$, and
Lemma~\ref{lem: cl(V) in U} imply that there is a prime $q'$ and
an essential  $V' \in \cA(q')$ containing $x$ and such that $\cl(V') \subset V.$
Hence we may choose a closed annulus $A_0$ containing $V'$ and
contained in $V$.  Applying Lemma~\ref{lem: 2ann} we
conclude that $\rho_{f_c}(x) = \rho_{h_c}(x)$ where $h = f|_V.$ 
\end{proof}

\begin{remark} In the following definition we assume that $\rho(U)$ is
non-trivial.   If we are willing to pass to a power of $F$
this  is a consequence of our standing hypothesis
that $F:S^2 \to S^2$ has at least three periodic points, because then
$F^q$ will have three fixed points and any $U \in \cA(q)$ will have
a non-trivial end which is sufficient by Lemma~\ref{lem: U rot int}
to imply $\rho(U)$ is non-trivial.
\end{remark}

\begin{defns} \label{defn: hat Y} Suppose that $\rho(U)$ is non-trivial, 
that $\partial_0U_c$ and $\partial_1U_c$ are the frontier components
of $U \in \cA$ and that $a_i$ is the rotation number of
$f$ on $\partial_iU_c$.  Choose $Q = Q(U)$ so that:
\begin{itemize}
\item If $a_i = p/q$ for some $0 < p < q$ then $q < Q$.
\item % If $\rho(U)$ is non-trivial, then for 
For every prime $q > Q$ there is an 
essential element $V \in \cA(q)$  
 such that $cl(V) \subset U$. (See  Lemma~\ref{lem: cl(V) in U}).)
 \item For every prime $q > Q$ there are distinct elements $V_0(q),
 V_1(q) \in \cA(q)$ (as in Lemma~\ref{lem: ends}) that are contained
 in $U$ such that $\fr(V_i(q)) \subset U_c$ is 
$\partial_i U_c \cup B_i(q)$ for $i=0,1$ where $B_0(q)
 \cap B_1(q) = \emptyset$.
\end{itemize}
 Define
 \[
\hat Y_i = \bigcap_{q > Q} \cl_{U_c}( V_i(q)),
\]
where the intersection is taken over all primes $q >Q$
and define 
\[
\Check U = U_c \setminus ( \hat Y_0 \cup \hat Y_1) \subset U.
\]
\end{defns}

Recall from Definition~\ref{defn: omega-free} that $\W_0$ is the set of   free disk recurrent points.

\begin{lemma}\label{lem: Check U} Assume notation as above and assume that $\rho(U)$ is non-trivial.  The following hold for $i=0$ and $i=1$:
\begin{enumerate}
\item $\hat Y_i$ is well-defined, i.e., independent of the
choice of $Q$.
\item$\Check U$ and $\Check U \cup (\hat Y_i \cap U)$  are 
essential open $f$-invariant annuli.
\item If $a_i = 0$  then $\hat Y_i \cap U$ has measure $0.$
\item If $a_i \ne 0$ then $ \hat Y_{i} \cap U \subset \W_0$.  
\item $\rho(y) = a_i$ for each $y \in \hat Y_i \cap U$.
\end{enumerate}
\end{lemma}

\begin{remark}   
If $x \in \Check U$ 
then it has non-zero rotation number.  To see this observe
that if $x \in \Check U$,  the $\omega$-limit set
$\omega(f_c, x)$ is  separated from  $\partial U_c$ because the orbit of $x$ is
separated from  $\partial U_c$ by $V_0(q) \cup V_1(q)$ for some
$q.$  Then Corollary~\ref{cor: ess V}
implies that for every sufficiently large prime $q$ the point
$x$ must lie in some essential $V \in \cA(q)$.  For large $q$
this $V$ must be interior, i.e. separated from $\partial U_c$.

 Let $h=f|_V$ and consider $h_c:V_c \to V_c$.
Observe that $0 \notin \rho_{h_c}(V_c)$,
because if it were Theorem~\ref{thm: translation interval}
would imply $h_c$ has a fixed point in $V_c$.
But then Lemma~\ref{lem: frontier fp}
implies there is a fixed point for $F$ in $\cl(V) \subset U$ 
which is a contradiction.

 Suppose now $\rho_{f_c}(x) = 0$ in $U_c$. Since 
$0 \notin \rho_{h_c}(\partial V_c)$,
Proposition~\ref{prop: well-defined} implies that
$x$ also has rotation number $0$ for $h_c: V_c \to V_c$,
but as noted above this is a contradiction.

If $a_0 = a_1 = 0$,
then item (5) of Lemma~\ref{lem: Check U} implies that $U \setminus \Check U$ 
consists of points with rotation number $0$. Hence in this case
$\Check U$ is precisely the subset of $U$ whose 
points have non-zero rotation number.
\end{remark}

\begin{proof}  By  part (3) of Lemma~\ref{lem: ends}  
there is a sequence of primes  $\{q_j\}$ tending to infinity such that 
\[
\hat Y_i = \bigcap_{q_j} \cl_{U_c}( V_i(q_j))
\]
and
$$
 V_i(q_{j+1}) \cup B_i(q_{j+1}) )\subset V_i(q_j)
$$
for all $j$, where $B_i(q)$ denotes the component of
the frontier of $V_i(q)$ which lies in $U$.  
This proves that $\hat Y_i$ does not depend on the choice of $Q$ in its definition and so is well-defined.
 
Let $Z_i(q)$ be the component of $S^2 \setminus B_i(q)$ that contains $Y_i$.   Then $\cl_{S^2}(Z_i(q_{j+1})) = Z_i(q_{j+1}) \cup B_i(q_{j+1}) \subset Z_i(q_j)$  for the 
sequence of primes $\{q_j\}$ chosen above.   Define
$$
\hat Z_i = \bigcap_{q_j} \cl_{S^2}Z_i(q_j)) = \bigcap_{q>Q} \cl_{S^2}(Z_i(q))
$$
Then $\hat Z_0$ and $\hat Z_1$ are disjoint, compact, connected sets
and $\Check U = U_c \setminus ( \hat Y_0 \cup \hat Y_1) = S^2
\setminus ( \hat Z_0 \cup \hat Z_1).$ Thus $\Check U$ is an open
subsurface of $S^2$ with two ends and hence an annulus.  The set
$\Check U$ separates $\hat Z_0$ and $\hat Z_1$ each of which contains
a point of $\Fix(F).$ Hence $\Check U$ is essential in $\M$ and
therefore in $U$.  The same argument applies to $\Check U \cup (\hat
Y_i \cap U)$: Its complement in $S^2$ has two components.  For
example, if $i=0$ then one of the components is $\hat Z_1$ and the
other is the component of the complement of $U$ that intersects
$\hat Z_0$. Each of these complementary components contains a point
of $\Fix(F)$ and hence $\Check U \cup (\hat Y_i \cap U)$ is an annulus
which is essential in $\M$ and therefore in $U$.

 To show $\Check U$  is $f$-invariant it suffices to
show $\hat Y_i$ is $f_c$-invariant, but this follows from
the definition of $\hat Y_i$ and the fact that $V_i(q)$ is
$f_c$-invariant. Having verified that $\Check U$ is $f$-invariant,
the same is true for $\Check U \cup (\hat Y_i \cap U)$.
This completes the proof of (2).

Item (5) follows from Lemma~\ref{lem: end limit}.  Item (3) then
follows from Proposition~\ref{prop: +meas} and that fact that $\Fix(f)
= \emptyset$.

For (4) suppose that  $a_i \ne 0$ and that $x \in \hat Y_i$.
If the $\omega$-limit set of $x$ contains a point in $U$ then $x \in
\W_0$.  Otherwise there is a non-fixed point
$z$ in $\omega(x, f_c)$.  If $D_0$ is a disk neighborhood of $z$ then
$D_0 \cap U$ is a free disk that the orbit of $x$ intersects more than
once and again $x \in \W_0$. 
\end{proof}

We are now prepared to complete the proof of Theorem~\ref{thm: annuli}.    For the definition of free disk recurrent and weakly free disk recurrent  see Definition~\ref{defn: omega-free}.

%\mhc{This is the latest version - (2) has been added and (3) strengthened.}   
%\jmfc{Seems ok.}
\bigskip
\noindent
{\bf Theorem~\ref{thm: annuli}.} {\em Suppose $F \in \newDiff$ has
entropy zero, infinite order and at least three periodic points.
Let $f = F|_{\M}$ where $\M = S^2
\setminus \Fix(F)$.   Then there is a countable collection $\cA$ of pairwise disjoint
open $f$-invariant annuli such that }
\begin{enumerate}
\item  {\em $\U = \bigcup_{U \in \cA} U$ is the set $\W$  of weakly free disk 
recurrent points for $f$. }
\item {\em  $\cA$ is the set of maximal $f$-invariant open annuli in $\M$.}
\item {\em  If $ z \not \in \U  $, there
are components $F_+(z)$ and $F_-(z)$ of $\Fix(F)$ so that $\omega(F,z)
\subset F_+(z) $ and $\alpha(F,z) \subset F_-(z)$.}
\item {\em For each $U \in \cA$ and each component $C_{\M}$ of the frontier of
$U$ in $\M$, $F_+(z)$ and $F_-(z)$ are independent of the choice of $z
\in C_{\M}$. }
\end{enumerate}

\begin{proof}  It suffices to verify items (1) - (4)  for one component $M$ of $\M$ at a time.  Items (2), (3) and (4) follow from the second, third and fifth items of   Proposition~\ref{intermediate} so it suffices   
 to prove (1).   %More precisely we will show that each $U \in \cA$ is the interior of the closure of a component of $\W_0$, the free disk recurrent set.
 
If  $x \in \W_0$ then there is  a free disk $D$ and $n > 0$ such that $x,f^n(x) \in D$.  Choose   lifts $\ti x \in \ti D$ to $\tiM$, let $\ti C$ be a home domain for $\ti x$ and let $T$ be a covering translation  such that $\ti f_{\ti C}^n(x) \in T(D)$.    Thus $T$ is an $\ti f_{\ti C}$-near cycle for $\ti x$.  Since $\ti x$ and $\ti f_{\ti C}^n(\ti x)$ are both contained in $\ti U$, $T$ preserves $\ti U$ and so preserves $\ti C$.  Lemma~\ref{more home domains} (2) implies that $x \in \U$ thereby proving that $\W_0 \subset \U$.  Lemma~\ref{ U is interior of closure} therefore implies that $\W \subset \U$.

%Given  $x \not \in \U$, choose a lift $\ti x \in \tiM$ and let $\ti C$ be the home domain for $\ti x$.  If there is  a free disk $D$ and $n > 0$ such that $x,f^n(x) \in D$,  then there is  a lift $\ti D$  of $D$ that contains $\ti  x$ and a covering translation $T$ such that $\ti f_{\ti C}^n(x) \in T(D)$.   Thus $T$ is an $\ti f_{\ti C}$-near cycle for $\ti x$.  Since $\ti x$ and $\ti f_{\ti C}^n(\ti x)$ are both contained in $\ti U$, $T$ preserves $U$ and so preserves $\ti C$ in contradiction to Lemma~\ref{more home domains}-(2).  We conclude that there is no such $n$ and $D$ which proves that $x \not \in \W_0$ .  Thus $\W_0 \subset \U$ and  the interior of the closure in $\M$ of any component of $\W_0$ is contained in some $U \in \cA$ by Lemma~\ref{U tilde properties}-(5).    This proves that $\W \subset \U$.

To prove the converse note that the $\omega$-limit set of any point in
$\Check U$ lies in $U$ and hence contains points that are not fixed by
$f$.  It follows that $\Check U
\subset \W_0$.  
%Suppose that $a_i \ne 0$ and that $x \in \hat Y_i$. If the $\omega$-limit set of $x$ contains a point in $U$ then $x \in \W_0$ as in the previous case.  Otherwise there is a non-fixed point $z$ in $\omega(x, f_c)$.  If $D_0$ is a disk neighborhood of $z$ then $D_0 \cap U$ is a free disk that the orbit of $x$ intersects more than once and again $x \in \W_0$. Thus $(\hat Y_i \cap U) \subset \W_0$  when $a_i \ne 0$.
  If both $a_0$ and $a_1$ are non-zero then $U = \Check U \cup (U
   \cap (\hat Y_0 \cup \hat Y_1)) \subset \W_0 \subset \W$ by  item   (4) of Lemma~\ref{lem: Check U}. If both $a_0$ and $a_1$ are zero then $\Check U$ is   dense in $U$ by  item   (3) of Lemma~\ref{lem: Check U}. Thus $U \subset \Int_M\cl_M(\Check U) \subset \W$ since $\Check U$ is a connected (item   (2) of Lemma~\ref{lem: Check U}) subset of $\W_0$.    For the remaining case we may assume
 that $a_0 =0$ and $a_1 \ne 0$.  Then $U \subset\Int_M\cl_M(\Check U
 \cup (\hat Y_1 \cap U))\subset \W$ because $\Check U
 \cup (\hat Y_1 \cap U)$ is a connected  subset of $\W_0$.
  \end{proof}

\noindent
{\bf Theorem \ref{thm: C(x)}.}
{\em
Suppose $F \in \newDiff$ has entropy zero, has
infinite order and at least three periodic points. Let $f =
F|_{\M}$ where $\M = S^2 \setminus \Fix(F)$ and let $\cA$ be
as in Theorem~\ref{thm: annuli}.  For $U \in \cA$, let $f_c
:U_c \to U_c$ be the annular compactification of $f|_U:U \to
U$.  Then}
\begin{enumerate}
\item  {\em  The rotation number $\rho_{f_c}(x)$  is defined and continuous at every  $x \in U_c$.}
\item   {\em If $\Fix(F)$ contains at least three points then
$\rho_{f_c}$ is non-constant.}
\item  {\em If $C$ is a component of a level set of $\rho_{f_c}$ 
then $C$ is $F$-invariant.  If $C$ does not contain a component
of $\partial U_c$ then it is essential in $U$, meaning that $U_c \setminus C$   has two components each containing a component of $\partial U_c$. }

 \end{enumerate}
 
%\jmfc{ I think $ \partial_i U_c \subset \hat Y_i$ so I replaced $\hat Y_i \cup \partial_i U_c$ with $\hat Y_i$ throughout.}

\begin{proof}   For notational simplicity we write $\rho(x) = \rho_{f_c}(x)$.  We know from Proposition~\ref{prop: well-defined}  that $\rho(x)$ is defined for all $x \in U_c$ so to prove (1) we must verify continuity of $\rho$.

Assume notation as in Definitions~\ref{defn: hat Y}. % Let $\rho =\rho_{f_c}$. 
 Recall that $a_i$ is the rotation number of $f_c$
on $\partial_i U_c.$ By item (5) of Lemma~\ref{lem: Check U}, $\rho(y) = a_i$
for all $y \in \hat Y_i$.  By construction, there is a
sequence of primes $\{q_k\}$ tending to infinity such that $N_k =
V_i(q_k) \cup \hat Y_i$ is a nested sequence of $f_c$-invariant
neighborhoods of $\hat Y_i$  whose frontiers have rotation number
$p/q_k$ with $0 < p < q_k$.  For sufficiently large $k$, $p/ q_k \ne
a_0$ and we conclude that $\hat Y_i$ is a level set for $\rho_{f_c}$.
Also $\partial_i U_c$ can be identified with $\partial_i V_i(q_k)_c$  since $N_k$ is a neighborhood of $Y_i$ in both $cl(U)$ and $cl(V_i(q_k))$.
It follows that $a_i \in \rho(V_i(q_k))$ for all $k$.  Proposition~\ref{prop: well-defined}  implies that  if $r \in \rho(V_i(q_k))$   then $r$ is in the rotation interval   for the induced action of $f$ on 
 $V_i(q_k)_c$.   Since the length of this interval tends to zero as $k \to \infty$ (Corollary~\ref{cor: V(q)}),  the same is true for the length of $\rho(V_i(q_k))$.   This proves continuity of  $\rho$ at points of $Y_i$.

The level sets   $C(x)$ for $x \in \Check U$ are defined similarly.   We specify $Q = Q(x) $ by a series of largeness conditions.
By Lemma~\ref{lem: ends} (3) and the assumption that $x \in
\Check U$, we may assume that $x \not \in V_0(q) \cup
V_1(q)$ for $q \ge Q$.  In particular, $\omega(x) \subset
U$.  We may also assume that $\rho(x) \ne p/q$ for $q \ge Q$
and $0 < p < q$.  Lemma~\ref{lem: not in V} therefore
implies that $x$ is contained in some $V(q,x) \in \cA(q)$
which is essential by Lemma~\ref{lem: cA_q}. Since $x \not
\in V_0(q) \cup V_1(q)$, we have $\cl(V(q,x)) \subset U$.% Since $x \not
%\in V_0(q) \cux\cl(V(q,x)) \subset U$.

  Define
 \[
C(x) = \bigcap_{q > Q} \cl( V(q,x)),
\]
where the intersection is over all primes $>Q.$

Given $q$ let $\delta$ be the minimum value of $|p/q - \rho(x)|$ for
$0 < p < q$.  If $q' >1/\delta$ then $\rho(V(q',x))$ does not contain
$p/q$ for $0 < p < q$ and so does not contain any points in the
frontier of $V(q,x)$.  It follows that $\cl(V(q',x)) \subset V(q,x)$.
We may therefore choose a sequence of primes $\{q_j\}$ tending to
infinity such that
\[
C(x) = \bigcap_{q_j} \cl( V(q_j,x))
\]
and
$$
 \cl(V(q_{j+1},x))   \subset V(q_j,x)
$$
for all $j$.  This proves that $C(x)$ is non-empty and does not depend on the choice of
$Q$ in its definition and so is well-defined.

 Item (1)  follows from the fact that  $|\rho(V(q,x))| \le 1/q$. 

If $y \notin C(x)$ then there exists $q
> Q$ such that $y \notin \cl(V(q,x)).$ The  frontier of
$V(q,x)$ separates $C(x)$ and $y$ and $q$ may be chosen so
that this frontier consists of points whose rotation number
is not equal to $\rho(x)$ so $y$ is not in the same
connected component as $x$ of the level set of $\rho$.  It
follows that $C(x)$ is a connected component of this level
set.  Since it is clear from the construction that $C(x)$ is $F$-invariant and  essential, we have proved (3).

If $\Fix(F)$ contains at least three points then $U$ has at least one non-singular end.
It then follows from Lemma~\ref{lem: U rot int} that $\rho(U)$ is non-trivial so 
$\rho_{f_c}$ is non-constant. This verifies (2) and completes the proof.  
\end{proof}

\section{Proof of Theorem~\ref{annulus with singularities}}

The proof of Theorem~\ref{annulus with singularities} is given at the end of this section following the statement and proofs of some preliminary lemmas.  Recall that $\mrA = \Int A$ and  that if $H \in \newDiffA$ then the homeomorphism $F:S^2 \to S^2$ obtained from $H$ by collapsing each component of $\partial A$ to a point satisfies   $F\in \newDiff$.   As throughout the paper, $\M = S^2 \setminus \Fix(F)$ and $f = F|_{\M}$.     We identify $\M$ with $A \setminus (\Fix(H) \cup \partial A)$ and $H|_{\M}$ with $f$.

\vspace{.1in}

%\noindent{\bf  Theorem~\ref{annulus theorem}} {\em Suppose that $F:A \to A$ is an area preserving $C^\infty$ diffeomorphism of the closed annulus $A$.  If $F $ has entropy zero   then the rotation number $\rho_F(x)$ is defined and continuous at each $x \in A$. }

%\mhc{  The notation $\Diff_{\mu}(A)$ does not fit with the rest of the paper.  Should we say "area preserving $C^\infty$ diffeomorphism".} \jmfc{See introduction; we should probably do $\Diff_\mu(A,P).$}

\noindent{\bf  Theorem~\ref{annulus with singularities}} {\em For each $H \in \newDiffA$ with entropy zero, the rotation number $\rho_H(x)$ is defined and continuous at each $x \in A$.  }

\vspace{.1in}

% Identify $\ti A$ with $\R \times [0,1]$ and let $p_1 : \R \times [0,1] \to \R$ be projection onto the first coordinate.   Let $T$ be the covering translation that $T(r,s) = (r+1,s)$.  

We assume without loss that $H$, and hence $F$, has infinite order.

\begin{lemma} \label{same rotation number}Suppose that $H \in \newDiffA$ has entropy zero and  that $\Fix(H)$ contains at  least one point 
in $\mrA$.   Let $\cA$ be as in Theorem~\ref{thm: C(x)} applied to the element $F \in\newDiff$ corresponding to $H$.  
\begin{enumerate}
\item \label{item:essential} If $U \in \cA$ is essential 
 in $A$ then $\rho_H(x) =  \rho_{f_c}(x)$ for all $x \in U$.  If a component $\partial_0A$ of $\partial A$ is a frontier component of $U$, then $\rho_H$ is defined and continuous on a neighborhood of $\partial_0A$.%  $\rho_F|_{U \cup \partial_0A} = \rho_{f_c}|_{U \cup \partial_0A}$.
\item \label{item:inessential} If $U \in \cA$ is inessential  in $A$ then $\rho_H(x) =  0$ for all $x \in U$.
\end{enumerate}
\end{lemma}

\proof   We first consider the case that $U \in \cA$ is essential.  Theorem~\ref{thm: C(x)} (2) and Theorem~\ref{thm: translation interval} imply that the image of $\rho_{f_c}$ is a non-trivial interval. Suppose that  $B$ is a component   of a level set  of    $\rho_{f_c}$.  If $B$ is disjoint from   $\partial U_c$, then $B$ is   $H$-invariant and essential  by Theorem~\ref{thm: C(x)} (3).    %If $C$ is a subannulus   bounded by any two of these level sets, then $B$ 
Since any such $B$ is contained in a closed  annulus in  $U$ that is essential in both $U$ and $A$,     Lemma~\ref{lem: 2ann} implies that $\rho_{f_c}|_B =   \rho_H|_B $.  It therefore suffices to prove the lemma for $B$  containing a component of $\partial U_c$.

Since $U$ is essential, $B$ corresponds to a singular end of $U$ if and only if  it corresponds to an end of $\mrA$ determined by a  component, say $\partial_0A$, of $\partial A$.  In this case,  Lemma~\ref{U is an annulus} (3) implies that $U \cup \partial_0A$ is a neighborhood of $\partial_0A$ in $A$. 
\ Since $\rho_{f_c}$ is not constant there is a core curve in $U$ which
separates $B$ from the end of $U$ which does not correspond to $B$.
Let $N$ be this curve together with the component
of its complement in $U$ which contains the end corresponding to $B$.
Note that $N$ is a neighborhood of an end of $U$ and also a neighborhood
of an end of $\mrA.$ The compactification of this end can be done in $N$, so
it is the same in $U_c$ and
in $A$.  Hence $\partial_0 A$ can be thought of as  a component of  $\partial U_c$ and $B$ can be thought of as a subset of $A$.  Since 
$N \cup \partial_0A$ is a closed  annulus in $U_c$ 
containing $B$ which is disjoint from the other component of $\partial U_c$, Lemma~\ref{lem: 2ann} implies that $\rho_H|_B=\rho_{f_c}|_B$. Since  $\rho_{f_c}$ is continuous, $\rho_H$ is defined and continuous on a neighborhood of $\partial_0A$.

We may therefore assume that $B$ corresponds to a non-singular end of $U$. Let $X_A$ be the component of the frontier of $U$ in $A$ determined by this non-singular end and let  $B_A  = X_A \cup \ (B\cap U)$. In other words, $B_A \subset A$ is obtained from $B\subset U_c$ by replacing a component of the frontier of $U$ in  $U_c$ with a component of the frontier of  $U$ in $A$.     Lemma~\ref{U is an annulus} (4) implies that $\rho_{f_c}|_B = 0$ so it suffices to show that $\rho_H|_{B_A} = 0$.     For reference below, we note that the remainder of the proof of \pref{item:essential} makes no use of the fact that $U$ is essential.

Let $\ti U_A \subset \ti A$ be the lift of $U$ to the universal (cyclic) cover $\ti A$ of $A$ and let $\ti X_A$ be the component of the frontier of $\ti U_A$ in $\ti A$  that projects to $X_A$.  Denote the full pre-image of $\Fix(H)$ by $\widetilde{\Fix(H)}$. We claim that there is a lift $\ti H : \ti A \to \ti A$ that fixes each point in $\ti X_A \cap \widetilde{\Fix(H)}$.  Up to isotopy rel $\Fix(H)$, $H$ is isotopic to a composition of Dehn twists along a finite set $\Sigma$ of disjoint simple closed curves  in $A \setminus \Fix(H)$. (See Section~\ref{sec: normal form}.)
 To prove the claim, it suffices to show that no two points in $\Fix(H) \cap X_A$  are separated from each other by an element   $\sigma$ of $\Sigma$.  

Let $M$ be the component of $\mrA \setminus \Fix(H)$ that contains $U$.  If $M$ is an annulus then $M=U$ and $X_A \subset  \Fix(H) \cup \partial A$.
In this case, $X_A \cap \Sigma= \emptyset$  so the claim is clear. We may therefore assume that $M$ has at least three ends.
Equip $M$ with a complete hyperbolic metric in which all isolated punctures are cusps, in which the core curve $\tau$ of $U$ is a geodesic and in which the  elements of $\Sigma$ are disjoint simple closed geodesics that have no transverse intersections with $\tau$.  If $\tau$ is not an element of $\Sigma$, let $C$ be the component of $M \setminus \Sigma$   that contains $\tau$.  Corollary~\ref{stays close} (1) implies that $U$ is contained in a bounded (as measured in the hyperbolic metric) neighborhood of $C$.   In particular,  all points in the intersection of $\Fix(H)$ with the closure of $U$ in $A$ are contained in the component of $A \setminus \Sigma$ that contains $C$.    Since any two such points can be connected by an arc in $A \setminus \Sigma$, the claim is proved.   If $\tau$ is   an element of $\Sigma$, let $C_1$ and $C_2$ be the components of $M \setminus \Sigma$ on either side of $\tau$.  Corollary~\ref{stays close} (2) implies that $U$ is contained in a bounded (as measured in the hyperbolic metric) neighborhood of $C_1 \cup C_2$.   Since $X_A$ is disjoint from $\tau$, we may assume that $X_A$  is contained in the closure of $C_1$ and the proof of the claim concludes as in the previous case.

 Identify $\ti A$ with $\R \times [0,1]$ and let $p_1 : \R \times [0,1] \to \R$ be projection onto the first coordinate.   Let $\ti B_A$ be the pre-image of $B_A$ and choose $\ti y \in \ti B_A$. For $\delta > 0$,   we say that {\em $k$ is $\delta$-good} if $p_1(\ti H^{k}(\ti y)) - p_1(\ti y) < k\delta$.     
We complete the proof of \pref{item:essential} by choosing $\epsilon > 0$ arbitrarily and showing that all sufficiently large $k$ are $\epsilon$-good.

    Lemma~\ref{more home domains} (3) implies that  for any neighborhood $W$ of $\Fix(H)$ there exists a positive integer $M$ so that  for all  $x \in X_A$,  we have $H^k(x) \in W$ for all but at most $M$ values of $k$.      We may therefore choose $K_1$ so that 
$$ p_1(\ti H^{K_1}(\ti x)) - p_1(\ti x) < K_1 \epsilon/2$$
for all $\ti x \in \ti X_A$ and hence for all $\ti x$ in a neighborhood $\ti V$ of $\ti X_A$.   Note that this inequality can be concatenated.  Thus, if $\ti H^{jK_1}(\ti x) \in\ti V$ for all $0 \le j \le J$ then $ p_1(\ti H^{JK_1}(\ti x)) - p_1(\ti x) < JK_1 \epsilon/2$.  In particular, if   the forward orbit of $\ti y$ is eventually contained in $\ti V$ then  all sufficiently large $k$ are $\epsilon$ good.    We may therefore assume that there exists arbitrarily large $k$ with $\ti H^k(\ti y) \not \in \ti V$.

There is a compact essential subannulus of $U$  whose lift to $\ti A$ contains $\ti B_A \cap (\ti A \setminus \ti V)$.  We may therefore assume that $p_1$ and the projection $p_1':\ti U_C \to \R$ used to define $\rho_{f_c}$ agree on $\ti B_A \cap (\ti A \setminus \ti V)$.  Since $\rho_{f_c}|B = 0$, we may assume after reducing the size of $\ti V$ if necessary, that there exists $K_2$ so that  $k$ is $\epsilon/2$-good whenever $k \ge K_2$ and $\ti H^k(\ti y)  \not \in  \ti V$.

Given arbitrary $k > K_2$, let $k'$ be the largest value between $k$ and $K_2$  such that $\ti H^k(\ti y)  \not \in  \ti V$ and let $m$ be the largest integer such that $l := K_2 + mK_1 < k$. Then $k'+mK_1$ is $\epsilon/2$-good, and $k-l$ is bounded by $K_1$.  It follows that $k$ is $\epsilon$-good for all sufficiently large $k$.  This completes the proof of  \pref{item:essential} 

 If $U$ is inessential then  the union of $U$ with one of the components of $A \setminus U$ is an open $H$-invariant disk.  If $B$ is a level set  of $\rho_{f_c}$ that does not contain a component of $\partial U_c$ then $B$  is contained in an $H$-invariant open disk  whose closure does not separate the boundary components of $A$.  It follows that the complete lift of this disk to $\ti A$ has bounded components and hence all points in the disk have $0$ rotation number.   The lemma therefore holds for all such $B$ and for the component of the frontier of $U$ in $A$ that is contained in the $H$-invariant disk.    We are therefore reduced to considering the level set $B$ corresponding to the other, necessarily non-singular, end of $U$.   The proof given above applies without change.	
 \endproof

\begin{lemma} \label{rotating boundary}   Suppose that $H \in \newDiffA$ has entropy zero and  that $\Fix(H)$ contains at  least one point 
in $\mrA$.   Let $\cA$ be as in Theorem~\ref{thm: C(x)} applied to the element $F \in\newDiff$ corresponding to $H$. %Suppose that $F \in \Diff_{\mu}(A)$ has entropy zero, that $\ Fix(F)$ contains at  least one  point in the interior of $A$  and that $\cA$ is as in Theorem~\ref{thm: C(x)}. 
  If $\partial_0A$ is a component of $\partial A$ and $\rho_H|_{\partial _0 A} \ne 0$ then $\partial_0A$ is a frontier component of some essential $U \in \cA$.
\end{lemma}

\proof     Since $\Fix(H) \cap \partial_0A = \emptyset$, there is a component $M$   of $\M$ that contains a deleted neighborhood $V$ of $\partial_0A$. % and let $f=F|_M$. 
 If $M$ is an annulus, then it is an element of $\cA$ and we are done.  We may therefore assume that $M$ has at least three ends.  Choose a component $\ti V$ of the full pre-image of $V$ in the universal cover   of $M$, let $T$ be the parabolic covering translation that preserves $\ti V$ and  let $P \in \sinfty $ be the unique fixed point of $T$.  After shrinking $V$ if necessary, we may assume that  $V$ is covered by a   finite collection of free disks, say $k$ free disks.   If $x \in V$ is sufficiently close to $\partial_0A$ then $f^j(x) \in V$ for all  $0 \le j \le k$ and so there exists $0 \le j_1\le j_2 \le k$ such that  $f^{j_1}(x)$ and  $f^{j_2}(x)$ belong to the same free disk.   It follows that some iterate of $T$ is a near cycle for all points in $\ti V$ that are sufficiently close to $P$.  The domain containing $P$ in its closure is a home domain for all such points by Corollary~\ref{stays close} and is obviously $T$-invariant.  Lemma~\ref{isolated puncture} therefore implies that $T$ is the covering translation associated to some $\ti U \in \ti A$ and Lemma~\ref{U is an annulus} (3) implies that $U$ contains a deleted neighborhood of $\partial_0A$.
\endproof

\begin{lemma} \label{short intervals} Suppose that $H :A \to A$ is a homeomorphism of the closed annulus and that $U_1, U_2,\ldots$ is an infinite sequence of disjoint invariant open essential annuli in $A$.   Let $H_i :U_i^c \to U_i^c$ be the annular compactification (see Notation~\ref{notn: annular comp})   of   $H|_{U_i}$ and let   $L_i$ be the length of the forward translation interval (see Definition~\ref{translation number}) of some (any) lift of $H_i$.  Then $L_i \to 0.
$\end{lemma}

\proof      If the lemma fails then, after passing to a subsequence, there  exists  $\epsilon > 0$ such that $ L_i> \epsilon$ for all $i$.   After replacing $H$ with an iterate, we may assume that   $ L_i >2$ for all $i$.  Since each $U_i$ contains points with rotation number $1/2$, we may choose $y_i \in U_i$ such that $H(y_i)$ is antipodal to $y_i$ in the $S^1 \times [0,1]$ structure of $A$.   After passing to a further subsequence we may assume that  $y_i  \to x$  for some non-fixed $x \in A$ and that $U_{i+1}$ separates $U_i$ from $x$ for all $i$.    Choose a free disc neighborhood of $x$ and an  arc $\nu$ in this free disk that begins at $x$  and ends at a point in some $\fr (U_{i_0})$.  For all $i > i_0$,  there is a subarc $\nu_i$ with interior in $U_i$ and with endpoints on both components of $\fr(U_i)$.   %We now fix such an $i$ and drop the $i$ subscript from the notation.
Transporting this to  $H_i :U_i ^c \to U_i ^c$, there is an arc $\nu_i $ with interior in $U_i $,  with endpoints on distinct components of $\partial U_i ^c$ and satisfying $H (\nu_i  ) \cap \nu_i  = \emptyset$.   

We now fix such an $i$ and drop the $i$ subscript from the notation, renaming
%To emphasize that the relationship between $h$ and $h_i$ is no longer relevant, we rename 
$H_i :U_i ^c \to U_i ^c$ by $h :A \to A$ and the arc $\nu_i$ by $\nu$.  Identify $\ti A$ with $\R \times [0,1]$ and let $p_1 : \R \times [0,1] \to \R$ be projection onto the first coordinate.   Let $T$ be the covering translation $T(r,s) = (r+1,s)$, let $\ti \nu$ be a lift of $\nu$ and let $\ti h$ be the lift  of $H$ such that $\ti H(\ti \nu)$ is contained in the interior of the region bounded by  $\ti \nu$  and $T(\ti \nu)$.  Then $\ti H^2(\ti \nu)$ is contained in the interior of the region bounded by $\ti h(\ti \nu)$ and  $\ti h(T(\ti \nu))=T(\ti h(\ti \nu))$ and so  is also  contained in the interior of the region bounded by  $\ti \nu$   and  $T^2(\ti \nu)$.  Continuing in this manner we conclude that $\ti h^k(\ti \nu)$ is  contained in the interior of the region bounded by  $\ti \nu$  and $T^k(\ti \nu)$, which implies that that forward translation interval for $\ti h$ has length at most one.  This contradiction completes the proof.
\endproof

\noindent {\bf Proof  of Theorem~\ref{annulus with singularities}} \ \  Let $F \in \newDiff$ be the  element $F \in\newDiff$ corresponding to $H$.  If $H|_{\mrA}$ has no interior periodic points then    every point in $A$ has the same irrational rotation number by Theorem~\ref{thm: translation interval}.  We may therefore assume that  $F$ satisfies the hypotheses of Theorems~\ref{thm: annuli} and \ref{thm: C(x)}.    Let $\cA$ be the set of annuli produced by those theorems,   let $\U = \cup_{U \in \cA}U$, let $\U_e$ be the union of all essential elements of $\cA$ and let $\U_e'$ be the union of $\U_e$  with any components of $\partial A$ for which $\rho_F$ is non-zero.  After replacing $H$ with an iterate, we may assume that $H$ has at least one interior fixed point.  Lemma~\ref{rotating boundary} implies that $\U_e'$ is an open subset of $A$.

     Each $x \in  A \setminus \U_e'$   satisfies one of the following.
\begin{itemize}
\item $x$ is contained in a components of $\partial A$ with zero rotation number.
\item $x \in \Fix(H)$.
\item $x \in \mrA \setminus \U$.
\item $x$ is contained in an inessential element of $A$.
\end{itemize}
In all of these cases $\rho_H(x) = 0$.  This is obvious for the first two and follows from Theorem~\ref{thm: annuli} (3) for the third and  Lemma~\ref{same rotation number} \pref{item:inessential} for the fourth.  

To complete the proof of  Theorem~\ref{annulus with singularities}, it  suffices to show that $\rho_H$ is defined and continuous on the closure of $\U'_e$.     By Theorem~\ref{thm: C(x)} (1), $\rho_H$ is defined and continuous on $\U_e$.  It is obvious that $\rho_H$ is defined on $\partial A$.  Lemma~\ref{same rotation number} \pref{item:essential} implies that $\rho_H$ is continuous at points in a component of $\partial A$ with  non-zero rotation number.  
It remains only to show that if $x_i \to x$ where $x_i \in \U_e$ and $x \in \fr(\U_e)$ does not belong to a boundary component with non-zero rotation number,   then $\rho_H(x_i) \to 0$.  By   Lemma~\ref{short intervals} we may assume that the $x_i$'s belong to a single essential $U \in A$.  Since $\rho_H(x_i) = \rho_{f_c}(x_i)$ and $\rho_{f_c}$ is continuous, it suffices to show that $\rho_{f_c}|_{\partial_0U_c} =0$ where $\partial_0U_c$ is the component of $\partial U_c$ to which the $x_i$'s converge.   This follows from Lemma~\ref{U is an annulus} (4) if $x$ is not contained in $\partial A$ and is obvious if $x \in \partial A$ because we have excluded components of $\partial A$ with non-zero rotation number.
\endproof

\section{The proof of Theorem~\ref{thm: ent0}.}

Recall that a group $G$ is called {\em indicable} if 
there is a non-trivial homomorphism $\phi: G \to \Z.$
We say $G$ is {\em virtually indicable} if it has
a finite index subgroup which is indicable.

\begin{prop}\label{prop: +ent}
Suppose that $S$ is a surface and 
$F: S \to S$ is $C^{1+\epsilon}$ and has positive
topological entropy.  Then every finitely generated 
infinite subgroup $H$ of the centralizer $Z(F)$of $F$  
is virtually indicable and has a finite index subgroup that 
has a global fixed point.
\end{prop}

\begin{proof}
  A result of Katok \cite{katok:horseshoe} asserts that $F^q$ has a
  hyperbolic  saddle fixed point $p$ for some $q\ge 1.$  The orbit of $p$ under
  $H$ consists of hyperbolic fixed points of $F^q$ at which the
  derivative of $DF^q$ has the same eigenvalues as $DF_p^q.$ If the
  $H$ orbit of $p$ were infinite, continuity of the derivative would
  imply that at any limit point of this orbit $DF^q$ would have the
  same eigenvalues and in particular would be hyperbolic.  But this is
  impossible since hyperbolic fixed points are isolated.  We conclude
  the orbit of $p$ under $H$ is finite and hence that the subgroup
  $H_0$ of $H$ that fixes $p$ has finite index.

  After passing to a further finite index subgroup we may assume that
  $Dh_p$ has positive eigenvalues and the same eigenspaces as $DF_p$
  for each $h \in H_0$.  For each eigenspace the function which
  assigns to $h$ the log of the eigenvalue of $Dh_p$ on that
  eigenspace is a homomorphism from $H_0$ to $\R$.  If this is
  non-trivial we are done.  Otherwise both eigenvalues are $1$ for
  each $Dh_p$.  Hence in the appropriate basis
  \[
  Dh_p =
  \begin{pmatrix}
    1 & r_h\\
    0 & 1
  \end{pmatrix}
  \]
  for some $r_h \in \R$.  The function $h \mapsto r_h$ defines a
  homomorphism from $H_0$ to $\R$, so we are done unless $r_h = 0$ for
  all $h\in H_0.$ But in this latter case $Dh_p = I$ for all $h \in H_0$
  so we may apply the Thurston stability theorem
  (\cite{thurston:stability}; see also Theorem 3.4 of
  \cite{franks:distortion}) to conclude there is a non-trivial
  homomorphism from $H_0$ to $\R.$   
\end{proof}

\begin{exs}\label{example}
{\rm Let $S = S^2$ be the unit sphere in $\R^3$.  Let $F: S \to S$
be a diffeomorphism whose restriction to each of the level sets
$z =c$ is a rotation of that circle and with the property that
$F = id$ for all points $(x,y,z)$ with $|z| \ge 3/4.$ We assume
that $F$ is not the identity on the equator $z = 0.$ Let
$g: S \to S$ be a rotation about the $z$-axis by an angle which
is an irrational multiple of $\pi.$ Let $h:S \to S$ be a 
diffeomorphism supported in the interior of the disks $|z| > 3/4$
with the property that $h$ preserves area and the $h$-orbits of $(0,0,1)$
and $(0,0,-1)$ are infinite. Let $G$ be the group of all rotations
about the $z$-axis through angles which are rational multiples
of $\pi.$}
\begin{enumerate}
\item  The group $H$ generated by $g$ and $h$ lies in the
centralizer $Z(F)$ of $F$ but has no finite index subgroup
with a global fixed point.
\item  The group $G$ is a subgroup of $Z(g)$. Every element
of $G$ has finite order so there are no
non-trivial homomorphisms from any subgroup of
$G$ to $\R$ and hence $G$ is not virtually indicable.
\end{enumerate}
\end{exs}

%\trycomment{ Is there an example which is not virtually indicableand which has no finite index subgroup with a global fixed point?}

 The first example above shows that we cannot generalize 
Proposition~\ref{prop: +ent} to the centralizer
of a diffeomorphism $F$ with zero entropy, even in the 
group of area preserving diffeomorphisms.  The second example
shows the necessity of the hypothesis of finitely generated in
the following.

\bigskip
\noindent
{\bf Theorem~\ref{thm: ent0}.}
{\em If $F \in \Diff_{\mu}(S^2)$ has infinite order then each finitely
generated infinite subgroup $H$ of $Z(F)$ is virtually indicable.}

\begin{proof}
The case that $F$ has positive entropy is covered
by Proposition~\ref{prop: +ent} so we need only
consider the case when $F$ has entropy zero.
We assume that every finite index subgroup of $H$ 
admits only the trivial homomorphism to $\R$ and
show this leads to a contradiction.

Assuming for now that $\Per(F)$ contains at least three points,  we may apply Theorem~\ref{thm: annuli} to $F$ and its iterates obtaining the families $\cA(q)$   of $F^q$-invariant annuli 
guaranteed by that theorem.   Since there is no loss in replacing  $F$ by an iterate, we may assume that $F$ has at least three fixed points.    Choose once and for all an element $U \in\cA = \cA(1)$.   Item (2) of Theorem \ref{thm: C(x)} implies that $f_c :U_c \to U_c$ has a non-trivial rotation interval so by Theorem~\ref{thm: translation interval}  we may choose $x \in U$   such that $\rho_{f_c}(x) =
\lambda$  is irrational and not equal to the rotation number of either component of $\partial U_c$.
   
   By Corollary~\ref{second centralizer invariance}, each $h \in H$ permutes the elements of 
$\cA.$ In particular, for any $h \in H$ the open annuli $U$ and
$h(U)$ must be disjoint or equal.
Since elements of $H$ preserve area the $H$ orbit of the open set $U$
must be finite.  We let $H'$ be the finite index subgroup of $H$ which
leaves $U$ invariant.

Let $C(x)$ be the component of the level set of $\rho_{f_c}$ that contains $x$.  Since  $h \in H'$  preserves level sets of $\rho_{f_c}$, $h(C(x))$ is either equal to or disjoint from $C(x)$.   By Theorem \ref{thm: C(x)} (3), $C(x)$ is essential in $U$.  Since $h$ preserves area, it cannot move $C(x)$ off of itself and we conclude that $C(x)$ is $h$-invariant.  

Choose a sequence of primes  $\{q_n\}$ tending to infinity.
By Corollary~\ref{cor: ess V}, for $n$ sufficiently large, $x \in V_n$ 
for some essential $V_n \in
\cA(q_n)$.  Lemma~\ref{lem: not in V}  implies that $C(x)$ is disjoint
from the frontier of $V_n$ and hence contained in $V_n$.  Since $h$
preserves $C(x)$ and permutes the elements of $\cA(q_n)$, it follows
that $V_n$ is $h$-invariant.

Choose one component, $V_n^+,$ of the complement of $C(x)$ 
in $V_n$ in such
a way that $V_{n+1}^+ \subset V_n^+$, i.e., always choose the component
on the same side of $C(x).$  Let
$\bar A_n$  denote $V_n^+$ with its ends compactified 
by the prime end compactification.
Let $\partial^+ A$ denote the
circle of prime ends added to the end corresponding to $C(x).$ The  
natural identification of these circles for different
$n$ is  reflected   in the notation which is independent of $n.$
Let $A_n = V_n^+ \cup \partial^+ A,$ i.e. $V_n^+$ with only
one end (the one corresponding to $C(x)$) compactified.  Then
$A_{n+1} \subset A_n$ and 
\[
\bigcap_{n>0} A_n = \partial^+ A.
\]
Let $\bar f: \bar A_n \to \bar A_n$ 
and $\bar h: \bar A_n \to \bar A_n$ denote the natural extensions of
$F$ and $h \in H'$ to $\bar A_n.$

The rotation number $\rho(\bar f|_{\partial^+ A})$ of the 
restriction of $\bar f$ to
$\partial^+ A$ must be $\lambda.$ This is because if it were not
and $p/q$ is between $\rho(\bar f|_{\partial^+  A})$ and $\lambda$ then
by Theorem~\ref{thm: translation interval} applied to $ \bar A_n$ 
there would be periodic points in the
interior of $\bar A_n \subset V_n$ with rotation number $p/q$ for all $n$, a
contradiction.

For each $n$ there is a homomorphism $\phi_n: H' \to S^1 = \R/\Z$
given by $h \mapsto \rho_\mu(h|_{\bar A_n})$ where $\rho_\mu(h|_{\bar A_n})$
denotes the mean rotation number of $h$ on the annulus $\bar A_n$
(see Definition~\ref{defn: mean rot}).
Let $H''$ denote the subgroup of $H'$ which is the kernel of the
canonical homomorphism from $H'$ to its abelianization.  Then
$\rho_\mu(h|_{\bar A_n}) = 0$ 
for all $h \in H''.$ Also the abelianization
of $H'$ must be finite since this is one of the equivalent conditions 
for $H'$ not to be indicable.  Therefore $H''$ has finite index in $H'$ (and
hence in $H$). 

Since $\rho_\mu(\bar h|_{\bar A_n})=0$ for each $n$ and 
each $h \in H''$ 
we conclude from Proposition~\ref{prop: mean rot}
that each $\bar h$ has a fixed point $x_n$ in 
 the interior of $\bar A_n$, i.e. in $A_n^+$, for all $n$.
Let $B$ be the closed disk which is the union
of  $\partial^+ A$ and the component of the complement
of $C(x)$ in $S^2$ which contains $V_n^+.$ 
Then $cl_B(V_n^+)$ contains
a fixed point $x_n$  of $\bar h.$ 

Taking the limit of a subsequence we note that for each $h \in H''$
there is a fixed point of $\bar h$ in $\partial^+ A$ .  But the rotation
number of $\bar f$ on $\partial^+ A$ is irrational so $\bar f$ has a
unique invariant minimal set which is the omega limit set $\omega(x,
\bar f)$ for each $x \in \partial^+ A$.  Since $\bar f$ preserves
$\Fix(\bar h)$ we conclude this minimal set is in $\Fix(\bar h)$.
Since the minimal set depends only on $\bar f$ and not $\bar h$
we conclude that the this minimal set is in $\Fix(\bar h)$ for
every $ h \in H''$.

We have found a prime end (in fact infinitely many ) in 
$\partial^+ A$ which is fixed by $\bar h$  for 
every $ h \in H''$. It follows from
Corollary~\ref{cor: frontier gfp} that there is a point of $\Fix(H'')$
in $cl(V_n^+)$ for each $n$.  Taking the limit of a subsequence again
we find a point of $\Fix(H'')$ which lies in $\bigcap_n cl(V_n^+) =C(x).$

Choosing an infinite collection $\{\lambda_i\}$ of distinct irrationals in
the rotation interval of $F|_U$ and repeating the construction
we obtain an infinite collection of global fixed points for 
$H''$ with distinct rotation numbers for $F|_U.$  They must
possess a limit point in $\Fix(H'').$

Proposition~(3.1) of \cite{fh:morita} asserts that if there is an
accumulation point of $\Fix(H'')$ then there is a homomorphism from
$H''$ to $\R$.  So $H''$ is indicable. 

 We are left with addressing the special case that $\Per(F)$ contains
only two points. Since $F$ cannot have an empty fixed point set
we conclude $\Per(F) = \Fix(F)$ and this set contains two points. 
If $H = Z(F)$ is the centralizer of $F$ then it
has an index two subgroup $H'$ which fixes both points and hence
the annulus $U = S^2 \setminus \Per(F)$ is invariant under $H'$ with
each element isotopic to the identity.
Let $f = F|_U$. Then $\rho_f(U)$ consists of a single point,
by Theorem~\ref{thm: translation interval}.  
By Proposition~\ref{prop: +meas}
(applied to iterates of $f$) it cannot consist of a single rational
in $\R/\Z.$ We conclude that $\rho_f(U)$ contains a single irrational
number $\lambda.$

Blowing up the two fixed points of
$F$ we obtain the annular compactification  
homeomorphism $f_c : U_c \to U_c.$
The restriction to the boundary component corresponding
to the fixed point $x$ is conjugate to the projectivization of $DF_x$.
It must have rotation number $\lambda$ since otherwise there would be
 additional periodic points in $U$  by Theorem~\ref{thm: translation interval}.

This map on the boundary circle is the projectivization of an element
of $SL(2,\R)$, i.e. a fractional linear transformation.  Since its
rotation number is irrational it
is an irrational rotation in appropriately chosen coordinates.  It follows that the restrictions of
blow-ups of elements of $H'$ to this circle are rotations, since the
centralizer of an irrational rotation consists of rotations.
Therefore this group of restrictions  is abelian.
It is finitely generated because it is the image under a homomorphism of a finitely generated
group. Since it admits no non-trivial homomorphisms to
$\R$ and is finitely generated it must be finite.  
We conclude there is a finite index subgroup
$H''$ of $H'$ whose restrictions to the boundary circle are all the
identity. In other words, the projectivization of $Dh_x$ is the
identity for all $h \in H''$. Since there are no non-trivial
homomorphisms from $H''$ to $\R$, $\det(Dh_x) =1$.  The Thurston
stability theorem (\cite{thurston:stability}) therefore produces a
non-trivial homomorphisms from $H''$ to $\R$ and we have arrived at
the desired contradiction.
\end{proof}

 We now provide the proof of Corollary~\ref{no actions}.

\bigskip
\noindent
{\bf Corollary~\ref{no actions}.}  
{\em 
If $\Sigma_g$ is the closed orientable surface of genus $g \ge 2$ then at least one of the following holds.
\begin{enumerate}
\item  No finite index subgroup of $\mcg(\Sigma_g)$ acts faithfully on $S^2$ by area preserving diffeomorphisms.
\item For all \  $1 \le k \le g-1$,  there is an indicable finite index subgroup $\Gamma$ of  the bounded mapping class group $\mcg(S_k,\partial S_k)$ where $S_k$ is the  surface  with genus $k$ and connected non-empty boundary.
\end{enumerate}
}

\begin{proof}
We assume that (1) fails, i.e., that there is a finite index subgroup
$G$ of $\mcg(\Sigma_g)$ which acts faithfully on $S^2$ by area
preserving diffeomorphisms, and show that this implies (2).

Suppose \ $1 \le k \le g-1$ and $S$ is the compact surface with
genus $k$ and a connected non-empty boundary, $\partial S.$ 
We assume $S$ is embedded in  $\Sigma_g$ with $\partial S$ a separating
closed curve and let $S'$ be the closure of the complement of $S$, 
a surface with genus $g -k$ and boundary $\partial S.$
There is a natural embedding of $\mcg(S, \partial S)$
into $\mcg(\Sigma_g)$ obtained
by extending a representative of an element of $\mcg(S, \partial S)$ to 
all of $\Sigma_g$ by letting it be the identity on the complement
of $S$. Similarly there is a natural embedding of $\mcg(S', \partial S')$
into $\mcg(\Sigma_g)$.  If $\Gamma_0$ and $\Gamma_0'$ are the images of
these two embeddings it is clear that every element of $\Gamma_0$ commutes
with every element of $\Gamma_0'$ since they have representatives in
$\Diff(\Sigma_g)$ which commute.

We let $\Gamma_1 = \Gamma_0 \cap G$ and $\Gamma_1' = \Gamma_0' \cap
G$.  Since $\Gamma_1'$ has finite index in $\mcg(S', \partial S)$ it
contains an element $\gamma$ of infinite order.  Suppose $\phi: G \to
\Diff_\mu(S^2)$ is the injective homomorphism defining the action of
$G$.  Let $F = \phi(\gamma)$.  Then $\phi(\Gamma_1)$ is in the
centralizer $Z(F).$ According to Theorem~\ref{thm: ent0}, $\Gamma_1$
is virtually indicable.  Therefore there is an indicable $\Gamma$ of
finite index in $\Gamma_1$.
\end{proof}

\bibliographystyle{plain}
%%%%% \bibliography{zero}

\begin{thebibliography}{10}

\bibitem{FLP}
{\em Travaux de {T}hurston sur les surfaces}, volume~66 of {\em Ast\'erisque}.
\newblock Soci\'et\'e Math\'ematique de France, Paris, 1979.
\newblock S{\'e}minaire Orsay, With an English summary.

\bibitem{bowen}
Rufus Bowen.
\newblock Entropy and the fundamental group.
\newblock In {\em The structure of attractors in dynamical systems ({P}roc.
  {C}onf., {N}orth {D}akota {S}tate {U}niv., {F}argo, {N}.{D}., 1977)}, volume
  668 of {\em Lecture Notes in Math.}, pages 21--29. Springer, Berlin, 1978.

\bibitem{brnkist}
M.~Brown and J.~M. Kister.
\newblock Invariance of complementary domains of a fixed point set.
\newblock {\em Proc. Amer. Math. Soc.}, 91(3):503--504, 1984.

\bibitem{casble:nielsen}
Andrew~J. Casson and Steven~A. Bleiler.
\newblock {\em Automorphisms of surfaces after {N}ielsen and {T}hurston},
  volume~9 of {\em London Mathematical Society Student Texts}.
\newblock Cambridge University Press, Cambridge, 1988.

\bibitem{farb:survey}
Benson Farb.
\newblock Some problems on the mapping class groups and moduli space.
\newblock In {\em Problems on Mapping Class Groups and Related Topics}, Proc.
  Symp. Pure and Applied Math., pages 10--58. 2006.

\bibitem{fisher:survey}
David Fisher.
\newblock Groups acting on manifolds: around the Zimmer program.
\newblock preprint.

\bibitem{franks:poincare}
John Franks.
\newblock Generalizations of the {P}oincar\'e-{B}irkhoff theorem.
\newblock {\em Ann. of Math. (2)}, 128(1):139--151, 1988.

\bibitem{F-Poincare}
John Franks.
\newblock Generalizations of the {P}oincar\'e-{B}irkhoff theorem.
\newblock {\em Ann. of Math. (2)}, 128(1):139--151, 1988.

\bibitem{franks:recurrence}
John Franks.
\newblock Recurrence and fixed points of surface homeomorphisms.
\newblock {\em Ergodic Theory Dynam. Systems}, 8$^*$(Charles Conley Memorial
  Issue):99--107, 1988.

\bibitem{franks:open_annulus}
John Franks.
\newblock Area preserving homeomorphisms of open surfaces of genus zero.
\newblock {\em New York J. Math.}, 2:1--19, electronic, 1996.

\bibitem{franks:distortion}
John Franks.
\newblock Distortion in groups of circle and surface diffeomorphisms.
\newblock In {\em Dynamique des diff\'eomorphismes conservatifs des surfaces:
  un point de vue topologique}, volume~21 of {\em Panor. Synth\`eses}, pages
  35--52. Soc. Math. France, Paris, 2006.

\bibitem{fh:periodic}
John Franks and Michael Handel.
\newblock Periodic points of {H}amiltonian surface diffeomorphisms.
\newblock {\em Geom. Topol.}, 7:713--756 (electronic), 2003.

\bibitem{fh:distortion}
John Franks and Michael Handel.
\newblock Distortion elements in group actions on surfaces.
\newblock {\em Duke Math. J.}, 131(3):441--468, 2006.

\bibitem{fh:morita}
John Franks and Michael Handel.
\newblock Global fixed points for centralizers and {M}orita's theorem.
\newblock {\em Geom. Topol.}, 13(1):87--98, 2009.

\bibitem{han:pseudo-circle}
Michael Handel.
\newblock A pathological area preserving {$C^{\infty }$} diffeomorphism of the
  plane.
\newblock {\em Proc. Amer. Math. Soc.}, 86(1):163--168, 1982.

\bibitem{handel:rotation}
Michael Handel.
\newblock The rotation set of a homeomorphism of the annulus is closed.
\newblock {\em Comm. Math. Phys.}, 127(2):339--349, 1990.

\bibitem{handel:commuting}
Michael Handel.
\newblock Commuting homeomorphisms of {$S^2$}.
\newblock {\em Topology}, 31(2):293--303, 1992.

\bibitem{han:fpt}
Michael Handel.
\newblock A fixed-point theorem for planar homeomorphisms.
\newblock {\em Topology}, 38(2):235--264, 1999.

\bibitem{Hirsch}
Morris~W. Hirsch.
\newblock {\em Differential topology}, volume~33 of {\em Graduate Texts in
  Mathematics}.
\newblock Springer-Verlag, New York, 1994.
\newblock Corrected reprint of the 1976 original.

\bibitem{Iv2}
Nikolai~V. Ivanov.
\newblock Fifteen problems about the mapping class groups.
\newblock In {\em Problems on mapping class groups and related topics},
  volume~74 of {\em Proc. Sympos. Pure Math.}, pages 71--80. Amer. Math. Soc.,
  Providence, RI, 2006.

\bibitem{katok:horseshoe}
A.~Katok.
\newblock Lyapunov exponents, entropy and periodic orbits for diffeomorphisms.
\newblock {\em Inst. Hautes \'Etudes Sci. Publ. Math.}, (51):137--173, 1980.

\bibitem{kirby}
Rob Kirby.
\newblock Problems in low dimensional manifold theory.
\newblock In {\em Algebraic and geometric topology ({P}roc. {S}ympos. {P}ure
  {M}ath., {S}tanford {U}niv., {S}tanford, {C}alif., 1976), {P}art 2}, Proc.
  Sympos. Pure Math., XXXII, pages 273--312. Amer. Math. Soc., Providence,
  R.I., 1978.

\bibitem{kork2}
Mustafa Korkmaz.
\newblock Problems on homomorphisms of mapping class groups.
\newblock In {\em Problems on mapping class groups and related topics},
  volume~74 of {\em Proc. Sympos. Pure Math.}, pages 81--89. Amer. Math. Soc.,
  Providence, RI, 2006.

\bibitem{mather:prime_end}
John~N. Mather.
\newblock Topological proofs of some purely topological consequences of
  {C}arath\'eodory's theory of prime ends.
\newblock In {\em Selected studies: physics-astrophysics, mathematics, history
  of science}, pages 225--255. North-Holland, Amsterdam, 1982.

\bibitem{mcC}
John~D. McCarthy.
\newblock On the first cohomology group of cofinite subgroups in surface
  mapping class groups.
\newblock {\em Topology}, 40(2):401--418, 2001.

\bibitem{Po}
Leonid Polterovich.
\newblock Growth of maps, distortion in groups and symplectic geometry.
\newblock {\em Invent. Math.}, 150(3):655--686, 2002.

\bibitem{simon:index}
Carl~P. Simon.
\newblock A bound for the fixed-point index of an area-preserving map with
  applications to mechanics.
\newblock {\em Invent. Math.}, 26:187--200, 1974.

\bibitem{thurston:stability}
William~P. Thurston.
\newblock A generalization of the {R}eeb stability theorem.
\newblock {\em Topology}, 13:347--352, 1974.

\bibitem{yomdin}
Y.~Yomdin.
\newblock Volume growth and entropy.
\newblock {\em Israel J. Math.}, 57(3):285--300, 1987.

\end{thebibliography}

\end{document}